\documentclass[a4paper, 12pt]{amsart}
\usepackage{amsfonts, amsthm, amssymb, amsmath, stmaryrd}
\usepackage{mathrsfs,array}
\usepackage{xy}
\usepackage{verbatim}
\usepackage{amscd} 
\usepackage{tikz}
\usepackage{tikz-cd}
\usetikzlibrary{matrix,arrows,decorations.pathmorphing}
\usepackage{textcomp}
\usepackage{mathtools}
\input xy
\xyoption{all}
\usepackage[unicode]{hyperref}

\usepackage{etoolbox}
\makeatletter
\patchcmd{\@bibitem}{\ignorespaces}{\label{bib-#1}\ignorespaces}{}{}
\makeatother

\numberwithin{equation}{subsection}

\newtheorem{thm}{Theorem}[subsection]

\newtheorem{cor}[thm]{Corollary}

\newtheorem{lem}[thm]{Lemma} 
\newtheorem{prop}[thm]{Proposition}
 \theoremstyle{definition}
 \theoremstyle{definition}
\newtheorem{defn}[thm]{Definition} \theoremstyle{remark}
\newtheorem{rem}[thm]{\bf Remark}

\newtheorem{para}[thm]{\bf} 
 
\newtheorem{exa}[thm]{\bf Example}

\DeclareMathOperator{\Id}{Id}
\DeclareMathOperator{\Hom}{Hom}
\DeclareMathOperator{\End}{End}
\DeclareMathOperator{\Ext}{Ext}

\DeclareMathOperator{\im}{im}
\DeclareMathOperator{\coker}{coker}

\DeclareMathOperator{\Spa}{Spa}

\DeclareMathOperator{\Spf}{Spf}
\DeclareMathOperator{\Fil}{Fil}
\DeclareMathOperator{\gr}{gr}
\DeclareMathOperator{\Gal}{Gal}

\DeclareMathOperator{\Lie}{Lie}
\DeclareMathOperator{\Frob}{Frob}

\DeclareMathOperator{\Sym}{Sym}

\def\cont{\mathrm{cont}}
\def\cris{\mathrm{cris}}

\def\rig{\mathrm{rig}}
\def\dR{\mathrm{dR}}
\def\Sen{\mathrm{Sen}}
\def\inf{\mathrm{inf}}

\def\HT{\mathrm{HT}}
\def\GM{\mathrm{GM}}
\def\LT{\mathrm{LT}}
\def\Dr{\mathrm{Dr}}

\def\Ind{\mathrm{Ind}}

\def\sm{\mathrm{sm}}

\def\lalg{\mathrm{lalg}}
\def\ord{\mathrm{ord}}
\def\Ig{\mathrm{Ig}}

\def\an{\mathrm{an}}
\def\la{\mathrm{la}}
\def\proet{\mathrm{pro\acute{e}t}}
\def\proket{\mathrm{prok\acute{e}t}}
\def\et{\mathrm{\acute{e}t}}

\newcommand{\Z}{\mathbb{Z}}
\newcommand{\F}{\mathbb{F}}
\newcommand{\bH}{\mathbb{H}}
\newcommand{\Q}{\mathbb{Q}}
\newcommand{\R}{\mathbb{R}}
\newcommand{\T}{\mathbb{T}}
\newcommand{\A}{\mathbb{A}}
\newcommand{\bC}{\mathbb{C}}

\newcommand{\B}{\mathbb{B}}

\newcommand{\tr}{\mathrm{tr}}
\newcommand{\GL}{\mathrm{GL}}
\newcommand{\SL}{\mathrm{SL}}

\newcommand{\Fl}{{\mathscr{F}\!\ell}}

\newcommand{\cF}{\mathcal{F}}

\newcommand{\cO}{\mathcal{O}}

\newcommand{\kb}{\mathfrak{b}}
\newcommand{\kh}{\mathfrak{h}}
\newcommand{\km}{\mathfrak{m}}
\newcommand{\kn}{\mathfrak{n}}

\newcommand{\sC}{\mathscr{C}}

\newcommand{\RNum}[1]{\uppercase\expandafter{\romannumeral #1\relax}}

\newcommand{\overbar}[1]{\mkern 1.5mu\overline{\mkern-1.5mu#1\mkern-1.5mu}\mkern 1.5mu}

\begin{document}

\title[On locally analytic vectors of completed cohomology II]{On locally analytic vectors of the completed cohomology of modular curves II}
\author{Lue Pan}
\address{Department of Mathematics, Princeton University, Fine Hall, Washington Road
Princeton NJ 08544}
\email{lpan@princeton.edu}
\begin{abstract} 
This is a continuation of our previous work on the locally analytic vectors of the completed cohomology of modular curves. We construct differential operators on modular curves with infinite level at $p$ in both ``holomorphic'' and ``anti-holomorphic'' directions. As applications, we reprove a classicality result of Emerton which says that every absolutely irreducible two dimensional Galois representation which is regular de Rham at $p$ and appears in the completed cohomology of modular curves comes from an eigenform. Moreover we give a geometric description of the locally analytic representations of $\GL_2(\Q_p)$ attached to such a Galois representation in the completed cohomology.
\end{abstract}

\maketitle

\tableofcontents

\section{Introduction}
\subsection{Main results}
\begin{para}
This is a continuation of our previous work \cite{Pan20} on the locally analytic vectors of the completed cohomology of modular curves. Roughly speaking, in \cite{Pan20} we study the \textit{Hodge-Tate} structure of the completed cohomology. The main goal of this work is to study the \textit{de Rham} structure. 

We introduce some notations first. See \cite[\S 1]{Pan20} for more details. Fix a prime number $p$ throughout this paper.  Let $K^p=\prod_{l\neq p}K_l$ be an open compact subgroup of $\GL_2(\A_f^p)$. Set
\[\tilde{H}^i(K^p,\Z/p^n):=\varinjlim_{K_p\subseteq\GL_2(\Q_p)} H^i(Y_{K^pK_p}(\bC),\Z/p^n),\]
where $K_p$ runs through all open compact subgroups of $\GL_2(\Q_p)$, and $Y_{K^pK_p}$ denotes the modular curve of level $K^pK_p$ over $\Q$ whose $\bC$-points are given by the usual quotient $Y_{K^pK_p}(\bC)=\GL_2(\Q)\setminus (\bH^{\pm}\times\GL_2(\A_f)/K)$, and $\bH^{\pm}=\bC-\R$ denotes the union of the upper and lower half planes. The completed cohomology of tame level $K^p$, introduced by Emerton, is defined as
\[\tilde{H}^i(K^p,\Z_p):=\varprojlim_n \tilde{H}^i(K^p,\Z/p^n).\]
It is p-adically complete and equipped with a natural continuous action of $\GL_2(\Q_p)\times G_{\Q}$, where $G_{\Q}=\Gal(\overline\Q/\Q)$ denotes the absolute Galois group of $\Q$.  We are interested in the finite dimensional Galois representations of $G_{\Q}$ inside of $\tilde{H}^1(K^p,\Z_p)$. 
\end{para}

\begin{thm}[Theorem \ref{MTCH}] \label{MTI}
Let $E$ be a finite extension of $\Q_p$ and 
\[\rho:G_{\Q}\to\GL_2(E)\]
be a two-dimensional continuous absolutely irreducible representation of $G_{\Q}$. By abuse of notation, we  also use $\rho$ to denote its underlying representation space. Suppose that
\begin{enumerate}
\item $\rho$ appears in $\tilde{H}^1(K^p,E):=\tilde{H}^1(K^p,\Z_p)\otimes_{\Z_p} E$, i.e. $\Hom_{E[G_{\Q}]}(\rho,\tilde{H}^1(K^p,E))\neq 0$.
\item $\rho|_{G_{\Q_p}}$ is de Rham of Hodge-Tate weights $0,k$ for some integer $k>0$, where $G_{\Q_p}\subseteq G_\Q$ denotes a decomposition group at $p$.
\end{enumerate}
Then  $\rho$ arises from a cuspidal eigenform of weight $k+1$.
\end{thm}

\begin{rem}
This result was first obtained by Matthew Emerton in his famous work \cite{Eme1} (at least in the generic  cases) and was used by him in the same paper to attack the Fontaine-Mazur conjecture. To do this, Emerton  proved a local-global compatibility result of the completed cohomology which implies that  $\Hom_{E[G_{\Q}]}(\rho,\tilde{H}^1(K^p,E))$  corresponds to $\rho|_{G_{\Q_p}}$ under the $p$-adic local Langlands correspondence for $\GL_2(\Q_p)$ (up to some multiplicities). Then the claim follows from Colmez's result on the existence of locally algebraic vectors and the relationship between locally algebraic vectors in the completed cohomology and the cohomology of certain local systems on modular curves. 

We remark that our method is geometric and  uses the intertwining operator constructed in this paper as a key input. In particular, we do not need the $p$-adic local Langlands correspondence for $\GL_2(\Q_p)$ in our argument. It is conceivable that our method can be generalized in more general contexts and hopefully can shed some light on the conjectural $p$-adic local Langlands correspondence for groups beyond $\GL_2(\Q_p)$.
\end{rem}

\begin{rem}
$\rho|_{G_{\Q_l}}$ is unramified for all but finitely many $l$ and $\rho|_{G_\mathbb{R}}$ is odd since $\rho$ appears in $\tilde{H}^1(K^p,E)$. Hence this Theorem is a special case of the Fontaine-Mazur conjecture. In fact, in Emerton's approach to the Fontaine-Mazur conjecture \cite{Eme1}, he first showed that any two-dimensional $p$-adic geometric absolutely irreducible odd Galois representation of $G_{\Q}$ appears in $\tilde{H}^1(K^p,E)$ (under some generic assumptions), then he invoked this result to deduce that it actually comes from a  cuspidal eigenform.
\end{rem}

\begin{para}
Let $\rho$ be as in Theorem \ref{MTI}. 
\[\Pi_\rho:=\Hom_{E[G_{\Q}]}(\rho,\tilde{H}^1(K^p,E))\]
 is an $E$-Banach space representation of $\GL_2(\Q_p)$. Our method also gives a geometric description of the $\GL_2(\Q_p)$-locally analytic vectors of $\Pi_\rho$. For the purpose of introduction, we will restrict ourselves to the case when
 \begin{itemize}
 \item $k=1$, i.e. $\rho$ corresponds an eigenform of weight $2$. 
  \end{itemize}
See Subsection \ref{rbpLLC} for results when $k\geq 2$.

Fix an embedding of $\tau:E\to C$, the $p$-adic completion of $\overline\Q_p$. By Theorem \ref{MTI}, $\rho$ corresponds to an irreducible representation $\pi^\infty=\otimes'_l \pi_l$ of $\GL_2(\A_f)$ over $C$, which is isomorphic to the finite part of an automorphic representation of $\GL_2(\A)$. See Section \ref{noraf} below for the precise definition of $\pi^\infty$. 
\end{para}

\begin{para}
We first explain the case when $\pi_p$ is supercuspidal. By the Jacquet-Langlands correspondence, $\pi^\infty$ transfers to an irreducible representation $\pi'^\infty=\otimes'_l \pi'_l$ of $(D\otimes_\Q \A_f)^\times$ where $D$ denotes the quaternion algebra over $\Q$ only ramified at $p,\infty$. In particular, $\pi'_p$ is an irreducible representation of $D_p^\times$.

Not surprisingly, in this case, $\Pi_\rho$ is closely related to the coverings of the Drinfeld upper half plane. Let $\Omega:=\mathbb{P}^1\setminus \mathbb{P}^1(\Q_p)$ be the Drinfeld upper half plane, viewed as an adic space over $\Spa(C,\cO_C)$. There is a natural action of $\GL_2(\Q_p)$ on $\Omega$. As explained by Drinfeld,  $\Omega$ has a sequence of finite \'etale Galois coverings $\{\mathcal{M}_{\Dr,n}^{(0)}\}_{n\geq 0}$ with Galois group $(\cO_{D_p}/p^n\cO_{D_p})^\times$. Let $\GL_2(\Q_p)^o$ denote the subgroup of $\GL_2(\Q_p)$ with determinants in $\Z_p^\times$. There are natural actions of $\GL_2(\Q_p)^o$ on $\mathcal{M}_{\Dr,n}^{(0)}$ so that the projection maps $\pi_{\Dr,n}^{(0)}:\mathcal{M}_{\Dr,n}^{(0)}\to \Omega$ are $\GL_2(\Q_p)^o$-equivariant. We denote by $j:\Omega\to\mathbb{P}^1$ the natural inclusion.
\end{para}

\begin{thm}[Theorem \ref{scch}] \label{scch1}
Suppose $k=1$ and $\pi_p$ is supercuspidal. 
There are a natural embedding 
\[i_\pi:\pi_p\to \left(\varinjlim_n  H^1\left(\mathbb{P}^1,j_!\pi^{(0)}_{\Dr,n*} \cO_{\mathcal{M}_{\Dr,n}^{(0)}}\right)\otimes_C \pi'_p\right)^{\cO_{D_p}^\times}\]
coming from the position of the Hodge filtration of $\rho|_{G_{\Q_p}}$ and a $\GL_2(\Q_p)^o$-equivariant isomorphism
\[\Pi_\rho^{\la}\widehat\otimes_{E,\tau} C\cong(\pi^{\infty,p})^{K^p}\otimes_C \left[\left(\varinjlim_n  H^1\left(\mathbb{P}^1,j_!\pi^{(0)}_{\Dr,n*} \cO_{\mathcal{M}_{\Dr,n}^{(0)}}\right)\otimes_C \pi'_p\right)^{\cO_{D_p}^\times}/i_\pi(\pi_p)\right]\]
where $\Pi_\rho^{\la}\subseteq\Pi_\rho$ is the subspace of $\GL_2(\Q_p)$-locally analytic vectors, and $j_!$ denotes the usual extension by zero, and the cohomology is computed on the analytic site of $\mathbb{P}^1$. This isomorphism can be upgraded to a $\GL_2(\Q_p)$-equivariant isomorphism, cf. Remark \ref{GL2oGL2}.
\end{thm}

\begin{rem}
$H^1\left(\mathbb{P}^1,j_!\pi^{(0)}_{\Dr,n*} \cO_{\mathcal{M}_{\Dr,n}^{(0)}}\right)$ is isomorphic to the compactly supported cohomology of $\cO_{\mathcal{M}_{\Dr,n}^{(0)}}$ on $\mathcal{M}_{\Dr,n}^{(0)}$.
 Hence this result essentially verifies a conjecture of Breuil-Strauch on $\Pi_\rho^{\la}$ in the dual form, which was completely solved by Dospinescu-Le Bras \cite{DLB17} before. See Remark \ref{scla} for more details. We emphasize again that our method does not rely on any known results from the $p$-adic local Langlands correspondence for $\GL_2(\Q_p)$.
\end{rem}

\begin{para}
Next assume that $\pi_p$ is a principal series. In this case we need the rigid cohomology of Igusa curves.
Let $\Gamma(p^n)=1+p^n\Z_p\subseteq \GL_2(\Z_p)$ be the principal congruence subgroup of level $n\geq 1$. The modular curve $Y_{K^p\Gamma(p^n)}$ has a natural compactification $X_{K^p\Gamma(p^n)}$. As explained by Katz-Mazur \cite[Theorem 13.7.6]{KM85}, $X_{K^p\Gamma(p^n)}\times_{\Q_p} C$ has a natural integral model $\mathfrak{X}_{K^p\Gamma(p^n)}$ over $\cO_C$ whose irreducible components of its special fiber are indexed by surjective homomorphisms $(\Z/p^n)^2\to \Z/p^n$ and meet at supersingular points. Let $\mathfrak{X}_{K^p\Gamma(p^n),\bar\F_p,c}\subseteq \mathfrak{X}_{K^p\Gamma(p^n),\bar\F_p}$ denote the open subset of the non-supersingular points of irreducible components with indices sending  $(1,0)\in (\Z/p^n)^2$ to $0$. It is an affine variety over $\bar\F_p$. In more classical languages, $\mathfrak{X}_{K^p\Gamma(p^n),\bar\F_p,c}$ is a union of Igusa curves of level $n$ (and tame level $K^p$).

Let $H^1_{\rig}(\Ig(K^p,n)):=H^1_{\rig}(\mathfrak{X}_{K^p\Gamma(p^n),\overline\F_p,c}/W(\bar\F_p)[\frac{1}{p}])\otimes_{W(\bar\F_p)[\frac{1}{p}]} C$ be the $C$-coefficients rigid cohomology of $\mathfrak{X}_{K^p\Gamma(p^n),\bar\F_p,c}$. It forms a direct system when $n$ varies. Let 
\[H^1_{\rig}(\Ig(K^p))=\varinjlim_n H^1_{\rig}(\Ig(K^p,n)).\]
It follows from the construction of $\mathfrak{X}_{K^p\Gamma(p^n),\bar\F_p,c}$ that the upper-triangular Borel subgroup $B\subseteq\GL_2(\Q_p)$ acts on $H^1_{\rig}(\Ig(K^p))$. The $B$-action is smooth, hence locally analytic.

On the other hand, fix a finite set of rational primes $S$ such that $K_l\cong\GL_2(\Z_l),l\notin S$. The spherical Hecke algebra $\T^S$ (cf. Subsection \ref{Hd} below) acts naturally on $H^1_{\rig}(\Ig(K^p))$. By Theorem \ref{MTI}, the Galois representation $\rho$ and the chosen embedding $\tau$ induces a homomorphism $\lambda_\tau:\T^S\to C$. See Subsection \ref{noraf} for the normalization.
\end{para}

\begin{thm}[Theorem \ref{PSla}] \label{PSla1}
Suppose $k=1$ and $\pi_p$ is a principal series.
There are a natural embedding
\[i_\pi:(\pi^{\infty})^{K^p}\to \Ind_B^{\GL_2(\Q_p)} H^1_{\rig}(\Ig(K^p)) [\lambda_\tau]\]
coming from the position of the Hodge filtration of $\rho|_{G_{\Q_p}}$ 
and a natural $\GL_2(\Q_p)$-equivariant isomorphism
\[\Pi_\rho^{\la}\widehat\otimes_{E,\tau} C\cong
\Ind_B^{\GL_2(\Q_p)} H^1_{\rig}(\Ig(K^p)) [\lambda_\tau]/i_\pi\left((\pi^{\infty})^{K^p}\right),\]
where $\Ind_B^{\GL_2(\Q_p)}$ denotes the locally analytic induction and $[\lambda_\tau]$ denotes the $\lambda_\tau$-isotypic part.
\end{thm}

\begin{rem}
There is an isomorphism $H^1_{\rig}(\Ig(K^p)) [\lambda_\tau]\cong (\pi^{\infty,p})^{K^p}\otimes_E D_{\cris}(\rho|_{G_{\Q_p}})$, which is $B$-equivariant if we equip the right hand side with the correct $B$-action using the congruence relation. See Remark \ref{psla} for more details. Hence
\[\Pi_\rho^{\la}\widehat\otimes_{E,\tau} C\cong (\pi^{\infty,p})^{K^p}\otimes \Ind_B^{\GL_2(\Q_p)} D_{\mathrm{cris}}(\rho|_{G_{\Q_p}})/\pi_p\]
and the embedding of $\pi_p\subseteq \Ind_B^{\GL_2(\Q_p)} D_{\mathrm{cris}}(\rho|_{G_{\Q_p}})$ comes from 
\[\Fil^1D_{\dR}(\rho|_{G_{\Q_p}})\subseteq D_\dR(\rho|_{G_{\Q_p}})\cong D_{\mathrm{cris}}(\rho|_{G_{\Q_p}}).\]
Such an isomorphism was previously known by Emerton's local-global compatibility result and the description of the local analytic representation of $\GL_2(\Q_p)$ associated to $\rho|_{G_{\Q_p}}$ under the $p$-adic local Langlands correspondence for $\GL_2(\Q_p)$. In this case, this was essentially conjectured by Berger-Breuil  and Emerton, and was proved by Liu-Xie-Zhang  and Colmez.
\end{rem}

\begin{rem}
We have a slightly weaker result when $\pi_p$ is special. See Remark \ref{scpsConj}.
\end{rem}

\begin{rem} \label{conunif}
In both Theorems \ref{scch1} and \ref{PSla1}, $\Pi_\rho^{\la}\widehat\otimes_{E,\tau} C$ is written as a quotient by $(\pi^\infty)^{K^p}$. This is not a coincidence. In fact, we will show that  there is a natural isomorphism (cf.  Remark \ref{DRstn}, Theorem \ref{HT0kerF})
\[\Pi_\rho^{\la}\widehat\otimes_{E,\tau} C\cong \bH^1(DR_{k-1})[\lambda]/\Fil^1 \bH^1(DR_{k-1})[\lambda]\]
for some two-terms complex $DR_{k-1}$ on $\mathbb{P}^1$, where $\Fil^\bullet$ denotes the Hodge filtration on $\bH^1(DR_{k-1})$. Let $\mathcal{X}_{K^pK_p}$ denote the adic space associated to $X_{K^pK_p}\times C$. When $k=1$,  the de Rham cohomology $H^1_{\dR}(\mathcal{X}_{K^pK_p})$ naturally embeds into $ \bH^1(DR_{0})$ and induces an isomorphism
\[\Fil^1 \bH^1(DR_0)[\lambda]\cong \Fil^1 \varinjlim_{K_p}H^1_{\dR}(\mathcal{X}_{K^pK_p})[\lambda] \cong (\pi^\infty)^{K^p}.\]
This shows that the Breuil-Strauch conjecture in the supercuspidal case and the conjecture of Berger-Breuil, Emerton in the principal series case can be formulated in a uniform way.  It is also interesting to point out that the cohomology group $\bH^1(DR_{k-1})[\lambda]$ itself does not see the information of Hodge filtration. Intuitively, $DR_0$ is some completed tensor product of the de Rham complex of $\mathcal{X}_{K^pK_p}$ and  the structure sheaf $\cO_{\mathbb{P}^1}$.
 See Remark \ref{KSreal} below for more details regarding $\bH^1(DR_{k-1})$ and its analogy with geometric constructions of representations of real groups, e.g. work of Kashiwara-Schmid in \cite{KS94}.
\end{rem}

Our method also yields a finiteness result. Let $\kn\subseteq\Lie(\GL_2(\Q_p))$ denote the upper-triangular nilpotent subalgebra. The following result was also obtained by Dospinescu-Pa\v{s}k\=unas-Schraen \cite{DPS22} very recently by a totally different method, cf. Remark \ref{relDPS}.

\begin{thm} [part (2) of Theorem \ref{MTCH}] \label{Intfin}
Let $\rho$ be as in Theorem \ref{MTI}. Then $ \Pi_\rho^{\Gamma(p^n),\kn}$ is finite-dimensional for any $n\geq 2$,
where $\Pi_\rho^{\Gamma(p^n),\kn}$ denotes the $\kn$-invariants of the $\Gamma(p^n)$-analytic vectors of $\Pi_\rho$.
\end{thm}

\subsection{Strategy}

\begin{para}
Now we explain our strategy for proving Theorem \ref{MTI}. Fix $k\geq 1$ throughout this subsection. The starting point is a characterization of de Rham Galois representations in Fontaine's classification of almost $B_{\dR}$-representations \cite{Fo04}. More precisely, let $V$ be a two-dimensional continuous representation of $G_{\Q_p}$ over $\Q_p$. Suppose that $V$ is Hodge-Tate of weights $0,k$, i.e. there is a natural decomposition
\[W:=V\otimes_{\Q_p} C=W_0\oplus W_k\]
with both $W_0$ and $W_k$ being non-zero, and  $W_i(i)=W_i(i)^{G_{\Q_p}}\otimes_{\Q_p} C$ for $i=0,k$. Fontaine proved that there is a natural $C$-linear, $G_{\Q_p}$-equivariant operator 
\[N:W_0\to W_k(k)\]
such that $V$ is de Rham exactly when $N=0$. For reader's convenience, we briefly recall Fontaine's construction here. Let $B_{\dR}$ denote the usual de Rham period ring and $H_{\Q_p}=\ker \varepsilon_p$, the kernel of the $p$-adic cyclotomic character $\varepsilon_p:G_{\Q_p}\to\Z_p^\times$. Let $\Gamma=G_{\Q_p}/H_{\Q_p}\cong\Z_p^\times$. Consider 
\[D_{\mathrm{pdR}}(V):= (V\otimes_{\Q_p}B_{\dR})^{H_{\Q_p},\Gamma-unipotent},\]
the subspace of $H_{\Q_p}$-invariants on which the action of $\Gamma$ is unipotent. The $t$-adic filtration on $B_{\dR}$ naturally defines a decreasing filtration on $D_{\mathrm{pdR}}(V)$.
Generalizing Sen's work on the decompletion of $C$-representations, Fontaine showed that $\dim_{\Q_p} D_{\mathrm{pdR}}(V)=\dim_{\Q_p} V$ and
\[\gr^i D_{\mathrm{pdR}}(V)\otimes_{\Q_p} C=W_i(i)\]
where $W_i(i)=0$ if $i\neq 0,k$. The action of $1\in\Q_p\cong\Lie\Gamma$ on $D_{\mathrm{pdR}}(V)$ is nilpotent and preserves the filtration, hence induces a map $\gr^0 D_{\mathrm{pdR}}(V)\to \gr ^k D_{\mathrm{pdR}}(V)$ whose tensor product with $C$ gives $N$.  We will call $N$ the Fontaine operator.
\end{para}

\begin{para} \label{IntFO}
Back to the completed cohomology.  For simplicity, we will assume $E=\Q_p$ in the introduction.
Since the completed cohomology is an admissible Banach space representation of $\GL_2(\Q_p)$, it follows that $\Pi_\rho$ is also admissible. Hence $\Pi^\la_\rho\neq 0$ by Schneider-Teitelbaum.
Moreover by   \cite[Proposition 6.1.5]{Pan20} (see also \cite{DPS20}), $\Pi^\la_\rho$ has an infinitesimal character $\tilde\chi_k:Z(U(\mathfrak{gl}_2(\Q_p)))\to\Q_p$ which is equal to the infinitesimal character of the $(k-1)$-th symmetric power of the dual of the standard representation. In fact, it was shown in \cite{Pan20} that  the $\tilde\chi_k$-isotypic part $\tilde{H}^1(K^p,\Q_p)^{\la,\tilde\chi_k}$ is Hodge-Tate of weights $0,k$, i.e 
there is a natural Hodge-Tate decomposition
\[\tilde{H}^1(K^p,\Q_p)^{\la,\tilde\chi_k}\widehat\otimes_{\Q_p} C=W_0\oplus W_k\]
such that $W_i(i)=W_i(i)^{G_{\Q_p}}\widehat\otimes_{\Q_p} C$, $i=0,k$. Fontaine's construction can also be generalized to this setting: we will show that there exists a natural continuous map
\[N:W_0\to W_k(k)\]
such that for any  two-dimensional irreducible sub-representation $V\subseteq \tilde{H}^1(K^p,\Q_p)^{\la,\tilde\chi_k}$ of $G_{\Q}$, the restriction $N|_{(V\otimes_{\Q_p}C)_0}$ agrees with the Fontaine operator $(V\otimes_{\Q_p}C)_0\to (V\otimes_{\Q_p}C)_k(k)$ for $V$. Therefore the following spectral decomposition theorem implies Theorem \ref{MTI} in view of Fontaine's result. 
\end{para}

\begin{thm} [Theorem \ref{sd}, Corollary \ref{sd'}]
There is natural generalized eigenspace decomposition with respect to the action of $\T^S$
\[\ker N=\bigoplus_{\lambda} \ker N\widetilde{[\lambda]}\]
where $\lambda:\T^S\to C$ runs over all systems of eigenvalues associated to eigenforms of weight $k+1$ and eigenvalues appearing in $\displaystyle \varinjlim_{K_p\subseteq\GL_2(\Q_p)} H^0(Y(K^pK_p)(\bC),\Q_p)$, and $\widetilde{[\lambda]}$ denotes the generalized eigenspace of $\lambda$.
\end{thm}

\begin{rem}
In fact, we will show that $\ker N\widetilde{[\lambda]}=\ker N{[\lambda]}$, the eigenspace of $\lambda$, if the associated automorphic representation is cuspidal and is not special at $p$. Using results from $p$-adic local Langlands correspondence for $\GL_2(\Q_p)$, we can even show that this is true in the Steinberg case. See Remark \ref{stn}.
\end{rem}

\begin{para}
It remains to prove this spectral decomposition. Here comes the main innovation of this paper: we will give another construction of the Fontaine operator $N$ using the $p$-adic geometry of the perfectoid modular curves. We need some construction in our previous work \cite{Pan20}. Let $\mathcal{X}_{K^p}\sim\varprojlim_{K_p\subset\GL_2(\Q_p)}\mathcal{X}_{K^pK_p}$ denote the modular curve of infinite level at $p$ introduced by Scholze in \cite{Sch15}. There is a $\GL_2(\Q_p)$-equivariant Hodge-Tate period map
\[\pi_{\HT}:\mathcal{X}_{K^p}\to \Fl\cong \mathbb{P}^1\]
where $\Fl$ denotes the flag variety of $\GL_2$, viewed as an adic space over $\Spa(C,\cO_C)$. Let $\cO_{K^p}:=\pi_{\HT*}\cO_{\mathcal{X}_{K^p}}$ and $\cO_{K^p}^\la\subseteq \cO_{K^p}$ be the subsheaf of $\GL_2(\Q_p)$-locally analytic sections. In \cite{Pan20}, using the relative Sen theory, we showed that $\cO^{\la}_{K^p} $ is annihilated by the horizontal nilpotent subalgebra $\mathfrak{n}^o$ on $\Fl$. Fix a Cartan subalgebra  $\mathfrak{h}:=\{\begin{pmatrix} * & 0 \\ 0 & * \end{pmatrix}\}\subseteq \{\begin{pmatrix} * & * \\ 0 & * \end{pmatrix}\}\subseteq\mathfrak{gl}_2(\Q_p)$. Hence there is a natural horizontal Cartan action $\theta_\kh$ of  $\mathfrak{h}$ on   $\cO^{\la}_{K^p} $  induced from $\kh\hookrightarrow \mathfrak{b}^o/\mathfrak{n}^o$, where $\kb^o$ denotes the horizontal Borel. Given a weight $\chi=(n_1,n_2)\in\Q_p^2:\kh\to \Q_p$, we will denote by $\cO^{\la,\chi}_{K^p}$ the $\chi$-isotypic part of $\cO^{\la}_{K^p}$ with respect to $\theta_\kh$.

It was shown in \cite{Pan20} that there are natural isomorphisms
\begin{itemize}
\item $W_0\cong H^1(\Fl,\cO^{\la,(1-k,0)}_{K^p})$ when $k\geq 2$;
\item $W_k\cong H^1(\Fl,\cO^{\la,(1,-k)}_{K^p})$.
\end{itemize}
When $k=1$, there is a natural map $H^1(\Fl,\cO^{\la,(1-k,0)}_{K^p})\to W_0$  and the difference between $W_0$ and  $H^1(\Fl,\cO^{\la,(1-k,0)}_{K^p})$ comes from $\tilde{H}^0(K^p,\Z_p)$. In particular,  $W_0[\lambda_\tau]\cong H^1(\Fl,\cO^{\la,(0,0)}_{K^p})[\lambda_\tau]$ still holds.  We recommend the reader to ignore this difference in the following discussion.
\end{para}

\begin{para}
Let $\cO^{\sm}_{K^p}\subseteq\cO_{K^p}$ be the subsheaf of $\GL_2(\Q_p)$-smooth sections. Equivalently, 
\[\displaystyle \cO^{\sm}_{K^p}=\varinjlim_{K_p\subseteq\GL_2(\Q_p)} \pi_{\HT*}\pi_{K_p}^{-1}\cO_{\mathcal{X}_{K^pK_p}},\]
where $\pi_{K_p}:\mathcal{X}_{K^p}\to \mathcal{X}_{K^pK_p}$ denotes the natural projection map. Similarly, set
\[\displaystyle \Omega^{1,\sm}_{K^p}(\mathcal{C})=\varinjlim_{K_p\subseteq\GL_2(\Q_p)} \pi_{\HT*}\pi_{K_p}^{-1}\Omega^1_{\mathcal{X}_{K^pK_p}}(\mathcal{C}_{K_p}),\]
where $\mathcal{C}_{K_p}$ denotes the cusps of $\mathcal{X}_{K^pK_p}$. Clearly $\Omega^{1,\sm}_{K^p}(\mathcal{C})$ is an $\cO^\sm_{K^p}$-module and there is a natural derivation $d:\cO^{\sm}_{K^p}\to \Omega^{1,\sm}_{K^p}(\mathcal{C})$. 

We note that there are natural inclusions $\cO^{\sm}_{K^p},\cO_{\Fl} \subseteq \cO^{\la,(0,0)}_{K^p}$. 
In Section \ref{Ioef} below, we will prove the following result.
\end{para}

\begin{thm}
The natural map $\cO^{\sm}_{K^p}\otimes_{C} \cO_{\Fl} \to\cO^{\la,(0,0)}_{K^p}$ has dense images. Moreover,
\begin{enumerate}
\item There exists a (necessarily unique) continuous map $d^1: \cO^{\la,(0,0)}_{K^p}\to \cO^{\la,(0,0)}_{K^p}\otimes_{\cO^{\sm}_{K^p}}\Omega^{1,\sm}_{K^p}(\mathcal{C})$ which is $\cO_{\Fl}$-linear and extends $d:\cO^{\sm}_{K^p}\to \Omega^{1,\sm}_{K^p}(\mathcal{C})$;
\item There exists a (necessarily unique) continuous map $\bar{d}^1: \cO^{\la,(0,0)}_{K^p}\to \cO^{\la,(0,0)}_{K^p}\otimes_{\cO_\Fl}\Omega^1_{\Fl}$ which is $\cO^{\sm}_{K^p}$-linear and extends $d_\Fl:\cO_{\Fl}\to \Omega^{1}_\Fl$.
\end{enumerate}
\end{thm}

\begin{rem}
$\bar{d}^1$ is essentially the pull-back of derivations on $\Fl$ along the Hodge-Tate period map $\pi_{\HT}$. As explained by Scholze in \cite{Sch16}, $\pi_{\HT}$ can be regarded as a $p$-adic analogue of the \textit{anti-holomorphic} Borel embedding. This explains the notation $\bar{d}$ here.
\end{rem}

\begin{para}
We define $I_0$ as the composite map
\[\cO^{\la,(0,0)}_{K^p} \xrightarrow{d^1} \cO^{\la,(0,0)}_{K^p}\otimes_{\cO^{\sm}_{K^p}}\Omega^{1,\sm}_{K^p}(\mathcal{C})\xrightarrow{\bar{d}^1\otimes 1} \Omega^1_{\Fl}\otimes_{\cO_{\Fl}}\cO^{\la,(0,0)}_{K^p}\otimes_{\cO^{\sm}_{K^p}}\Omega^{1,\sm}_{K^p}(\mathcal{C}).\]
There is a natural isomorphism $\Omega^1_{\Fl}\otimes_{\cO_{\Fl}}\cO^{\la,(0,0)}_{K^p}\otimes_{\cO^{\sm}_{K^p}}\Omega^{1,\sm}_{K^p}(\mathcal{C})\cong \cO^{\la,(1,-1)}_{K^p}(1)$ coming from the Kodaira-Spencer isomorphisms: \begin{itemize}
\item $\Omega^1_{\mathcal{X}_{K^pK_p}}(\mathcal{C}_{K_p})\cong \omega^2$ on $\mathcal{X}_{K^pK_p}$, where $\omega$ denotes the usual automorphic line bundle;
\item $\Omega^1_{\Fl}\cong \omega_{\Fl}^{-2}$ on $\Fl$, where $\omega_\Fl$ denotes the tautological ample line bundle on $\Fl=\mathbb{P}^1$.
\end{itemize}

It follows from the construction of $\pi_\HT$ that the pull-backs of $\omega$ and $\omega_{\Fl}$ to $\mathcal{X}_{K^p}$ are isomorphic to each other (up to a Tate twist), which gives the desired isomorphism. Hence $I_0=\bar{d}'^1\circ d^1$ can be viewed as a map $\cO^{\la,(0,0)}_{K^p} \to \cO^{\la,(1,-1)}_{K^p}(1)$, where $\bar{d}'^1=\bar{d}^1\otimes 1$. In general, a standard BGG construction allows us to define $I_{k-1}=\bar{d}'^{k+1}\circ d^{k+1}: \cO^{\la,(1-k,0)}_{K^p} \to \cO^{\la,(1,-k)}_{K^p}(k)$. We denote $H^1(I_{k-1})$ by $I^1_{k-1}$.

Recall that  in \ref{IntFO}, we have the Fontaine operator $N: W_0\cong H^1(\Fl,\cO^{\la,(1-k,0)}_{K^p} ) \to W_k(k)\cong H^1(\Fl,\cO^{\la,(1,-k)}_{K^p}(k))$ when $k\geq 2$. Here is the main result of Section \ref{IopHti}.
\end{para}

\begin{thm}
$I^1_{k-1}\cong c_k N$ when $k\geq 2$ and $I^1_0|_{H^1(\Fl,\cO^{\la,(0,0)}_{K^p})[\lambda_\tau]}\cong c_1 N|_{W_0[\lambda_\tau]}$ when $k=1$ for some $c_k\in\Q_p^\times$.
\end{thm}

\begin{para} \label{kerdinf}
It remains to prove that $\ker I^1_{k-1}$ has a spectral decomposition. In fact, we will show such a decomposition on the sheaf level. For simplicity we assume $k=1$ here. The following observations for $\bar{d}^1:\cO^{\la,(0,0)}_{K^p}\to \cO^{\la,(0,0)}_{K^p}\otimes_{\cO_\Fl}\Omega^1_\Fl$ are very useful:

\begin{itemize}
\item $\bar{d}^1$ is surjective;
\item $\ker \bar{d}^1=\cO^{\sm}_{K^p}$.
\end{itemize}
We give some very informal explanations here. We view $\cO^{\la,(0,0)}_{K^p}\xrightarrow{\bar{d}^1} \cO^{\la,(0,0)}_{K^p}\otimes_{\cO_\Fl}\Omega^1_\Fl$ as a ``relative de Rham complex of $\mathcal{X}_{K^p}$ over $\mathcal{X}_{K^pK_p}$''. Each fiber of the projection map $\mathcal{X}_{K^p}\to\mathcal{X}_{K^pK_p}$ is a profinite set, i.e. zero-dimensional, hence $\cO^{\la,(0,0)}_{K^p}\xrightarrow{\bar{d}^1} \cO^{\la,(0,0)}_{K^p}\otimes_{\cO_\Fl}\Omega^1_\Fl$  has no $H^1$, which implies the first claim. For the second claim, by Beilinson-Bernstein's theory of localization, the derivation $\bar{d}^1$ on $\Fl$ essentially comes from the action of $\mathfrak{gl}_2(\Q_p)$. Since $\ker{\bar{d}^1}$ is $\GL_2(\Q_p)$-invariant,  the first claim follows.

From this, it is easy to see that $H^1(\bar{d}'^1)$ is injective. Therefore $\ker I^1_0=\ker H^1(d^1)$. Consider the following de Rham complex
\[DR_0: \cO^{\la,(0,0)}_{K^p} \xrightarrow{d^1} \cO^{\la,(0,0)}_{K^p}\otimes_{\cO^{\sm}_{K^p}}\Omega^{1,\sm}_{K^p}(\mathcal{C}).\]
Then $\ker I^1_0=\ker H^1(d^1)=\bH^1(DR_0)/\Fil^1\bH^1(DR_0)$, where $\Fil^\bullet$ denotes the Hodge filtration.
This in particular explains Remark \ref{conunif}.
\end{para}

\begin{para}
Now it is enough to show that $H^\bullet(DR_0)$ has a spectral decomposition. Roughly speaking, for $y\in\Fl$, $H^\bullet(DR_0)_y$ computes the ``de Rham cohomology of the fiber  $\pi_\HT^{-1}(y)$''. Hence we need to stratify $\Fl=\mathbb{P}^1=\Omega\bigsqcup \mathbb{P}^1(\Q_p)$ according to the fibers.

\begin{itemize}
\item On $\Omega$, using the uniformization of supersingular locus in terms of the Lubin-Tate space, we get a natural  spectral decomposition of $H^\bullet(DR_0)|_{\Omega}$ with eigenvalues corresponding to automorphic representations on $(D\otimes_\Q \A)^\times$. Note that each fiber is a profinite set in this case, hence $H^1(DR_0)|_{\Omega}=0$.
\item On $ \mathbb{P}^1(\Q_p)$, the fiber is closely related to the Igusa curves, and we have a natural spectral  decomposition of $H^\bullet(DR_0)|_{\mathbb{P}^1(\Q_p)}$ with eigenvalues appearing in the rigid cohomology $H^1_{\rig}(\Ig(K^p)) $ of Igusa curves.
\end{itemize}

For a $C$-point $y\in\Omega$, there is an isomorphism $\pi_{\HT}^{-1}(y)\cong D^\times\setminus (D\otimes_\Q \A_f)^\times/K^p$. The key point is to show that any element in $H^0(DR_0)_y=(\ker d^1)_y$, viewed as a continuous function on $D^\times\setminus (D\otimes_\Q \A_f)^\times/K^p$, is an automorphic function, i.e. is $D_p^\times$-smooth. To see this, we investigate the duality isomorphism between the Lubin-Tate towers and Drinfeld towers at the infinite level in Subsection \ref{Drinfty} and show that the roles of $\bar{d}^1$ and $d^1$ are swapped under this isomorphism. In particular, that $\ker d^1$ consists of $D_p^\times$-smooth elements is equivalent to that $\ker\bar{d}^1$ consists of $\GL_2(\Q_p)$-smooth elements, which we have already seen in \ref{kerdinf}!

For $y\in  \mathbb{P}^1(\Q_p)$, the fiber $\pi_{\HT}^{-1}(y)$ is some perfect  Igusa curve. The main point here is to show that $H^\bullet(DR_0)_y$ computes the rigid cohomology of the classical, \textit{non-perfect} Igusa curves. I believe this is a consequence of our differential equation \cite[Lemma 4.3.3]{Pan20}.
\end{para}

\subsection{Organization of the paper}
\begin{para}
This paper is organized as follows. The next two sections are preliminary and can be skipped on first reading (although some notations will be introduced in Section \ref{Rf20}). In Section \ref{Lav}, we collect some basic facts about locally analytic representations over a Hausdorff LB-space. In Section \ref{Rf20}, 
we recall some results and constructions from  our previous work \cite{Pan20} and make some slight generalizations.
In Section \ref{Ioef}, we will prove the existence of the differential operators $d^1$ and $\bar{d}^1$ and study the $\mathfrak{n}$-invariants of $\ker H^1(d^{k})$.  The spectral decomposition will be proved in Section \ref{Iosd}. As mentioned above, we will also study a ``local picture": the Lubin-Tate towers and Drinfeld towers. In Section \ref{IopHti}, we prove that the Fontaine operator  has the simple formula $H^1(\bar{d}'^k\circ d^k)$ (up to a non-zero scalar). Not surprisingly, one main ingredient is the period sheaf $\cO\B_\dR^+$. Finally, we collect everything and prove the main results in Section \ref{mrs}.

Although this paper is quite long, I hope it is clear to the reader that at least the main idea is very simple.
\end{para}

\begin{para}
In this work, we also make  improvements in some results obtained in \cite{Pan20}, which may simplify and clarify many arguments in our previous work.
\begin{enumerate}
\item In \cite[\S 2.2]{Pan20}, we introduced two notions: $\mathfrak{LA}$-acyclic and strongly $\mathfrak{LA}$-acyclic. In Proposition \ref{LAasLAa}, we will show that both notions are equivalent. 
\item In \cite[Theorem 4.2.7]{Pan20}, we calculated the differential equation that $\cO^{\la}_{K^p}$ satisfies using a result of Faltings. We will give a more conceptual explanation in the proof of Theorem \ref{LTde} below by exploring the relationship between the Higgs bundle and the variation of Hodge structure attached to a de Rham local system. Roughly speaking, we will show that the differential equation  is essentially given by the Kodaira-Spencer class. This point of view is more flexible and works for general Shimura varieties.
\item In our previous work, the inclusion map in the relative Hodge-Tate sequence (See sequence \eqref{rHT} below) is not Hecke-equivariant. More precisely, we implicitly fix a trivialization of the $H^2$ of the relative de Rham cohomology of the universal elliptic curves. We will make everything independent of choices in Section \ref{HEt} below and explain how our choice of the trivialization affects the Hecke action. In particular, the Hodge-Tate sequence should be sequence \ref{rHTH}.
\item In \cite[Proof of Theorem 5.1.11]{Pan20}, we had to take the largest separated quotient for some map in a \v{C}ech complex. In Subsection \ref{Cech} below, we will use the Weyl group to show that all the \v{C}ech complexes considered are strict, hence there is no need taking the separated quotient. We also explain how to carry out a Cartan-Serre type of argument in this context, cf. Proof of Lemma \ref{Lemcom} (1).
\end{enumerate}

\end{para}

\begin{para}
After finishing this work, Juan Esteban Rodriguez Camargo explained to me a more direct and conceptual construction of  the differential operators $d^1$ and $\bar{d}^1$ using $\B_{\dR}^+/t^2$. It would be very interesting to simplify arguments in Section \ref{IopHti} with his construction.
\end{para}

\subsection*{Acknowledgement} 
I would like to thank Yiwen Ding, Matthew Emerton, Arthur-C\'esar Le-Bras, Juan Esteban Rodriguez Camargo, Junliang Shen for useful discussions and suggestions. I would also like to thank Gabriel Dospinescu for informing me of their recent work \cite{DPS22} and thank Yiweng Ding and  Arthur-C\'esar Le-Bras
for several comments and corrections on an earlier version of this paper. I am supported by a Simons Junior Faculty Fellows award from the Simons Foundation.

\subsection*{Notation}
For a topological space, we denote by $\pi_0(X)$ the set of connected components of $X$. For a map $f:\mathcal{F}\to\mathcal{G}$ of sheaves of abelian groups on $X$, we will denote by $H^i(f):H^i(X,\mathcal{F})\to H^i(X,\mathcal{G})$ the induced map of the $i$-th cohomology groups.
 
Fix an algebraic closure $\overbar\Q$ of $\Q$. Denote by $G_{\Q}$ the absolute Galois group $\Gal(\overbar\Q/\Q)$. For each rational prime $l$, fix an algebraic closure $\overbar\Q_l$ of $\Q_l$ with ring of integers $\overbar\Z_l$, an embedding $\overbar\Q\to\overbar\Q_l$ which determines a decomposition group $G_{\Q_l}\subset G_{\Q}$ at $l$, and a lift of geometric Frobenius $\Frob_l\in G_{\Q_l}$. Our convention for local class field theory sends a uniformizer to a lift of geometric Frobenius.  
\begin{itemize}
\item $\varepsilon:\A_f^\times/\Q_{>0}^{\times}\to\Z_p^\times$ denotes the character corresponding to the $p$-adic cyclotomic character  via global class field theory. By abuse of notation, we will also use $\varepsilon$ to denote its composite with the determinant map  $\GL_2(\A_f)\xrightarrow{\det}\A_f^\times\xrightarrow{\varepsilon}\Z_p^\times$.
\item  $\varepsilon_p:G_{\Q_p}\to\Z_p^\times$ denotes the $p$-adic cyclotomic character.
\item   $\varepsilon'_p:=\varepsilon|_{\Q_p^\times}:\Q_p^\times\to\Z_p^\times$ denotes the composite of $\varepsilon_p$ with the local Artin map. Explicitly, its sends $x\in\Q_p^\times$ to $x|x|$. By abuse of notation, we will also use $\varepsilon'_p$ to denote its composite with the determinant map  $\GL_2(\Q_p)\xrightarrow{\det}\Q_p^\times\xrightarrow{\varepsilon'_p}\Z_p^\times$.
\item For a $\GL_2(\A_f)$-representation $M$, we will denote by $M\cdot D_0^k$ the twist of $M$ by the character  $|\cdot|_{\A}^{-k}\circ \det$, where $|\cdot|_\A:\A^\times_f\to \Q^\times$ denotes the usual adelic norm map.
\end{itemize}
Suppose $S$ is a finite set of rational primes. We denote by $G_{\Q,S}$ the the Galois group of the maximal extension of $\Q$ unramified outside $S$ and $\infty$. 

 Let $B\subseteq\GL_2(\Q_p)$ denote the upper triangular Borel subgroup. For $i=1,2$,  $e_i':B\to \Q_p^\times$ will denote the character sending $\begin{pmatrix} a_1 & b \\ 0 & a_2\end{pmatrix}$ to $a_i$.

We fix an isomorphism $\bar{\Q}_p\cong \mathbb{C}$.

\section{Locally analytic vectors} \label{Lav}
Let $G$ be a $p$-adic Lie group of dimension $d$. In \cite[\S 2]{Pan20}, we consider locally analytic vectors in a $\Q_p$-Banach space representation of $G$. For the purpose of this paper, we also need to study locally analytic vectors of representations over Hausdorff LB-spaces. One good reference is \cite{Eme17}. 

\subsection{Representations on Hausdorff LB-spaces} \label{roHLB}
\begin{para} \label{GanBan}
First we recall some constructions from \cite[\S 2]{Pan20}. Fix a compact open subgroup $G_0$ of $G$ equipped with an integer valued, saturated $p$-valuation. Let $G_n=G_0^{p^n}$. We denote by $\sC^{\an}(G_n,\Q_p)$ the space of $\Q_p$-valued analytic functions on $G_n$. There are two left actions of $G_n$ on $\sC^{\an}(G_n,\Q_p)$: left translation action and right translation action. Unless otherwise stated, we will always use the left translation action. 

For a $\Q_p$-Banach space representation $W$ of $G$, we denote by $\sC^{\an}(G_n,W)$ the space of $W$-valued analytic functions on $G_n$. The group $G_n$ acts on it by
\[(g\cdot f)(h)=g\left(f(g^{-1}h)\right)\]
for any $g,h\in G_n$ and $f\in \sC^{\an}(G_n,W)$. There is a natural $G_n$-equivariant isomorphism $\sC^{\an}(G_n,W)\cong \sC^{\an}(G_n,\Q_p)\widehat{\otimes}_{\Q_p}W$. The subspace of $G_n$-analytic vectors in $W$ (denoted by $W^{G_n-\an}$) is defined as the image of the $G_n$-invariants under the evaluation map at the identity element
\[\sC^{\an}(G_n,W)^{G_n}\xrightarrow{ev_{\mathbf{1}}} W.\]
$G_n$ acts continuously on $W^{G_n-\an}$. One way to see this is by identifying this action with the right translation action of $G_n$ on $\sC^{\an}(G_n,W)$.
If $W^{G_n-\an}=W$, we say the action of $G_n$ on $W$ is analytic. In this case, the evaluation map ${ev_{\mathbf{1}}}$ is an isomorphism, hence it induces a natural isomorphism
\[\sC^{\an}(G_n,\Q_p)\widehat{\otimes}_{\Q_p}W\cong \sC^{\an}(G_n,W)\]
equipped with the usual actions of $G_n$  on $\sC^{\an}(G_n,\Q_p)$ and $ \sC^{\an}(G_n,W)$ but with the trivial action on $W$ on the left hand side. Explicitly, if we identify $\sC^{\an}(G_n,\Q_p)\widehat{\otimes}_{\Q_p}W$ with $\sC^{\an}(G_n,W)$, this isomorphism sends $f\in \sC^{\an}(G_n,W)$ on the left hand side to the function $g\mapsto g(f(g)),g\in G_n$.

If $W$ is a $C$-Banach space representation of $G$, we can view it as a $\Q_p$-Banach space representation. Then it is clear that $W^{G_n-\an}$ is also a $C$-Banach space.
\end{para}

\begin{para} \label{indVn}
Now we extend everything to representations on LB-spaces. By an LB-space, we mean a locally convex topological $\Q_p$-vector space which is isomorphic to a locally convex inductive limit of a countable inductive system of $\Q_p$-Banach spaces. Our reference here is \cite{Eme17}. In this paper, we will only consider Hausdorff LB-space.  Such a space can be written as the inductive limit of a sequence of $\Q_p$-Banach spaces with injective transition maps. See the discussion below \cite[Definition 1.1.16]{Eme17} for more details about this notion.
\end{para}

\begin{exa} \label{exala}
Let $W$ be a $\Q_p$-Banach space representation of $G$. Then the subspace of $G$-locally analytic vectors in $W$
\[W^{\la}=\varinjlim_{n}W^{G_n-\an}\]
is naturally a Hausdorff LB-space. (To see that $W^{\la}$ is Hausdorff, one simply observes that the natural inclusion $W^{\la}\subseteq W$ is continuous because $\sC^{\an}(G_n,\Q_p)\subseteq\sC(G_n,\Q_p)$ is continuous, where $\sC(G_n,\Q_p)$ denotes the space of $\Q_p$-valued continuous functions on $G_n$.)
\end{exa}

\begin{para} \label{discLB}
Let $\displaystyle V=\varinjlim_i V_i$ be a Hausdorff LB-space, where $\{V_i\}_{i\geq 1}$ is an inductive sequence of Banach spaces. We denote by $\sC^{\an}(G_n,V)$ the space of $V$-valued analytic functions on $G_n$, cf. \cite[Definition 2.1.11]{Eme17}. Then it is easy to see that
\[\sC^{\an}(G_n,V)\cong\varinjlim_i \sC^{\an}(G_n,V_i)\cong \varinjlim_i \sC^{\an}(G_n,\Q_p)\widehat{\otimes}_{\Q_p}V_i\]
is a Hausdorff LB-space. As the notation suggests, we need to justify that this  does not depend on the choice of $\{V_i\}_{i\geq 1}$. Let $U_i$ be the kernel of the natural map $V_i\to V$ and $W_i=V_i/U_i$. Since $V$ is Hausdorff, $U_i$ is a closed subspace hence there is a natural Banach space structure on $W_i$.  Moreover $W_i$ can be viewed as a subspace of $V$ and $\bigcup_{i\geq 1} W_i=V$.  We claim that 
\begin{itemize}
\item for any $i$, there exists $i'\geq i$ such that the image of $U_i$ in $U_{i'}$ is zero, i.e. there is an induced map $W_i\to V_{i'}$ through which the inclusion $W_i\subseteq W_{i'}$ factors.
\end{itemize}
Indeed,  $U_i$ is equal to the increasing union $\bigcup_{j\geq i}\ker(V_i\to V_{j})$. Hence our claim follows from the Baire category theorem. In particular, this implies that the natural map
\[\varinjlim_{j} \sC^{\an}(G_n,V_j)\to \varinjlim_{j} \sC^{\an}(G_n,W_j)\]
is an isomorphism. Now it suffices to show $\sC^{\an}(G_n,V)$ does not depend on the choice of $\{W_i\}_{i\geq 1}$. But this is a direct consequence of \cite[Proposition 1.1.10]{Eme17}. We are going to use this result a few times, so we recall it here.
\end{para}

\begin{prop} \label{1.1.10}
Let $V$ be a Hausdorff locally convex topological $\Q_p$-vector space. Suppose that there exists a sequence of $\Q_p$-Banach spaces $\{V_n\}_{n\geq 1}$, and for each $n$, an injective continuous linear mapping $v_n:V_n\to V$ such that $\bigcup_{n\geq 1}v_n(V_n)=V$. Let $W$ be a $\Q_p$-Banach space and $u:W\to V$ a continuous linear mapping. Then $u(W)\subseteq v_n(V_n)$ for some $n$.
\end{prop}

\begin{proof}
This is essentially \cite[Proposition 1.1.10]{Eme17} which follows from \cite[prop. 1, p. I.20]{Bou87} by observing that the graph of $u$ is closed in $W\times V$ by our assumption.
\end{proof}

\begin{para}
Suppose everything admits a continuous action of $G_0$, i.e. $\{V_i\}_{i\geq 1}$ is an inductive sequence of Banach space representations of $G_0$. There is a natural action of $G_n$ on 
$\sC^{\an}(G_n,V)$ defined by
\[(g\cdot f)(h)=g\left(f(g^{-1}h)\right)\]
for any $g,h\in G_n$ and $f\in \sC^{\an}(G_n,V)$. We define the subspace of $G_n$-analytic vectors in $V$ as the image of the $G_n$-invariants under the evaluation map at the identity element
\[\sC^{\an}(G_n,V)^{G_n}\xrightarrow{ev_{\mathbf{1}}} V\]
equipped with the subspace topology from $\sC^{\an}(G_n,V)$ and denote it by $V^{G_n-\an}$. Equivalently, 
\[V^{G_n-\an}=\varinjlim_{i} V_i^{G_n-\an}\subseteq V.\]
It is clear from this definition that  $V^{G_n-\an}$ is a Hausdorff LB-space and $\{V^{G_n-\an}\}_n$ forms a inductive sequence. The subspace of locally analytic vectors in $V$ is defined as the inductive limit $\displaystyle \varinjlim_n V^{G_n-\an}$ and will be denoted by $V^{\la}$. This is a Hausdorff LB-space. It is easy to check that $V^{\la}$ does not depend on the choice of $G_0$ in the sense that we get the same $V^{\la}$ if we replace $G_0$ by a  compact open subgroup of $G_0$ equipped with an integer valued, saturated $p$-valuation.
\end{para}

\subsection{LB-space of compact type}
\begin{para}
We will need a sufficient condition for an LB-space being Hausdorff. For simplicity, all the discussions in this subsection are restricted to $\Q_p$-coefficients. But all the results actually hold with $\Q_p$ replaced by a finite extension.

Recall that a linear operator $f:X\to Y$ between two $\Q_p$-Banach spaces is called compact if the closure of $f(X^o)$ in $Y$ is compact. Here $X^o\subseteq X$ denotes the unit ball of $X$. Equivalently, $f$ is compact if and only if the closure of $f(T)$ is compact for any bounded subset $T$ of $X$. From this description, it is clear that the notion of compact operator only depends on the topology on $X$ and $Y$, not the norms. 

A compact operator is necessarily continuous. Conversely, suppose $f:X\to Y$ is a continuous linear map between two $\Q_p$-Banach spaces. Then $f(X^o)\subseteq p^kY^o$ for some integer $k$. By definition, $f$ is compact if and only if the image of $X^o\xrightarrow{f} p^kY^o/p^{k+n}Y^o$ is finite for any $n\geq 0$. 

Suppose $g:Y\to Z$ and $h:W\to X$ are continuous operators between Banach spaces. Then $g\circ f$ and $f\circ h$ are also compact if $f$ is compact.
\end{para}

\begin{exa} \label{compexa}
Here is a standard example of a compact operator in rigid analytic geometry. Consider the natural inclusion map between Tate algebras $\Q_p\langle T \rangle \to \Q_p \langle \frac{T}{p}\rangle$. It is easy to see that this map is compact. Geometrically, it corresponds to the restriction from the closed unit disc to the closed disc with radius $\|p\|$.

More generally, suppose $B$ is a $\Q_p$-Banach algebra. Then a continuous $\Q_p$-algebra homomorphism $f:\Q_p\langle T_1,\cdots, T_n \rangle \to B$ is compact  if $f(T_1),\cdots, f(T_n)$ are topologically nilpotent. We will need a variant of this in our later application. Consider the $\Q_p$-Banach algebra $A=\Z_p\langle T_1,\cdots, T_n\rangle [[x]]\otimes_{\Z_p}\Q_p$ with unit open ball $\Z_p\langle T_1,\cdots, T_n\rangle [[x]]$. It is easy to see that again a continuous $\Q_p$-algebra homomorphism $g:A \to B$ is compact  if $g(T_1),\cdots, g(T_n)$ and $g(x)$ are topologically nilpotent.  
\end{exa}

The following results will only be used in \ref{LBH}. 

\begin{prop} \label{comHau}
Suppose $\{V_i\}_{i\geq 1}$ is an inductive sequence of $\Q_p$-Banach spaces with injective compact transition maps. Then the locally convex inductive limit $\displaystyle \varinjlim_i V_i$ is Hausdorff.
\end{prop}

\begin{proof}
\cite[Lemma 16.9]{Sch02}.
\end{proof}

We will also need the following generalization.

\begin{cor} \label{comWHau}
Suppose $\{V_i\}_{i\geq 1}$ is an inductive sequence of $\Q_p$-Banach spaces with injective compact transition maps and $W$ is a $\Q_p$-Banach space. Then the locally convex inductive limit $\displaystyle \varinjlim_i V_i\widehat\otimes_{\Q_p}W$ is Hausdorff.
\end{cor}

\begin{proof}
Given a non-zero vector $v\in V_j\widehat\otimes_{\Q_p}W$, we can find a linear functional $l:W\to\Q_p$ such that the image of $v$ under the induced map $1\otimes l:V_j\widehat\otimes_{\Q_p}W\to V_j\otimes\Q_p=V_j$ is non-zero. Consider the locally convex inductive limit over $i$
\[ \varinjlim_i V_i\widehat\otimes_{\Q_p}W\xrightarrow{1\otimes l}  \varinjlim_i V_i.\]
This is a  continuous map and $v$ has non-zero image. Since $\displaystyle \varinjlim_i V_i$ is Hausdorff by Proposition \ref{comHau}, there exists an open neighborhood of $0\in \varinjlim_i V_i\widehat\otimes_{\Q_p}W$ not containing $v$. This shows that $\varinjlim_i V_i\widehat\otimes_{\Q_p}W$ is Hausdorff.
\end{proof}

In our later application, we will prove some inductive sequence of $\Q_p$-Banach spaces has  compact transition maps by a Cartan-Serre type argument. Here is the key lemma.

\begin{lem} \label{CStarg}
Let $\{V_i\}_{i\geq 1}$ and $\{W_i\}_{i\geq 1}$ be inductive sequences of $\Q_p$-Banach spaces with continuous linear transition maps. Suppose there are continuous linear maps $f_i:V_i\to W_i,i\geq 1$ with the following properties
\begin{enumerate}
\item $\{f_i\}_{i\geq 1}$ commute with transition maps;
\item For any $i\geq 1$, there exists $j\geq i$ such that $\im(W_i\to W_j)\subseteq f_j(V_j)$;
\item For any $i\geq 1$, the composite map $V_i\xrightarrow{f_i}W_i\to W_j$ is compact for $j$ sufficiently large.
\end{enumerate}
Then the transition map $W_i\to W_j$ is compact for any $i\geq 1$ and $j$ sufficiently large.
\end{lem}

\begin{proof}
Fix $i\geq 1$. By the second condition, $\im(W_i\to W_k)\subseteq f_k(V_k)$ for some $k\geq i$. We equip $f_k(V_k)=V_k/\ker(f_k)$ with the quotient topology.  Hence it is also a $\Q_p$-Banach space. The induced map 
\[W_i\to f_k(V_k)\]
is continuous by the closed graph theorem. Indeed, the graph of $W_i\to W_k$ is closed and the inclusion $W_i\times f_k(V_k)\subseteq W_i\times W_k$ is continuous. Now take $j\geq k$ such that $V_k\xrightarrow{f_k}W_k\to W_j$ is compact. Clearly  $W_i\to W_j$ is compact as it is the composite $W_i\to f_k(V_k)\to W_j$.
\end{proof}

\subsection{Some miscellaneous results}
\begin{para}
In this subsection, we will collect some simple results which will be used later. Besides, we will show  the equivalence between $\mathfrak{LA}$-acyclicity and strongly $\mathfrak{LA}$-acyclicity introduced in \cite[\S 2.2]{Pan20}. We keep the same notation used in the previous subsection.
\end{para}

\begin{lem} \label{Gnanprecle}
Let $V\subseteq W$ be a closed embedding of $\Q_p$-Banach space representations of $G$. Then the natural maps $V^{G_n-\an}\to W^{G_n-\an},n\geq 0$ and $V^{\la}\to W^{\la}$ are also closed embeddings.
\end{lem}

\begin{proof}
It suffices to prove that $V^{G_n-\an}\to W^{G_n-\an}$ is a closed embedding. By our assumption, the natural map
\[\sC^{\an}(G_n,\Q_p)\widehat{\otimes}_{\Q_p} V \to \sC^{\an}(G_n,\Q_p)\widehat{\otimes}_{\Q_p} W\]
is a closed embedding by our assumption. Hence the $G_n$-invariants of this map is also a closed embedding which is nothing but $V^{G_n-\an}\to W^{G_n-\an}$.
\end{proof}

\begin{prop} \label{tenscomHcont}
Let $H$ be a profinite group and $K$ a finite extension of $\Q_p$. Suppose  $W$ is a $K$-Banach space representation of $H$ and
\[H^i_\cont(H,W)=0,i\geq1.\]
Then
\[H^i_\cont(H,W\widehat{\otimes}_{K}V)=0,i\geq 1\]
for any $K$-Banach space $V$ equipped with a trivial action of $H$.
\end{prop}

\begin{proof}
By definition, $H^i_\cont(H,W)$ is computed by a complex $C^\bullet(H,W)$ with 
$C^{i}(H,W)=\sC(\underbrace{H\times\cdots\times H}_\text{$n+1$ times},W)^H\cong \sC(\underbrace{H\times\cdots\times H}_\text{$n$ times},W)$. Then $W^{H}\to C^\bullet(H,W)$ is a strict exact complex by our assumption and the open mapping theorem.
There are natural isomorphisms
\[\sC(\underbrace{H\times\cdots\times H}_\text{$n$ times},W)\widehat{\otimes}_{K}V\cong \sC(\underbrace{H\times\cdots\times H}_\text{$n$ times},W\widehat{\otimes}_{K}V),\]
i.e. $C^\bullet(H,W)\widehat{\otimes}_{K}V\cong C^\bullet(H,W\widehat{\otimes}_{K}V)$. Hence $(W\widehat{\otimes}_{K}V)^H\cong W^H\widehat{\otimes}_{K}V\to  C^\bullet(H,W\widehat{\otimes}_{K}V)$ is a strict exact complex as well.
\end{proof}

\begin{cor} \label{injnoHi}
Let $H$ be a profinite group and $W$ be a $\Q_p$-Banach space representation of $H$. Then 
\[H^i_\cont(H,\sC(H,\Q_p)\widehat\otimes_{\Q_p}W)=0,i\geq 1.\]
\end{cor}

\begin{proof}
Let $W'=W$ equipped with the \textit{trivial} action of $H$. Then there is a natural $H$-equivariant isomorphism 
\[\sC(H,\Q_p)\widehat\otimes_{\Q_p}W'\cong \sC(H,\Q_p)\widehat\otimes_{\Q_p}W.\]
Explicitly it sends $f\in\sC(H,W')=\sC(H,\Q_p)\widehat\otimes_{\Q_p}W'$ to the function $g\mapsto g\left(f(g)\right),g\in H$, viewed as an element in $\sC(H,W)=\sC(H,\Q_p)\widehat\otimes_{\Q_p}W$. Hence we may assume the action of $H$ on $W$ is trivial. Now our claim follows from  Proposition \ref{tenscomHcont} since $H^i_\cont\left(H,\sC(H,\Q_p)\right)=0,i\geq 1$, cf. Proof of Proposition 1.1.3 of \cite{Eme06}.
\end{proof}

\begin{para}
For a $\Q_p$-Banach space representation $W$ of $G$ and $i\geq 1$, we define (following \cite[\S 2.2]{Pan20})
\[R^i\mathfrak{LA}(W):=\varinjlim_n H^i_{\cont}(G_n,W\widehat{\otimes}_{\Q_p}\sC^{\an}(G_n,\Q_p))\]
which measures the failure of taking locally analytic vectors in $W$. Note that there is an isomorphism $W\widehat{\otimes}_{\Q_p}\sC^{\an}(G_n,\Q_p)\cong \sC^{\an}(G_n,W)$. Hence  
\[R^i\mathfrak{LA}(W):=\varinjlim_n H^i_{\cont}\left(G_n,\sC^{\an}(G_n,W)\right).\]

We say $W$  is
\begin{itemize}
\item $\mathfrak{LA}$-acyclic if $R^i\mathfrak{LA}(W)=0$ for any $i\geq 1$;
\item strongly $\mathfrak{LA}$-acyclic if the direct system $\left\{H^i_{\cont}(G_n,W\widehat{\otimes}_{\Q_p}\sC^{\an}(G_n,\Q_p))\right\}_n$ is essentially zero for any $i\geq1$, i.e. for any $i\geq 1$ and $n\geq0$, we can find  $m\geq n$ such that
\[H^i_{\cont}(G_n,\sC^{\an}(G_n,W)) \to H^i_{\cont}(G_m,\sC^{\an}(G_m,W))\]
is  the zero map.
\end{itemize}
Both notions do not depend on the choice of $G_0$. Clearly $W$ is  $\mathfrak{LA}$-acyclic if it is strongly  $\mathfrak{LA}$-acyclic.
\end{para}

\begin{prop} \label{LAasLAa}
Suppose $W$ is  a $\mathfrak{LA}$-acyclic  $\Q_p$-Banach space representation  of $G$. Then $W$ is  strongly $\mathfrak{LA}$-acyclic.
\end{prop}

\begin{proof}
Consider the natural closed embedding $\Q_p\to \sC(G_0,\Q_p)$ which maps $\Q_p$ to constant functions. Take the completed tensor product with $W$ over $\Q_p$. We get a closed embedding
\[W\to \sC(G_0,\Q_p)\widehat\otimes_{\Q_p} W\cong \sC(G_0,W).\]
Denote the quotient by $Q$ equipped with the quotient topology. This is naturally a $\Q_p$-Banach representation of $G_0$. Hence we get a short exact sequence
\[0\to W\to \sC(G_0,W)\to Q\to 0.\]
Passing to the $G_n$-analytic vectors, we have
\[  \sC(G_0,W)^{G_n-\an}\to Q^{G_n-\an}\to H^1_{\cont}(G_n,\sC^{\an}(G_n,W))\to H^1_{\cont}\left(G_n,\sC^{\an}(G_n, \sC(G_0,W))\right)\to.\]
Since $\sC^{\an}\left(G_n, \sC(G_0,W)\right)\cong \sC(G_0,\Q_p)\widehat\otimes_{\Q_p}\sC^\an(G_n,W)$, it follows from Corollary \ref{injnoHi} that 
\[H^i_{\cont}\left(G_n,\sC^{\an}(G_n, \sC(G_0,W))\right)=0,i\geq 1.\]
(Note that $ \sC(G_0,\Q_p)\cong  \sC(G_n,\Q_p)^{\oplus [G_0:G_n]}$.)
In particular,  we have 
\[H^{i+1}_\cont\left(G_n,\sC^\an(G_n,W)\right)\cong H^i_\cont\left(G_n,\sC^\an(G_n,Q)\right),i\geq 1,\]
and an exact sequence 
\[ 0\to W^{G_n-\an}\to\sC(G_0,W)^{G_n-\an}\to Q^{G_n-\an}\to H^1_{\cont}(G_n,\sC^{\an}(G_n,W))\to 0.\]
Denote the quotient $\sC(G_0,W)^{G_n-\an}/W^{G_n-\an}$ by $Q_n$ equipped with the quotient topology. This is a $\Q_p$-Banach space. Now consider the direct limit over $n$:
\[0\to \varinjlim_n Q_n\to\varinjlim_n Q^{G_n-\an}\to R^1\mathfrak{LA}(W)=0~~~~\mbox{(by our assumption)}.\]
Hence $\displaystyle \varinjlim_n Q_n\to\varinjlim_n Q^{G_n-\an}\cong  \varinjlim_n Q^{G_n-\an}=Q^{\la}$ is an isomorphism.
Example \ref{exala} shows  that $Q^{\la}$ is a Hausdorff LB-space. Thus 
$Q^{G_n-\an}\subseteq Q_m$ for some $m\geq n$  by Proposition \ref{1.1.10}. Equivalently, the image of 
\[H^1_{\cont}(G_n,\sC^{\an}(G_n,W))\to H^1_{\cont}(G_m,\sC^{\an}(G_m,W))\]
is zero. This proves that $\left\{H^1_{\cont}(G_n,W\widehat{\otimes}_{\Q_p}\sC^{\an}(G_n,\Q_p))\right\}_n$ is essentially zero. The general case can be proved by induction on $i$ using the isomorphism 
\[H^{i+1}_\cont\left(G_n,\sC^\an(G_n,W)\right)\cong H^i_\cont\left(G_n,\sC^\an(G_n,Q)\right),i\geq 1.\]
We omit the detail here.
\end{proof}

\section{Some results of \texorpdfstring{\cite{Pan20}}{} and generalizations} \label{Rf20}
\subsection{Horizontal Cartan action}\label{R20}

\begin{para} \label{brr}
First we recall some notation in  \cite[\S 4]{Pan20}. As in the introduction, let $C$ be the completion of $\overbar\Q_p$. For a neat open compact subgroup $K$ of $\GL_2(\A_f)$, we denote by $X_K$ the complete modular curve of level $K$ over $\Q$ and by $\mathcal{X}_K$ the adic space associated to $X_K\times_{\Q}C$. Fix an open compact subgroup $K^p$ of $\GL_2(\A_f^p)$  contained in the level-$N$-subgroup $\{g\in\GL_2(\hat{\Z}^p)=\prod_{l\neq p}\GL_2(\Z_l)\,\vert\, g\equiv1\mod N\}$ for some $N\geq 3$ prime to $p$ throughout this paper. Then there exists a unique perfectoid space $\mathcal{X}_{K^p}$ over $C$ such that
\[\mathcal{X}_{K^p}\sim\varprojlim_{K_p\subset\GL_2(\Q_p)}\mathcal{X}_{K^pK_p},\]
where $K_p$ runs through all open compact subgroups of $\GL_2(\Q_p)$. Very loosely speaking, on non-cusp points, $\mathcal{X}_{K^p}$ parametrizes elliptic curves with level-$K^p$-structure away from p and a trivialization of the $p$-adic Tate module. Let $V=\Q_p^{\oplus 2}$ be the standard representation of $\GL_2(\Q_p)$. Our convention here is that the local system associated to $V$ agrees with the first (relative) $p$-adic \'etale cohomology of the universal family of elliptic curves. Hence the universal $p$-adic Tate module corresponds to the representation $V(1)$, where  $(1)$ denotes the Tate twist by $1$, or $V^*$, the dual representation of $V$.

We denote the automorphic line bundle on finite level of modular curves by $\omega^1$ and its pull-back to $\mathcal{X}_{K^p}$ by $\omega_{K^p}$.  On the infinite level $\mathcal{X}_{K^p}$,  we have the following exact sequence coming from the relative Hodge-Tate filtration: 
\begin{eqnarray}\label{rHT}
0\to\omega_{K^p}^{-1}(1)\to V(1)\otimes_{\Q_p} \cO_{\mathcal{X}_{K^p}} \to \omega_{K^p}\to 0,
\end{eqnarray}
or equivalently,
\[0\to\omega_{K^p}^{-1}\to V\otimes_{\Q_p} \cO_{\mathcal{X}_{K^p}} \to \omega_{K^p}(-1)\to 0.\]
Then the variation of Hodge-Tate decompositions, or more precisely the position of $\omega^{-1}_{K^p}$ in $V\otimes_{\Q_p} \cO_{\mathcal{X}_{K^p}}$, induces a $\GL_2(\Q_p)$-equivariant Hodge-Tate period map
\[\pi_{\HT}:\mathcal{X}_{K^p}\to\Fl.\]
Here $\Fl=\mathbb{P}^1$ is the adic space over $C$ associated to the usual flag variety for $\GL_2$. We note that there is an ample line bundle $\omega_{\Fl}$  on  $\Fl$ (tautological line bundle) whose pull-back along $\pi_{\HT}$ agrees with $\omega_{K^p}(-1)$.

Let $\cO_{K^p}={\pi_{\HT}}_*\cO_{\mathcal{X}_{K^p}}$ be the  push-forward of the structure sheaf of $\mathcal{X}_{K^p}$ along $\pi_{\HT}$ and $\cO^{\la}_{K^p}\subset \cO_{K^p}$ be the subsheaf of $\GL_2(\Q_p)$-locally analytic sections introduced in \cite[4.2.6]{Pan20}. Let $\mathfrak{g}:=\mathfrak{gl}_2(C)$ be the complexified Lie algebra of $\GL_2(\Q_p)$. Then there is a natural action of $\mathfrak{g}^0:=\cO_{\mathrm{Fl}}\otimes_{C}\mathfrak{g}$ on $\cO^{\la}_{K^p}$. Let  $\mathfrak{b}^0$ (resp. $\mathfrak{n}^0$) be the subsheaf of horizontal Borel subalgebra (resp. subsheaf of horizontal nilpotent subalgebra). By \cite[Theorem.4.2.7]{Pan20}, $\mathfrak{n}^0$ acts trivially on $\cO^{\la}_{K^p}$, hence we get an action of $\mathfrak{b}^0/\mathfrak{n}^0$ on $\cO^{\la}_{K^p}$. Let  $\mathfrak{h}:=\{\begin{pmatrix} * & 0 \\ 0 & * \end{pmatrix}\}$ be a Cartan subalgebra of the Borel subalgebra $\mathfrak{b}:=\{\begin{pmatrix} * & * \\ 0 & * \end{pmatrix}\}$. It acts on $\cO^{\la}_{K^p}$ via the natural embedding $\mathfrak{h}\to\cO_{\mathrm{\Fl}}\otimes_{C}\mathfrak{h}=\mathfrak{b}^0/\mathfrak{n}^0$. This horizontal action will be denoted by $\theta_\kh$ and encodes the infinitesimal character, cf. \cite[Corollary 4.2.8]{Pan20}.
\end{para}

\begin{para} \label{omegakla}
In this paper, we also need to consider twists of  $\cO^{\la}_{K^p}$. Fix an integer $k$. Let $U$ be an affinoid open subset of $\Fl$. Then $\omega_{\Fl}$ is trivial on $U$. Hence by choosing a generator of $\omega_{\Fl}|_U$, we get an isomorphism $\cO_{\mathcal{X}_{K^p}}(\pi_{\HT}^{-1}(U))\cong \omega^k_{K^p}(\pi_{\HT}^{-1}(U))$. This defines a Banach space structure on $\omega_{K^p}(\pi_{\HT}^{-1}(U))$. It is clear that the topology defined on $\omega^k_{K^p}(\pi_{\HT}^{-1}(U))$ does not depend on the choice of the generator of $\omega_{\Fl}|_U$.  In particular, the subspace of locally analytic vectors in $\omega^k_{K^p}(\pi_{\HT}^{-1}(U))$ is well-defined.
Let $\omega^{k,\la}_{K^p}$ be the subsheaf of  $\GL_2(\Q_p)$-locally analytic sections of ${\pi_{\HT}}_*(\omega_{K^p})^{\otimes k}$. Similar to $\theta_\kh$, there is a natural horizontal Cartan action of $\kh$ on $\omega^{k,\la}_{K^p}$. There are two natural ways viewing $\omega^{k,\la}_{K^p}$ as a twist of $\cO^{\la}_{K^p}$: 
\begin{enumerate}
\item (twist on $\Fl$) $\omega^{k,\la}_{K^p}(-k)\cong \cO^{\la}_{K^p}\otimes_{\cO_{\Fl}}(\omega_{\Fl})^{\otimes k}$;
\item (twist on modular curves) $\omega^{k,\la}_{K^p}\cong \omega^{k,\sm} _{K^p}\otimes_{\cO^{\sm} _{K^p}}\cO^{\la}_{K^p}$.
\end{enumerate}
Here  $\cO^{\sm}_{K^p}=\omega^{0,\sm} _{K^p}$, and $\omega^{k,\sm} _{K^p}$ were introduced in \cite[5.3.3]{Pan20}. As the superscript suggests, $\omega^{k,\sm} _{K^p}\subset\omega^{k,\la} _{K^p}$ is the subsheaf of $\GL_2(\Q_p)$-smooth sections, i.e. sections annihilated by $\mathfrak{g}$. Equivalently, 
\[\omega^{k,\sm}_{K^p}={\pi_{\HT}}_{*} (\varinjlim_{K_p\subset \GL_2(\Q_p)}(\pi_{K_p})^{-1} \omega^{ k}),\] 
where $\pi_{K_p}:\mathcal{X}_{K^p}\to\mathcal{X}_{K^pK_p}$ denotes the natural projection and $\pi_{K_p}^{-1}$ is the pull-back as sheaf of abelian groups. The first isomorphism (twist on $\Fl$) is clear and the second isomorphism will be proved later in the next Subsection. It's easy to compute that the horizontal action of $\kh$  on $\omega^{k,\sm}_{K^p}$ is trivial and the action on $\omega_{\Fl}$ is  via the character sending $\begin{pmatrix} a & 0\\ 0 & d \end{pmatrix}\in\kh$ to $d$. Both isomorphisms are $\kh$-equivariant with respect to these actions.
\end{para}

\subsection{Local structure of \texorpdfstring{$\cO^{\la,\chi}_{K^p}$}{Lg}}
\begin{para} \label{exs}
For a weight $\chi$ of $\kh$, we write $\chi(\begin{pmatrix} a & 0\\ 0 & d \end{pmatrix}) =n_1a+n_2d$ for some $n_1,n_2\in C$ and identify $\chi$ with an ordered pair $(n_1,n_2)\in C^2$. Throughout the paper, we will only consider the case of \textit{integral weights}, i.e. 
\[(n_1,n_2)\in\Z^2.\] 
Denote by $\cO^{\la,\chi}_{K^p}$ the weight-$\chi$ subsheaf of $\cO^{\la}_{K^p}$ under $\theta_\kh$ and define $\omega^{k,\la,\chi}_{K^p}\subset\omega^{k,\la}_{K^p}$ similarly. An explicit description of $\cO^{\la,\chi}_{K^p}$ was given in \cite[4.3,5.1]{Pan20}. We recall it here and generalize to $\omega^{k,\la,\chi}_{K^p}$ as well.

Recall that we have the following exact sequence: 
\begin{eqnarray*}
0\to\omega_{K^p}^{-1}(1)\to V(1)\otimes_{\Q_p} \cO_{\mathcal{X}_{K^p}} \to \omega_{K^p}\to 0
\end{eqnarray*}
coming from the relative Hodge-Tate filtration. Taking $\wedge^2$ of $V(1)\otimes_{\Q_p} \cO_{\mathcal{X}_{K^p}}$ in this exact sequence \eqref{rHT}, we get an isomorphism 
\[\cO_{\mathcal{X}_{K^p}}(1)=\omega_{K^p}^{-1}(1)\otimes_{\cO_{\mathcal{X}_{K^p}}}\omega_{K^p}\cong \wedge^2 V(1) \otimes_{\Q_p} \cO_{\mathcal{X}_{K^p}}=\cO_{\mathcal{X}_{K^p}}\otimes\det(2),\]
i.e. $\cO_{\mathcal{X}_{K^p}}\cong \cO_{\mathcal{X}_{K^p}}\otimes\det(1)$, where $\det$ denotes the determinant representation of $\GL_2(\Q_p)$.
Fix  a basis $b$ of $\Q_p(1)$ from now on. Then under this isomorphism,  $1\otimes b$ defines  an invertible function $\mathrm{t}\in H^0(\mathcal{X}_{K^p},\cO_{\mathcal{X}_{K^p}})$, on which the Galois group $G_{\Q_p}$ acts via the cyclotomic character. We remark that $\GL_2(\A_f)$ acts on $\mathrm{t}$ via a non-trivial character because implicitly we identify $\wedge^2 D$ with a trivial line bundle.  See \ref{HEt} below for more details.  
Note that the notation for $\mathrm{t}$ in \cite[4.3.1]{Pan20} was $t$. We decide to change the notation here because $t$ will be used for Fontaine's $2\pi i$ later. We will use the notation $\cdot \mathrm{t}$ to denote the twist by $\mathrm{t}$ to remember the action of $G_{\Q_p}\times\GL_2(\A_f)$.

Let $(1,0),(0,1)$ be the standard basis of $V$. We denote their images in $H^0(\mathcal{X}_{K^p},\omega_{K^p})$ under the surjective map in \eqref{rHT} by $e_1,e_2$. (Here we identify $V(1)$ with $V$ using $b$.) Hence we may view $e_1,e_2$ as sections of $\omega_{\Fl}$ on $\Fl$ and $x:=\frac{e_2}{e_1}$ defines rational function on $\Fl$. 

By \cite[Theorem III.1.2]{Sch15} (see also \cite[Theorem 4.1.7]{Pan20}), there exists a basis of open affinoid subsets  $\mathfrak{B}$  of $\Fl$ stable under finite intersections such that for any $U\in\mathfrak{B}$,
\begin{itemize}
\item its preimage $V_\infty=\pi_{\HT}^{-1}(U)$ is affinoid perfectoid;
\item $V_\infty$ is the preimage of an affinoid subset $V_{K_p}\subset\mathcal{X}_{K^pK_p}$ for sufficiently small open subgroup $K_p$ of $\GL_2(\Q_p)$;
\item the map $\varinjlim_{K_p}H^0(V_{K_p},\cO_{\mathcal{X}_{K^pK_p}})\to H^0(V_\infty,\cO_{\mathcal{X}_{K^p}})$
has dense image.
\end{itemize}
We briefly recall the choice of $\mathfrak{B}$ in \cite[Theorem 4.1.7]{Pan20}. Let $U_1,U_2\subset \Fl$ be the affinoid open subsets defined by $\{\|x\|\leq1\},\{\|x\|\geq1\}$. Then $\mathfrak{B}$ consists of finite intersections of rational subsets of $U_1,U_2$.

Fix $U\in\mathfrak{B}$ and assume that
\begin{itemize}
\item $U$ is an open subset of $U_1$, hence $e_1$ generates $H^0(V_\infty, \omega_{K^p})$ and $x=\frac{e_2}{e_1}$ is a regular function on $V_\infty$.
\end{itemize}
Now we would like to give an explicit description of $\cO^{\la,\chi}_{K^p}(U)$ and $\omega^{k,\la,\chi}_{K^p}(U)$. Let $G_0=1+p^mM_2(\Z_p),m\geq 2$ be an open subgroup of $\GL_2(\Q_p)$ so that $V_\infty$ is the preimage of an affinoid subset $V_{G_0}\subset\mathcal{X}_{K^pG_0}$. Let $G_n=G_0^{p^n}=1+p^{m+n}M_2(\Z_p)$ and $V_{G_n}\subset\mathcal{X}_{K^pG_n}$ be the preimage of $V_{G_0}$. Then $\varinjlim_{n}H^0(V_{G_n},\cO_{\mathcal{X}_{K^pG_n}})\to H^0(V_\infty,\cO_{\mathcal{X}_{K^p}})$
has dense image.  As explained in \cite[4.3.5]{Pan20}, for any $n\geq0$, we can find
\begin{itemize}
\item an integer $r(n)>r(n-1)>0$;
\item $x_n\in H^0(V_{G_{r(n)}},\cO_{\mathcal{X}_{K^pG_{r(n)}}})$ such that $\|x-x_n\|_{G_{r(n)}}=\|x-x_n\|\leq p^{-n}$ in $H^0(V_\infty,\cO_{\mathcal{X}_{K^p}})$,
\end{itemize}
and assume that $\omega^1$ is a trivial line bundle on $V_{K^pG_{r(n)}}$. Here $\|\cdot\|_{G_{r(n)}}$ denotes the norm on $G_{r(n)}$-analytic vectors and $\|\cdot\|$ is the usual norm on $\cO_{\mathcal{X}_{K^p}}$.  
\end{para}

\begin{thm} \label{str}
Suppose $\chi=(n_1,n_2)\in\Z^2$. 
\begin{enumerate}
\item For any $n\geq0$, given a sequence of sections 
\[c_i\in H^0(V_{G_{r(n)}},\omega^{n_1-n_2}),i=0,1,\cdots\] 
such that the norm of $e_1^{n_2-n_1}c_i p^{(n-1)i}\in \cO_{\mathcal{X}_{K^p}}(V_\infty),i\geq 0$ is uniformly bounded, then
\[f=\mathrm{t}^{n_1}e_1^{n_2-n_1}\sum_{i=0}^{+\infty}c_i(x-x_n)^i=\mathrm{t}^{n_1}\sum_{i=0}^{+\infty}e_1^{n_2-n_1}c_i(x-x_n)^i\]
converges in $\cO^{\la,\chi}_{K^p}(U)^{G_{r(n)-\an}}$ and any $G_{n}$-analytic vector in $\cO^{\la,\chi}_{K^p}(U)$ arises in this way. Moreover $f=0$ if and only if all $c_i=0$.
\item Similarly, for any $n\geq0$, given a sequence of sections $c'_i\in H^0(V_{G_{r(n)}},\omega^{n_1-n_2+k}),i=0,1,\cdots$ such that  $\|e_1^{n_2-n_1-k}c'_i \|p^{-(n-1)i}$ is uniformly bounded, then 
\[f=\mathrm{t}^{n_1}e_1^{n_2-n_1}\sum_{i=0}^{+\infty}c'_i(x-x_n)^i\]
converges in $\omega^{k,\la,\chi}_{K^p}(U)^{G_{r(n)}-\an}$ and any $G_n$-analytic vector in $\omega^{k,\la,\chi}_{K^p}(U)$ arises in this way. Moreover $f=0$ if and only if all $c'_i=0$.
\end{enumerate}
 \end{thm}
 
 \begin{proof}
 The first part follows from Theorem 4.3.9 and Lemma 5.1.2 of \cite{Pan20}. Note that $c_i$ here is $(\frac{1}{t_N})^{n_1}(\frac{1}{e_{1,N}})^{n_2-n_1}c_i^{(n)}(f)$ in \cite[Lemma 5.1.2]{Pan20}. The second part can be reduced to the first part since multiplication by $e_1^{k}$ induces an isomorphism $\cO^{\la}_{K^p}(U)\cong \omega^{k,\la}_{K^p}(U) $. 
 \end{proof}
 
 \begin{cor} \label{density}
 Let $\chi=(n_1,n_2)\in\Z^2,k\in\Z$.  We denote by $\omega^{k,\la,\chi}_{K^p}(U)^{\kn-fin}\subseteq\omega^{k,\la,\chi}_{K^p}(U)$ the subspace consisting of elements of the form 
 \[\mathrm{t}^{n_1}e_1^{n_2-n_1}\sum_{i=0}^lc_ix^i,~~~~~c_i\in H^0(V_{G_n},\omega^{n_1-n_2+k})\mbox{ for some }l,n\geq 0.\]
 Then
 \begin{itemize}
 \item $\omega^{k,\la,\chi}_{K^p}(U)^{\kn-fin}$ is dense in $\omega^{k,\la,\chi}_{K^p}(U)$.
 \item There is a natural isomorphism 
 \[\omega^{k,\la,\chi}_{K^p}(U)^{\kn-fin}\cong \omega^{n_1-n_2+k,\sm}_{K^p}(U)\otimes_{\Q_p} M^{\vee}_{(n_2,n_1)},\]
 where $M^{\vee}_{(n_2,n_1)}\subseteq \omega^{n_2-n_1,\la,\chi}_{K^p}(U)$ consists of elements of the form
 \[\mathrm{t}^{n_1}e_1^{n_2-n_1}\sum_{i=0}^lc_ix^i,~~~~~c_i\in \Q_p \mbox{ for some }l\geq 0.\]
 (The notation comes from the category $\cO$, cf. \ref{CatO} below.)
\end{itemize} 
 \end{cor}
 
\begin{proof}
This is obvious in view of Theorem \ref{str}.
\end{proof}
 
\begin{rem} \label{tn1}
It is clear that $\cO^{\la,(n_1+k,n_2+k)}_{K^p}=\cO^{\la,(n_1,n_2)}_{K^p}\cdot\mathrm{t}^k$ for any integer $k$. We will use  later to reduce calculations to the case when $n_1=0$.
\end{rem}

\begin{para} \label{lalg0k}
Suppose $k$ is a non-negative integer and consider the case when $\chi=(0,k)$, i.e. $n_1=0,n_2=k$. Recall that $V=\Q_p^{\oplus 2}$ denotes the standard representation of $\GL_2(\Q_p)$.  The relative Hodge-Tate filtration induces a map $V(1)\to \omega^1_{K^p}$. Taking its $k$-th symmetric power, we get a map $\Sym^k V(k)\to \omega^k_{K^p}$, which is nothing but the natural map $H^0(\Fl,\omega^k_{\Fl})(k)\to \omega^k_{K^p}$.  This defines a map $\Sym^k V(k)\otimes_{\Q_p}\omega^{-k,\sm}_{K^p}\to \cO^{\la}_{K^p}$ whose image lies in the weight-$(0,k)$ part. In fact, this map is injective, hence we get a natural inclusion 
\[\Sym^k V(k)\otimes_{\Q_p}\omega^{-k,\sm}_{K^p}\subseteq \cO^{\la,(0,k)}_{K^p}\]
and it is easy to see that the image exactly consists of $\GL_2(\Q_p)$-locally algebraic vectors of $\cO^{\la,(0,k)}_{K^p}$. In general, if $\chi=(n_1,n_1+k)$, we have an inclusion
\[(\Sym^k V(k)\otimes_{\Q_p}\omega^{-k,\sm}_{K^p})\cdot\mathrm{t}^{n_1}\subseteq \cO^{\la,\chi}_{K^p}.\]
\end{para}

\begin{defn} \label{lalgKp}
We denote the image by $\cO^{\lalg,\chi}_{K^p}\subseteq \cO^{\la,\chi}_{K^p}$.
\end{defn}
 
 \begin{para}[Definition of $A^n$] \label{An}
Let  $A^n\subset \cO^{\la,\chi}_{K^p}(U)$ be the subset of $f$ that can be written as $\mathrm{t}^{n_1}e_1^{n_2-n_1}\sum_{i=0}^{+\infty}c_i(x-x_n)^i$ for some $c_i\in H^0(V_{G_{r(n)}},\omega^{n_1-n_2})$ such that $\|e_1^{n_2-n_1}c_i p^{(n-1)i}\|$ is uniformly bounded. \cite[Remark 4.3.11]{Pan20} implies that $A^n$ is independent of the choice of $x_n$. This is a Banach space with respect to the norm
\[\|\mathrm{t}^{n_1}e_1^{n_2-n_1}\sum_{i=0}^{+\infty}c_i(x-x_n)\|_n:=\sup_{i\geq 0}\|e_1^{n_2-n_1}c_i p^{(n-1)i}\|.\]
Then there are continuous inclusions 
\[\cO^{\la,\chi}_{K^p}(U)^{G_n-\an}\subset A^n \subset \cO^{\la,\chi}_{K^p}(U)^{G_{r(n)}-\an}.\]
In particular, $\cO^{\la,\chi}_{K^p}(U)=\varinjlim_n A^n$ as LB-spaces. Note that $A^n$ can be viewed as a subspace of $\prod_{i=0}^{+\infty} H^0(V_{G_{r(n)}},\omega^{n_1-n_2})$ by sending $f$ to $\{c_i\}_{i=0,1,\cdots}$. More precisely, if we fix a generator $s\in H^0(V_{G_{r(n)}},\omega^{n_1-n_2})$ i.e. an isomorphism $H^0(V_{G_{r(n)}},\cO_{\mathcal{X}_{K^pG_{r(n)}}})\cong H^0(V_{G_{r(n)}},\omega^{n_1-n_2})$, then we have an isomorphism
\begin{eqnarray} \label{exAn}
\left(\prod_{i=0}^\infty \cO^+_{\mathcal{X}_{K^pG_{r(n)}}}(V_{G_{r(n)}})\right)\otimes_{\Z_p}\Q_p\cong A^n
\end{eqnarray}
by sending $(a_i)_{i\geq 0}\in \prod_{i=0}^\infty \cO^+_{\mathcal{X}_{K^pG_{r(n)}}}(V_
{G_{r(n)}})$ to
\[\mathrm{t}^{n_1}e_1^{n_2-n_1}s\sum_{i=0}^{+\infty}a_ip^{-(n-1)i}(x-x_n)^i.\]
In other words, we have
\begin{eqnarray} \label{exAn[[]]}
\cO^+_{\mathcal{X}_{K^pG_{r(n)}}}(V_{G_{r(n)}})[[\frac{x-x_n}{p^{n-1}}]] \otimes_{\Z_p} \Q_p\cong A^n
\end{eqnarray}

Similarly we can define a subspace $B^n\subset \omega^{k,\la,\chi}_{K^p}(U)$, which can also be viewed as a subspace of $\prod_{i=0}^{+\infty} H^0(V_{G_{r(n)}},\omega^{n_1-n_2+k})$. Then the natural isomorphism
\[H^0(V_{G_{r(n)}},\omega^{k})\otimes_{\cO(V_{G_{r(n)}})}\prod_{i=0}^{+\infty} H^0(V_{G_{r(n)}},\omega^{n_1-n_2})\cong \prod_{i=0}^{+\infty} H^0(V_{G_{r(n)}},\omega^{n_1-n_2+k})\]
induces an isomorphism $H^0(V_{G_{r(n)}},\omega^{k})\otimes_{\cO(V_{G_{r(n)}})}A^n\cong B^n$. Recall that we introduced the sheaf $\omega^{k,\sm}_{K^p}$ in \ref{omegakla}.  By definition, $\omega^{k,\sm}_{K^p}(U)=\varinjlim_n H^0(V_{G_{n}},\omega^{k})$. Hence by passing to the direct limit over $n$ of  the previous isomorphisms, we get the following natural isomorphism
\[\omega^{k,\sm}_{K^p}(U)\otimes_{\omega^{0,\sm}_{K^p}(U)} \cO^{\la,\chi}_{K^p}(U)\cong \omega^{k,\la,\chi}_{K^p}(U).\]
Same argument works when $e_2$ is a generator of  $H^0(V_\infty,\omega_{K^p})$. Therefore, we have

 \end{para}
\begin{lem} \label{tensomegaksm}
\[\omega^{k,\sm}_{K^p}\otimes_{\cO^{\sm}_{K^p}} \cO^{\la,\chi}_{K^p}\cong \omega^{k,\la,\chi}_{K^p}.\]
\end{lem}

\begin{rem}
One can repeat the same argument without fixing weight $\chi$ (using \cite[Theorem 4.3.9]{Pan20} as an input), then one has the natural isomorphism
$\omega^{k,\la}_{K^p}\cong \omega^{k,\sm} _{K^p}\otimes_{\cO^{\sm} _{K^p}}\cO^{\la}_{K^p}$, which was claimed in \ref{omegakla}. 
\end{rem}

\begin{para} \label{vCech}
As pointed out in \cite[Lemma 4.3.14]{Pan20}, we can use $A^n$ to prove the following vanishing result of the \v{C}ech cohomology. This will only be used in the proof of Proposition \ref{Haus} below. Let $\mathfrak{U}=\{U^1,\cdots,U^l\}$ be a finite subset of $\mathfrak{B}$.  We can find $G_0$ of the form $1+p^mM_2(\Z_p)$ such that each $U^{i}\in\mathfrak{U}$ is $G_0$-stable. Let $C^\bullet(\mathfrak{U},\cO^{\la,\chi}_{K^p})$ be the \v{C}ech complex for $\cO^{\la,\chi}_{K^p}$ with respect to $\mathfrak{U}$. For $n\geq 0$, let
$C^\bullet(\mathfrak{U},\cO^{G_n-\an,\chi}_{K^p})$ be the subcomplex defined by restricting to the $G_n$-analytic vectors, i.e.
\[C^\bullet(\mathfrak{U},\cO^{G_n-\an,\chi}_{K^p}):=C^\bullet(\mathfrak{U},\cO^{\la,\chi}_{K^p})^{G_n-\an},\]
where $G_n=G_0^{p^n}$ as before.
We denote the cohomology of this complex by $\check{H}^\bullet(\mathfrak{U},\cO^{G_n-\an,\chi}_{K^p})$. Clearly there is a natural map $C^\bullet(\mathfrak{U},\cO^{G_n-\an,\chi}_{K^p})\to C^\bullet(\mathfrak{U},\cO^{G_{m}-\an,\chi}_{K^p})$ for $m\geq n$.
\end{para}

\begin{prop} \label{VCech}
Suppose $\mathfrak{U}=\{U^1,\cdots,U^l\}\subseteq \mathfrak{B}$ is a finite cover of $U\in\mathfrak{B}$. 
For any $i\geq 1$ and $n\geq 0$, the natural map 
\[\check{H}^i(\mathfrak{U},\cO^{G_n-\an,\chi}_{K^p})\to\check{H}^i(\mathfrak{U},\cO^{G_m-\an,\chi}_{K^p})\]
is zero when $m\geq n$ is sufficiently large.
\end{prop}

\begin{proof}
After possibly shrinking $G_0$, we may assume that $\pi_{\HT}^{-1}(U^{i})$ (resp. $\pi_{\HT}^{-1}(U)$) is the preimage of an affinoid subset $V^i_{G_0}\subseteq\mathcal{X}_{K^pG_0}$ (resp. $V_{G_0}\subseteq\mathcal{X}_{K^pG_0}$). Let $V^i_{G_n}\subseteq\mathcal{X}_{K^pG_n}$ be the preimage of $V^i_{G_0}$. Clearly $\{V^i_{G_n}\}_{i=1,\cdots,l}$ forms a finite affinoid cover of $V_{G_n}$. Using the action of $\GL_2(\Q_p)$, we may assume $U\subseteq U_1$ and apply the construction in \ref{exs}. In particular, we have sections $x_n$'s and can define $A^n\subseteq \cO^{\la,\chi}_{K^p}(U)$ as in \ref{An}.

Now suppose $\mathcal{I}$ is a subset of $\{1,\cdots,l\}$. We denote $\bigcap_{i\in \mathcal{I}} U^i$ by $U^{\mathcal{I}}$ and $\bigcap_{i\in \mathcal{I}} V^i_{G_{n}}$ by $V^\mathcal{I}_{G_n}$. ($U^{\emptyset}$ is understood as $U$ and $V^\emptyset_{G_n}=V_{G_n}$.) Using $x_n|_{U^\mathcal{I}}$, we can define $A^{\mathcal{I},n}\subseteq \cO^{\la,\chi}_{K^p}(U^i)$ similarly. Suppose $\mathcal{J}\subseteq \mathcal{I}$. Clearly there is a natural restriction map $A^{\mathcal{I},n}\to A^{\mathcal{J},n}$. Hence we can replace $\cO^{\la,\chi}_{K^p}(U^\mathcal{I})$ by  $A^{\mathcal{I},n}$ in the construction of the \v{C}ech complex $C^\bullet(\mathfrak{U},\cO^{\la,\chi}_{K^p})$ and get a subcomplex $C^{\bullet}(\mathfrak{U},A^n)$. Since $\cO^{\la,\chi}_{K^p}(U)^{G_n-\an}\subset A^n \subset \cO^{\la,\chi}_{K^p}(U)^{G_{r(n)}-\an}$ and similar statements hold for $U^\mathcal{I}$, we have natural inclusions
\[C^\bullet(\mathfrak{U},\cO^{G_n-\an,\chi}_{K^p})\subseteq C^{\bullet}(\mathfrak{U},A^n)\subseteq C^\bullet(\mathfrak{U},\cO^{G_{r(n)}-\an,\chi}_{K^p}).\]
Therefore it suffices to prove that for $i\geq 1$, 
\[H^i(C^{\bullet}(\mathfrak{U},A^n))=0.\]

Fix a generator $s$ of $\omega^{n_1-n_2}|_{V_{G_{r(n)}}}$. Then $s|_{V^\mathcal{I}_{G_{r(n)}}}$ is a generator  of $\omega^{n_1-n_2}|_{V^\mathcal{I}_{G_{r(n)}}}$. There is an isomorphism \eqref{exAn}
\[\left(\prod_{i=0}^\infty \cO^+_{\mathcal{X}_{K^pG_{r(n)}}}(V^\mathcal{I}_{G_{r(n)}})\right)\otimes_{\Z_p}\Q_p\cong A^{\mathcal{I},n}\]
which is compatible with replacing $\mathcal{I}$ by a subset. Note that  $\mathfrak{V}^n=\{V^i_{G_n}\}_{i=1,\cdots,l}$ is a finite affinoid cover of the affinoid $V_{G_n}$. Let $C^\bullet(\mathfrak{V}^n,\cO^+)$ be the \v{C}ech complex for $\cO^+_{\mathcal{X}_{K^pG_{r(n)}}}$ with respect to $\mathfrak{V}^n$. Then the above isomorphism implies that 
\[\left(\prod_{i=0}^\infty C^\bullet(\mathfrak{V}^n,\cO^+)\right)\otimes_{\Z_p}\Q_p\cong C^{\bullet}(\mathfrak{U},A^n).\]
Our claim follows by noting that  $H^i(C^\bullet(\mathfrak{V}^n,\cO^+)),i\geq 1$ is annihilated by some power of $p$ by Tate's acyclicity result, cf. \cite[Proof of Lemma 4.3.14]{Pan20}.
\end{proof}

\begin{para} \label{mathfrakU'}
Here is how we are going to apply this result. Let $\mathfrak{U}=\{U_1,U_2\}$. Recall that $U_1$ (resp. $U_2$) is defined by $\|x\|\leq 1$ (resp. $\|U_2\|\geq 1$). Let $U'_1\subseteq\Fl$ (resp. $U'_2\subseteq\Fl$) be the affinoid open subset defined by  $\|x\|\leq \|p^{-1}\|$ (resp. $\|U_2\|\geq \|p\|$) and $\mathfrak{U}'=\{U'_1,U'_2\}$. Then both $\mathfrak{U}$ and $\mathfrak{U}'$ are covers of $\Fl$. Note that $U'_1,U'_2$ can also be obtained by applying the action of $\begin{pmatrix} p & 0\\ 0 & 1 \end{pmatrix}$ and its inverse to $U_1,U_2$. In particular, even though $U'_1,U'_2\notin \mathfrak{B}$, results we proved for open subsets in $\mathfrak{U}$ are also valid for $U'_1,U'_2$.

Take $G_0=1+p^mM_2(\Z_p)$ so that $U'_1,U_1,U_2,U'_2$ are $G_0$-stable. As in Section \ref{vCech}, we have the \v{C}ech complexes $C^\bullet(\mathfrak{U},\cO^{G_n-\an,\chi}_{K^p}), C^\bullet(\mathfrak{U'},\cO^{G_n-\an,\chi}_{K^p})$ and their cohomology groups $\check{H}^i(\mathfrak{U},\cO^{G_n-\an,\chi}_{K^p})$, $\check{H}^i(\mathfrak{U}',\cO^{G_n-\an,\chi}_{K^p})$, where $G_n=G_0^{p^n}$. The restriction of $U'_1$ (resp. $U'_2$) to $U_1$ (resp. $U_2$) induces a natural map $\check{H}^i(\mathfrak{U'},\cO^{G_n-\an,\chi}_{K^p}) \to \check{H}^i(\mathfrak{U},\cO^{G_n-\an,\chi}_{K^p})$. 
\end{para}

\begin{cor} \label{CSarg}
For any $n\geq 0$, there exists $m\geq n$ such that the image of 
\[\check{H}^1(\mathfrak{U},\cO^{G_n-\an,\chi}_{K^p})\to \check{H}^1(\mathfrak{U},\cO^{G_m-\an,\chi}_{K^p})\] is contained in the image of $\check{H}^1(\mathfrak{U}',\cO^{G_m-\an,\chi}_{K^p})\to \check{H}^1(\mathfrak{U},\cO^{G_m-\an,\chi}_{K^p})$.
\end{cor}

\begin{proof}
This is a formal consequence of Proposition \ref{VCech}. Equivalently, we need to find $m\geq n$ such that for any $s\in \cO^{\la,\chi}_{K^p}(U_1\cap U_2)^{G_n-\an}$, there exist sections $s_i\in \cO^{\la,\chi}_{K^p}(U_i)^{G_m-\an}$, $i=1,2$ so that
\[s+s_1|_{U_1\cap U_2}+s_2|_{U_1\cap U_2}\]
can be extended to a section in $ \cO^{\la,\chi}_{K^p}(U'_1\cap U'_2)^{G_m-\an}$. Consider the cover $\{U_1,U'_1\cap U_2\}$ of $U'_1$. Then Proposition \ref{VCech} implies that there exists $m'\geq n$ such that for any $s\in \cO^{\la,\chi}_{K^p}(U_1\cap U_2)^{G_n-\an}$, we can find $s_1\in \cO^{\la,\chi}_{K^p}(U_1)^{G_{m'}-\an}$ so that 
$s+s_1|_{U_1\cap U_2}$
can be extended to $U'_1\cap U_2$. Similarly, by applying Proposition \ref{VCech} to the cover $\{U_2,U'_1\cap U'_2\}$ of $U_2$, we find $m\geq m'$ and $s_2\in \cO^{\la,\chi}_{K^p}(U_2)^{G_{m}-\an}$ such that 
$s+s_1|_{U_1\cap U_2}+s_2|_{U_1\cap U_2}$
can be extended to $U'_1\cap U'_2$, which is exactly what want.
\end{proof}

\subsection{Hodge-Tate structure}

\begin{para} \label{stpHT}
Since modular curves $\mathcal{X}_{K^pK_p}$ are defined over $\Q_p$, there is a natural continuous action of $G_{\Q_p}$ on $\mathcal{X}_{K^p}$ which commutes with the action of $\GL_2(\Q_p)$. Hence it also acts on the sheaf $\cO^{\la,\chi}_{K^p}$, where $\chi=(n_1,n_2)$ is an integral weight. In this subsection, we mainly study this semilinear action with respect to the $C$-vector space structure on $\cO^{\la,\chi}_{K^p}$.  Our reference here is \cite[Theorem 5.1.8]{Pan20} and its proof.

Let $U\in\mathfrak{U}$. We keep the same notation as in \ref{exs}. Fix $G_0=1+p^mM_2(\Z_p)$ so that $\pi_{\HT}^{-1}(U)$ is the preimage of an affinoid subset $V_{G_0}\subset\mathcal{X}_{K^pG_0}$. Let $G_n=G_0^{p^n}$, One useful observation is that 
\begin{itemize}
\item $V_{G_n}$, hence also its preimage $V_{G_n}$ in $\mathcal{X}_{K^pG_n}$, are defined over a finite extension of $\Q_p$. 
\end{itemize}
Indeed, this is clear when $U=U_1=\{||x||\leq 1\}$ or $U_2=\{||x||\geq 1\}$ because $\overbar\Q_p$ is dense in $C$ hence we can approximate $x$ by a section defined over a finite extension of $\Q_p$ on a finite level modular curve. A similar argument implies cases when $U$ is a rational subset of $U_1$ or $U_2$. In general, our assertion follows from the construction of $\mathfrak{B}$ that $U$ is a finite intersection of rational subsets of $U_1$ and $U_2$, cf. \ref{exs}. 

Fix a finite extension $K$ of $\Q_p$ in $C$ over which $V_{G_0}$ is defined. Then $G_K$ acts continuously  on the LB-space $\cO^{\la,\chi}_{K^p}(U)$ and  $C$-Banach space $\cO^{\la,\chi}_{K^p}(U)^{G_n-\an}$.
\end{para}

\begin{para} \label{HTsetup}
In order to state our result, we need a general discussion on semilinear representations of $G_K$ on $C$-Banach spaces. Let $K$ bw a finite extension of $\Q_p$ in C. We denote by $K_\infty\subseteq C$ the maximal $\Z_p$-extension of $K$ in $K(\mu_{p^\infty})$. Equivalently, let $\varepsilon_p:G_{\Q_p}\to\Z_p^\times$ be the $p$-adic cyclotomic character. There is a natural decomposition $\Z_p=(\Z/p\Z)^\times \times(1+2p\Z_p)$. We denote the composite map $G_K\xrightarrow{\varepsilon_p}\Z_p^\times\to 1+2p\Z_p$ by $\tilde\varepsilon_p$. Let $H_K$ be the kernel of $\tilde\varepsilon_p$. Then $K_\infty=\overbar\Q_p{}^{H_K}$. 
Hence $\Gal(K_\infty/K)$ is (non-canonically) isomorphic to $\Z_p$. 

Suppose $W$ is a $C$-Banach space equipped with a continuous semilinear action of $G_{K}$. Denote by
\[W^{K}\subseteq W\]
the subspace of $G_{K_\infty}$-fixed, $G_K$-analytic vectors \footnote{In the classical Sen theory for finite-dimensional $C$-representations, people usually consider $G_{K_\infty}$-fixed, $G_K$-finite vectors. As pointed out by Berger-Colmez in \S 3.1 of \cite{BC16}, the correct notion in general should be analytic vectors.} in $W_i$. This is naturally a $K$-Banach space. The $C$-Banach space structure on $W$ induces a map
\[\varphi^K:C\widehat{\otimes}_K W^{K} \to W.\]
We are particularly interested in $W$ satisfying that
\begin{itemize}
\item $\varphi^K$ is an isomorphism.
\end{itemize}
Then a standard argument using Tate's normalized trace shows that the natural map
\[L\otimes_{K}W^K\xrightarrow{\cong} W^L\]
is an isomorphism for any finite extension $L$ of $K$ in $C$. See for example the proof of \cite[Th\'eor\`eme 3.2]{BC16} \footnote{The $K_\infty$ in \cite{BC16} is different from our $K_\infty$ by a finite group in general. But the same argument works.}. In particular, $\varphi^L$ is an isomorphism for any finite extension $L$ of $K$.  More generally, one can use have  Tate's normalized trace to prove the following result.
\end{para}

\begin{prop} \label{exeTtr}
Let $X$ be a $K$-Banach space representation of $G_K$ such that $X^K=X$, i.e. the action of $G_{K_\infty}$ on $X$ is trivial and the induced action of $\Gal({K_\infty}/K)$ is analytic. Then $(C\widehat\otimes_K X)^K=X$.
\end{prop}

\begin{proof}
Exercise.
\end{proof}

\begin{cor} \label{strHT}
Let $K$ be a finite extension of $\Q_p$ in $C$. Let $W_1,W_2$ be  $C$-Banach spaces equipped with continuous semilinear actions of $G_{K}$. Suppose that the natural maps
\[\varphi_i^K:C\widehat{\otimes}_K W_i^K \to W,i=1,2,\]
are isomorphisms. Let $\phi:W_1\to W_2$ be a $G_{\Q_p}$-equivariant,  $C$-linear continuous map. Then
\begin{enumerate}
\item $\ker(\phi)= C\widehat\otimes_K \ker(\phi)^K$.
\item Suppose $\phi$ is a closed embedding. Let  $Z=W_2/\phi(W_1)$ equipped with the quotient topology. Then 
$Z= C\widehat\otimes_K Z^K$.
\item In general, suppose $\phi$ is strict. Let $Y=W_2/\phi(W_1)$ equipped with the quotient topology. Then 
$Y= C\widehat\otimes_K Y^K$.

\end{enumerate}
\end{cor}

\begin{proof}
For the first two parts, in view of the previous proposition, it suffices to prove that if we have a short exact sequence 
\[0\to M'\to M\xrightarrow{f} M\]
of $K$-Banach spaces with continuous homomorphisms, then 
\[0\to C\widehat\otimes_K M' \to C\widehat\otimes_K M\xrightarrow{1\otimes f} C\widehat\otimes_K M''\]
is still exact, and moreover, $1\otimes f$ is surjective if  $f$ is surjective. This is clear by choosing an orthonormal basis of $C$ over $K$.

The last part follows from the first two parts. 
\end{proof}

\begin{para}
Back to the setup in \ref{HTsetup}. By definition, there is a natural action of the Lie algebra $\Lie(\Gal({K_\infty}/K))$ on $W^K$. Note that we can naturally identify $\Lie(\Gal({K_\infty}/K))$ with $\Lie(\Z_p^\times)=\Q_p$ via $\varepsilon_p$. The action of $1\in\Q_p\cong \Lie(\Gal({K_\infty}/K))$ is classically called the Sen operator. Clearly the Sen operator is $K$-linear.
\end{para}

\begin{defn} \label{HTg}
Let $K$ be a finite extension of $\Q_p$ in $C$ and let $W$ be a $C$-Banach space equipped with a continuous semilinear action of $G_{K}$.
\begin{enumerate}
\item We say $W$ is Hodge-Tate of weight $\{w_1,\cdots,w_n\}\subseteq\Z$ if there exists a finite extension $L$ of $K$ in $C$ such that
\begin{itemize}
\item $\varphi^L:C\widehat{\otimes}_L W^{L} \to W$ is an isomorphism;
\item The action of the Sen operator $1\in\Q_p\cong \Lie(\Gal(L_\infty/L))$ on $W^L$ is semi-simple with eigenvalues $-w_1,\cdots,-w_n$, i.e. there is a natural decomposition 
\[W^L=W^L_{-w_1}\oplus\cdots\oplus W^L_{-w_n}\]
such that the Sen operator  acts on $W^L_{-w_i}$ via multiplication by $-w_i$ for $i=1,\cdots,n$. All $W^{L}_{-w_i}$ are non-zero.
\end{itemize}
Moreover, we have a decomposition for $W$ if it is Hodge-Tate of weight $\{w_1,\cdots,w_n\}$:
\[W=W_{-w_1}\oplus\cdots W_{-w_n},\]
where $W_{-w_i}=C\widehat\otimes_{L}W^L_{-w_i}\subseteq W$. It is easy to see that $W_{-w_i}$ is Hodge-Tate of weight $w_i$ and does not depend on the choice of $L$.
\item We say $W$ is Hodge-Tate if $W$ is Hodge-Tate of some weight $\{w_1,\cdots,w_n\}\subseteq\Z$ in the above sense.
\end{enumerate}
\end{defn}

\begin{rem} \label{clHT}
Keep the same notation as in Definition \ref{HTg}. Comparing with the classical definition of Hodge-Tate representations, we remark that $W$ is Hodge-Tate of weight $0$ if and only if 
\[W= C\widehat\otimes_{K} W^{G_K}.\]
Indeed, suppose $W=C \widehat\otimes_L W^L_0$ for some finite Galois extension $L$ of $K$. Then $\Gal(L_\infty/L)$ acts trivially on $W^L_0$ as this action is analytic and its infinitesimal action is trivial. Hence $W^L_0=W^{G_L}$. By Galois descent, we have 
\[W^{G_L}=L\otimes_K (W^{G_L})^{G_K}=L\otimes_K W^{G_K}.\]
It follows that $W= C\widehat\otimes_{K} W^{G_K}$. More generally, $W$ is Hodge-Tate of weight $w\in\Z$ if and only if its $w$-th Tate twist $W(w)$  is Hodge-Tate of weight $0$, equivalently, 
\[W=C(-w) \widehat\otimes_K W(w)^{G_K}. \]

In general, $W$  is Hodge-Tate of weight $\{w_1,\cdots,w_n\}\subseteq\Z$ if and only if 
\[W=\bigoplus_{i=1}^n C(-w_i) \widehat\otimes_K W(w_i)^{G_K}.\]
It recovers the classical notion of Hodge-Tate representations when $W$ is finite-dimensional. This also implies that 
\[W=C\widehat\otimes_{K} W^K\]
if $W$ is Hodge-Tate.
\end{rem}

Now we can state our main results.

\begin{prop} \label{GnanHT}
Let $\chi=(n_1,n_2)$ be an integral weight and $U\in\mathfrak{B}$ as in  \ref{stpHT}. Then $\cO^{\la,\chi}_{K^p}(U)^{G_n-\an}$ is Hodge-Tate of weight $-n_2$ for any $n\geq 0$.
\end{prop}

\begin{proof}
We may  assume $U\subseteq U_1$.  Keep the same notation in \ref{stpHT}. In particular, $V_{G_n}$ is defined over some finite extension $K$ of $\Q_p$. We may simply assume $K$ contains $\mu_{p^2}$ hence $K_\infty=K(\mu_{p^\infty})$. Now choose $x_n\in H^0(V_{G_{r(n)}},\cO_{\mathcal{X}_{K^pG_{r(n)}}})$ as in \ref{exs}. Since $\overbar\Q_p$ is dense in $C$, after possibly replacing $K$ by a finite extension, $x_n$ can be chosen to be defined over  $K$.  Then we have $A^n \subset \cO^{\la,\chi}_{K^p}(U)$ as \ref{An}. It follows from the construction that $G_K$ acts on $\mathrm{t},e_1,e_2$ via the cyclotomic character $\varepsilon_p$. Hence $G_K$ acts trivially on $x$. This implies that $A^n$ is $G_K$-stable. 

\begin{lem} \label{AnHT}
$A^n$ is Hodge-Tate of weight $-n_2$.
\end{lem}

\begin{proof}
Recall that $A^n$ consists of  elements of the form $\mathrm{t}^{n_1}e_1^{n_2-n_1}\sum_{i=0}^{+\infty}c_i(x-x_n)^i$ for some $c_i\in H^0(V_{G_{r(n)}},\omega^{n_1-n_2})$ such that $\|e_1^{n_2-n_1}c_i p^{(n-1)i}\|$ is uniformly bounded. Since $V_{G_{r(n)}}$ and $\omega^{n_1-n_2}$ are defined over $K$, we have
\[H^0(V_{G_{r(n)}},\omega^{n_1-n_2})=C\widehat\otimes_K H^0(V_{G_{r(n)}},\omega^{n_1-n_2})^{G_K}.\]
Let $A^{n,K}\subseteq A^n$ be the subset of $f=\mathrm{t}^{n_1}e_1^{n_2-n_1}\sum_{i=0}^{+\infty}c_i(x-x_n)^i$ with $c_i\in H^0(V_{G_{r(n)}},\omega^{n_1-n_2})^{G_K}$. This is a $K$-Banach space and $G_K$ acts on it via $\varepsilon_p^{n_1}\cdot \varepsilon_p^{n_2-n_1}=\varepsilon_p^{n_2}$. It follows from Proposition \ref{exeTtr} that 
\[A^n = C\widehat\otimes_K A^{n,K}\]
and by definition, $A^n$ is Hodge-Tate of weight $-n_2$.
\end{proof}

Note that $\cO^{\la,\chi}_{K^p}(U)^{G_n-\an}\subseteq A^n$. By the previous lemma, we have
\[\cO^{\la,\chi}_{K^p}(U)^{G_n-\an}=(A^n)^{G_n-\an}=(C\widehat\otimes_K A^{n,K})^{G_n-\an}=C\widehat\otimes_K (A^{n,K})^{G_n-\an}\]
where the last identity is clear by choosing an orthonormal basis of $C$ over $K$. Our claim follows as $G_K$ acts on $A^{n,K}$ via $\varepsilon_p^{n_2}$.
\end{proof}

The following result will be used later in a Cartan-Serre type argument. Roughly speaking, this is a generalization of the well-known fact in analytic geometry that relatively compact inclusions between $\Q_p$-affinoid spaces induce compact operators between $\Q_p$-affinoid algebras.

\begin{prop} \label{strcomp}
Let $\chi=(n_1,n_2)$ be an integral weight. Let  $U,U'\subseteq \mathfrak{B}$ be $G_0=1+p^lM_2(\Z_p)$-stable open subsets of $\Fl$ with $l\geq 2$ and $\pi_\HT^{-1}(U),\pi_\HT^{-1}(U')$ are the preimages of affinoid subsets $V_{G_0},V'_{G_0}\subseteq \mathcal{X}_{K^pG_0}$. Suppose that $V_{G_0}$ and $V'_{G_0}$ are defined over a finite extension $K$ of $\Q_p$ in $C$.
Assume that 
\begin{itemize}
\item $U$ is a strict neighborhood of $U'$, i.e. $\overbar{U'}\subseteq U$ as adic spaces.
\end{itemize}
Then 
\begin{enumerate}
\item $\cO^{\la,\chi}_{K^p}(U)^{G_n-\an}=C\widehat\otimes_K \cO^{\la,\chi}_{K^p}(U)^{G_n-\an,K}$ for any $n\geq 0$, where $G_n=G_0^{p^n}$;
\item $\cO^{\la,\chi}_{K^p}(U')^{G_n-\an}=C\widehat\otimes_K \cO^{\la,\chi}_{K^p}(U')^{G_n-\an,K}$ for any $n\geq 0$;
\item For any $n\geq0 $, the restriction map
$\cO^{\la,\chi}_{K^p}(U)^{G_n-\an,K}\to \cO^{\la,\chi}_{K^p}(U')^{G_m-\an,K}$
is compact for $m\geq n$ sufficiently large
\end{enumerate}
\end{prop}

\begin{rem}
This proposition roughly says that the restriction map $\cO^{\la,\chi}_{K^p}(U)^{G_n-\an}\to \cO^{\la,\chi}_{K^p}(U')^{G_m-\an}$ is essentially the completed tensor product of a compact operator with $C$.
\end{rem}

\begin{proof}
The first two claims follow from Lemma \ref{AnHT} and Remark \ref{clHT}. It remains to prove the last assertion. Note that  we are free to replace $K$ by a finite extension. Indeed, suppose $L$ is a finite extension of $K$ in $L$. Then $ \cO^{\la,\chi}_{K^p}(U)^{G_n-\an,K}$ is closed subspace of $\cO^{\la,\chi}_{K^p}(U)^{G_n-\an,L}$ because
\[ \cO^{\la,\chi}_{K^p}(U)^{G_n-\an,L}=L\otimes_K  \cO^{\la,\chi}_{K^p}(U)^{G_n-\an,K},\]
cf. the discussion above Proposition \ref{exeTtr}.

We may assume $U\subseteq U_1$. Denote by $V_{G_{n}}$ (resp. $V'_{G_n}$) the preimage of $V_{G_0}$ (resp. $V'_{G_0}$) in $\mathcal{X}_{K^pG_n}$.  As in \ref{exs}, we can choose  $x_n\in H^0(V_{G_{r(n)}},\cO_{\mathcal{X}_{K^pG_{r(n)}}})$ and define $A^n\subseteq \cO^{\la,\chi}_{K^p}(U)$. Using $x_n|_{V'_{G_{r(n)}}}$, we can define $A'{}^{n}\subseteq \cO^{\la,\chi}_{K^p}(U')$ similarly. Then $\cO^{\la,\chi}_{K^p}(U)^{G_n-\an}\subseteq A^n$  and $A'{}^n \subseteq \cO^{\la,\chi}_{K^p}(U')^{G_{r(n)}-\an}$.

Fix $n\geq 0$. After possibly replacing $K$ by a finite extension, we may assume that $V_{G_0},V'_{G_0},x_n,x_{n+1}$ are defined over $K$. It suffices to prove the natural map
\[f:A^{n,K}\to A'{}^{n+1,K}\]
is compact. Indeed, this will imply that $\cO^{\la,\chi}_{K^p}(U)^{G_n-\an,K}\to \cO^{\la,\chi}_{K^p}(U)^{G_m-\an,K}$ is compact with $m=r(n+1)$. To see the compactness of $f$, we fix a generator $s$ of $\omega^{n_1-n_2}|_{V_{G_{r(n)}}}$ defined over $K$. By \ref{exAn[[]]} and the proof of Lemma \ref{AnHT}, there is an isomorphism 
\[A^{n,K}\cong \cO^+_{\mathcal{X}_{K^pG_{r(n)}}}(V_{G_{r(n)}})^{G_K}[[\frac{x-x_n}{p^{n-1}}]] \otimes_{\Z_p} \Q_p.\]
Similarly, using the pull-back of $s$ to $V_{G_{r(n+1)}}$, we have 
\[A'{}^{n+1,K}\cong \cO^+_{\mathcal{X}_{K^pG_{r(n+1)}}}(V'_{G_{r(n+1)}})^{G_K}[[\frac{x-x_{n+1}}{p^{n}}]] \otimes_{\Z_p} \Q_p.\]
By our assumption, $U$ is a strict neighborhood of $U'$. Hence $\overbar{\pi_{\HT}^{-1}(U')}\subseteq \pi_\HT^{-1}(U)$. Both are $G_{r(n)}$-stable subsets. Note that $\displaystyle \pi_\HT^{-1}(U)\sim\varprojlim_{l}V_{G_l}$ and $V_{l}$ is the quotient of $V_{k}$ by $G_k/G_l, k\geq l$. Hence $V_{G_{r(n)}}=\pi_\HT^{-1}(U)/G_{r(n)}$ as topological spaces. From this, we conclude that $V_{G_{r(n)}}$ is a strict neighborhood of $V'_{G_{r(n)}}$. Therefore 
\[\cO^+_{\mathcal{X}_{K^pG_{r(n)}}}(V_{G_{r(n)}})^{G_K}\to   \cO^+_{\mathcal{X}_{K^pG_{r(n+1)}}}(V'_{G_{r(n+1)}})^{G_K}\]
is compact because modular curves are of finite type. On the other hand, we have
\[\frac{x-x_n}{p^{n-1}}=p\left(\frac{x-x_{n+1}}{p^n}\right)+\frac{x_{n+1}-x_n}{p^{n-1}}.\]
By our choice of $x_n$ in \ref{exs},
\[\|x_{n+1}-x_{n}\|\leq\max\{\|x-x_n\|,\|x-x_{n+1}\|\} \leq p^{-n}.\]
In particular, $\frac{x_{n+1}-x_n}{p^{n-1}}\in p \cO^+_{\mathcal{X}_{K^pG_{r(n+1)}}}(V'_{G_{r(n+1)}})^{G_K}$. Hence $\frac{x-x_n}{p^{n-1}}$ is topologically nilpotent in $A'{}^{n+1,K}$. This shows that $f:A^{n,K}\to A'{}^{n+1,K}$ is compact, cf. Example \ref{compexa}.
\end{proof}

\subsection{A digression: topology on the cohomology of \texorpdfstring{$\cO^{\la,\chi}_{K^p}$}{Lg}} \label{LBH}
\begin{para} \label{Cech}
$H^0(U,\cO^{\la,\chi}_{K^p})=H^0(U,\cO_{K^p})^{\la,\chi}$ is a Hausdorff LB-space if $U\in\mathfrak{B}$. The goal of this subsection is to define such a topology on  the cohomology $H^i(\Fl,\cO^{\la,\chi}_{K^p})$. 

We first consider the cohomology group $H^i(\Fl,\cO^+_{K^p})$, where $\cO^+_{K^p}={\pi_{\HT}}_*\cO^+_{\mathcal{X}_{K^p}}$. By Scholze's work, $H^i(\Fl,\cO^+_{K^p})$ is isomorphic to the $\cO_C$-coefficient completed cohomology as almost $\cO_C$-modules, cf. \cite[Corollary 4.4.3]{Pan20} and the proof of Theorem 4.4.6 \textit{ibid.}. In particular, $H^i(\Fl,\cO^+_{K^p})$ has bounded $p$-torsion and $H^i(\Fl,\cO_{K^p})$ is a $C$-Banach space. Recall that $U_1$ (resp. $U_2$) denotes the affinoid open subset  $\{\|x\|\leq1\}$ (resp. $\{\|x\|\geq1\}$) in $\Fl$. Let $U_{12}=U_1\cap U_2$. As proved in the proof of Theorem 4.4.6 of \cite{Pan20},  the cohomology $H^\bullet(\Fl,\cO^+_{K^p})$ can be computed by the \v{C}ech complex 
\[\cO^+_{K^p}(U_1)\oplus \cO^+_{K^p}(U_2)\to \cO^+_{K^p}(U_{12}).\]
Therefore after inverting $p$, we see that the homomorphism in the  \v{C}ech complex 
\[C^\bullet:\cO_{K^p}(U_1)\oplus \cO_{K^p}(U_2)\to \cO_{K^p}(U_{12})\]
is strict. 
\end{para}

\begin{lem} \label{rescle}
The restriction map $\cO_{K^p}(U_1)\to \cO_{K^p}(U_{12})$ is a  closed embedding.
\end{lem}

\begin{proof}
First we show that this map is injective. Suppose $f\in\ker\left(\cO_{K^p}(U_1)\to \cO_{K^p}(U_{12})\right)$. Then $f$ can be extended to a global section $\tilde{f}\in H^0(\Fl,\cO_{K^p})$ with $\tilde{f}|_{U_2}=0$ (and $\tilde{f}|_{U_1}=f$). Since $H^0(\Fl,\cO_{K^p})$ is isomorphic to the $C$-coefficient completed cohomology $\tilde{H}^0(K^p,C)$ by \cite[Theorem 4.4.6]{Pan20}, on which the action of $\SL_2(\Q_p)$ is trivial, cf. \cite[4.2]{Eme06}, we conclude that $\tilde{f}$ is fixed by $\SL_2(\Q_p)$. However $\SL_2(\Q_p)\cdot U_2=\Fl$. Hence $f=0$.

It remains to show the strictness of $\cO_{K^p}(U_1)\to \cO_{K^p}(U_{12})$. We use the Weyl group.
\begin{lem} \label{keyweyl}
Let $i:A\to B$ be a continuous map of $\Q_p$-Banach spaces. Let $w$ be a continuous involution on $B$.  Then $i$ is strict if and only if $(i,w\circ i):A\oplus A\to B$ is strict.
\end{lem}

\begin{proof}
Suppose $(i,w\circ i):A\oplus A\to B$ is strict, then $i$ is strict as it is the composite $A\xrightarrow{x\mapsto (x,0)} A\oplus A\xrightarrow{(i,w\circ i)} B$. Now suppose $i$ is strict. There is a continuous involution $w'$ on $A\oplus A$ sending $(x_1,x_2)$ to $(x_2,x_1)$. It is clear that $(i,w\circ i)$ intertwines $w'$ and $w$. Hence we only need to prove that the induced maps $(A\oplus A)^{w'=1}\to B^{w=1}$ and $(A\oplus A)^{w'=-1}\to B^{w=-1}$ are strict. Under the isomorphism $A\xrightarrow{x\mapsto (x,x)}(A\oplus A)^{w'=1}$, one can check easily that $(A\oplus A)^{w'=1}\to B^{w=1}$ agrees with the composite
\[A\xrightarrow{i} B \xrightarrow{1+w} B^{w=1}.\]
Note that $(1+w)/2$ is the projection map of $B$ onto $B^{w=1}$ hence a strict homomorphism. It follows from our assumption that $(1+w)\circ i$ is strict. Similarly $(A\oplus A)^{w'=-1}\to B^{w=-1}$ is strict as it essentially agrees with $(1-w)\circ i$.
\end{proof}
We apply this lemma with $A=\cO_{K^p}(U_1)$, $B=\cO_{K^p}(U_{12})$, $i$ being the restriction map and $w=\begin{pmatrix} 0 & 1\\ 1 & 0\end{pmatrix}\in\GL_2(\Q_p)$. In this case, $(i,w\circ i):A\oplus A\to B$ is nothing but the \v{C}ech complex $\cO_{K^p}(U_1)\oplus \cO_{K^p}(U_2)\to \cO_{K^p}(U_{12})$ up to a sign and we have seen this is a strict complex. Hence $\cO_{K^p}(U_1)\to \cO_{K^p}(U_{12})$ is strict.
\end{proof}

\begin{para} \label{Celachi}
Let $G_0=1+p^2M_2(\Z_p)\subseteq \GL_2(\Z_p)$ and $G_n=G_0^{p^n}$. There are unitary continuous actions of $G_n$ on $\cO_{K^p}(U_1)$ and $\cO_{K^p}(U_{12})$. Then Lemma \ref{rescle} and Lemma \ref{Gnanprecle} imply that 
\[\cO_{K^p}(U_1)^{G_n-\an}\to \cO_{K^p}(U_{12})^{G_n-\an}\]
is also a closed embedding. Applying  Lemma \ref{keyweyl} to this map, we see  that the complex obtained by taking $G_n$-analytic vectors in $C^{\bullet}$
\[\cO_{K^p}(U_1)^{G_n-\an}\oplus \cO_{K^p}(U_2)^{G_n-\an}\to \cO_{K^p}(U_{12})^{G_n-\an}\]
is still a strict complex. 

Similarly, let $\chi$ be a weight of $\kh$. Then $\cO_{K^p}(U_1)^{G_n-\an,\chi}\to \cO_{K^p}(U_{12})^{G_n-\an,\chi}$ is a closed embedding. We get a strict homomorphism
\[\cO^{\la,\chi}_{K^p}(U_1)^{G_n-\an}\oplus \cO^{\la,\chi}_{K^p}(U_2)^{G_n-\an}\to \cO^{\la,\chi}_{K^p}(U_{12})^{G_n-\an}.\]
Therefore the cohomology groups of this complex are also Banach spaces.
If we pass to the inductive limit over $n$, we get the \v{C}ech complex
\[\cO^{\la,\chi}_{K^p}(U_1)\oplus \cO^{\la,\chi}_{K^p}(U_2)\to \cO^{\la,\chi}_{K^p}(U_{12}).\]
Note that this actually computes the cohomology $H^i(\Fl,\cO^{\la,\chi}_{K^p})$. Indeed, the proof of Theorem 4.4.6 of \cite{Pan20} says that the \v{C}ech cohomology $\check{H}^j(U,\cO^{\la}_{K^p})=0$ for any $j>0$ and $U\in\mathfrak{B}$. By \cite[Theorem 5.1.2.(1)]{Pan20}, this shows that $\check{H}^j(U,\cO^{\la,\chi}_{K^p})=0,~j>0,~U\in\mathfrak{B}$. Hence by the same argument used in the proof of Theorem 4.4.6 of \cite{Pan20}, $H^i(\Fl,\cO^{\la,\chi}_{K^p})$ can be computed by the \v{C}ech cohomology.
\end{para}

\begin{defn} \label{chGanchi}
We define  $\check{H}^\bullet(\cO^{G_n-\an}_{K^p})$ to be the cohomology of 
\[\cO_{K^p}(U_1)^{G_n-\an}\oplus \cO_{K^p}(U_2)^{G_n-\an}\to \cO_{K^p}(U_{12})^{G_n-\an}\]
and $\check{H}^\bullet(\cO^{G_n-\an,\chi}_{K^p})$ (which was denoted by $\check{H}^i(\mathfrak{U},\cO^{G_n-\an,\chi}_{K^p})$ in \ref{vCech}) to be the cohomology of
\[\cO^{\la,\chi}_{K^p}(U_1)^{G_n-\an}\oplus \cO^{\la,\chi}_{K^p}(U_2)^{G_n-\an}\to \cO^{\la,\chi}_{K^p}(U_{12})^{G_n-\an}.\]
These are natural Banach spaces over $C$ equipped with an analytic action of $G_n$. This defines an LB-space structure on   $H^i(\Fl,\cO^{\la,\chi}_{K^p})$  by 
\[H^i(\Fl,\cO^{\la,\chi}_{K^p})=\varinjlim_n \check{H}^i(\cO^{G_n-\an,\chi}_{K^p}).\]
\end{defn}

\begin{rem} \label{LBstreqv}
Similarly, we can define an LB-space structure on $H^i(\Fl,\cO^{\la}_{K^p})$ by 
\[H^i(\Fl,\cO^{\la}_{K^p})=\varinjlim_{n}\check{H}^i(\cO^{G_n-\an}_{K^p}).\]
On the other hand, it follows from \cite[Theorem 4.4.6]{Pan20} that we can define another LB-space structure by
\[H^i(\Fl,\cO^{\la}_{K^p})\cong  H^i(\Fl,\cO_{K^p})^\la =\varinjlim_n H^i(\Fl,\cO_{K^p})^{G_n-\an}.\] 
Not surprisingly, both LB-spaces structures are equivalent.  This is clear when $i=0$ because the natural map
\[\check{H}^0(\cO^{G_n-\an}_{K^p})\to H^0(\Fl,\cO_{K^p})^{G_n-\an}.\]
is an isomorphism. When $i=1$, the natural map $\check{H}^1(\cO^{G_n-\an}_{K^p})\to H^1(\Fl,\cO_{K^p})^{G_n-\an}$ is continuous. It suffices to prove that for $n\geq0$, there exist $m\geq n$ and a continuous map
\[H^1(\Fl,\cO_{K^p})^{G_n-\an}\to \check{H}^1(\cO^{G_m-\an}_{K^p})\] 
whose composite with  $\check{H}^1(\cO^{G_m-\an}_{K^p})\to H^1(\Fl,\cO_{K^p})^{G_m-\an}$ agrees with the natural inclusion $H^1(\Fl,\cO_{K^p})^{G_n-\an}\subseteq H^1(\Fl,\cO_{K^p})^{G_m-\an}$. This follows from the strongly $\mathfrak{LA}$-acyclicity of $\cO_{K^p}(U_i)$ and $H^0(\Fl,\cO_{K^p})$ cf. the beginning of the proof of Theorem 5.1.11 of \cite{Pan20}.
\end{rem}

\begin{rem}
We can also equip a topology on $H^i(\Fl,\omega^{k,\la,\chi}_{K^p})$ similar to Definition \ref{chGanchi}. This will not be used in this paper but might be of some independent interest.  Again we need to show the \v{C}ech complex associated to the cover $\{U_1,U_2\}$
\[{\pi_{\HT}}_*(\omega^k_{K^p})(U_1)\oplus {\pi_{\HT}}_*(\omega^k_{K^p})(U_2)\to {\pi_{\HT}}_*(\omega^k_{K^p})(U_{12})\]
has strict homomorphism. The topology on ${\pi_{\HT}}_*(\omega^k_{K^p})(U_?),?=1,2,12$ is defined as follows: by choosing  an invertible section of $\omega^k_{K^p}$ on $\pi_{\HT}^{-1}(U_?)$, we identify ${\pi_{\HT}}_*(\omega^k_{K^p})(U_?)$ with $\cO_{K^p}(U_?)$ and thus define a Banach space structure on ${\pi_{\HT}}_*(\omega^k_{K^p})(U_?)$.  It is easy to see that the topology obtained in this way does not depend on the choice of the invertible section. By Lemma \ref{keyweyl} again, we see that 
\[{\pi_{\HT}}_*(\omega^k_{K^p})(U_1)\oplus {\pi_{\HT}}_*(\omega^k_{K^p})(U_2)\to {\pi_{\HT}}_*(\omega^k_{K^p})(U_{12})\]
is strict. From this, we can proceed as before. Note that this argument shows that we can put a Banach space structure on $H^i(\Fl,{\pi_{\HT}}_*(\omega^k_{K^p}))$. However the action of $\GL_2(\Q_p)$ is only continuous but not unitary in general.
\end{rem}

\begin{rem}
One can also use Lemma \ref{keyweyl} to prove the claim in Remark 5.1.14 of \cite{Pan20} without any calculations.
\end{rem}

\begin{para} \label{GKcoh}
As explained in \ref{stpHT},  the Galois group $G_K$ acts naturally on $\cO^{\la,\chi}_{K^p}(U_i),i=1,2$ for some finite extension $K$ of $\Q_p$ in $C$. Assume $\chi=(n_1,n_2)$ is integral. We have seen in Proposition \ref{GnanHT} that $\cO^{\la,\chi}_{K^p}(U_?)^{G_n-\an}, ?=1,2,12$ are Hodge-Tate of weight  $-n_2$. 
\end{para}

\begin{prop} \label{chcohHT}
$\check{H}^\bullet(\cO^{G_n-\an,\chi}_{K^p})$ is Hodge-Tate of weight $-n_2$.
\end{prop}

\begin{proof}
This follows from  Corollary \ref{strHT} directly.
\end{proof}

Back to $H^i(\Fl,\cO^{\la,\chi}_{K^p})$. Here are the main results of this subsection.

\begin{prop} \label{Haus}
$H^i(\Fl,\cO^{\la,\chi}_{K^p})$ is a Hausdorff LB-space.
\end{prop}

\begin{proof}
The continuous  inclusion $\cO^{\la,\chi}_{K^p}(U_{12})\subseteq \cO_{K^p}(U_{12})$ induces  a  natural continuous map 
\[H^i(\Fl,\cO^{\la,\chi}_{K^p})\to H^i(\Fl,\cO_{K^p}).\]
By \cite[Corollary 5.1.3.(2)]{Pan20}, this is injective unless $i=1$ and $\chi(h)=0$, where $h=\begin{pmatrix} 1& 0\\ 0&-1\end{pmatrix}$. Hence $H^i(\Fl,\cO^{\la,\chi}_{K^p})$ is Hausdorff except this case as $H^i(\Fl,\cO_{K^p})$ is a $C$-Banach space. By Remark \ref{tn1}, it remains to show that $\displaystyle \varinjlim_n \check{H}^1(\cO^{G_n-\an,(0,0)}_{K^p})$ is Hausdorff. We need  several lemmas. By the discussion in \ref{GKcoh}, $G_K$ acts on $\check{H}^1(\cO^{G_n-\an,(0,0)}_{K^p}),n\geq 0$ for some finite extension $K$ of $\Q_p$ in $C$. 

\begin{lem}
For $n\geq 0$, there is a natural isomorphism 
\[\check{H}^1(\cO^{G_n-\an,(0,0)}_{K^p})\cong \check{H}^1(\cO^{G_n-\an,(0,0)}_{K^p})^{G_{K}}\widehat\otimes_{K}C.\] 
\end{lem}

\begin{proof}
This follows from Proposition \ref{chcohHT} and Remark \ref{clHT}.
\end{proof}

\begin{lem} \label{Lemcom}
For $m\geq n$, we denote by $i_{n,m}$ the transition map 
\[\check{H}^1(\cO^{G_n-\an,(0,0)}_{K^p})^{G_{K}}\to \check{H}^1(\cO^{G_m-\an,(0,0)}_{K^p})^{G_{K}}.\]
Then
\begin{enumerate}
\item $i_{n,m}$ is compact if $m$ is sufficiently large.
\item $\ker(i_{n,m})$ stabilizes when $m$ is sufficiently large, i.e. 
\[\ker(i_{n,m})=\ker\left(\check{H}^1(\cO^{G_n-\an,(0,0)}_{K^p})^{G_{K}}\to H^1(\Fl,\cO^{\la,(0,0)}_{K^p})\right),~~~~\mbox{if }m>>n.\]
Equivalently, $\ker\left(\check{H}^1(\cO^{G_n-\an,(0,0)}_{K^p})^{G_{K}}\to H^1(\Fl,\cO^{\la,(0,0)}_{K^p})\right)$ is a closed subspace.
\end{enumerate}
\end{lem}

Assuming this lemma at the moment, we finish the proof of Proposition \ref{Haus}. Denote by $V_n$ the image of $\check{H}^1(\cO^{G_n-\an,(0,0)}_{K^p})^{G_{K}}\to H^1(\Fl,\cO^{\la,(0,0)}_{K^p})$ equipped with the quotient topology.  Then it follows from the previous two lemmas that 
\[\varinjlim_n V_n\widehat\otimes_K C\cong \varinjlim_n \check{H}^1(\cO^{G_n-\an,(0,0)}_{K^p})\]
as locally convex inductive limits. Note that $V_n\to V_m$ is injective and compact for sufficiently large $m\geq n$. We can apply Corollary \ref{comWHau} here with $W=C$ and conclude that $\displaystyle \varinjlim_n \check{H}^1(\cO^{G_n-\an,(0,0)}_{K^p})$ is Hausdorff.
\end{proof}

\begin{proof}[Proof of Lemma \ref{Lemcom} (1)]
We will use a Cartan-Serre type argument and invoke Lemma \ref{CStarg}. Our starting point is that $H^i(\Fl,\cO^{\la,\chi}_{K^p})$ can be computed by another cover $\mathfrak{U}'$ of $\Fl$ introduced in \ref{mathfrakU'}. Recall that $\mathfrak{U}'=\{U_1',U_2'\}$ with $U'_1$ (resp. $U'_2$) defined by $\|x\|\leq\|p^{-1}\|$ (resp. $\|x\|\geq \|p\|$). Then we can redo everything in this subsection with $U_1,U_2$ replaced by $U'_1,U'_2$. As in \ref{vCech}, we denote the quotient of 
\[\cO_{K^p}(U'_1)^{G_n-\an}\oplus \cO_{K^p}(U'_2)^{G_n-\an}\to \cO_{K^p}(U'_{1}\cap U'_{2})^{G_n-\an}\]
by $\check{H}^1(\mathfrak{U}',\cO^{G_n-\an,\chi}_{K^p})$.
Then $\displaystyle H^1(\Fl,\cO^{\la,\chi}_{K^p})=\varinjlim_n \check{H}^1(\mathfrak{U}',\cO^{G_n-\an,\chi}_{K^p})$. Also $G_K$ acts on these cohomology groups and there are natural  isomorphisms 
\[\check{H}^1(\mathfrak{U}',\cO^{G_n-\an,\chi}_{K^p})=\check{H}^1(\mathfrak{U}',\cO^{G_n-\an,\chi}_{K^p})^{G_K}\widehat\otimes_K C.\]
Moreover, we have natural restriction maps
\[\check{H}^1(\mathfrak{U}',\cO^{G_n-\an,\chi}_{K^p}) \to \check{H}^1(\cO^{G_n-\an,\chi}_{K^p}).\]
By Corollary \ref{CSarg}, the image  of $\check{H}^1(\mathfrak{U}',\cO^{G_m-\an,\chi}_{K^p}) \to \check{H}^1(\cO^{G_m-\an,\chi}_{K^p})$ contains the image of $\check{H}^1(\cO^{G_n-\an,\chi}_{K^p}) \to \check{H}^1(\cO^{G_m-\an,\chi}_{K^p})$ for sufficiently large $m\geq n$. We claim that this is also true for the $G_K$-invariants, i.e. 
\[\im\left( \check{H}^1(\cO^{G_n-\an,\chi}_{K^p})^{G_K} \to \check{H}^1(\cO^{G_m-\an,\chi}_{K^p})\right) \subseteq
\im\left(\check{H}^1(\mathfrak{U}',\cO^{G_m-\an,\chi}_{K^p})^{G_K} \to \check{H}^1(\cO^{G_m-\an,\chi}_{K^p})\right).\]
Indeed by the last part of Corollary \ref{strHT}, the image  of $\check{H}^1(\mathfrak{U}',\cO^{G_m-\an,\chi}_{K^p}) \to \check{H}^1(\cO^{G_m-\an,\chi}_{K^p})$, equipped with the quotient topology, is Hodge-Tate of weight $0$. 

Note that $U'_1\cap U'_2$ is a strict neighborhood of $U_1\cap U_2$.  Hence by Proposition \ref{strcomp}, 
$\cO_{K^p}(U'_{1}\cap U'_{2})^{G_n-\an,G_K}\to \cO_{K^p}(U_{1}\cap U_{2})^{G_m-\an,G_K}$
is compact for sufficiently large $m\geq n$. Therefore 
\[\check{H}^1(\mathfrak{U}',\cO^{G_n-\an,\chi}_{K^p})^{G_K} \to \check{H}^1(\cO^{G_m-\an,\chi}_{K^p})^{G_K}\]
is compact for $m>>n$.

Finally our claim follows from Lemma \ref{CStarg} with $V_i= \check{H}^1(\mathfrak{U}',\cO^{G_i-\an,\chi}_{K^p})^{G_K}$ and $W_i= \check{H}^1(\cO^{G_i-\an,\chi}_{K^p})^{G_K}$. 
\end{proof}

\begin{proof}[Proof of Lemma \ref{Lemcom} (2)]
Fix $n\geq 0$. It suffices to prove that
\[U:=\ker\left(\check{H}^1(\cO^{G_n-\an,(0,0)}_{K^p})\to H^1(\Fl,\cO^{\la,(0,0)}_{K^p})\right)\] 
is a closed subspace.

By \cite[Corollary 5.1.3.(2)]{Pan20}, there is a $\GL_2(\Q_p)$-equivariant exact sequence
\begin{eqnarray} \label{chio}
~\,~\,~\,~\,~\,~\,~\,~\, 0\to \varinjlim_{K_p\subset\GL_2(\Q_p)} H^0(\mathcal{X}_{K^pK_p},\cO_{\mathcal{X}_{K^pK_p}}) \to H^1(\Fl,\cO^{\la,(0,0)}_{K^p}) \to H^1(\Fl,\cO_{K^p})^{\la,(0,0)}\to 0
\end{eqnarray}
where $K_p$ runs through all open compact subgroups of $\GL_2(\Q_p)$. Recall that 
\[H^1(\Fl,\cO^{\la,(0,0)}_{K^p})=\varinjlim_n \check{H}^1(\cO^{G_n-\an,(0,0)}_{K^p}).\]
Let $V$ be the kernel of $\check{H}^1(\cO^{G_n-\an,(0,0)}_{K^p})\to H^1(\Fl,\cO_{K^p})^{\la,(0,0)}$. Then $V$ is closed and contains $U$ . We claim that $V/U$ is a finite-dimensional $C$-vector space. This will imply our claim by a standard argument using the open mapping theorem.

Note that $V/U$  naturally embeds into $\varinjlim_{K_p\subset\GL_2(\Q_p)} H^0(\mathcal{X}_{K^pK_p},\cO_{\mathcal{X}_{K^pK_p}})$. Since the action of $\GL_2(\Q_p)$  on $\varinjlim_{K_p\subset\GL_2(\Q_p)} H^0(\mathcal{X}_{K^pK_p},\cO_{\mathcal{X}_{K^pK_p}})$ is smooth, the action of $G_n$ on $V/U$ is also smooth. On the other hand, the action of $G_n$ on $\check{H}^1(\cO^{G_n-\an,(0,0)}_{K^p})$ is analytic. It is easy to see that $G_{n+1}$ acts trivially on $V/U$. Now $V/U$ is finite-dimensional because
\[ \left(\varinjlim_{K_p\subset\GL_2(\Q_p)} H^0(\mathcal{X}_{K^pK_p},\cO_{\mathcal{X}_{K^pK_p}})\right)^{G_{n+1}}= H^0(\mathcal{X}_{K^pG_{n+1}},\cO_{\mathcal{X}_{K^pG_{n+1}}})\]
is finite-dimensional.
\end{proof}

Now we can talk about $G_n$-analytic vectors in $H^i(\Fl,\cO^{\la,\chi}_{K^p})$. Our next result gives a control of these vectors. 

\begin{prop} \label{contrGnan}
For any $n\geq0$, the subspace of $G_n$-analytic vectors $H^i(\Fl,\cO^{\la,\chi}_{K^p})^{G_n-\an}$ is contained in  the image of  $\check{H}^i(\cO^{G_m-\an,\chi}_{K^p})\to H^i(\Fl,\cO^{\la,\chi}_{K^p})$ for some $m\geq n$.
\end{prop}

\begin{proof}
Note that $\check{H}^0(\cO^{G_m-\an,\chi}_{K^p})=H^i(\Fl,\cO_{K^p})^{G_n-\an,\chi}=H^0(\Fl,\cO^{\la,\chi}_{K^p})^{G_n-\an}$ and there is no $H^i,i\geq2$ as $\Fl$ is one-dimensional. Hence we only need to consider the case $i=1$. First we claim that $H^1(\Fl,\cO_{K^p})^{G_n-\an,\chi}$ is contained in the image of the composite map
\[f_n:\check{H}^1(\cO^{G_{m'}-\an,\chi}_{K^p}) \to H^1(\Fl,\cO^{\la,\chi}_{K^p})\to H^1(\Fl,\cO_{K^p})^{\la,\chi}\]
for some $m'\geq n$.
Indeed, it follows from the exact sequence \eqref{chio} that 
\[H^1(\Fl,\cO_{K^p})^{\la,\chi}=\bigcup_{n=0}^\infty \im(f_n).\] 
Since $H^1(\Fl,\cO_{K^p})^{G_n-\an,\chi}$ is a Banach space and $f_n$ is continuous by Remark \ref{LBstreqv}, our claim follows from Proposition \ref{1.1.10}. We denote the image of  $\check{H}^1(\cO^{G_m-\an,\chi}_{K^p})\to H^1(\Fl,\cO^{\la,\chi}_{K^p})$ by $V_m$. Note that
\[ \left(\varinjlim_{K_p\subset\GL_2(\Q_p)} H^0(\mathcal{X}_{K^pK_p},\cO_{\mathcal{X}_{K^pK_p}})\right)^{G_{m'}}= H^0(\mathcal{X}_{K^pG_{m'}},\cO_{\mathcal{X}_{K^pG_{m'}}})\]
is finite-dimensional over $C$, hence is contained in $V_m$ via \eqref{chio} for some $m\geq m'$.  Now we claim $H^1(\Fl,\cO^{\la,\chi}_{K^p})^{G_n-\an}\subseteq V_m$. Take the $G_n$-analytic vectors in \eqref{chio}.
\[0\to \left(\varinjlim_{K_p\subset\GL_2(\Q_p)} H^0(\mathcal{X}_{K^pK_p},\cO_{\mathcal{X}_{K^pK_p}})\right)^{G_n-\an} \to H^1(\Fl,\cO^{\la,\chi}_{K^p})^{G_n-\an} \to H^1(\Fl,\cO_{K^p})^{G_n-\an,\chi}.\]
For any $v\in H^1(\Fl,\cO^{\la,\chi}_{K^p})^{G_n-\an}$, we can find $v'\in V_{m'}$ such that 
\[v-v'\in \left(\varinjlim_{K_p\subset\GL_2(\Q_p)} H^0(\mathcal{X}_{K^pK_p},\cO_{\mathcal{X}_{K^pK_p}})\right)^{G_{m'}-\an}=\left(\varinjlim_{K_p\subset\GL_2(\Q_p)} H^0(\mathcal{X}_{K^pK_p},\cO_{\mathcal{X}_{K^pK_p}})\right)^{G_{m'}}.\]
Hence $v-v'\in V_m$. Thus $v=v'+(v-v')\in V_m$.
\end{proof}

\section{Intertwining operators: explicit formulas} \label{Ioef}

\subsection{Differential operators I} \label{Do1}
\begin{para}
In this subsection, we will construct an operator from $\cO^{\la,\chi}_{K^p}$ to a twist of $\omega^{2k+2,\la,\chi}_{K^p}$ for some integer $k$ depending on $\chi$. Roughly speaking, this operator comes from  \textit{differential operators on modular curves} and $\cO_{\Fl}$-linearly extends the operator $\theta^{k+1}$ in the classical theory of modular forms. Let me first recall the definition of $\theta^{k+1}$.
\end{para}

\begin{para} \label{XCY}
Fix an open compact subgroup $K_p$ of $\GL_2(\Q_p)$. Let  $\mathcal{X}=\mathcal{X}_{K^pK_p}$ and $\mathcal{C}\subset \mathcal{X}$ be the subset of cusps. There is a universal family of elliptic curves $\mathcal{E}$ over the modular curve $\mathcal{X}-\mathcal{C}$. We denote its first relative de Rham cohomology by $D_{\dR}$. This is a rank two vector bundle equipped with a Hodge filtration and a Gauss-Manin connection. Its canonical extension to $\mathcal{X}$ will be denoted by $D$. In particular, it admits a logarithmic connection
\[\nabla:D\to D\otimes_{\cO_{\mathcal{X}}}\Omega^1_{\mathcal{X}}(\mathcal{C}),\]
and a decreasing filtration 
\[\Fil^n(D)=\left\{
	\begin{array}{lll}
		D~,n\leq0\\
		\omega,~n=1\\
		0,~n\geq 2
	\end{array}.\right.\]
Here as usual, $\Omega^1_{\mathcal{X}}(\mathcal{C})$ denotes the sheaf of differentials on $\mathcal{X}$ with simple poles at $\mathcal{C}$. It is well-known that the composite map
\[\omega=\Fil^1 D\subset D\xrightarrow{\nabla}D\otimes_{\cO_{\mathcal{X}}}\Omega^1_{\mathcal{X}}(\mathcal{C})\to \gr^0 D\otimes_{\cO_{\mathcal{X}}}\Omega^1_{\mathcal{X}}(\mathcal{C})=\wedge^2 D\otimes_{\cO_{\mathcal{X}}}\omega^{-1}\otimes_{\cO_{\mathcal{X}}}\Omega^1_{\mathcal{X}}(\mathcal{C})\]
is an isomorphism (Kodaira-Spencer isomorphism): $\omega\stackrel{\sim}{\to}\wedge^2 D\otimes_{\cO_{\mathcal{X}}}\omega^{-1}\otimes_{\cO_{\mathcal{X}}}\Omega^1_{\mathcal{X}}(\mathcal{C})$. Under this isomorphism, we sometimes identify $\Omega^1_{\mathcal{X}}(\mathcal{C})$ with $(\wedge^2 D)^{-1}\otimes_{\cO_{\mathcal{X}}}\omega^2$. We also note that the cup product induces a natural isomorphism between $\wedge^2 D$ and the canonical extension of $H^2_{\dR}(\mathcal{E}/\mathcal{X})$.
Everything here is compatible under varying $K_p$.

Let $k$ be a non-negative integer. Then the $k$-th symmetric power $\Sym^k D$  also admits a logarithmic connection $\nabla_k:\Sym^kD\to \Sym^k D\otimes_{\cO_{\mathcal{X}}}\Omega^1_{\mathcal{X}}(\mathcal{C})$ and a decreasing filtration with $\Fil^{k} \Sym^k D=\omega^k$. The Kodaira-Spencer isomorphism implies that the composite map
\[\Fil^1\Sym^k D\xrightarrow{\nabla_k}\Sym^k D\otimes_{\cO_{\mathcal{X}}}\Omega^1_{\mathcal{X}}(\mathcal{C})\to (\Sym^k D/\Fil^k)\otimes_{\cO_{\mathcal{X}}}\Omega^1_{\mathcal{X}}(\mathcal{C})\]
is an isomorphism. In particular, given a section $s$ of $\gr^0 \Sym^k D=\Sym^k D/\Fil^1 \Sym^k D$, there exists a unique lift $\tilde{s}\in \Sym^k D$ such that $\nabla_k(\tilde{s})\in \Fil^k \Sym^k D\otimes_{\cO_{\mathcal{X}}}\Omega^1_{\mathcal{X}}(\mathcal{C})=\omega^k\otimes_{\cO_{\mathcal{X}}}\Omega^1_{\mathcal{X}}(\mathcal{C})$. This defines a map 
\[\theta'^{k+1}: \omega^{-k}\otimes_{\cO_{\mathcal{X}}}(\wedge^2 D)^{k}\cong \gr^0 \Sym^k D\to \omega^k\otimes_{\cO_{\mathcal{X}}}\Omega^1_{\mathcal{X}}(\mathcal{C})\cong \omega^{-k}\otimes_{\cO_{\mathcal{X}}}(\wedge^2 D)^{k}\otimes_{\cO_{\mathcal{X}}}\Omega^1_{\mathcal{X}}(\mathcal{C})^{\otimes k+1}\]
which can be viewed as a log differential operator on $ \omega^{-k}\otimes_{\cO_{\mathcal{X}}}(\wedge^2 D)^{k}$ of order $k+1$.
Note that $\wedge^2 D$ is a line bundle equipped with a logarithmic connection. If we apply this construction to $\Sym^k D\otimes_{\cO_\mathcal{X}} (\wedge^2 D)^{-k}=\Sym^k D^*$ where $D^*$ denotes the dual of $D$, we get a map
\[\theta^{k+1}:\omega^{-k}\to  \omega^k\otimes_{\cO_{\mathcal{X}}}\Omega^1_{\mathcal{X}}(\mathcal{C})\otimes_{\cO_\mathcal{X}} (\wedge^2 D)^{-k}\cong\omega^{k+2}\otimes_{\cO_{\mathcal{X}}} (\wedge^2 D)^{-k-1}.\]
Equivalently, since $\wedge^2 D$ is isomorphic to $\cO_{\mathcal{X}}$  with the standard logarithmic connection, $\theta^{k+1}$ is nothing but the twist of $\theta'^{k+1}$ on $ \omega^{-k}\otimes_{\cO_{\mathcal{X}}}(\wedge^2 D)^{-k}$ by $(\wedge^2 D)^{k}$.
When $k=0$, $\theta^1$ is simply the composite of $d:\cO_{\mathcal{X}}\to \Omega^1_{\mathcal{X}}(\mathcal{C})$ with the Kodaira-Spencer isomorphism.
\end{para}
 
\begin{para} \label{KSsm}
One can take the direct limit of previous paragraph over $K_p$ in the following way. Since the vector bundle $D$ on $\mathcal{X}=\mathcal{X}_{K^pK_p}$ is compatible with varying $K_p$,  
\[D^{\sm}_{K^p}={\pi_{\HT}}_{*} (\varinjlim_{K_p\subset \GL_2(\Q_p)}(\pi_{K_p})^{-1} D)\]
defines a locally free $\cO^{\sm}_{K^p}$-module on $\Fl$. See \ref{omegakla} for the notation here. It has a natural decreasing filtration with $\Fil^1=\omega^{1,\sm}_{K^p}$ and a connection
\[\nabla:D^{\sm}_{K^p}\to D^{\sm}_{K^p}\otimes_{\cO^\sm_{K^p}}\Omega_{K^p}^{1}(\mathcal{C})^{\sm},\]
where 
\[\Omega_{K^p}^{1}(\mathcal{C})^{\sm}={\pi_{\HT}}_{*} (\varinjlim_{K_p\subset \GL_2(\Q_p)}(\pi_{K_p})^{-1} \Omega_{\mathcal{X}}^{1}(\mathcal{C})).\]
Then we can repeat our construction in the previous paragraph. For example, the Kodaira-Spencer isomorphism induces  a natural isomorphism $\omega^{2,\sm}_{K^p}\otimes_{\cO^\sm_{K^p}}(\wedge^2 D^{\sm}_{K^p})^{-1}\cong \Omega^1_{K^p}(\mathcal{C})^{\sm}$. We note that $\wedge^2 D^{\sm}_{K^p}$ is isomorphic to $\cO^{\la}_{K^p}$, but the isomorphism will not be Hecke-equivariant.
Now we can state the main result of this subsection, which basically says that all the constructions here can be extended in an $\cO_{\Fl}$-linear way.
\end{para}
 
\begin{thm} \label{I1}
Suppose $\chi=(n_1,n_2)\in\Z^2$ and $k=n_2-n_1\geq 0$. Then there exists a unique natural continuous operator
\[d^{k+1}:\cO^{\la,\chi}_{K^p}\to \omega^{2k+2,\la,\chi}_{K^p}\otimes_{\cO^{\sm}_{K^p}} (\wedge^2 D^{\sm}_{K^p})^{-k-1}\cong \cO^{\la,\chi}_{K^p}\otimes_{\cO^{\sm}_{K^p}}(\Omega^1_{K^p}(\mathcal{C})^\sm)^{\otimes k+1}\]
satisfying the following properties:
\begin{enumerate}
\item $d^{k+1}$ is $\cO_{\Fl}$-linear;
\item $d^{k+1}(\mathrm{t}^{n_1}e_1^ie_2^{k-i} s)=\mathrm{t}^{n_1}e_1^ie_2^{k-i}  \theta^{k+1}(s)$ for any section $s\in \omega^{-k,\sm}_{K^p}$ and $i=0,1,\cdots,k$. See Theorem \ref{str} for notation here.
\end{enumerate}
\end{thm}

\begin{rem} \label{dd'}
Since $d^{k+1}$ is $\cO_{\Fl}$-linear, it can be twisted by a line bundle on $\Fl$. Let $\chi'=\chi-(0,k)=(n_1,n_2-k)=(n_1,n_1)$. If we twist by $(\omega_{\Fl})^{-k}(-k)$ and use the isomorphisms in \ref{omegakla}, we get a continuous operator
\[{d'}^{k+1}:\omega^{-k,\la,\chi'}_{K^p}\to \omega^{k+2,\la,\chi'}_{K^p}\otimes_{\cO^{\sm}_{K^p}} (\wedge^2 D^{\sm}_{K^p})^{-k-1}\]
satisfying the following properties
\begin{enumerate}
\item ${d'}^{k+1}$ is $\cO_{\Fl}$-linear;
\item ${d'}^{k+1}(\mathrm{t}^{n_1} s)=\mathrm{t}^{n_1} \theta^{k+1}(s)$ for any section $s\in \omega^{-k,\sm}_{K^p}$.
\end{enumerate}
Clearly the existence of ${d'}^{k+1}$ is equivalent with the existence of $d^{k+1}$.
\end{rem}

\begin{proof}
By Theorem \ref{str}, given an open set $U\in\mathfrak{B}$  introduced in \ref{exs}, elements of the form $\mathrm{t}^{n_1}e_1^ie_2^{k-i} s$ generate a dense $\cO_{\Fl}(U)$-submodule in $\cO^{\la,\chi}_{K^p}(U)$. In particular, $d^{k+1}$ is unique if it exists. 

First we  prove the case $k=0$. By the uniqueness we just showed, it suffices to prove the existence of $d^1$ on each $U\in\mathfrak{B}$. Fix an open set $U$ as in \ref{exs}. In particular, $e_1$ is a generator of $\omega^1_{K^p}$ on $V_\infty=\pi_{\HT}^{-1}(U)$. In general, one can use the action of $\GL_2(\Q_p)$ to reduce to this case. We will freely use the notation in Theorem \ref{str} and \ref{An}. Let $n$ be a non-negative integer. Consider
\[f=\mathrm{t}^{n_1}\sum_{i=0}^{+\infty}c_i(x-x_n)^i\in A^n,\] 
where $c_i\in H^0(V_{G_{r(n)}},\cO_{\mathcal{X}_{K^pG_{r(n)}}}),i=0,1,\cdots$ such that  $\|c_i p^{(n-1)i}\|$ are uniformly bounded. We use $d$ to denote the derivation $\cO_{V_{G_{r(n)}}}\to \Omega^1_{V_{G_{r(n)}}}(\mathcal{C})$. Define 
\[d_n(f)=\mathrm{t}^{n_1}(\sum_{i=0}^{+\infty}(x-x_n)^idc_i-\sum_{i=0}^{+\infty}(i+1)c_{i+1}(x-x_n)^idx_n).\]
Formally, this formula is obtained by applying the Leibniz rule to $c_i(x-x_n)^i$. Since derivation $d$ is continuous, it follows from our choice of $c_i$ that $p^{(n-1)i}dc_i$ are also uniformly bounded. Hence the expression of $d_n(f)$ converges to a $G_{r(n)}$-analytic vector in $\cO^{\la,\chi}_{K^p}\otimes_{\cO^{\sm}_{K^p}}\Omega^1_{K^p}(\mathcal{C})^\sm(U)$, i.e. $d_n$ defines a continuous map
\[d_n:A^n\to \cO^{\la,\chi}_{K^p}\otimes_{\cO^{\sm}_{K^p}}\Omega^1_{K^p}(\mathcal{C})^\sm(U)^{G_{r(n)}-\an}.\]
Note that if $f=\mathrm{t}^{n_1}\sum_{i=0}^l c_ix^i$ is  a finite sum, $d_n(\mathrm{t}^{n_1}\sum_{i=0}^kc_ix^i)$ is simply $\mathrm{t}^{n_1}\sum_{i=0}^kx^idc_i$. Since these finite sums are dense in $A^n$, this proves that  $d_{n+1}|_{A^{n}}=d_n$. In particular, by passing to the limit over $n$, we get a continuous map 
\[d^1:\cO^{\la,\chi}_{K^p}(U)\to \cO^{\la,\chi}_{K^p}\otimes_{\cO^{\sm}_{K^p}}\Omega^1_{K^p}(\mathcal{C})^\sm(U),\]
which is independent of the choice of $x_n$. Clearly $d^1$ commutes with restriction of $U$ to an affinoid subdomain in $\mathfrak{B}$. Hence $d^1$ defines  a map of sheaves
\[d^1:\cO^{\la,\chi}_{K^p}\to \cO^{\la,\chi}_{K^p}\otimes_{\cO^{\sm}_{K^p}}\Omega^1_{K^p}(\mathcal{C})^\sm.\]

We claim that  $d^1$ is $\cO_{\Fl}$-linear.  This is a local question on $\Fl$. Again using the action of $\GL_2(\Q_p)$, we only need to prove that $d^1:\cO^{\la,\chi}_{K^p}(U)\to \cO^{\la,\chi}_{K^p}\otimes_{\cO^{\sm}_{K^p}}\Omega^1_{K^p}(\mathcal{C})^\sm(U)$ commutes with $\cO_{\Fl}(U)$ for any rational subdomain $U$ of $U_1$. When $U=U_1=\{\|x\|\leq 1\}=\Spa(C\langle x\rangle,\cO_C\langle x\rangle)$, it follows from our construction that  $d^1$ commutes with multiplication by a polynomial of $x$ over $C$, hence commutes with $\cO_{\Fl}(U_1)$ as well by continuity. In general, if $U=U_1(\frac{f_1,\cdots,f_n}{g})$ is a rational subset of $U$, it can be proved in a similar way by noting that the image of $C\langle x\rangle[\frac{1}{g}]$ in $\cO_{\Fl}(U)$ is dense.  This finishes the proof when $k=0$.

In general, one can repeat the construction of $\theta^{k+1}$ to prove the existence of $d^{k+1}$. In fact, we will construct ${d'}^{k+1}$ in Remark \ref{dd'}. Let $\chi'=\chi-(0,k)=(n_1,n_1)$ introduced in Remark \ref{dd'}. Consider $D^{\chi'}:=D^{\sm}_{K^p}\otimes_{\cO^{\sm}_{K^p}}\cO^{\la,\chi'}_{K^p}$. The filtration on $D^{\sm}_{K^p}$ extends $\cO^{\la,\chi'}_{K^p}$-linearly to $D^{\chi'}$. In particular, 
\[\Fil^1 D^{\chi'}=\omega^{1,\sm}_{K^p}\otimes_{\cO^{\sm}_{K^p}}\cO^{\la,\chi'}_{K^p}=\omega^{1,\la,\chi'}_{K^p},\]
\[\gr^0 D^{\chi'}\cong \omega^{-1,\la,\chi'}_{K^p}\otimes_{\cO^{\sm}_{K^p}}\wedge^2 D^{\sm}_{K^p}.\]
Recall that there is a connection $\nabla:D^{\sm}_{K^p}\to D^{\sm}_{K^p}\otimes_{\cO^{\sm}_{K^p}}\Omega^1_{K^p}(\mathcal{C})^\sm$. Hence using $d^1:\cO^{\la,\chi'}_{K^p}\to \cO^{\la,\chi'}_{K^p}\otimes_{\cO^{\sm}_{K^p}}\Omega^1_{K^p}(\mathcal{C})^\sm$, we have a connection on $D^{\chi'}$:
\[\nabla^{\chi'}:D^{\chi'}\to D^{\chi'}\otimes_{\cO^{\sm}_{K^p}}\Omega^1_{K^p}(\mathcal{C})^\sm\]
such that the composite 
\[\Fil^1 D^{\chi'}\xrightarrow{\nabla^{\chi'}} D^{\chi'}\otimes_{\cO^{\sm}_{K^p}}\Omega^1_{K^p}(\mathcal{C})^\sm\to \gr^0 D^{\chi'} \otimes_{\cO^{\sm}_{K^p}}\Omega^1_{K^p}(\mathcal{C})^\sm\]
is an isomorphism. Therefore, if we apply the construction of $\theta^{k+1}$ to the $k$-th  symmetric power $\Sym^k (D^{\chi'}\otimes_{\cO^{\sm}_{K^p}}(\wedge^2 D^{\sm}_{K^p})^{-1})$ as a $\cO^{\la,(0,0)}_{K^p}$-module, or equivalently, $\Sym^k D^{*,\chi'}$, where $D^{*,\chi'}:=\Hom_{\cO^{\sm}_{K^p}}(D^{\sm}_{K^p},\cO^{\sm}_{K^p})\otimes_{\cO^{\sm}_{K^p}}\cO^{\la,\chi'}_{K^p}$, we get an operator 
\[\omega^{-k,\la,\chi'}_{K^p}\to \omega^{k+2,\la,\chi'}_{K^p}\otimes_{\cO^{\sm}_{K^p}} (\wedge^2 D^{\sm}_{K^p})^{-k-1}\]
which satisfy all the properties of ${d'}^{k+1}$ listed in Remark \ref{dd'}. By the same remark, this implies the existence of $d^{k+1}$.
\end{proof}

We record the following result obtained in the proof which will be used in Subsection \ref{PTII}.
\begin{lem} \label{KSsec}
Recall that 
\[D^{*,(0,0)}=D^{(0,0)}\otimes_{\cO^{\sm}_{K^p}}(\wedge^2 D^{\sm}_{K^p})^{-1}\cong\Hom_{\cO^{\la,(0,0)}_{K^p}}(D^{(0,0)},\cO^{\la,(0,0)})\]
denotes the dual of $D^{(0,0)}$. We normalize the filtration on $D^{*,(0,0)}$ so that $\gr^iD^{*,(0,0)}=0,~i<0$ and $\gr^0D^{*,(0,0)}=\omega^{-1,\la,(0,0)}_{K^p}$.
Then the surjective map 
\[\Sym^k D^{*,(0,0)} \to \gr^0 \Sym^k D^{*,(0,0)}\cong \omega^{-k,\la,(0,0)}_{K^p}\]
has a unique  continuous left inverse 
\[r_k: \omega^{-k,\la,(0,0)}_{K^p}\to \Sym^k D^{*,(0,0)}\]
such that $\nabla_k\circ r_k\subseteq \Fil^{k}\Sym^k D^{*,(0,0)}\otimes_{\cO^{\sm}_{K^p}} \Omega^1_{K^p}(\mathcal{C})^{\sm}$, where by abuse of notation $\nabla_k$ denotes the logarithmic connection on $\Sym^k D^{*,(0,0)}$.  

Similarly, if we twist $D^{*,(0,0)}$ by $\wedge^2 D^{\sm}_{K^p}\cong\cO^{\sm}_{K^p}$, we obtain
\[r'_k: \omega^{-k,\la,(0,0)}_{K^p}\otimes_{\cO^{\sm}_{K^p}}(\wedge^2 D^{\sm}_{K^p})^k=\gr^0 \Sym^k D^{(0,0)}\to \Sym^k D^{(0,0)}=\Sym^k D^{*,(0,0)}\otimes_{\cO^{\sm}_{K^p}}(\wedge^2 D^{\sm}_{K^p})^k\]
which is the unique left inverse of $\Sym^k D^{(0,0)}\to\gr^0 \Sym^k D^{(0,0)}$ and satisfies $\nabla'_k\circ r'_k\subseteq \Fil^{k}\Sym^k D^{(0,0)}\otimes_{\cO^{\sm}_{K^p}} \Omega^1_{K^p}(\mathcal{C})^{\sm}$. Here $\nabla'_k$ denotes the logarithmic connection on $\Sym^k D^{(0,0)}$. Explicitly, let $c$ be a global invertible section of $\wedge^2 D$ (which is necessarily horizontal). Then $r'_k(f\otimes c^k)=r_k(f)\otimes c^k$.
\end{lem}

\begin{proof}
Only the continuity of $r_k$ needs to be shown. Consider the map  
\[r:\Sym^k D^{*,(0,0)}\to  \gr^0 \Sym^k D^{*,(0,0)}\oplus \left ((\Sym^k D^{*,(0,0)}/\Fil^{k})\otimes_{\cO^{\sm}_{K^p}} \Omega^1_{K^p}(\mathcal{C})^{\sm}\right)\]
sending $f\in\Sym^k D^{*,(0,0)}$ to $(f\mod \Fil^1,\nabla_k(f)\mod \Fil^k)$. The Kodaira-Spencer isomorphism implies that this map is an isomorphism hence an isometry by the open mapping theorem. From this, it is easy to see that $r_k$ is continuous because $r\circ r_k(f)=(f,0)$.
\end{proof}

\subsection{Differential operators II} \label{DII}
\begin{para}
In this subsection, we construct an operator from $\cO^{\la,\chi}_{K^p}$ to a twist of $\omega^{-2k-2,\la,-w\cdot(-\chi)}_{K^p}$ for some weight $-w\cdot(-\chi)$ and some integer $k$ depending on $\chi$. Roughly speaking, this operator is $\cO^{\sm}_{K^p}$-linear and essentially comes from \textit{differential operators on the flag variety} $\Fl$. The construction here is much simpler than $d^{k+1}$ because  by Beilinson-Bernstein's theory of localization \cite{BB83}, differential operators on $\Fl$ can be expressed by the Lie algebra $\mathfrak{g}$.
\end{para}

\begin{para} \label{KSFl}
Denote by $\Omega^1_{\Fl}$ the sheaf of differential forms on $\Fl$. Then there is a $\GL_2$-equivariant isomorphism of line bundles on $\Fl$:
 \[\Omega^1_{\Fl}\cong \omega_{\Fl}^{-2}\otimes\det{}.\]
%Explicitly, by abuse of notation, we may view $e_1,e_2$ as a basis of $H^0(\Fl,\omega_{\Fl})$. We require this isomorphism identifies $dx=d(\frac{e_2}{e_1})$ with $\frac{\mathrm{1}}{e_1^2}$. 

Consider $\cO^{\la,(0,0)}_{K^p}$. It follows from Beilinson-Bernstein's theory that differential operators on $\Fl$ act on this sheaf. In particular, we get a natural map 
\[\bar{d}^1:\cO^{\la,(0,0)}_{K^p} \to \cO^{\la,(0,0)}_{K^p}\otimes_{\cO_{\Fl}}\Omega^1_\Fl.\]
Concretely, note that the derivation $\frac{d}{dx}$ on $\cO_{\Fl}$ agrees with the action of $u^+:=\begin{pmatrix} 0& 1 \\ 0 & 0 \end{pmatrix}\in\mathfrak{g}$. Hence on any open subset of  $\Fl$ where $e_1$ is invertible,
\[\bar{d}^1:s\in \cO^{\la,(0,0)}_{\Fl}\mapsto u^+(s)\otimes dx.\]
In other words, $\bar{d}^1$ extends $\cO^{\sm}_{K^p}$-linearly  the de Rham sequence $\cO_{\Fl}\xrightarrow{d_\Fl} \Omega^1_\Fl$ on $\Fl$, where $d_\Fl$ notes the derivation on $\Fl.$
\end{para}

\begin{para} \label{HEt}
Before moving on, we need to modify the relative Hodge-Tate filtration \eqref{rHT} to make everything Hecke-equivariant. It follows from the construction in \cite[4.1.3]{Pan20} that the $\omega^{-1}$ in \eqref{rHT} should be replaced by $\gr^1D=\wedge^2 D\otimes_{\cO_\mathcal{X}} \omega^{-1}$:
\begin{eqnarray}\label{rHTH}
0\to\wedge^2 D\otimes_{\cO_{\mathcal{X}}}\omega_{K^p}^{-1}(1)\to V(1)\otimes_{\Q_p} \cO_{\mathcal{X}_{K^p}} \to \omega_{K^p}\to 0,
\end{eqnarray}
where $\mathcal{X}=\mathcal{X}_{K^pK_p}$ is a modular curve of finite level. Implicitly in the discussion of \ref{R20}, we already fixed a trivialization $\wedge^2 D\cong \cO_{\mathcal{X}}$, equivalently a global non-vanishing section $c$ of $\wedge^2 D$. Then our choice of $c$ is that it is the pull-back of a global section of $\wedge^2 D$ on ``$\mathcal{X}_{\GL_2(\hat\Z)}$'', i.e. a $\GL_2(\hat\Z)$-fixed non-zero global section.  Recall that $\wedge^2 D$ extends $H^2_{\dR}(\mathcal{E}/\mathcal{Y})$.  It follows that $\GL_2(\A_f)$ acts on $c$ via the degree map $|\cdot|_{\A}^{-1}\circ \det$, where $|\cdot|_\A:\A^\times_f\to \Q^\times$ denotes the usual adelic norm map on finite ideles.
If we take the wedge product of $V(1)\otimes_{\Q_p} \cO_{\mathcal{X}_{K^p}}$, cf. \ref{exs}, we get a natural isomorphism 
\[\cO_{\mathcal{X}_{K^p}}\cong (\wedge^2 D)^{-1}\otimes_{\cO_\mathcal{X}} \cO_{\mathcal{X}_{K^p}}\otimes \det(1).\]
The element $\mathrm{t}$ introduced in \ref{exs} should be regarded as the image of $c^{-1}\otimes 1\otimes b$ under this isomorphism, i.e. $\mathrm{t}=c^{-1}\otimes 1\otimes b$. We remind the reader that $b\in\Q_p(1)$ is our choice of a basis. Hence we have the following natural isomorphism on $\Fl$:
\begin{eqnarray} \label{rmtwddet1}
\cO^{\la,(1,1)}_{K^p}=\cO^{\la,(0,0)}_{K^p}\cdot \mathrm{t}\cong \cO^{\la,(0,0)}_{K^p}\otimes_{\cO_{K^p}^{\sm}}(\wedge^2 D^{\sm}_{K^p})^{-1}\otimes \det(1),
\end{eqnarray}
which is independent of the choice of $\mathrm{t}$. Recall that we have the Kodaira-Spencer isomorphism in \ref{KSsm}: $\omega^{2,\sm}_{K^p}\otimes_{\cO^\sm_{K^p}}(\wedge^2 D^{\sm}_{K^p})^{-1}\cong \Omega^1_{K^p}(\mathcal{C})^{\sm}$. Therefore, it follows isomorphisms in \ref{omegakla} and \ref{KSFl} that there is a natural isomorphism
\begin{eqnarray} \label{period}
\Omega^1_{K^p}(\mathcal{C})^{\sm}\otimes_{\cO^{\sm}_{K^p}}\cO^{\la,(0,0)}_{K^p}\otimes_{\cO_{\Fl}}\Omega^1_{\Fl}\cong \cO^{\la,(1,-1)}_{K^p}(1).
\end{eqnarray}
Explicitly, we require 
\[s\otimes1\otimes dx\in \Omega^1_{K^p}(\mathcal{C})^{\sm}\otimes_{\cO^{\sm}_{K^p}}\cO^{\la,(0,0)}_{K^p}\otimes_{\cO_{\Fl}}\Omega^1_{\Fl}\mapsto\frac{s\mathrm{t}}{e_1^2}\in \cO^{\la,(1,-1)}_{K^p}(1)\] 
on an open set where $e_1$ is invertible. Here we view $s$ as a section of $\omega^{2,\sm}_{K^p}$ using the Kodaira-Spencer isomorphism and our choice of generator of $\wedge^2 D$. It is easy to check that $\frac{s\mathrm{t}}{e_1^2}$ does not depend on the choice of the trivializations of $\wedge^2 D$ and $\Q_p(1)$.
\end{para}

\begin{para} \label{dbar}
For a non-negative integer $k$, it is well-known that ``$(k+1)$-th derivation'' defines a $\GL_2$-equivariant map
\[\omega_\Fl^{k}\to\omega_\Fl^k\otimes_{\cO_\Fl}(\Omega_{\Fl}^1)^{k+1}\cong \omega_{\Fl}^{-k-2}\otimes \det{}^{k+1}\]
sending $s\in \omega_\Fl^{k}$ to $(u^+)^{k+1}(s)\otimes (dx)^{k+1}$. The operator we are going to construct essentially $\cO^{\sm}_{K^p}$-linearly extends this map up to a constant. We follow the idea of  Bernstein-Gelfand-Gelfand (BGG). Let $U(\mathfrak{g})$ be the universal enveloping algebra of $\mathfrak{g}$ and $Z(U(\mathfrak{g}))$ be its centre. Then $Z(U(\mathfrak{g}))$ acts via the same characters on 
\[\cO^{\la,(0,0)}_{K^p}\otimes _{\cO_{\Fl}}\omega_{\Fl}^{k}\cong  \omega^{k,\la,(0,k)}_{K^p}(-k)=(\omega^{1,\la,(0,1)}_{K^p}(-1))^{\otimes k}\]
 and 
\begin{eqnarray*}
\cO^{\la,(0,0)}_{K^p}\otimes _{\cO_{\Fl}}\omega_\Fl^k\otimes_{\cO_\Fl}(\Omega_{\Fl}^1)^{\otimes k+1} &\cong& \omega^{-k,\la,(0,-k)}_{K^p}\otimes_{\cO_{\Fl}} \Omega^1_{\Fl} \otimes \det{}^{k}(k) \\
&\xrightarrow[\times \mathrm{t}^k]{\simeq} &(\wedge^2 D^{\sm}_{K^p})^{k}\otimes_{\cO^{\sm}_{K^p}}\omega_{K^p}^{-k,\la,(k,0)}\otimes  \Omega^1_{\Fl},
\end{eqnarray*}
which will be denote by $\lambda_k$.  The last isomorphism uses the $k$-th power of \eqref{rmtwddet1}. Now we  take the locally analytic vectors  of the push forward of  the sequence \eqref{rHTH} along $\pi_{\HT}$ and twist by $(-1)$
\[0\to \wedge^2 D^{\sm}_{K^p}\otimes_{\cO^{\sm}_{K^p}}\omega^{-1,\la}_{K^p}\to V \otimes_{\Q_p} \cO_{K^p}^{\la}\to \omega_{K^p}^{1,\la}(-1)\to 0\]
Restricting the middle term to $V\otimes_{\Q_p} \cO_{K^p}^{\la,(0,0)}$ induces
\begin{eqnarray} \label{rHTcc}
0\to \wedge^2 D^{\sm}_{K^p}\otimes_{\cO^{\sm}_{K^p}}\omega^{-1,\la,(1,0)}_{K^p}\to V\otimes_{\Q_p} \cO_{K^p}^{\la,(0,0)}\to \omega_{K^p}^{1,\la,(0,1)}(-1)\to 0,
\end{eqnarray}
which is essentially the same as the tensor product of $\cO^{\la,(0,0)}_{K^p}$ and $0\to \omega^{-1}_{\Fl}\otimes \det\to V\otimes_{\Q_p}\cO_{\Fl}\to\omega_{\Fl}\to 0$. 

By Harish-Chandra's work and the relation between $\theta_\kh$ and infinitesimal characters \cite[Corollary 4.2.8]{Pan20}, the weight-$(n_1,n_2)$ part and  weight-$(k_1,k_2)$ part of $\cO^{\la}_{K^p}$ have the same infinitesimal characters if and only if $n_1+n_2=k_1+k_2$ and $(n_1-n_2-1)^2=(k_1-k_2-1)^2$. Hence the $\lambda_1$-isotypic component of $V\otimes_{\Q_p}\cO_{K^p}^{\la,(0,0)}$ is canonically isomorphic to $\omega_{K^p}^{1,\la,(0,1)}(-1)$ since the infinitesimal character of $\omega^{-1,\la,(1,0)}_{K^p}$ is different. Similarly, if we twist this exact sequence by $\Omega^1_{\Fl}$, we see that  the $\lambda_1$-isotypic component of $V\otimes_{\Q_p}\cO_{K^p}^{\la,(0,0)}\otimes_{\cO_{\Fl}}\Omega^1_\Fl$ is naturally identified with $\wedge^2 D^{\sm}_{K^p}\otimes_{\cO^{\sm}_{K^p}}\omega^{-1,\la,(1,0)}_{K^p}\otimes_{\cO_\Fl}\Omega^1_\Fl$. Therefore, if we take the $\lambda_1$-isotypic component of the tensor product of $V$ and $\bar{d}^1:\cO^{\la,(0,0)}_{K^p} \to \cO^{\la,(0,0)}_{K^p}\otimes_{\cO_{\Fl}}\Omega^1_\Fl$, we get 
\[\bar{d}^2:\omega_{K^p}^{1,\la,(0,1)}(-1)\to \wedge^2 D^{\sm}_{K^p}\otimes_{\cO^{\sm}_{K^p}}\omega^{-1,\la,(1,0)}_{K^p}\otimes_{\cO_\Fl}\Omega^1_\Fl\cong \omega_{K^p}^{1,\la,(0,1)}(-1)\otimes_{\cO_\Fl}(\Omega_{\Fl}^1)^{\otimes 2}.\]
A direct calculation shows that this operator agrees with the explicit formula $s\in \omega_{K^p}^{1,\la,(0,1)}\mapsto(u^+)^{2}(s)\otimes (dx)^{2}$ up to a non-zero constant in $\Q$ on any open subset of $\Fl$ where $e_1$ is invertible, or equivalently $x$ is a regular function. Clearly, $\bar{d}^2$ is $\cO^{\sm}_{K^p}$-linear, hence it can be twisted by locally free $\cO^{\sm}_{K^p}$-modules. By abuse of notation, we still denote them by $\bar{d}^2$. For example, the twist by $\omega^{-1,\sm}_{K^p}(1)$ is
\[\bar{d}^2:\cO^{\la,(0,1)}_{K^p}\to \cO_{K^p}^{\la,(0,1)}\otimes_{\cO_\Fl}(\Omega_{\Fl}^1)^{\otimes 2}\xrightarrow[\times \mathrm{t}^{-2}]{\simeq}(\wedge^2 D^{\sm}_{K^p})^{2}\otimes_{\cO^{\sm}_{K^p}}\omega^{-4,\la,(2,-1)}_{K^p}(2).\]
\end{para}

\begin{para} \label{genklambdak}
For general $k\geq1$, one can repeat the same construction with $V$ replaced by its $k$-th symmetric power and $\lambda_1$ replaced by $\lambda_k$.  It follows from the sequence \eqref{rHTcc} that there are a natural surjective map
\[\phi_{k,1}:\Sym^k V\otimes_{\Q_p} \cO_{K^p}^{\la,(0,0)}\to \omega_{K^p}^{k,\la,(0,k)}(-k)\]
and a natural injective map
\[\phi_{k,2}:\omega_{K^p}^{k,\la,(0,k)}(-k)\otimes_{\cO_\Fl}(\Omega_{\Fl}^1)^{\otimes k+1}\to
\Sym^k V\otimes_{\Q_p} \cO_{K^p}^{\la,(0,0)}\otimes_{\cO_\Fl}\Omega_{\Fl}^1.\]
\end{para}

\begin{lem} \label{secinf}
Both $\phi_{k,1}$ and $\phi_{k,2}$ induce isomorphisms on the $\lambda_k$-isotypic components. In particular, $\phi_{k,1}$ has a natural $U(\mathfrak{g})$-equivariant left inverse
\[\phi'_{k}: \omega_{K^p}^{k,\la,(0,k)}(-k)
\to \Sym^k V\otimes_{\Q_p} \cO_{K^p}^{\la,(0,0)}.\]
\end{lem}

\begin{proof}
For $\phi_{k,1}$,
we note that \eqref{rHTcc} implies that $\Sym^k V\otimes_{\Q_p} \cO_{K^p}^{\la,(0,0)}$ is filtered by 
\[ (\wedge^2 D^{\sm}_{K^p})^{\otimes i}\otimes_{\cO^{\sm}_{K^p}}\omega_{K^p}^{k-2i,\la,(i,k-i)}(i-k),~i=0,\cdots,k\]
and only $\omega_{K^p}^{k,\la,(0,k)}(-k)$ contributes to the $\lambda_k$-isotypic component. The same argument works for $\phi_{k,2}$.
\end{proof}

As a consequence, we get
\[\bar{d}^{k+1}:\omega_{K^p}^{k,\la,(0,k)}(-k)\cong \cO^{\la,(0,0)}_{K^p}\otimes_{\cO_{\Fl}}\omega^k_{\Fl}\to  \omega_{K^p}^{k,\la,(0,k)}(-k)\otimes_{\cO_\Fl}(\Omega_{\Fl}^1)^{\otimes k+1},\]
which agrees with the explicit formula $s\in \omega_{K^p}^{k,\la,(0,k)}\mapsto(u^+)^{k+1}(s)\otimes (dx)^{k+1}$ up to a non-zero constant in $\Q$ by a direct calculation. It is $\cO^{\sm}_{K^p}$-linear and its twist by $\omega^{k,\sm}_{K^p}(k)$ becomes
\[\bar{d}^{k+1}:\cO^{\la,(0,k)}_{K^p}\to\cO_{K^p}^{\la,(0,k)}\otimes_{\cO_\Fl}(\Omega_{\Fl}^1)^{\otimes  k+1}\xrightarrow[\times \mathrm{t}^{-k-1}]{\simeq}(\wedge^2 D^{\sm}_{K^p})^{k+1}\otimes_{\cO^{\sm}_{K^p}}\omega^{-2k-2,\la,(k+1,-1)}_{K^p}(k+1).\]
It is clear from the construction that $\bar{d}^{k+1}$ is $\GL_2(\Q_p)$ and Hecke-equivariant. Since we are free to multiply by powers of $\mathrm{t}$, we obtain the following theorem.

\begin{thm} \label{I2}
Suppose $\chi=(n_1,n_2)\in\Z^2$ and $k=n_2-n_1\geq 0$. Let $-w\cdot(-\chi)=(n_2+1,n_1-1)$. Then there exists a natural continuous operator
\[\bar{d}^{k+1}:\cO^{\la,\chi}_{K^p}\to (\wedge^2 D^{\sm}_{K^p})^{k+1}\otimes_{\cO^{\sm}_{K^p}}\omega^{-2k-2,\la,-w\cdot(-\chi)}_{K^p}(k+1)\cong \cO^{\la,\chi}_{K^p}\otimes_{\cO_{\Fl}}(\Omega^1_{\Fl})^{\otimes k+1}\]
which formally can be regarded as $(d_\Fl)^{k+1}$
satisfying the following properties:
\begin{enumerate}
\item $\bar{d}^{k+1}$ is $\cO^{\sm}_{K^p}$-linear;
\item there exists a non-zero constant $c\in\Q$ such that $\bar{d}^{k+1}(s)=c(u^+)^{k+1}(s)\otimes (dx)^{k+1}$ for any $s\in \cO^{\la,\chi}_{K^p}$ defined over an open subset of $\Fl$ where $x$ is a regular function. 
\end{enumerate}
Moreover, $\bar{d}^{k+1}$ commutes with $\GL_2(\Q_p)$ and Hecke actions away from $p$.
\end{thm}

\begin{rem} \label{ddbarcom}
It's natural to compare the differential operator $d^{k+1}$ introduced in Theorem \ref{I1} with $\bar{d}^{k+1}$. It follows from the 
 explicit formulae that $\bar{d}^{k+1}$ commutes with $d^{k+1}$.
In fact, both differential operators can be defined in a uniform way. This is probably best explained in the local picture,
See Subsections \ref{LTinfty}, \ref{Drinfty} below.
\end{rem}

\begin{prop} \label{dbarsurj}
Suppose $\chi=(0,k)\in\Z^2$ and $k\geq 0$. We have the following results for
\[\bar{d}^{k+1}:\cO^{\la,(0,k)}_{K^p}\to \cO^{\la,(0,k)}_{K^p}\otimes_{\cO_{\Fl}}(\Omega^1_{\Fl})^{\otimes k+1}.\]
\begin{enumerate}
\item $\bar{d}^{k+1}$ is surjective.
\item $\ker(\bar{d}^{k+1})=\cO^{\lalg,(0,k)}_{K^p}=\Sym^k V(k)\otimes_{\Q_p}\omega^{-k,\sm}_{K^p}$. See \ref{lalg0k}, \ref{lalgKp} for the notation here.
\end{enumerate}
In general, if $\chi=(n_1,n_2)\in\Z^2$ with $k=n_2-n_1\geq 0$, then the same results hold: $\ker(\bar{d}^{k+1})=\cO^{\lalg,\chi}_{K^p}=\Sym^k V(k)\otimes_{\Q_p}\omega^{-k,\sm}_{K^p}\cdot \mathrm{t}^{n_1}$. 
\end{prop}

\begin{proof}
Since $\bar{d}^{k+1}$ is $\GL_2(\Q_p)$-equivariant, it suffices to prove these claims on open sets on which $e_1\in\omega^1_{K^p}$ is a generator. Then both claims follow from the explicit formula $\bar{d}^{k+1}=(u^+)^{k+1}\otimes (dx)^{k+1}$ up to a non-zero constant. The first part follows from \cite[Proposition 5.2.9]{Pan20} which says that $u^+$ is surjective on such open sets . For the second part, since $\ker(\bar{d}^{k+1})$ is $\mathfrak{gl}_2(\Q_p)$-stable and annihilated by $(u^+)^{k+1}$, the action of $\mathfrak{gl}_2(\Q_p)$ on it is locally finite. A consideration of the infinitesimal character shows that $\ker(\bar{d}^{k+1})$ is exactly the $\Sym^k V$-isotypic subspace. Hence $\ker(\bar{d}^{k+1})=\cO^{\lalg,(0,k)}_{K^p}$. One can also prove this by a direct calculation. 
\end{proof}

\begin{rem}\label{dbarres}
When $k=0$, it follows that there is a natural exact sequence
\[0\to \cO^{\sm}_{K^p} \to \cO^{\la,(0,0)}_{K^p}\xrightarrow{\bar{d}}  \cO^{\la,(0,k)}_{K^p}\otimes_{\cO_{\Fl}}\Omega^1_{\Fl}\to 0\]
which can be regarded as a $\bar{d}$-resolution of $\cO^{\sm}_{K^p}$ in this setting.  In general, $\bar{d}^{k+1}$ gives rise to a resolution
\[0\to \Sym^k V(k)\otimes_{\Q_p}\omega^{-k,\sm}_{K^p}\to \cO^{\la,(0,k)}_{K^p}\xrightarrow{\bar{d}^{k+1}} \cO^{\la,(0,k)}_{K^p}\otimes_{\cO_{\Fl}}(\Omega^1_{\Fl})^{\otimes k+1}\to 0.\]

\end{rem}

\subsection{Intertwining operator}
\begin{defn} \label{I}
Suppose $\chi=(n_1,n_2)\in\Z^2$ and $k=n_2-n_1\geq 0$.  We define 
\[I_k:\cO^{\la,\chi}_{K^p}\to \cO^{\la,-w\cdot(-\chi)}_{K^p}(k+1)\]
as  the composite of 
\[\cO^{\la,\chi}_{K^p}\xrightarrow{d^{k+1}}  \omega^{2k+2,\la,\chi}_{K^p}\otimes_{\cO^{\sm}_{K^p}} (\wedge^2 D^{\sm}_{K^p})^{-k-1} \xrightarrow{\bar{d}'^{k+1}\otimes 1} \cO^{\la,-w\cdot(-\chi)}_{K^p}(k+1),\]
where $\bar{d}'^{k+1}:=\bar{d}^{k+1}\otimes 1: \omega^{2k+2,\la,\chi}_{K^p}\otimes_{\cO^{\sm}_{K^p}} (\wedge^2 D^{\sm}_{K^p})^{-k-1}\to \cO^{\la,-w\cdot(-\chi)}_{K^p}(k+1)$.
\end{defn}
Note that $\bar{d}'^{k+1}$ makes sense because $\bar{d}^{k+1}$ is $\cO^{\sm}_{K^p}$-linear by Theorem \ref{I2}.
Formally, $I_k=\bar{d}'^{k+1}\circ d^{k+1}$ and  commutes with $\GL_2(\Q_p)$ and Hecke actions away from $p$. Note that $d^{k+1}$ commutes with $\bar{d}^{k+1}$ by \ref{ddbarcom}. Hence $I_k$ also agrees with the composite:
\begin{eqnarray*}
\cO^{\la,\chi}_{K^p}\xrightarrow{\bar{d}^{k+1}} \cO^{\la,\chi}_{K^p}\otimes_{\cO_{\Fl}}(\Omega^1_{\Fl})^{\otimes k+1}&\xrightarrow{d^{k+1}\otimes 1} & 
(\Omega_{K^p}^1(\mathcal{C})^\sm)^{\otimes k+1}\otimes_{\cO^{\sm}_{K^p}}\cO^{\la,\chi}_{K^p}\otimes_{\cO_{\Fl}}(\Omega^1_{\Fl})^{\otimes k+1}\\
&\cong &\cO^{\la,-w\cdot(-\chi)}_{K^p}(k+1),
\end{eqnarray*}
where the last isomorphism follows from \eqref{period}. Hence $I_k$ can also be regarded as $d^{k+1}\circ\bar{d}^{k+1}$.

\begin{rem}
We will give a $p$-adic Hodge-theoretic interpretation of $I_k$  in the next section. Roughly speaking, it comes from the action of $\mathbb{G}_a$ (a monodromy operator) in Fontaine's classification of almost de Rham $B_{\dR}$-representations \cite{Fo04}. The Tate twist $(k+1)$ in the definition is better to be understood as $\gr^k B_{\dR}$.
\end{rem}

\begin{defn} \label{I^1k}
 We denote $H^1(I_k)$ by
 \[I^1_k:H^1(\Fl,\cO^{\la,\chi}_{K^p})\to H^1(\Fl,\cO^{\la,-w\cdot(-\chi)}_{K^p}(k+1)).\]
 \end{defn}

 \begin{para}[A variant of $I_k$]  \label{VIk}
If $k>0$, then $H^1(\Fl,\cO^{\la,\chi}_{K^p})=H^1(\Fl,\cO_{K^p})^{\la,\chi}$ and similar results hold for $H^1(\Fl,\cO^{\la,-w\cdot(-\chi)}_{K^p}(k+1))$ by \cite[Corollary 5.1.3.(2)]{Pan20}. Thus we can also view $I^1_k$ as a map
\[I^{1}_k:H^1(\Fl,\cO_{K^p})^{\la,\chi} \to H^1(\Fl,\cO_{K^p})^{\la,-w\cdot(-\chi)}(k+1).\]
If $k=0$, it follows from the exact sequence \eqref{chio} and Remark \ref{A00rinv} below that $I^1_0$ factors through the quotient $H^1(\Fl,\cO_{K^p})^{\la,\chi}$ of $H^1(\Fl,\cO^{\la,\chi}_{K^p})$ and we will denote by 
\[{I}'_0:H^1(\Fl,\cO_{K^p})^{\la,\chi} \to H^1(\Fl,\cO_{K^p})^{\la,-w\cdot(-\chi)}(1)\]
the induced map.
 \end{para}

\subsection{\texorpdfstring{$\kn$}{Lg}-cohomology}
\begin{para} \label{I1k}
Fix a non-negative integer $k$ and a weight $\chi=(n_1,n_2)\in\Z^2$ with $n_2-n_1=k$ throughout this subsection. Let $\kn=Cu^+=\{\begin{pmatrix} 0 & * \\ 0 & 0 \end{pmatrix}\}\subset \mathfrak{g}$ be the upper triangular nilpotent subalgebra. We will determine the $\kn$-invariants of $\ker I^1_k$ below as a representation of the upper-triangular  Borel subgroup $B$ of $\GL_2(\Q_p)$ and deduce the finiteness result mentioned in the introduction Theorem \ref{Intfin}. We note that the main result will not be used in the next two sections. People who are only interested in the spectral decomposition can skip this subsection at first.

It suffices to study the case when $n_1=0$, i.e. $\chi=(0,k)$ as the general case can be reduced to this case by multiplying by $\mathrm{t}^{-n_1}$. Thus we will simply assume  $n_1=0$
 from now on.

Recall that  $d^{k+1}: \cO^{\la,\chi}_{K^p}\to\omega^{2k+2,\la,\chi}_{K^p}\otimes_{\cO^{\sm}_{K^p}} (\wedge^2 D^{\sm}_{K^p})^{-k-1}$. 
\end{para}

\begin{lem} \label{kerdk+1}
\[\ker I^1_k=\ker \left(H^1(d^{k+1})\right).\]
\end{lem}

\begin{proof}
$I_k=\bar{d}'^{k+1}\circ d^{k+1}$ (Definition \ref{I}). It suffices to prove that the kernel of 
\[H^1(\bar{d}'^{k+1}):  H^1(\Fl,\omega^{2k+2,\la,\chi}_{K^p}\otimes_{\cO^{\sm}_{K^p}} (\wedge^2 D^{\sm}_{K^p})^{-k-1})\to H^1(\Fl,\cO^{\la,-w\cdot(-\chi)}_{K^p}(k+1))\]
is zero. Since $\bar{d}^{k+1}\otimes 1$ is surjective by Proposition \ref{dbarsurj}, we only need to show that 
\[H^1(\Fl,\ker\bar{d}'^{k+1})=0.\]
On the other hand, by the second part of Proposition \ref{dbarsurj} and the isomorphism $\omega^{2k+2,\la,\chi}_{K^p}\cong \cO^{\la,\chi}_{K^p}\otimes_{\cO^{\sm}_{K^p}}\omega^{2k+2,\sm}_{K^p}$, we have
\[\ker(\bar{d}^{k+1}\otimes 1)\cong\Sym^k V(k)\otimes_{\Q_p} \omega^{k+2,\la,\chi}_{K^p}\otimes_{\cO^{\sm}_{K^p}} (\wedge^2 D^{\sm}_{K^p})^{-k-1}.\]
Since $\wedge^2 D^{\sm}_{K^p}\cong\cO^{\sm}_{K^p}$ (non-canonically), it is enough to prove 
\[H^1(\Fl,\omega^{k+2,\sm}_{K^p})=0,\]
which is exactly \cite[Corollary 5.3.6]{Pan20} as  $k\geq 0$.
\end{proof}

\begin{para}
We are going to compute $(\ker I^1_k)^\kn$, the $\kn$-invariants of $\ker I^1_k$. This is a subspace of the $\kn$-invariants $H^1(\Fl,\cO^{\la,\chi}_{K^p})^\kn$, which was determined in \S 5.3 of \cite{Pan20}. We briefly recall the computation here. One ingredient is the following simple lemma.
\end{para}

\begin{lem} \label{ninvlem}
Let $\mathcal{F},\mathcal{G}$ be sheaves of abelian groups on $\Fl$ (or any topological space with cohomological dimension one) and $X:\mathcal{F}\to\mathcal{G}$ be a map. Then  there are natural homomorphisms $r:H^0(\Fl,\mathcal{G})\to H^0(\Fl,\coker X)$ and $\delta:\ker r =H^0(\Fl,X(\cF))\to H^1(\Fl,\ker X)$, and the following natural short exact sequence:
\[0\to \coker(\delta) \to \ker H^1(X) \to \coker(r)\to 0.\]
In particular, 
\begin{enumerate}
\item If $H^1(\Fl,\ker X)=0$, then $\ker H^1(X)\cong \coker(r)$.
\item If $H^0(\Fl,\mathcal{G})=0$, then $H^0(\Fl,X(\cF))=0$ and $0\to H^1(\Fl,\ker X)\to \ker H^1(X)\to H^0(\Fl,\coker X)\to 0$.
\end{enumerate}
\end{lem}

\begin{proof}
The map $r:H^0(\Fl,\mathcal{G})\to H^0(\Fl,\coker X)$ simply comes from the quotient map $\mathcal{G}\to\coker(X)$. For $\delta$, consider the short exact sequence $0\to \ker(X)\to\cF\to X(\cF) \to 0$ and the long exact sequence of its cohomology groups:
\[H^0(\Fl,X(\cF))\xrightarrow{\delta} H^1(\Fl,\ker X)\to H^1(\Fl,\cF) \to H^1 (\Fl,X(\cF))\to 0.\]
The first connecting homomorphism will be our $\delta$. Hence $\coker(\delta)$ is isomorphic to the  kernel of $H^1(\Fl,\cF) \to H^1 (\Fl,X(\cF))$. On the other hand, consider the cohomology groups of $0\to X(\cF)\to \mathcal{G} \to \coker X\to 0$:
\[H^0(\Fl,\mathcal{G})\xrightarrow{r} H^0(\Fl,\coker X)\to H^1(\Fl,X(\cF))\to H^1(\Fl,\mathcal{G}).\]
Hence $\coker(r)$ is isomorphic to the  kernel of $H^1(\Fl,X(\cF))\to H^1(\Fl,\mathcal{G})$. By writing $H^1(X)$ as the composite map $H^1(\Fl,\cF)\to H^1(\Fl,X(\cF))\to H^1(\Fl,\mathcal{G})$, 
we get $0\to \coker(\delta) \to \ker(H^1(X)) \to \coker(r)$. Since $H^1(\Fl,\cF) \to H^1 (\Fl,X(\cF))$ is surjective, the last map is surjective as well.
\end{proof}

\begin{rem} \label{hyclem}
Here is a more direct way to prove this lemma.
If we denote the hypercohomology of the complex $\mathcal{F}\xrightarrow{X}\mathcal{G}$ by $\bH^*(X)$, then there is natural exact sequence.
\[H^0(\Fl,\mathcal{F})\xrightarrow{H^0(X)} H^0(\Fl,\mathcal{G}) \to \bH^1(X)\to \ker (H^1(X))\to 0.\]
Hence
\[ \ker (H^1(X))\cong  \bH^1(X)/\Fil^1\bH^1(X),\]
where $\Fil^1\bH^1(X)$ denotes the image of $H^0(\Fl,\mathcal{G})$ in $\bH^1(X)$.
On the other hand, since there is no $H^2$ on $\Fl$, we have another exact sequence
\[0\to H^1(\Fl,\ker X)\to \bH^1(X)\to H^0(\Fl,\coker X)\to 0.\]
The composite map $H^0(\Fl,\mathcal{G}) \to \bH^1(X)\to H^0(\Fl,\coker X)$ is nothing but $r$. From this we easily deduce the lemma.
\end{rem}

\begin{para} \label{e_1e_2ch}
We apply this lemma to $\cF=\mathcal{G}=\cO^{\la,\chi}_{K^p}$ and $X=u^+$, it remains to determine the $\kn$-invariants $\cO^{\la,\chi,\kn}_{K^p}$, the coinvariants $(\cO^{\la,\chi}_{K^p})_{\kn}$ and their cohomology.  This was computed in \S 5.2, 5.3 of \cite{Pan20}. To state the result, we denote by $\infty\in \Fl$ the point where $e_1$ vanishes and by $i_{\infty}:\infty\to \Fl$ the natural embedding. Following \cite[Definition 5.2.5, 5.3.8]{Pan20}, we define 
\begin{eqnarray*}
M^\dagger_k(K^p)&:=&(\omega^{k,\sm}_{K^p})_\infty \mbox{, the stalk of } \omega^{k,\sm}_{K^p} \mbox{ at }\infty,\\
M_k(K^p)&:=&H^0(\Fl,\omega^{k,\sm}_{K^p})=\varinjlim_{K_p\subset\GL_2(\Q_p)}H^0(\mathcal{X}_{K^pK_p},\omega^k).
\end{eqnarray*}
$M^\dagger_k(K^p)$ (resp. $M_k(K^p)$) is our space of overconvergent modular forms (resp. classical modular forms) of weight $k$ of tame level $K^p$. Then there is a natural  restriction map $M_k(K^p)\to M^\dagger_k(K^p)$. Clearly $M_k(K^p)=0$ if $k<0$.

Following \cite[Remark 5.2.7]{Pan20}, for a $B\times G_{\Q_p}$-representation $W$ and integers $i,j$, we use $W\cdot e_1^ie_2^j$ to denote the twist of $W(i+j)$ by the character sending $\begin{pmatrix} a & b\\ 0 & d \end{pmatrix}\in B$ to $a^id^j$. This agrees with the definition in \cite[Remark 5.2.7]{Pan20}. 
%by our choice of basis of $\Q_p(1)$.
\end{para}

\begin{prop} \label{kninvcoinv}
$\chi=(0,k)\in\Z^2$ and $k\geq 0$.
\begin{enumerate}
\item There is a natural isomorphism 
\begin{eqnarray*}
\omega^{-k,\sm}_{K^p}\cdot e_1^k&\cong&\cO^{\la,\chi,\kn}_{K^p} \\
s& \mapsto&e_1^ks;
\end{eqnarray*}
\item $(\cO^{\la,\chi}_{K^p})_{\kn}$ is a sky-scrapper sheaf supported at $\infty$ and there is a natural isomorphism 
\begin{eqnarray*}
(i_\infty)_* M^\dagger_{-k}(K^p)\cdot e_2^{k} \oplus (i_\infty)_* M^\dagger_{-k}(K^p)\cdot e_1^{1+k}e_2^{-1} &\cong& (\cO^{\la,\chi}_{K^p})_{\kn} \\
(s_1,s_2)&\mapsto& e_2^{k}s_1+ e_1^{1+k}e_2^{-1}s_2.
\end{eqnarray*}
\item $H^0(\Fl,\cO^{\la,\chi}_{K^p})=0$ unless $k=0$. If $k=0$, then there is a natural isomorphism
\begin{eqnarray*}
H^0(\Fl,\cO^{\sm}_{K^p})&\cong&H^0(\Fl,\cO^{\la,\chi}_{K^p}) 
\end{eqnarray*}
induced from the natural inclusion $\cO^{\sm}_{K^p}\subset \cO^{\la,\chi}_{K^p}$
In particular, $H^0(\Fl,u^+(\cO^{\la,\chi}_{K^p}))=0$.
\end{enumerate}
\end{prop}

\begin{proof}
See Proposition 5.2.6 ,5.2.9, 5.2.11, Corollary 5.1.3 of \cite{Pan20}.
\end{proof}

We note that all these isomorphisms are $\cO^{\sm}_{K^p}$-linear.

\begin{cor} \label{gso2k2}
$\chi=(0,k)\in\Z^2$ and $k\geq 0$. There are natural isomorphisms:
\begin{enumerate}
\item $(\omega^{2k+2,\la,\chi}_{K^p})^{\kn} \cong \omega^{k+2,\sm}_{K^p}\cdot e_1^k$.
\item $(\omega^{2k+2,\la,\chi}_{K^p})_{\kn}\cong (i_\infty)_* M^\dagger_{k+2}(K^p)\cdot  e_2^{k} \oplus (i_\infty)_* M^\dagger_{k+2}(K^p)\cdot e_1^{1+k}e_2^{-1} $.
\item $H^0(\Fl,\omega^{2k+2,\la,\chi}_{K^p})\cong \Sym^k V(k) \otimes_{\Q_p} M_{k+2}(K^p)$
 induced from the natural inclusion $\Sym^k V(k)\otimes_{\Q_p}\omega^{-k,\sm}_{K^p}\subseteq \cO^{\la,\chi}_{K^p}$, cf. \ref{lalg0k}.
\end{enumerate}
\end{cor}

\begin{proof}
The first two claims follow from Proposition \ref{kninvcoinv} and the $\cO^{\sm}_{K^p}$- linear isomorphism $\omega^{2k+2,\la,\chi}_{K^p}\cong \cO^{\la,\chi}_{K^p}\otimes_{\cO^{\sm}_{K^p}}\omega^{2k+2,\sm}_{K^p}$. For the last part, consider 
\[\bar{d}'^{k+1}:\omega^{2k+2,\la,\chi}_{K^p}\to  (\wedge^2 D^{\sm}_{K^p})^{k+1}\otimes_{\cO^{\sm}_{K^p}}\cO^{\la,-w\cdot(-\chi)}_{K^p}(k+1)\cong \cO^{\la,-w\cdot(-\chi)}_{K^p}(k+1).\]
By \cite[Corollary 5.1.3]{Pan20}, we have  $H^0(\Fl,\cO^{\la,-w\cdot(-\chi)}_{K^p})=0$. Hence 
\[H^0(\Fl,\omega^{2k+2,\la,\chi}_{K^p})=\ker(H^0(\bar{d}'^{k+1}))=H^0(\Fl,\ker(\bar{d}'^{k+1}))=H^0(\Fl, \Sym^k V(k) \otimes_{\Q_p}\omega^{k+2,\sm}_{K^p})\]
by Proposition \ref{dbarsurj}.
\end{proof}

\begin{rem} \label{resmap}
There is a natural restriction map 
\[\mathrm{res}:H^0(\Fl,\omega^{2k+2,\la,\chi}_{K^p})\to H^0(\Fl,(\omega^{2k+2,\la,\chi}_{K^p})_{\kn}). \]
Using the isomorphisms  in the previous corollary, we get
\[\Sym^k V(k) \otimes_{\Q_p} M_{k+2}(K^p) \to M^\dagger_{k+2}(K^p)\cdot e_2^{k} \oplus  M^\dagger_{k+2}(K^p)\cdot e_1^{1+k}e_2^{-1}.\]
It is important to know this map explicitly. The $\kn$-coinvariants $(\Sym^k V)_{\kn}$ are isomorphic to $\Q_p$  by sending $(0,1)^{\otimes k}$ to $1$. (Recall that $V=\Q_p^{\oplus 2}$.) One can check easily that $\Sym^k V(k) \otimes_{\Q_p} M_{k+2}(K^p)$ factors through the quotient 
\[(\Sym^k V)_{\kn}(k) \otimes_{\Q_p} M_{k+2}(K^p) \cong M_{k+2}(K^p) \cdot e_2^k\]
 and $\mathrm{res}$ is the composite with the natural map $M_{k+2}(K^p)\cdot e_2^k\to M^\dagger_{k+2}(K^p)\cdot e_2^k$.
\end{rem}

\begin{para} \label{chininv}
Now Lemma \ref{ninvlem} and Proposition \ref{kninvcoinv} imply that there is a natural exact sequence
\[0\to H^1(\Fl,\omega^{-k,\sm}_{K^p})\cdot e_1^k\to H^1(\Fl,\cO^{\la,\chi}_{K^p})^\kn \to M^\dagger_{-k}(K^p)/M_{-k}(K^p)\cdot  e_1^{-1}e_2^{k+1} \oplus M^\dagger_{-k}(K^p)\cdot e_1^{k}  \to 0.\]
We note that to make everything $B$-equivariant, $H^0(\Fl,(\cO^{\la,\chi}_{K^p})_{\kn})$ needs to be twisted by $\kn^*$, or formally $\cdot e_1^{-1}e_2$. We remark that by \cite[Lemma 5.3.5]{Pan20},   
\[H^1(\Fl,\omega^{-k,\sm}_{K^p})= \varinjlim_{K_p\subset\GL_2(\Q_p)} H^1(\mathcal{X}_{K^pK_p},\omega^{-k}).\]

Similarly, we can also compute the $\kn$-invariants $H^1(\Fl,\omega^{2k+2,\la,\chi}_{K^p})^{\kn}$. Since
\[H^1(\Fl,\omega^{2k+2,\la,\chi,\kn}_{K^p})\cong H^1(\Fl,\omega^{k+2,\sm}_{K^p})=0\]
by  \cite[Corollary 5.3.6]{Pan20}, Lemma \ref{ninvlem} simply becomes
\[H^1(\Fl,\omega^{2k+2,\la,\chi}_{K^p})^{\kn}\cong  M^\dagger_{k+2}(K^p)/M_{k+2}(K^p)\cdot  e_1^{-1}e_2^{k+1} \oplus  M^\dagger_{k+2}(K^p)\cdot e_1^{k}\]
by taking Remark \ref{resmap} into account.

We are ready to compute $(\ker I^1_k)^\kn$. By Lemma \ref{kerdk+1}, it is the same as $(\ker H^1(d^{k+1}))^{\kn}$. Fix a trivialization $\wedge^2 D^{\sm}_{K^p}\cong \cO^{\sm}_{K^p}$. In particular,  the operator $\theta^{k+1}$ in \ref{Do1} defines a map $M_{-k}^{\dagger}(K^p)\to M_{k+2}^\dagger(K^p)$ which we still call $\theta^{k+1}$ by abuse of notation. Note that Lemma \ref{ninvlem} is functorial in the pair $(\cF,X)$. It follows from Theorem \ref{I1} that we have the following commutative diagram 
\[\begin{tikzcd}
 H^1(\Fl,\omega^{-k,\sm}_{K^p})\cdot e_1^k \arrow [d] \arrow[r] & H^1(\Fl,\cO^{\la,\chi}_{K^p})^\kn \arrow[d,"H^1(d^{k+1})"] \arrow[r] & M^\dagger_{-k}/M_{-k}\cdot  e_1^{-1}e_2^{k+1} \oplus M^\dagger_{-k}\cdot e_1^{k} \arrow[d,"\theta^{k+1}\oplus\theta^{k+1}"]\\
 0 \arrow[r] & H^1(\Fl,\omega^{2k+2,\la,\chi}_{K^p})^\kn \arrow[r,"\simeq"] & M^\dagger_{k+2}/M_{k+2}\cdot  e_1^{-1}e_2^{k+1} \oplus  M^\dagger_{k+2}\cdot e_1^{k} 
\end{tikzcd}.\]
Here we drop all $(K^p)$  in $M^\dagger_{-k}(K^p),M^\dagger_{k+2}(K^p),M_{k+2}(K^p),M_{-k}(K^p)$ for simplicity.
Recall that $\kb\subset\mathfrak{g}$ is the Lie algebra of $B$ and $\kh\subset \kb$ is a Cartan subalgebra, cf. the last paragraph of \ref{brr}. Then there is a natural action of $\kh$ on $H^1(\Fl,\cO^{\la,\chi}_{K^p})^\kn$ and induces a decomposition of $H^1(\Fl,\cO^{\la,\chi}_{K^p})^\kn$  into weight $(k,0)$ and $(-1,k+1)$ components.  Note that this constant $\kh$-action is different from the horizontal action $\theta_\kh$.
\end{para}

\begin{defn} \label{D0}
\[D_0:=H^0(\Fl,\wedge^2 D^{\sm}_{K^p}).\] 
This is a free module of $M_0(K^p)$ of rank-one with a generator on which $\GL_2(\A_f)$ acts via $|\cdot|^{-1}\circ\det$.
\end{defn}

\begin{prop}
There is a natural weight decomposition
\[(\ker I^1_k)^\kn=A_{(k,0)}\cdot e_1^k\oplus A_{(-1,k+1)}\cdot e_1^{-1}e_2^{k+1},\]
where 
\[A_{(-1,k+1)}\cong\ker\left(M^\dagger_{-k}(K^p)/M_{-k}(K^p)\xrightarrow{\theta^{k+1}}(M^\dagger_{k+2}(K^p)/M_{k+2}(K^p))\otimes_{M_0(K^p)}D_0^{\otimes -k-1}\right),\]
and $A_{(k,0)}$ sits inside the exact sequence 
\[0\to \varinjlim_{K_p\subset\GL_2(\Q_p)} H^1(\mathcal{X}_{K^pK_p},\omega^{-k})\to A_{(k,0)} \to M^{\dagger}_{-k}(K^p)\xrightarrow{\theta^{k+1}}M^\dagger_{k+2}(K^p)\otimes_{M_0(K^p)}D_0^{\otimes -k-1}.\]
All the maps  here are  Hecke and $B$-equivariant.
\end{prop}

\begin{proof}
Apply the Snake lemma to the commutative diagram above.
\end{proof}

To simplify this result, we need the following well-known result on the kernel of $\theta^{k+1}$.

\begin{lem} \label{kertheta}
Suppose $k\geq 0$. Then the kernel of $\theta^{k+1}:M_{-k}^{\dagger}(K^p)\to M_{k+2}^\dagger(K^p)$ is $M_{-k}(K^p)$. In particular, $\theta^{k+1}$ is injective if $k>0$.
\end{lem}

\begin{proof}
As explained in  \cite[5.2.4]{Pan20}, $M^\dagger_k(K^p)$ has the following equivalent description. Let $\Gamma(p^n)=1+p^n M_2(\Z_p)$. We defined an open subset $\mathcal{X}_{K^p\Gamma(p^n),c}$ (canonical locus) of $\mathcal{X}_{K^p\Gamma(p^n)}$. For each connected component $\mathcal{Z}$ of $\mathcal{X}_{K^p\Gamma(p^n)}$, the intersection $\mathcal{Z}\cap \mathcal{X}_{K^p\Gamma(p^n),c}$ is the generic fiber of an Igusa curve in the integral model of $\mathcal{X}_{K^p\Gamma(p^n)}$ defined by Katz-Mazur. In particular, the irreducibility of the Igusa curve implies that the connected components of $\mathcal{X}_{K^p\Gamma(p^n),c}$ are naturally in bijection with the connected components of $\mathcal{X}_{K^p\Gamma(p^n)}$, i.e. $\pi_0(\mathcal{X}_{K^p\Gamma(p^n),c})=\pi_0(\mathcal{X}_{K^p\Gamma(p^n)})$. We denote by $M^\dagger_k(K^p\Gamma(p^n))$ the set of sections of $\omega^k$ defined in a strict neighborhood of $\mathcal{X}_{K^p\Gamma(p^n),c}$. Then  $\{M^\dagger_k(K^p\Gamma(p^n))\}_n$ form a direct system and $\varinjlim_n  M^\dagger_k(K^p\Gamma(p^n))=M^\dagger_k(K^p)$. Note that $\theta^{k+1}$ maps  $M^\dagger_{-k}(K^p\Gamma(p^n))$ to  $M^\dagger_{k+2}(K^p\Gamma(p^n))$.

Suppose $k=0$. Then $\theta^1$ is essentially the derivation. Hence $\ker(\theta^1|_{M^\dagger_0(K^p\Gamma(p^n))})$ is the set of locally constant functions on $\mathcal{X}_{K^p\Gamma(p^n),c}$. By our previous discussion, this is also the set of locally constant functions on $\mathcal{X}_{K^p\Gamma(p^n)}$, i.e. $H^0(\mathcal{X}_{K^p\Gamma(p^n)},\cO_{\mathcal{X}_{K^p\Gamma(p^n)}})$. By passing to the limit over $n$, we get $\ker(\theta^1)=M_0(K^p)$.

It remains to show that $\theta^{k+1}$ is injective when $k\geq1$. This can be deduced by several different ways. We sketch a proof here using infinitesimal characters, or equivalently, using weights. 

\begin{lem} \label{unifb}
$\dim_C \ker(\theta^{k+1}|_{M^\dagger_{-k}(K^p\Gamma(p^n))})\leq (k+1)|\pi_0(\mathcal{X}_{K^p\Gamma(p^n)})|$ which only depends on $\det(K^p\Gamma(p^n))\Q^\times_{>0}\subset\A_f^\times$.
\end{lem}

Let's assume this lemma at the moment.  Suppose $k\geq 1$. Consider the direct limit over all open compact subgroups $K^p$ of $\GL_2(\A^p_f)$:
\[N:=\varinjlim_{K^p\subseteq\GL_2(\A_f^p)}\ker(\theta^{k+1}|_{M^\dagger_{-k}(K^p \Gamma(p^n))}),\]
which is a representation of $\GL_2(\A_f^p)$.
Let $\psi:(\A_f^p)^\times\to C^\times$ be a smooth character. Then the previous lemma implies that its $\psi$-isotypic part  $N[\psi]$ is a finite-dimensional smooth representation of $\GL_2(\A_f^p)$. Hence the action of $\GL_2(\A_f^p)$ on $N[\psi]$ factors through the determinant homomorphism.  Therefore, if $\ker(\theta^{k+1}|_{M^\dagger_{-k}(K^p\Gamma(p^n))})\neq 0$, we can find a system of spherical Hecke eigenvalues appearing in $\ker(\theta^{k+1}|_{M^\dagger_{-k}(K^p\Gamma(p^n))})$ and its associated two-dimensional $p$-adic Galois representation has consecutive Hodge-Tate weights $a,a+1$. But this contradicts Theorem 1.1 of \cite{Pan2020N} as $|-k-1|>1$ when $k\geq 1$. Hence $\ker(\theta^{k+1})=0$.
\end{proof}

\begin{proof}[Proof of Lemma \ref{unifb}]
Since $\pi_0(\mathcal{X}_{K^p\Gamma(p^n),c})=\pi_0(\mathcal{X}_{K^p\Gamma(p^n)})$,  it suffices to show that on each connected component of $\mathcal{X}_{K^p\Gamma(p^n),c}$, the kernel of $\theta^{k+1}$ has dimension at most $k+1$. It is enough to prove this on the non-cusp points. Note that outside of the cusps, computing the kernel of
\[\theta^{k+1}:\omega^{-k}\to\omega^{-k}\otimes_{\cO_{\mathcal{X}_{K^pK_p}}}(\Omega^1(\mathcal{C}))^{\otimes k+1}\]
is the same as solving a homogeneous linear ordinary differential equation of order $k+1$ (using Kodaira-Spencer isomorphism). Hence the space of solutions is at most $(k+1)$-dimensional on a connected component by our knowledge of linear ODE.
\end{proof}

Hence we get the following theorem.

\begin{thm} \label{kernk}
For $\chi=(0,k)$, there is a natural weight decomposition
\[(\ker I^1_k)^\kn=A_{(k,0)}\cdot e_1^k\oplus A_{(-1,k+1)}\cdot e_1^{-1}e_2^{k+1},\]
where  
\begin{itemize}
\item $A_{(-1,k+1)}\cong\ker\left(M^\dagger_{-k}(K^p)/M_{-k}(K^p)\xrightarrow{\theta^{k+1}}(M^\dagger_{k+2}(K^p)/M_{k+2}(K^p))\otimes_{M_0(K^p)}D_0^{\otimes -k-1}\right)$ is isomorphic to a subspace of $M_{k+2}(K^p)\otimes_{M_0(K^p)}D_0^{\otimes -k-1}$ via $\theta^{k+1}$.
\item  $A_{(k,0)}$ sits inside the exact sequence 
\[0\to \varinjlim_{K_p\subset\GL_2(\Q_p)} H^1(\mathcal{X}_{K^pK_p},\omega^{-k})\to A_{(k,0)} \to M_{-k}(K^p)\to 0.\]
\end{itemize}
All maps  here are Hecke and $B$-equivariant. In particular, the Hecke action on $(\ker I^1_k)^\kn$ is ``classical'' in the sense that systems of Hecke eigenvalues appearing in $(\ker I^1_k)^\kn$ all come from  automorphic forms on $\GL_2(\A)$.
\end{thm}

\begin{rem} \label{A00rinv}
When $k>0$, we have $A_{(k,0)}=\varinjlim_{K_p\subset\GL_2(\Q_p)} H^1(\mathcal{X}_{K^pK_p},\omega^{-k})$ because $M_{-k}(K^p)=0$. When $k=0$, it turns out that the surjective map $A_{(0,0)}\to M_{0}(K^p)$ has a natural right inverse $M_0(K^p)\to H^1(\Fl,\cO^{\la,(0,0)}_{K^p})$ constructed in \cite[Corollary 5.1.3.(2)]{Pan20} which is $\GL_2(\Q_p)$-equivariant. See the discussion before Theorem 5.3.16 \textit{ibid.}. Hence we have 
\[A_{(0,0)}=\varinjlim_{K_p\subset\GL_2(\Q_p)} H^1(\mathcal{X}_{K^pK_p},\cO_{\mathcal{X}_{K^pK_p}})\oplus M_0(K^p).\]
\end{rem}

\begin{rem}
It is natural to ask whether one can give a more precise description of $A_{(-1,k+1)}$. Essentially one needs to understand the cokernel of $\theta^{k+1}$ on $\mathcal{X}_{K^pK_p,c}$, which is the $H^1$ of the de Rham complex $\Sym^k D\xrightarrow{\nabla} \Sym^k D \otimes_{\cO_{\mathcal{X}_{K^pK_p}}} \Omega^1_{\mathcal{X}_{K^pK_p}}(\mathcal{C})$ on the dagger space associated to $\mathcal{X}_{K^pK_p,c}$. One can also interpret this as some rigid cohomology on the Igusa curve. Coleman was able to determine the primitive part in his famous work \cite{Cole97} and his result roughly says that $\coker(\theta^{k+1})$ is closely related to the finite slope part of classical modular forms.
\end{rem}

\begin{cor} \label{knfin}
Let $K_p=1+p^lM_2(\Z_p)$ for some integer $l\geq 2$ and $(\ker I^1_k)^{K_p-\an}\subseteq \ker I^1_k$ the subspace of $K_p$-analytic vectors.  Then
\begin{enumerate}
\item $(\ker I^1_k)^{K_p-\an,\kn}$ is finite-dimensional.
\item The subspace of $\kn$-finite vectors 
\[(\ker I^1_k)^{K_p-\an,\kn-\mathrm{fin}}:=\{v\in(\ker I^1_k)^{K_p-\an}, (u^+)^k\cdot v=0,\mbox{ for some }k\}\]
is a finitely generated $U(\mathfrak{g})$-module.
\end{enumerate}
\end{cor}

\begin{proof}
The second part follows from the first part. For the first part, we first explain the rough idea. The $K_p$-analytic vectors in a smooth representation of $\GL_2(\Q_p)$ are equal to the $K_p$-fixed vectors. In particular, the $K_p$-analytic vectors in $M_{k+2}(K^p)$ are
\[M_{k+2}(K^pK_p):=H^0(\mathcal{X}_{K^pK_p},\omega^{k+2}),\]
a finite-dimensional vector space over $C$ as $\mathcal{X}_{K^pK_p}$ is proper over $C$. Similarly, the subspace of $K_p$-analytic vectors in $\varinjlim_{K'_p\subset\GL_2(\Q_p)} H^1(\mathcal{X}_{K^pK'_p},\omega^{-k})$ is also finite-dimensional. Therefore our finiteness claim will follow from Theorem \ref{kernk} if we can show maps in Theorem \ref{kernk} preserve the property of being $K_p$-analytic.

For $A_{(k,0)}$, there is a natural inclusion $\Sym^k V(k)\otimes_{\Q_p}\omega^{-k,\sm}_{K^p}\subseteq \cO^{\la,(0,k)}_{K^p}$ by  \ref{lalg0k}. Taking the $H^1$, we get a $\GL_2(\Q_p)$-equivariant inclusion
\[\Sym^k V(k)\otimes_{\Q_p}H^1(\Fl,\omega^{-k,\sm}_{K^p})\cong H^1(\Fl,\Sym^k V(k)\otimes_{\Q_p}\omega^{-k,\sm}_{K^p})\subseteq H^1(\Fl,\cO^{\la,(0,k)}_{K^p})\]
which recovers  $\varinjlim_{K_p\subset\GL_2(\Q_p)} H^1(\mathcal{X}_{K^pK_p},\omega^{-k})\to A_{(k,0)}$ in Theorem \ref{kernk} by taking
the $\kn$-invariants of this map. Note that $\Sym^k V$ is an algebraic representation of $\GL_2(\Q_p)$. Hence it is clear from this description that 
\[H^1(\Fl,\cO^{\la,(0,k)}_{K^p})^{K_p-\an}\cap H^1(\Fl,\omega^{-k,\sm}_{K^p})\cdot e_1^k=H^1(\mathcal{X}_{K^pK_p},\omega^{-k})\cdot e_1^k\]
is finite-dimensional. If $k=0$, by Remark \ref{A00rinv}, the inclusion $M_{0}(K^p)\subset H^1(\Fl,\cO^{\la,(0,0)}_{K^p})$ is $\GL_2(\Q_p)$-equivariant. Hence the same argument works. This proves that $(A_{(k,0)}\cdot e_1^k)^{K_p-\an}$ is finite-dimensional. (Essentially, we are showing that $A_{(k,0)}\cdot e_1^k$ are all contributed by the locally algebraic vectors.)

For $A_{(-1,k+1)}$, the situation is a bit more complicated. We need several lemmas. Note that $e_1^k\neq e_1^{-1}e_2^{k+1}$ as $k\geq 0$. Hence the exact sequence in the beginning of \ref{chininv} implies that there is a natural inclusion
\[M^\dagger_{-k}(K^p)/M_{-k}(K^p)\cdot  e_1^{-1}e_2^{k+1} \subseteq H^1(\Fl,\cO^{\la,\chi}_{K^p})^\kn.\]
We need to understand which vectors in $M^\dagger_{-k}(K^p)/M_{-k}(K^p)\cdot  e_1^{-1}e_2^{k+1}$ are $K_p$-analytic.
Recall that  $M^\dagger_k(K^p\Gamma(p^m))$ denotes the set of sections of $\omega^k$ defined in a strict neighborhood of $\mathcal{X}_{K^p\Gamma(p^m),c}$.

\begin{lem} \label{bdan}
There exist an integer $m>0$ and an affinoid strict open neighborhood $V$ of $\mathcal{X}_{K^p\Gamma(p^m),c}$ such that the $K_p$-analytic vectors
\[M^\dagger_{-k}(K^p)/M_{-k}(K^p)\cdot  e_1^{-1}e_2^{k+1} \cap H^1(\Fl,\cO^{\la,\chi}_{K^p})^{K_p-\an,\kn}\]
are contained in the image of $H^0(V,\omega^{-k})\to M^\dagger_{-k}(K^p\Gamma(p^m)) \to M^\dagger_{-k}(K^p)/M_{-k}(K^p)$.
\end{lem}

We also need an irreducibility result. Denote the natural map $\mathcal{X}_{K^p\Gamma(p^m)} \to \mathcal{X}_{K^p\Gamma(p^n)}, m\geq n$ by $\pi_{m/n}$.

\begin{lem} \label{irredigsnbhd}
Let $V$ be an affinoid subset of $\mathcal{X}_{K^p\Gamma(p^m)}$ for some $m$. Suppose for each connected component $X$ of $\mathcal{X}_{K^p\Gamma(p^m)}$, the intersection $V\cap X$ is either empty or a connected strict open neighborhood of $\mathcal{X}_{K^p\Gamma(p^m),c}\cap X$. (This is possible as Igusa curves are irreducible.) Then there exists an integer $m'\geq m$ such that for each connected component $V'$ of $\pi_{m'/m}^{-1}(V)$ and any integer $m''\geq m'$, the intersection of $\pi_{m''/m'}^{-1}(V')$ with each connected component of  $\mathcal{X}_{K^p\Gamma(p^{m''})}$ is either connected or empty.
\end{lem}

Both lemmas will be proved later.  Back to the proof of Corollary \ref{knfin}. Let $V$ be an open subset of $\mathcal{X}_{K^p\Gamma(p^m)}$  as in Lemma \ref{bdan}. We may assume that it satisfies the assumption in Lemma \ref{irredigsnbhd}, i.e. $V\cap X$ is connected for each connected component $X$ of $\mathcal{X}_{K^p\Gamma(p^m)}$. Fix a trivialization $D_0\cong M_0(K^p)$. It suffices to show the kernel of the composite
\[H^0(V,\omega^{-k})\xrightarrow{\theta^{k+1}} H^0(V,\omega^{k+2}) \to M^\dagger_{k+2}(K^p)/M_{k+2}(K^p)\]
is finite-dimensional. In fact, we claim that $\ker\left(H^0(V,\omega^{k+2}) \to M^\dagger_{k+2}(K^p)/M_{k+2}(K^p)\right)$ is finite-dimensional. Indeed, by Lemma \ref{irredigsnbhd}, after possibly replacing $m$ by $m'$ and $V$ by the union of irreducible components  of $\pi_{m'/m}^{-1}(V)$ whose intersection with  $\mathcal{X}_{K^p\Gamma(p^{m'}),c}$ is non-empty, we may assume further that
\begin{itemize}
\item the preimage of $V$ in each connected component of $\mathcal{X}_{K^p\Gamma(p^{m''})}$ is connected for any $m''\geq m$.
\end{itemize}
Suppose $s\in\ker\left(H^0(V,\omega^{k+2}) \to M^\dagger_{k+2}(K^p)/M_{k+2}(K^p) \right)$. This means that the pull-back of $s$ to  $\mathcal{X}_{K^p\Gamma(p^{m''})}$ is equal to  some  $s'\in H^0(\mathcal{X}_{K^p\Gamma(p^{m''})},\omega^{k+2})$ when restricted to a strict open neighborhood of  $\mathcal{X}_{K^p\Gamma(p^{m''}),c}$ for some $m''\geq m$. On the other hand, it follows from our assumption that we can choose this open neighborhood to be  $\pi_{m''/m}^{-1}(V)$, which is $\Gamma(p^m)$-stable. Since $s$ is fixed by $\Gamma(p^m)$, so is $s'$, i.e. $s'$ is the pull-back of some element in $M_{k+2}(K^p\Gamma(p^m))$. This implies our claim as $M_{k+2}(K^p\Gamma(p^m))$ is finite-dimensional.
\end{proof}

\begin{proof}[Proof of Lemma \ref{bdan}]
By Proposition \ref{contrGnan}, $H^1(\Fl,\cO^{\la,\chi}_{K^p})^{K_p-\an}$ is contained in the image of $\check{H}^1(\cO^{\Gamma(p^m)-\an,\chi})\to H^1(\Fl,\cO^{\la,\chi}_{K^p})$ for some $m$. Moreover after possibly enlarging $m$, we may even assume there is a natural lifting $H^1(\Fl,\cO^{\la,\chi}_{K^p})^{K_p-\an}\subseteq \check{H}^1(\cO^{\Gamma(p^m)-\an,\chi})$ by the discussion in \ref{discLB}. In particular, we may view $H^1(\Fl,\cO^{\la,\chi}_{K^p})^{K_p-\an,\kn}$ as a subspace of  $\check{H}^1(\cO^{\Gamma(p^m)-\an,\chi})^\kn$. Let $H=\Gamma(p^m)$. By definition \ref{chGanchi}, there is an exact sequence
\[0\to \cO^{\la,\chi}_{K^p}(U_1)^{H-\an}\oplus \cO^{\la,\chi}_{K^p}(U_2)^{H-\an}/\check{H}^0(\cO^{H-\an,\chi})\to \cO^{\la,\chi}_{K^p}(U_{12})^{H-\an}\to \check{H}^1(\cO^{H-\an,\chi})\to 0.\]
Take the $\kn$-cohomology. 
\[\cO^{\la,\chi}_{K^p}(U_{12})^{H-\an,\kn}\to  \check{H}^1(\cO^{H-\an,\chi})^\kn \to  \left(\cO^{\la,\chi}_{K^p}(U_1)^{H-\an}\right)_\kn \oplus \left( \cO^{\la,\chi}_{K^p}(U_2)^{H-\an}/\check{H}^0(\cO^{H-\an,\chi})\right)_\kn.\]
Recall that  $H^1(\Fl,\cO^{\la,\chi}_{K^p})$ has a natural weight decomposition into weight $(k,0)$ and $(-1,k+1)$ components and we are interested in the $(-1,k+1)$ part. Since $\cO^{\la,\chi}_{K^p}(U_{12})^{H-\an,\kn}$ has weight $(k,0)$, it is easy to see that we have the following commutative diagram 
\[\begin{tikzcd}
 \check{H}^1(\cO^{H-\an,\chi})^\kn_{(-1,k+1)} \arrow [d] \arrow[r] &  \left( \cO^{\la,\chi}_{K^p}(U_2)^{H-\an}/\check{H}^0(\cO^{H-\an,\chi})\right)_\kn \arrow[d] \\
 H^1(\Fl,\cO^{\la,\chi}_{K^p})_{(-1,k+1)}  \arrow[r,"\simeq"] & M^\dagger_{-k}(K^p)/M_{-k}(K^p)\cdot e_1^{-1}e_2^{k+1}
 \end{tikzcd},\]
 where the subscript $(-1,k+1)$ means the weight-$(-1,k+1)$ component and the bottom line follows from the beginning of \ref{chininv}. ( There is no contribution from $(\cO^{\la,\chi}_{K^p}(U_1)^{H-\an})_\kn$ because $(\cO^{\la,\chi}_{K^p})_\kn$ is supported outside of $U_1$.) Hence it remains to study the image of 
 \[ev_\infty:\left( \cO^{\la,\chi}_{K^p}(U_2)^{H-\an}/\check{H}^0(\cO^{H-\an,\chi})\right)_\kn \to M^\dagger_{-k}(K^p)/M_{-k}(K^p).\]
 By \cite[Proposition 5.2.10.(2)]{Pan20}, this map comes from the composite map
 \[ \left( \cO^{\la,\chi}_{K^p}(U_2)^{H-\an}\right)_\kn \to \cO^{\la,\chi}_{K^p}(U_2)/(y)\cong M^\dagger_{-k}(K^p),\]
 where $y=1/x$ vanishes at $\infty$. Now our claim essentially follows from the proof of \cite[Proposition 5.2.6]{Pan20}, i.e. the construction of the isomorphism $\cO^{\la,\chi}_{K^p}(U_2)/(y)\cong M^\dagger_{-k}(K^p)$.
 Indeed, by Theorem \ref{str} (with $x$ replaced by $y$ and $x_n=0$ here), we can find sufficiently large integers $m'<r(m')$ such that for any $f\in \cO^{\la,\chi}_{K^p}(U_2)^{H-\an}$, its restriction on $U'_2:=\{\|y\|\leq p^{-{m'}}\} \subseteq U_2$ can be written as
 \[f|_{U'_2}=e_2^{k}\sum_{i=0}^\infty c_i y^i,\]
 where $c_i\in\omega^{-k}(V'_{2})$ and $e_2^kc_ip^{(m'-1)i}$ is uniformly bounded, $i\geq 0$. Here $V'_2\subseteq \mathcal{X}_{K^p\Gamma(p^{r(m')})}$ denotes the open affinoid subset whose preimage in $\mathcal{X}_{K^p}$ is $\pi_{\HT}^{-1}(U'_2)$. The map $ev_\infty$ is simply sending $f$ to $c_0$.  See the proof of  \cite[Proposition 5.2.10.(3),(4)]{Pan20} for more details. Note that $V'_2$ is a strict neighborhood of $\mathcal{X}_{K^p\Gamma(p^{r(m')})_c}$ by \cite[Lemma 5.2.9]{Pan20}. We thus prove Lemma \ref{bdan} with $m=r(m')$ and $V=V'_2$.
 \end{proof}
 
 \begin{proof}[Proof of Lemma \ref{irredigsnbhd}]
 We may assume $V$ is connected. Let 
 \[S_1=\varprojlim_{n\geq m} \pi_0 \left(\pi_{n/m}^{-1}(V)\right),~~~~~~~~~~~~S_2=\varprojlim_{n} \pi_0 (\mathcal{X}_{K^p\Gamma(p^n)}).\]
$\Gamma(p^m)$ acts naturally on both profinite sets. The inclusion of $V\subseteq \mathcal{X}_{K^p\Gamma(p^m)}$ induces a natural $\Gamma(p^m)$-equivariant map
\[f:S_1\to S_2.\]
Our claim is equivalent to the existence of integer $m'\geq m$ such that $f|_T$ is injective for any $\Gamma(p^{m'})$-orbit $T$ of $S_1$. By our knowledge of connected components of modular curves, for any $s\in S_2$, the stabilizer of $\Gamma(p^m)$ with respect to $s$ is $\Gamma(p^m)\cap\SL_2(\Z_p)$. Since $V$ is connected, $\Gamma(p^m)$ acts transitively on $S_1$. Hence $S_1=\Gamma(p^m)/H$ for some closed subgroup $H$ of $\Gamma(p^m)$. It suffices to prove that 
\[H\cap\Gamma(p^{m'})=\Gamma(p^{m'})\cap\SL_2(\Z_p)\]
for some $m'\geq m$. Equivalently $\Lie(H)=\mathfrak{sl}_2(\Q_p)$. To see this, we consider the space of $\Gamma(p^m)$-locally analytic $\Q_p$-valued functions $\sC(S_1,\Q_p)^{\la}$ on $S_1$. It is enough to show that $\mathfrak{sl}_2(\Q_p)$ acts trivially on $\sC(S_1,\Q_p)^{\la}$. 

Let $V_\infty$ be the preimage of $V$ in $\mathcal{X}_{K^p}$. By shrinking $V$ if necessary, we may assume $e_2$ is a generator on $V_\infty$, cf. \cite[Lemma 5.2.9]{Pan20} and its proof. There is a natural continuous map $c:|V_\infty|\to S_1$ hence we may view $\sC(S_1,\Q_p)^{\la}$ as a subspace of $\cO_{\mathcal{X}_{K^p}}(V_\infty)^\la$. In \cite[\S 3]{Pan20}, we showed that elements in  $\cO_{\mathcal{X}_{K^p}}(V_\infty)^\la$ satisfy a first-order differential equation and calculated this differential equation in \cite[Theorem 4.2.4]{Pan20}. Note that $e_2$ is a generator. Hence  by \textit{loc. cit.}
\[\begin{pmatrix} y & 1 \\ -y^2 & -y\end{pmatrix}\cdot f=0,~~~~~f\in\cO_{\mathcal{X}_{K^p}}(V_\infty)^\la.\]
Recall $y=e_1/e_2$ is a coordinated function on $\pi_\HT(V_\infty)$. We can rewrite this equation as
\begin{eqnarray} \label{sl2eqn}
\begin{pmatrix} 0 & 1 \\ 0 & 0\end{pmatrix}\cdot f +y\begin{pmatrix} 1 & 0 \\ 0 & -1\end{pmatrix}\cdot f +y^2\begin{pmatrix} 0 & 0 \\ -1 & 0\end{pmatrix}\cdot f=0.
\end{eqnarray}

\begin{lem} \label{largeim}
$\pi_{\HT}\left(c^{-1}(s)\right)$ contains at least $3$ classical  points for any $s\in S_1$, where $c$ denotes the natural map $|V_\infty|\to S_1$.
\end{lem}

\begin{proof}
Since $\Gamma(p^m)$ acts transitively on $S_1$, it is enough to find $3$ classical points in $\pi_{\HT}(V_\infty)$ which belong to different orbits of $\Gamma(p^m)$.
In fact, it follows from Scholze's construction of $\mathcal{X}_{K^p}$ in \cite[Chapter III]{Sch15} that $\pi_{\HT}(V_\infty)$ contains an open subset of $\Fl$. We give a more direct proof here. By our assumption $V$ contains the canonical locus of a connected component of $\mathcal{X}_{K^p\Gamma(p^m)}$, hence $\infty\in \pi_{\HT}(V_\infty)$. Take a classical point outside of the canonical locus in $V$. Its preimage in $\mathcal{X}_{K^p}$ goes to a classical non-$\Q_p$-rational point $s_2$ in $\Fl$ under $\pi_\HT$. Note that $q=\displaystyle \inf_{g\in\Gamma(p^m)} |y(g\cdot s_2)|>0$. We apply Lemma 5.2.9 of \cite{Pan20} with $U=\{\|y\|\leq q/2\}\subseteq\Fl$ and conclude that there is a classical point outside of the canonical locus in $\pi^{-1}_{m'/m}(V)$ for some $m'$, whose preimage in $\mathcal{X}_{K^p}$ maps to a point $s_3\in U$ under $\pi_\HT$. Clearly $\infty,s_2,s_3$ belong to different $\Gamma(p^m)$-orbits.
\end{proof}

Now  let $f\in\sC(S_1,\Q_p)^{\la}$. For any $s\in S_1$,  by Lemma \ref{largeim}, there exist $3$ points $s_1,s_2,s_3$ in $c^{-1}(s)\subseteq V_\infty$ such that $y(s_1),y(s_2),y(s_3)\in C$ are distinct, Then we can evaluate \eqref{sl2eqn} at $s_1,s_2,s_3$ and conclude that
\[\left(\begin{pmatrix} 0 & 1 \\ 0 & 0\end{pmatrix}\cdot f\right)(s)=\left(\begin{pmatrix} 1 & 0 \\ 0 & -1\end{pmatrix}\cdot f \right)(s)=\left(\begin{pmatrix} 0 & 0 \\ -1 & 0\end{pmatrix}\cdot f\right)(s)=0.\]
This immediately implies our claim as  $\begin{pmatrix} 0 & 1 \\ 0 & 0\end{pmatrix}, \begin{pmatrix} 1 & 0 \\ 0 & -1\end{pmatrix},\begin{pmatrix} 0 & 0 \\ -1 & 0\end{pmatrix}$ generate $\mathfrak{sl}_2(\Q_p)$.
 \end{proof}

\begin{rem}
The same argument gives a purely $p$-adic proof of the well-known fact that the action of $\SL_2(\Q_p)$ on $\varinjlim_{K_p}\pi_0(\mathcal{X}_{K^pK_p})$ is trivial.
\end{rem}
 
\section{Intertwining operators: spectral decomposition} \label{Iosd}
In this section, we continue our study of the intertwining operators $I_k$ and $I^1_k$. Our main result (Theorem \ref{sd}) below gives a decomposition of $\ker I^1_k$ with respect to the Hecke action. To do this, probably not surprisingly, we will use the Newton stratification on the flag variety $\Fl=\mathbb{P}^1=\mathbb{P}^1(\Q_p)\sqcup \Omega$, where $ \Omega=\mathbb{P}^1\setminus \mathbb{P}^1(\Q_p)$ denotes the usual Drinfeld upper half plane, and compute the kernel and cokernel of $I_k$ on each strata. 

Recall some notations. Fix a non-negative integer $k$ throughout this section. For $\chi=(n_1,n_2)\in\Z^2$ with $n_2-n_1=k$, we have
\[d^{k+1}: \cO^{\la,(n_1,n_2)}_{K^p}\to  \cO^{\la,(n_1,n_2)}_{K^p}\otimes_{\cO^{\sm}_{K^p}} (\Omega^1_{K^p}(\mathcal{C})^{\sm})^{\otimes k+1}\]  
\[d'^{k+1}:\cO^{\la,(n_1,n_2)}_{K^p}\otimes_{\cO_{\Fl}}(\Omega^1_{\Fl})^{\otimes k+1}\to\cO^{\la,(n_2+1,n_1-1)}_{K^p}(k+1),\]
Similarly,
\[\bar{d}^{k+1}:  \cO^{\la,(n_1,n_2)}_{K^p}\to \cO^{\la,(n_1,n_2)}_{K^p}\otimes_{\cO_{\Fl}}(\Omega^1_{\Fl})^{\otimes k+1},\]
\[\bar{d}'^{k+1}: \cO^{\la,(n_1,n_2)}_{K^p}\otimes_{\cO^{\sm}_{K^p}} (\Omega^1_{K^p}(\mathcal{C})^{\sm})^{\otimes k+1}\to \cO^{\la,(n_2+1,n_1-1)}_{K^p}(k+1).\]
Hence $I_k=d'^{k+1}\circ \bar{d}^{k+1}=\bar{d}'^{k+1}\circ d^{k+1}$.

\subsection{\texorpdfstring{$I_k$}{Lg} on \texorpdfstring{$\mathbb{P}^1(\Q_p)$}{Lg}}

%\begin{lem} \label{Ikdk}

\begin{para}
We first compute the stalks of $\ker d^{k+1}$ and $\coker d^{k+1}$ at points in $\mathbb{P}^1(\Q_p)$. Since everything is $\GL_2(\Q_p)$-equivariant and $\GL_2(\Q_p)$ acts transitively on $\mathbb{P}^1(\Q_p)$, it's enough to determine the stalks at one point. Here we choose it to be $\infty\in\Fl$, the vanishing locus of $e_1$. Note that $y=\frac{e_1}{e_2}$ is a local coordinate around $0$.

We need the following description of the stalk of $\cO^{\la,(n_1,n_2)}_{K^p}$ at $\infty$.  Note that
\[(\cO_{\Fl})_{\infty}=\varinjlim_{n} \cO_C[[\frac{y}{p^n}]][\frac{1}{p}]\]
and we  view $\cO_C[[\frac{y}{p^n}]][\frac{1}{p}]$ as a $C$-Banach space with unit open ball $\cO_C[[\frac{y}{p^n}]]$.
On the other hand, for $l\in\Z$, recall that 
\[(\omega^{l,\sm}_{K^p})_\infty=M^{\dagger}_{l}(K^p)=\varinjlim_{n}\varinjlim_{U\supseteq \bar{\mathcal{X}}_{K^p\Gamma(p^n),c}}\omega^l(U),\]
where $U$ runs through all strict affinoid open neighborhood of $\mathcal{X}_{K^p\Gamma(p^n),c}$.
See the notation in the proof of Lemma \ref{kertheta}. Denote by 
\[(\cO_{\Fl})_{\infty}\widehat\otimes_{C} M^{\dagger}_{l}(K^p):= \varinjlim_{n}\varinjlim_{U\supseteq \bar{\mathcal{X}}_{K^p\Gamma(p^n),c}} \cO_C[[ \frac{y}{p^n}]][\frac{1}{p}] \widehat\otimes_C~ \omega^l(U)\]
equipped with the inductive limit topology.
Equivalently, any element of $(\cO_{\Fl})_{\infty}\widehat\otimes_{C} M^{\dagger}_{l}(K^p)$ can be written as 
$\displaystyle \sum_{i=0}^{+\infty}c_iy^{i}$,
where all $c_i\in \omega^l(U)$ for some strict affinoid open neighborhood $U$ of some $\mathcal{X}_{K^p\Gamma(p^n),c}$ and $c_ip^{mi},i\geq 0$ are uniformly bounded for some $m$.
Take $l=n_1-n_2$. Then there is  a natural continuous map
\[(\cO_{\Fl})_{\infty}\widehat\otimes_{C} M^{\dagger}_{n_1-n_2}(K^p) \cdot \mathrm{t}^{n_1}e_2^{n_2-n_1}\to (\cO^{\la,(n_1,n_2)}_{K^p})_\infty\]
sending $\displaystyle \sum_{i=0}^{+\infty}c_iy^{i}$ to $\displaystyle \mathrm{t}^{n_1}e_2^{n_2-n_1}\sum_{i=0}^{+\infty}c_iy^{i}$. Note that $(\cO_{\Fl})_{\infty}\cdot e_2^{n_2-n_1}=(\omega^{n_2-n_1}_{\Fl})_{\infty}(n_2-n_1)$. We can also rewrite this map as 
\[(\omega^{n_2-n_1}_{\Fl})_{\infty}(n_2-n_1)\widehat\otimes_{C} M^{\dagger}_{n_1-n_2}(K^p) \cdot \mathrm{t}^{n_1}\to (\cO^{\la,(n_1,n_2)}_{K^p})_\infty\]
\end{para}

\begin{lem} \label{Olainfty}
This is an isomorphism, i.e. 
\[(\cO^{\la,(n_1,n_2)}_{K^p})_\infty\cong (\omega^{n_2-n_1}_{\Fl})_{\infty}\widehat\otimes_{C} M^{\dagger}_{n_1-n_2}(K^p)(n_2-n_1) \cdot \mathrm{t}^{n_1}.\]
Similarly, for $l\in\Z$, 
\[(\omega^{l,\la,(n_1,n_2)}_{K^p})_\infty\cong (\omega^{n_2-n_1}_{\Fl})_{\infty}\widehat\otimes_{C} M^{\dagger}_{n_1-n_2+l}(K^p)(n_2-n_1) \cdot \mathrm{t}^{n_1}.\]
\end{lem}

\begin{proof}
The first claim is essentially proved in the proof of \cite[Proposition 5.2.10]{Pan20}. We give a sketch here. By Theorem \ref{str} (after applying $\begin{pmatrix} 0 & 1\\ 1& 0\end{pmatrix}$), a section of $\cO^{\la,(n_1,n_2)}_{K^p}$  defined in an open neighborhood of $\infty$ can be written as $\displaystyle \mathrm{t}^{n_1}e_2^{n_2-n_1}\sum_{i=0}^{+\infty}c_i(y-y_n)^{i}$ for some $y_n$.  After shrinking this open neighborhood, we may choose $y_n=0$. Thus we see that any element of $(\cO^{\la,(n_1,n_2)}_{K^p})_\infty$ has the desired form. This implies the second claim by Lemma \ref{tensomegaksm}.
\end{proof}

\begin{para}
Suppose $(n_1,n_2)\in\Z^2$ and $n_2-n_1=k$. Fix a trivialization of $\wedge^2 D^{\sm}_{K^p}$ as before, i.e. an isomorphism $D_0=H^0(\wedge^2 D^{\sm}_{K^p})\cong M_0(K^p)$.  Consider $d^{k+1}:\cO^{\la,\chi}_{K^p}\to \omega^{2k+2,\la,\chi}_{K^p}$ at $\infty$. Under the isomorphisms in Lemma \ref{Olainfty}, we may identify $(d^{k+1})_\infty:(\cO^{\la,\chi}_{K^p})_\infty\to (\omega^{2k+2,\la,\chi}_{K^p})_\infty$ with the map
\[1\otimes\theta^{k+1}:(\omega^k_{\Fl})_{\infty}\widehat\otimes_{C} M^{\dagger}_{-k}(K^p) \cdot \mathrm{t}^{n_1} \to (\omega^k_{\Fl})_{\infty}\widehat\otimes_{C} M^{\dagger}_{k+2}(K^p) \cdot \mathrm{t}^{n_1}\]
induced from $\theta^{k+1}:M^{\dagger}_{-k}(K^p)\to M^{\dagger}_{k+2}(K^p)$. By Lemma \ref{kertheta}, this map is injective if $k\geq 1$, and when $k=0$, the kernel of $M^{\dagger}_0(K^p\Gamma(p^n))\to M^{\dagger}_2(K^p\Gamma(p^n))$ is $M_0(K^p\Gamma(p^n))$, which is finite-dimensional. Hence
\[\ker(d^{1})_\infty= (\cO_{\Fl})_{\infty}\otimes_{C} M_{0}(K^p) \cdot \mathrm{t}^{n_1}\]
(the usual tensor product). Since $\GL_2(\Q_p)$ acts transitively on $\mathbb{P}^1(\Q_p)$, this also implies the following result.
\end{para}

\begin{prop} \label{kerdc}
For $d^{k+1}:\cO^{\la,(n_1,n_2)}_{K^p}\to \cO^{\la,(n_1,n_2)}_{K^p}\otimes_{\cO^{\sm}_{K^p}}(\Omega^1_{K^p}(\mathcal{C})^\sm)^{\otimes k+1}$
, we have
\[(\ker d^{k+1})|_{\mathbb{P}^1(\Q_p)}=0,~~~k\geq 1.\]
and 
\[\ker(d^{1})|_{\mathbb{P}^1(\Q_p)}= (\cO_{\Fl})|_{\mathbb{P}^1(\Q_p)}\otimes_{C} M_{0}(K^p) \cdot \mathrm{t}^{n_1}.\]
\end{prop}
 
\begin{para}
Next we consider the cokernel of $d^{k+1}$ at $\infty$. Using Lemma \ref{Olainfty}, essentially we need to understand the cokernel of $\theta^{k+1}:M_{-k}^{\dagger}(K^p)\to M_{k+2}^{\dagger}(K^p)$. 
\end{para} 
 
\begin{defn} \label{H1Ig}
\[H^1_{\rig}(\Ig(K^p\Gamma(p^n)),\Sym^k):=\coker\left(M_{-k}^{\dagger}(K^p\Gamma(p^n))\xrightarrow{\theta^{k+1}} M_{k+2}^{\dagger}(K^p\Gamma(p^n))\otimes D_0^{-k-1}\right)\]
for $n\geq 0$, and 
\[H^1_{\rig}(\Ig(K^p),\Sym^k):=\coker\left(M_{-k}^{\dagger}(K^p)\xrightarrow{\theta^{k+1}} M_{k+2}^{\dagger}(K^p)\otimes D_0^{-k-1}\right).\]
Clearly, $\displaystyle H^1_{\rig}(\Ig(K^p),\Sym^k)=\varinjlim_n H^1_{\rig}(\Ig(K^p\Gamma(p^n)),\Sym^k)$ and there is a natural smooth action of $B:=\{\begin{pmatrix} * & * \\ 0 & * \end{pmatrix}\}\subseteq \GL_2(\Q_p)$ on it. Note that we put $\otimes D_0^{-k-1}$ here to make sure that everything does not depend on the trivialization of $D_0$.

When $k=0$, we will simply write $H^1_{\rig}(\Ig(K^p))$ instead of $H^1_{\rig}(\Ig(K^p),\Sym^0)$.
\end{defn}

$H^1_{\rig}(\Ig(K^p\Gamma(p^n)),\Sym^k)$ was studied in detail by Coleman in his famous works \cite{Cole96,Cole97}. He basically shows that it is closely related to the finite slope part of classical modular forms. For our purpose, we only need the following finiteness result.

\begin{prop} \label{Colefin}
$H^1_{\rig}(\Ig(K^p\Gamma(p^n)),\Sym^k)$ is a finite-dimensional vector space over $C$.
\end{prop}

\begin{proof}
Coleman determined its dimension in the proof of \cite[Theorem 2.1]{Cole97}, which is based on his previous work \cite[\S 8]{Cole96}. Roughly speaking, as Coleman explained,  $(\Sym^k D,\nabla_k)$ defines an overconvergent  logarithmic F-isocrystal on the special fiber of $\mathcal{X}_{K^p\Gamma(p^n),c}$ (a union of Igusa curves) and $H^1_{\rig}(\Ig(K^p\Gamma(p^n)),\Sym^k)$ is simply its log-rigid cohomology (after extending the coefficients to $C$). The claim here follows from the general finiteness  result on rigid cohomology. %\textbf{REFERENCE?}
\end{proof}

\begin{prop} \label{cokdk+1infty}
For $d^{k+1}:\cO^{\la,(n_1,n_2)}_{K^p}\to \cO^{\la,(n_1,n_2)}_{K^p}\otimes_{\cO^{\sm}_{K^p}}(\Omega^1_{K^p}(\mathcal{C})^\sm)^{\otimes k+1}$
, we have
\[(\coker d^{k+1})_{\infty}=(\omega^k_\Fl)_\infty\otimes_C H^1_{\rig}(\Ig(K^p),\Sym^k)(k) \cdot \mathrm{t}^{n_1}.\]
\end{prop}

\begin{proof}
By Proposition \ref{Colefin}, the image of $M_{-k}^{\dagger}(K^p\Gamma(p^n))$ in $M_{k+2}^{\dagger}(K^p\Gamma(p^n))$ under $\theta^{k+1}$ is closed. From this, by the standard argument using open mapping theorem,  it's easy to deduce that 
\[\varinjlim_{U\supseteq \bar{\mathcal{X}}_{K^p\Gamma(p^n),c}} \cO_C[[ \frac{y}{p^n}]][\frac{1}{p}] \widehat\otimes_C~ \omega^{-k}(U)\xrightarrow{1\otimes\theta^{k+1}} \varinjlim_{U\supseteq \bar{\mathcal{X}}_{K^p\Gamma(p^n),c}} \cO_C[[ \frac{y}{p^n}]][\frac{1}{p}] \widehat\otimes_C~ \omega^{k+2}(U)\]
has closed images with cokernel $\cO_C[[ \frac{y}{p^n}]][\frac{1}{p}] \widehat\otimes_C H^1_{\rig}(\Ig(K^p\Gamma(p^n)),\Sym^k)$. Taking the inductive limit over $n$ gives the proposition.
\end{proof}

\begin{para} \label{H1ord}
To deduce a description of the sheaf $(\coker d^{k+1})|_{\mathbb{P}^1(\Q_p)}$, we  note that  $B$ acts smoothly on $H^1_{\rig}(\Ig(K^p),\Sym^k)$. The smooth induction of it to $\GL_2(\Q_p)$ naturally defines a $\GL_2(\Q_p)$-equivariant sheaf on $\mathbb{P}^1(\Q_p)$ which will be denoted by $\mathcal{H}^1_{\ord}(K^p,k)$. 
Explicitly, the (right) action of $\GL_2(\Q_p)$ on $\infty$ induces a map $\pi:\GL_2(\Q_p)\to\mathbb{P}^1(\Q_p)$. Let $U$ be an open subset of $\mathbb{P}^1(\Q_p)$. Then
$\mathcal{H}^1_{\ord}(K^p,k)(U)$ is the set of smooth functions
\[f:\pi^{-1}(U)\to H^1_{\rig}(\Ig(K^p),\Sym^k)\]
such that $f(bg)=b\cdot f(g)$. It's clear that
\[H^0(\mathbb{P}^1(\Q_p),\mathcal{H}^1_{\ord}(K^p,k))=\Ind_B^{\GL_2(\Q_p)} H^1_{\rig}(\Ig(K^p),\Sym^k),\]
where $\Ind$ denotes the smooth induction. By abuse of notation, we will also view $\mathcal{H}^1_{\ord}(K^p,k)$ as a sheaf on $\Fl$ via the closed embedding $i:\mathbb{P}^1(\Q_p)\subseteq \Fl$. Rougly speaking, $\mathcal{H}^1_{\ord}(K^p,k)$ is defined by the rigid cohomology of the ordinary locus of modular curves.
\end{para}

\begin{para}[Definition of $e'_1,{e'_2}$]
Note that $B$ acts on the fiber of $\omega^l_\Fl$ at $\infty$ via the character ${e'_2}^l$ sending $\begin{pmatrix} a & b\\ 0 &d\end{pmatrix}\in B$ to $d^l$. We remark that  $e'_2=e_2$  in \ref{e_1e_2ch} as a character of $B$ but we don't put any Galois action on $e'_2$. As in \ref{e_1e_2ch},  $W\cdot {e'_2}^l$ denotes the twist of $W$ by ${e'_2}^l$ for any $B$-representation $W$ over $\Q_p$. Similarly, $e'_1:B\to\Q_p^\times$ denotes the character sending $\begin{pmatrix} a & b\\ 0 &d\end{pmatrix}\in B$ to $a$.

There is a Hausdorff LB-space structure on $ H^1_{\rig}(\Ig(K^p),\Sym^k)\cdot {e'_2}^l$ and the action of $B$ on it is locally analytic. Its locally analytic induction from $B$ to $\GL_2(\Q_p)$ defines naturally a $\GL_2(\Q_p)$-equivariant sheaf on $\mathbb{P}^1(\Q_p)$ whose global sections are given by the locally analytic induction
\[\Ind_B^{\GL_2(\Q_p)}H^1_{\rig}(\Ig(K^p),\Sym^k)\cdot {e'_2}^l . \]
The sheaf defined by $\Ind_B^{\GL_2(\Q_p)} {e'_2}^l$ is nothing but $\omega^l_{\Fl}|_{\mathbb{P}^1(\Q_p)}$. Hence Proposition \ref{cokdk+1infty} has the following corollaries.
\end{para}

\begin{cor} \label{cokerdk+1P1}
There are natural isomorphisms
\[(\coker d^{k+1})|_{\mathbb{P}^1(\Q_p)}\cong\omega^k_{\Fl}|_{\mathbb{P}^1(\Q_p)}\otimes_C \mathcal{H}^1_{\ord}(K^p,k)(k)\cdot \mathrm{t}^{n_1},\]
\[H^0(\mathbb{P}^1(\Q_p),\coker d^{k+1})\cong\Ind_B^{\GL_2(\Q_p)} H^1_{\rig}(\Ig(K^p),\Sym^k) (k)\cdot {e'_2}^k\mathrm{t}^{n_1}.\]
\end{cor}

Recall that $d'^{k+1}$ denotes the twist of $d^{k+1}$ by $(\Omega^1_{\Fl})^{\otimes k+1}\cong\omega_\Fl^{-2k-2}\otimes\det{}^{k+1}$.

\begin{cor} \label{cokerd'k+1}
For $\chi=(-k,0)$, 
\[(\coker d'^{k+1})|_{\mathbb{P}^1(\Q_p)}\cong\omega^{-k-2}_\Fl|_{\mathbb{P}^1(\Q_p)}\otimes_C \mathcal{H}^1_{\ord}(K^p,k)(k)\otimes\det{}^{k+1}\cdot \mathrm{t}^{-k}\]
\[H^0(\mathbb{P}^1(\Q_p),\coker d'^{k+1})\cong\Ind_B^{\GL_2(\Q_p)} H^1_{\rig}(\Ig(K^p),\Sym^k) (k)\cdot {e'_1}^{k+1}{e'_2}^{-1}\mathrm{t}^{-k}.\]
\end{cor}

\subsection{The Lubin-Tate space at infinite level} \label{LTinfty}
\begin{para}
Now we study the intertwining operator $I_k$ on the Drinfeld upper half plane $\Omega$. Since it is well-known that $\pi_\HT^{-1}(\Omega)$ is a finite disjoint union of the Lubin-Tate space $\mathcal{M}_{\LT,\infty}$ at infinite level, we essentially need to understand the geometry of the Hodge-Tate period map for $\mathcal{M}_{\LT,\infty}$. The key point here is that  the differential operators $d^{k+1}$ and $\bar{d}^{k+1}$ are swapped under the isomorphism between Lubin-Tate and Drinfeld towers, hence results we proved for $\bar{d}^{k+1}$ before (Proposition \ref{dbarsurj}) can be applied directly to $d^{k+1}$. 

In this subsection, we will restrict ourselves to this local picture and come back to modular curves in next subsection.
\end{para}

\begin{para}\label{LTsetup}
First we recall the definition of the Lubin-Tate towers. One reference is \cite[\S 6]{SW13}. However, we warn the readers that we work with \textit{contravariant} objects instead of \textit{covariant} objects used in the reference. Hence some statements might be slightly different from \cite{SW13}.

Let $H_0$ be a one-dimensional formal group over $\bar{\F}_p$ of height $2$. I is unique up to isomorphism. Let $\mathrm{Nilp}_{W(\bar{\F}_p)}$ be the category of $W(\bar{\F}_p)$-algebras on which $p$ is nilpotent. Consider the functor which  assigns $R\in\mathrm{Nilp}_{W(\bar{\F}_p)}$ to the set of isomorphism classes of deformations of $H_0$ to $R$, i.e. pairs $(G,\rho)$ where $G$ is a $p$-divisible group over $R$ and $\rho:H_0\otimes_{\bar{\F}_p} R/p \to G \otimes_{R} R/p$ is a quasi-isogeny.
 This is represented by a formal scheme $\mathcal{M}$ over $\Spf W(\bar{\F}_p)$. Let  $\mathcal{M}^{(i)}\subseteq  \mathcal{M}$ be the subset parametrizing quasi-isogenies of height $i\in\Z$. Then each $\mathcal{M}^{(i)}$ is non-canonically isomorphic to $\Spf W(\bar{\F}_p)[[T]]$ and $\mathcal{M}$ is the disjoint union of  $\mathcal{M}^{(i)},i\in\Z$. For $*=(i)$ or empty, we denote by $\mathcal{M}^{*}_{\LT}$ the base change of the generic fiber of $\mathcal{M}^{*}$ (viewed as an adic space)  from $\Spa(W(\bar{\F}_p)[\frac{1}{p}],W(\bar{\F}_p))$  to $\Spa(C,\cO_C)$. 

Let $(\mathcal{G},\rho)$ be a universal deformation of $H_0$ on $\mathcal{M}$. Its covariant Dieudonn\'e crystal defines a vector bundle $M(\mathcal{G})$ of rank $2$ on $\mathcal{M}_{\LT}$ equipped with an integrable connection $\nabla_{\LT,0}$, the (dual of) Gauss-Manin connection. Let $M(H_0)$ denote the covariant Dieudonn\'e module of $H_0$, a free $W(\bar{\F}_p)$-module of rank $2$. Then $\rho$ induces a natural trivialization $M(\mathcal{G})\cong M(H_0)\otimes_{W(\bar{\F}_p)}\cO_{\mathcal{M}_{\LT}}$ under which $\nabla_{\LT,0}$ is identified with the standard connection on $\cO_{\mathcal{M}_{\LT}}$.

The Lie algebra of $\mathcal{G}$ defines a line bundle on $\mathcal{M}_{\LT}$, whose dual will be denoted by $\omega_{\LT,0}$. By the Grothendieck-Messing theory, $(\mathcal{G},\rho)$ gives rises to a surjection
\[M(H_0)\otimes_{W(\bar{\F}_p)}\cO_{\mathcal{M}_{\LT}} \cong M(\mathcal{G})\to (\omega_{\LT,0})^{-1} \]
whose kernel is a line bundle on $\mathcal{M}_{\LT}$. This induces the so-called \textit{Gross-Hopkins period map} \cite{GH94}
\[\pi_{\GM}:\mathcal{M}_{\LT}\to \Fl_{\GM},\]
where $\Fl_{\GM}$ is the adic space over $\Spa(C,\cO_C)$ associated to the flag variety parametrizing $1$-dimensional quotients of the $2$-dimensional $C$-vector space $M(H_0)\otimes_{W(\bar{\F}_p)} C$. Again it follows from the Grothendieck-Messing theory that this map is  an \'etale morphism of adic spaces locally
of finite type. Moreover it admits local sections, cf. \cite[Lemma 6.1.4]{SW13}. By definition, the pull-back of the tautological ample line bundle on $\Fl_{\GM}$ to $\mathcal{M}_{\LT}$ is $(\omega_{\LT,0})^{-1} $.

Let $D(\mathcal{G})$ be the dual of $M(\mathcal{G})$ as a vector bundle. Then there is a natural inclusion $\omega_{\LT,0}\subset D(\mathcal{G})$ which gives rise to a decreasing filtration on  $D(\mathcal{G})$ (the Hodge filtration) with $\Fil^0=D(\mathcal{G})$, $\Fil^1=\omega_{\LT,0}$ and $\Fil^2=0$. It trivially satisfies the Griffiths transversality with respect to the connection $\nabla_{\LT}$ on $D(\mathcal{G})$ induced by $\nabla_{\LT,0}$. We have the usual Kodaira-Spencer map defined as the composite map
\begin{eqnarray} \label{KSmap}
KS:\Fil^1D(\mathcal{G})\xrightarrow{\nabla_{\LT}} D(\mathcal{G})\otimes_{\cO_{\mathcal{M}_{\LT}}} \Omega^1_{\mathcal{M}_{\LT}} \to \gr^0  D(\mathcal{G}) \otimes_{\cO_{\mathcal{M}_{\LT}}} \Omega^1_{\mathcal{M}_{\LT}}.
\end{eqnarray}
An easy and standard computation on $\Fl_{\GM}$ shows the following well-known result.
\end{para}

\begin{prop} \label{KSLT}
$KS$ is an isomorphism.
\end{prop}

\begin{para} \label{LTet}
Next we turn to the \'etale side. Note that $\mathcal{G}$ defines a $\Z_p$-local system of rank $2$ on the \'etale site of $\mathcal{M}_{\LT}$, whose dual will be denoted by $V_\LT$. Here we take a dual in order to be consistent with our normalization used for modular curves.
For $n\geq0$, we get a $\GL_2(\Z/p^n\Z)=\GL_2(\Z_p)/\Gamma(p^n)$-\'etale covering $\mathcal{M}_{\LT,n}$ of $\mathcal{M}_{\LT}$ by considering isomorphisms between $V_\LT/p^nV_\LT$ and $(\Z/p^n)^{2}$.  One main result of the work of Scholze-Weinstein \cite[Theorem 6.3.4]{SW13} shows that there exists a perfectoid space $\mathcal{M}_{\LT,\infty}$ over $\Spa(C,\cO_C)$ such that
\[\mathcal{M}_{\LT,\infty} \sim\varprojlim_n \mathcal{M}_{\LT,n}.\]
Strictly speaking, $\mathcal{M}_{\LT,\infty}$ is the strong completion of the base change to $C$ of the preperfectoid space constructed in the reference, cf. Proposition 2.3.6 of \cite{SW13}.  There is a natural continuous right action of  $\GL_2(\Z_p)$ on $\mathcal{M}_{\LT,\infty}$. A purely $p$-adic Hodge-theoretic description of $\mathcal{M}_{\LT,\infty}$ can be found in  \cite[Proposition 6.3.9]{SW13}. Using this, one can extend the action of $\GL_2(\Z_p)$ to $\GL_2(\Q_p)$. As a consequence of our normalization, the action of $\GL_2(\Q_p)$ here differs from the one in the reference by $g\mapsto (g^{-1})^t$.

Let $\omega_{\LT,\infty}$ be the pull-back of  $\omega_{\LT,0}$ to $\mathcal{M}_{\LT,\infty}$. Since the Tate module of $\mathcal{G}$ gets trivialized on $\mathcal{M}_{\LT,\infty}$,  the dual of the Hodge-Tate sequence induces a surjection
\[\Z_p^2\otimes_{\Z_p} \cO_{\mathcal{M}_{\LT,\infty}} \to \omega_{\LT,\infty}(-1)\]
whose kernel is isomorphic to $\omega^{-1}_{\LT}$ if we choose an isomorphism between $\mathcal{G}$ and its dual, i.e. a principal polarization here.
 Recall that $\Fl=\mathbb{P}^1$ denotes the flag variety of $\GL_2$. Then  the above surjection defines a $\GL_2(\Q_p)$-equivariant map, called the \textit{Hodge-Tate period map}
\[\pi_{\LT,\HT}: \mathcal{M}_{\LT,\infty}\to \Fl.\]
The image is exactly the Drinfeld upper half plane $\Omega$. This will be clear from the point of view of the Drinfeld towers, cf. Theorem \ref{LTDrdual} below.  It follows from the construction that the pull-back of the tautological ample line bundle $\omega_{\Fl}$ along $\pi_{\LT,\HT}$ is $\omega_{\LT,\infty}(-1)$.
\end{para}

\begin{para} \label{AA+}
Let $X=\Spa(A,A^+)$ be an open affinoid subset of $\mathcal{M}_{\LT}$ and $\tilde{X}=\Spa(B,B^+)$ be its preimage in $\mathcal{M}_{\LT,\infty}$.
Then $X$ is a one-dimensional smooth affinoid adic space over $\Spa(C,\cO_C)$ and $\tilde{X}$ is a $G=\GL_2(\Z_p)$-Galois pro-\'etale perfectoid covering of $X$. In \cite[Theorem 3.1.2]{Pan20}, we show that the $\GL_2(\Z_p)$-locally analytic vectors $B^{\la}\subseteq B$ satisfy a first-order differential equation $\theta_{\tilde{X}}=0$ for some $\theta_{\tilde{X}}\in B\otimes_{\Q_p}\Lie(G)$ (under some smallness assumption on $X$ which can be easily removed here, cf. \cite[Remark 3.1.5]{Pan20}).  To describe this differential equation, we recall the following construction on the flag variety $\Fl$ of $\GL_2/C$ in the last paragraph of \ref{brr}.  For a $\Spa(C,\cO_C)$-point $x$ of $\Fl$,  let $\mathfrak{b}_x,\kn_x\subseteq\mathfrak{gl}_2(C)$  be its corresponding Borel subalgebra and nilpotent subalgebra. Let
\begin{eqnarray*}
\mathfrak{g}^0&:=&\cO_{\Fl}\otimes_{C}\mathfrak{gl}_2(C),\\
\mathfrak{b}^0&:=&\{f\in \mathfrak{g}^0\,| \, f_x\in \mathfrak{b}_x,\mbox{ for all }\Spa(C,\cO_C)\mbox{-point }x\in\Fl\},\\
\mathfrak{n}^0&:=&\{f\in \mathfrak{g}^0\,| \, f_x\in \mathfrak{n}_x,\mbox{ for all }\Spa(C,\cO_C)\mbox{-point }x\in\Fl\}.
\end{eqnarray*}
$\mathfrak{g}^0$ acts naturally on $B^\la$ through $\pi_{\LT,\HT}$ in the following sense: suppose that $U$ is an open subset of $\Fl$ containing $\pi_{\LT,\HT}(\tilde{X})$, then there is a natural map $\cO_{\Fl}(U)\to B^{\la}$ induced by  $\pi_{\LT,\HT}$ and induces a natural action of $\mathfrak{g}^0(U)=\cO_{\Fl}(U)\otimes_{C}\mathfrak{gl}_2(C)$ on $B^{\la}$. Hence we get natural actions of $\mathfrak{b}^0$ and $\mathfrak{n}^0$ on $B^\la$. Note that $\mathfrak{n}^0$ is an invertible sheaf. Locally it is generated by one differential operator.
\end{para}

\begin{thm} \label{LTde}
Let $\tilde{X}=\Spa(B,B^+)$ be as above. Then $\theta_{\tilde{X}}$ is  given by a generator of $\mathfrak{n}^0$ up to $B^\times$, i.e.  $B^{\la}$ is annihilated by $\mathfrak{n}^0$. 
\end{thm}

\begin{proof}
This follows from the corresponding result for the modular curves \cite[Theorem 4.2.7]{Pan20} by observing that the Lubin-Tate space $\mathcal{M}^({0)}_{\LT}$ can be embedded into modular curves. In fact, one can also repeat the argument in \cite[\S 4.2]{Pan20}. The key point in the computation there is that the \textit{Kodaira-Spencer map} is an isomorphism (Proposition \ref{KSLT}). We give a sketch of proof here, which is more conceptual than the original argument and was first suggested to me by Michael Harris.

To compute $\theta_{\tilde{X}}$, as in \cite[Remark 3.3.7]{Pan20}, we can choose a faithful finite dimensional continuous $\Q_p$-representation $V$ of $G$ and compute the corresponding \textit{Higgs field}  \cite[Remark 3.1.7]{Pan20}
\[\phi_{V}:B\otimes_{\Q_p} V \to \Omega^1_{A/C}\otimes_{A}B\otimes_{\Q_p} V(-1). \]
Then $\theta_{\tilde{X}}$ is obtained by choosing a generator of $\Omega^1_{A/C}$. In our situation, we take $V=\Q_p^2$, the $2$-dimensional representation associated to $V_{\LT}$. To compute the Higgs field, we note that $V_{\LT}$ is \textit{de Rham} and make use of the corresponding variation of $p$-adic Hodge structure $(D(\mathcal{G}),\nabla_{\LT})$. Recall that there is the dual of the Hodge-Tate sequence $0\to \omega^{-1}_{\LT}\to V\otimes \cO_{\mathcal{M}_{\LT,\infty}} \to \omega_{\LT,\infty}(-1)\to 0$ (after choosing a polarization).  This gives rise to an ascending filtration on $V\otimes_{\Q_p} \cO_{\mathcal{M}_{\LT,\infty}}$ with $\Fil_0=\omega^{-1}_{\LT}$, $\Fil_{-1}=0$ and $\Fil_1$ is everything. Observe that when we evaluate this at $\tilde{X}$,
\begin{itemize}
\item $\theta_{V}(\Fil_0(\Q_p^2\otimes B))=0$ and  $\theta_V$ induces a map
\[ \omega_{\LT,\infty}(B)(-1)=\gr_1 (\Q_p^2\otimes B)\xrightarrow{\theta_V} \Fil_0 (\Q_p^2\otimes B)\otimes_A \Omega^1_{A/C}(-1)= \omega^{-1}_{\LT} (B)\otimes_A \Omega^1_{A/C}(-1)\]
which is the unique $B$-linear map extending the Kodaira-Spencer map $KS$ \eqref{KSmap}.
\end{itemize}
In the classical theory of non-abelian Hodge theory over complex numbers, this result follows directly from the construction of the Higgs bundle from a variation of Hodge structures (which was called a system of Hodge bundles by Simpson). In this $p$-adic case, it is not too hard to deduce this from  the work of Liu-Zhu \cite[Theorem 2.1, 3.8]{LZ17} by observing that the Higgs field is obtained from the $0$th graded piece of  the connection $(V\otimes_{\Q_p}\cO\B_{\dR},1\otimes\nabla)$. \footnote{Here we need a compatibility between the Higgs bundle in Liu-Zhu's work and the Higgs field constructed in \cite{Pan20}. This can be checked on a toric chart. We plan to provide more details in a future work.} Now since $KS$ is an isomorphism,  we see that $\phi_V$ is essentially a generator of $\kn^0$ up to $B^\times$ by unraveling the definition of $\kn^0$. This is exactly what we need to show.
\end{proof}

\begin{para} \label{khLT}
Keep the same notation as above.
Once we have Theorem \ref{LTde}, we can repeat our work in \cite[\S4]{Pan20}. The action of $\mathfrak{b}^0$ on $B^{\la}$ factors through  the quotient $\mathfrak{b}^0/\mathfrak{n}^0$. As in  \ref{brr}, let  $\mathfrak{h}:=\{\begin{pmatrix} * & 0 \\ 0 & * \end{pmatrix}\}$ be a Cartan subalgebra of the Borel subalgebra $\mathfrak{b}:=\{\begin{pmatrix} * & * \\ 0 & * \end{pmatrix}\}$. It acts on $B^{\la}$ via the natural embedding $\mathfrak{h}\to\cO_{\mathrm{\Fl}}\otimes_{C}\mathfrak{h}\cong\mathfrak{b}^0/\mathfrak{n}^0$ and we denote this action by $\theta_{\LT,\kh}$. Similarly as in \ref{omegakla}, the same construction defines a natural action of $\kh$ on the $\GL_2(\Z_p)$-locally analytic vectors of $\omega_{\LT,\infty}$, which is also denoted by $\theta_{\LT,\kh}$.

We can also describe elements in $B^{\la}$ as in \cite[\S4.3]{Pan20}. By fixing a principal polarization of $\mathcal{G}$, we get a natural map $\mathcal{M}_{\LT,\infty}\to\mathrm{Isom}(\Z_p,\Z_p(1))$ of topological spaces, which classically can be defined in terms of the connected components of $\mathcal{M}_{\LT,n}$. Choosing a generator of $\Z_p(1)$ induces an identification $\mathrm{Isom}(\Z_p,\Z_p(1))=\Z_p^\times$ and $1\in\Z_p^\times$ defines a global section  of $\cO_{\mathcal{M}_{\LT,\infty}}$ which will be denoted by $\mathrm{t}$. As in \cite[\S4.3.1]{Pan20}, for $n\geq 1$, we can find
\begin{itemize}
\item $\mathrm{t}_n\in H^0(\mathcal{M}_{\LT,n},\cO_{\mathcal{M}_{\LT,n}})$ which factors through $\pi_0(\mathcal{M}_{\LT,n})$ so that $||\mathrm{t}-\mathrm{t}_n||\leq p^{-n}$.
\end{itemize}

Fix a generator of $\Z_p(1)$. Let $e_1,e_2\in H^0(\mathcal{M}_{\LT,\infty},\omega_{\LT,\infty})$ be the images of $(1,0),(0,1)\in\Z_p^2$ under the map $\Z_p^2(1)\otimes_{\Z_p} \cO_{\mathcal{M}_{\LT,\infty}} \to \omega_{\LT,\infty}$ in the dual of the Hodge-Tate sequence. Note that $x=e_2/e_1$ is a standard coordinate function on $\Fl$ and  is an invertible function on $\mathcal{M}_{\LT,\infty}$, or equivalently $e_1$ is invertible, because $\pi_{\LT,\HT}(\mathcal{M}_{\LT,\infty})= \Omega$. 

Let $G_n=\Gamma(p^n)$ and $X_n$ be the preimage of $X$ in $\mathcal{M}_{\LT,n}$. Then $X_n=\Spa(B^{G_n},(B^+)^{G_n})$.  Following \cite[\S4.3.5]{Pan20}, for each $n\geq1$, we can find 
\begin{itemize}
\item an integer $r(n)>r(n-1)>0$;
\item $x_n\in B^{G_{r(n)}}$ such that  $\|x-x_n\|_{G_{r(n)}}=\|x-x_n\|\leq p^{-n}$ in $B$;
\item $e_{1,n}\in \omega_{\LT,\infty}(\tilde{X})^{G_{r(n)}}$ invertible such that  $\|1-e_{1}/e_{1,n}\|_{G_{r(n)}}=\|1-e_{1}/e_{1,n}\|\leq p^{-n}$ in $B$. This implies that $\log(\frac{e_1}{e_{1,n}}):=-\sum_{i=1}^{+\infty}(-1)^i\frac{1}{i}(\frac{e_1}{e_{1,n}}-1)^i$ converges.
\item $||\mathrm{t}-\mathrm{t}_n||_{G_{r(n)}}=||\mathrm{t}-\mathrm{t}_n||\leq p^{-n}$. Similarly, $\log(\frac{\mathrm{t}}{\mathrm{t}_{n}})$ converges.
\end{itemize}
Here $||\cdot||_{G_{r(n)}}$ denotes the norm on $G_{r(n)}$-analytic vectors. Using these elements, we have the following description of $B^{\la}$.
\end{para}

\begin{thm} \label{expGL2}
 For any $n\geq0$, given a sequence of sections 
\[c_{i,j,k}\in B^{G_{r(n)}},i,j,k=0,1,\cdots\] 
such that the norms of $c_{i,j,k}p^{(n-1)(i+j+k)}$, $i,j,k\geq 0$ are uniformly bounded, then
\[f=\sum_{i,j,k\geq 0} c_{i,j,k}(x-x_n)^i \left(\log(\frac{e_1}{e_{1,n}})\right)^j\left(\log(\frac{\mathrm{t}}{\mathrm{t}_{n}})\right)^k\]
converges in $B^{G_{r(n)-\an}}$ and any $G_{n}$-analytic vector in $B$ arises in this way.
\end{thm}

\begin{proof}
Same proof as \cite[Theorem 4.3.9]{Pan20}.
\end{proof}

\begin{para}
Let $O_{D_p}:=\End(H_0)$ and $D_p:=\End(H_0)\otimes_{\Z}\Q$. Then $D_p$ is  a non-split quaternion algebra over $\Q_p$. The multiplicative group $D_p^\times$ is the group of self-quasi-isogenies of $H_0$ and acts naturally on $\mathcal{M}_{\LT}$ and $\Fl_{\GM}$ (on the right) through its action on $H_0$. By choosing a basis of $M(H_0)$, one may identify $D_p^\times$ as a subgroup of $\GL_2(C)$ and identify $\Fl_{\GM}$ with the projective space $\mathbb{P}^1$ over $\Spa(C,\cO_C)$, on which the action  of $D_p^\times$ is the usual linear action. In particular $D_p^\times$  acts continuously on $\Fl_{\GM}$ in the sense of \cite[Proposition 6.5.5]{SW13}. Its action on  $\mathcal{M}_{\LT}$ is also continuous by \cite[Proposition 19.2]{GH94}. Clearly $\pi_{\GM}$ is $D_p^\times$-equivariant.  
We remark that the left action of $D_p^\times$ on $M(H_0)\otimes_{W(\bar{F}_p)}C$ is irreducible and its action on $\wedge^2 M(H_0)\otimes_{W(\bar{F}_p)}C$ is the reduced norm map.

It follows from the constructions that  $\mathcal{M}_{\LT,n},n\geq 0$  are $D^\times_p$-equivariant finite coverings. The continuity of the action of $D^\times_p$ on $\mathcal{M}_{\LT}$ implies the continuity of the action on each $\mathcal{M}_{\LT,n}$ and hence on $\mathcal{M}_{\LT,\infty}$. Note that  $\omega_{\LT,0}$  and its pull-back to $\mathcal{M}_{\LT,n}$ and $\mathcal{M}_{\LT,\infty}$ are $D_p^\times$-equivariant line bundles. We denote by 
\[\pi_{\LT,\GM}:\mathcal{M}_{\LT,\infty}\to \Fl_{\GM}\]
the composite of the projection map $\mathcal{M}_{\LT,\infty}\to \mathcal{M}_{\LT}$ and $\pi_{\GM}$. It is $\GL_2(\Q_p)\times D_p^\times$-equivariant with respect to the trivial action of $\GL_2(\Q_p)$ on $\Fl_{\GM}$. We note that the centers $\Q_p^\times\subseteq \GL_2(\Q_p)$ and $\Q_p^\times\subseteq D_p^\times$ act in the same way on $\mathcal{M}_{\LT,\infty}$.

Keep the same notation as in \ref{AA+}. The continuity of the action of $D^\times_p$ shows that we may find an open compact subgroup $K\subseteq D^\times_p$ such that $X$ is $K$-stable. Then $K$ acts continuously on $B$. We denote by $B^{D^\times_p-\la}\subseteq B$ the subspace of $K$-locally analytic vectors, which does not depend on the choice of $K$. One consequence of Theorem \ref{expGL2} is an one-direction inclusion between the $K$-locally analytic vectors and $\GL_2(\Z_p)$-locally analytic vectors. 
\end{para}

\begin{cor} \label{BlasubsetBDpla}
$B^{\la}\subseteq B^{D^\times_p-\la}$.
\end{cor}

We will see (Corollary \ref{GL2Dpan}) that $B^{D^\times_p-\la}\subseteq B^{\la}$ by swapping the roles of $D^\times_p$ and $\GL_2(\Q_p)$. 

\begin{proof}
Recall that $X_{r(n)}=\Spa(B^{G_{r(n)}},(B^+)^{G_{r(n)}})$ is affinoid of finite type over  $\Spa(C,\cO_C)$. In particular, there exists an open subgroup $K'\subseteq K$ of the form $1+p^m\cO_{D_p}$ such that the action of $K'$ on $(B^+)^{G_{r(n)}}/p$ is trivial. This implies that the action of $K'$ on $B^{G_{r(n)}}$ is analytic. Hence the $K'$-analytic norm on $B^{G_{r(n)}}$ is equivalent with the norm induced from $B$. Note that $D^\times_p$ acts trivially on $x,e_1$. Its action on $\mathrm{t}$ factors through the reduced norm map hence agrees the reduced norm map because the center of $D^\times_p$ acts in the same way as the center of $\GL_2(\Q_p)$.  By shrinking $K'$ if necessary, we may assume that the $K'$-analytic norm $||x-x_n||_{K'}$ is equal to $||x-x_n||$ and similar statements hold for $\log(e_1/e_{1,n})$ and $\log(\mathrm{t}/\mathrm{t}_n)$. Hence the function $f$ in Theorem \ref{expGL2} is $K'$-analytic. Our claim now follows as any $f\in B^{\la}$ has this form.
\end{proof}

\begin{rem}
So far we are assuming that $X$ is an affinoid subset of $\mathcal{M}_{\LT}$. It is easy to see that all of these results are still true if $X$ is an open affinoid subset of $\mathcal{M}_{\LT,n}$ for some $n\geq0$.
\end{rem}

\begin{para} \label{OLTla}
Let $\mathcal{M}_{\LT,\infty}^{(0)}$ be the preimage of $\mathcal{M}_{\LT}^{(0)}$ in $\mathcal{M}_{\LT,\infty}$. Recall that the superscript $(0)$ means the subset parametrizing quasi-isogenies  of height $0$.  Note that $\mathcal{M}_{\LT,\infty}^{(0)}$ is $\GL_2(\Q_p)^0\times \cO^\times_{D_p}$-stable, where $\GL_2(\Q_p)^0\subseteq \GL_2(\Q_p)$ denotes the subgroup of elements with determinants in $\Z_p^\times$. Since $\pi_{\LT,\HT}$ has images in $\Omega$, we denote by 
\[\pi_{\LT,\HT}^{(0)}:\mathcal{M}_{\LT}^{(0)}\to \Omega\] 
the restriction of $\pi_{\LT,\HT}$ to $\mathcal{M}_{\LT}^{(0)}$ and $\Omega$.
By embedding Lubin-Tate spaces into the modular curves (Theorem \ref{ssunif} below), we may invoke \cite[Theorem III.1.2]{Sch15} and conclude that there exists a basis of open affinoid subsets $\mathfrak{B}$ of $\Omega$ such that for any $U\in\mathfrak{B}$, the preimage $(\pi^{(0)}_{\LT,\HT})^{-1}(U)$ in $\mathcal{M}_{\LT,\infty}^{(0)}$ is affinoid perfectoid and is also the preimage of some open affinoid subset of $\mathcal{M}^{(0)}_{\LT,n}$. In particular, we can apply previous results for $\tilde{X}$ to $(\pi^{(0)}_{\LT,\HT})^{-1}(U)$. Let 
\[\cO_{\LT}:=\pi^{(0)}_{\LT,\HT}{}_* \cO_{\mathcal{M}_{\LT,\infty}^{(0)}}.\]
This is a $\GL_2(\Q_p)^0$-equivariant sheaf on $\Omega$ equipped with an action of $\cO^\times_{D_p}$. We denote by 
\[\cO^{\la}_{\LT}\subseteq \cO_{\LT}\]
the subsheaf of $\GL_2(\Q_p)^0$-locally analytic sections. By Theorem \ref{LTde}, $\cO^{\la}_{\LT}$ is annihilated by $\mathfrak{n}^0$ and we get an induced action of $\kh$ on it via $\mathfrak{h}\to\cO_{\mathrm{\Fl}}\otimes_{C}\mathfrak{h}=\mathfrak{b}^0/\mathfrak{n}^0$, which will also be denoted by $\theta_{\LT,\kh}$ by abuse of notation. As before, given a weight $\chi=(n_1,n_2):\kh\to C$, we denote the $\chi$-isotypic part of $\cO^{\la}_{\LT}$  by $\cO^{\la,\chi}_{\LT}$. 
\end{para}

\begin{para}[Compare \ref{lalg0k}] \label{LTlalg}
Suppose $\chi=(n_1,n_2)\in\Z^2$ and $k=n_2-n_1\geq 0$. By Theorem \ref{expGL2}, for $i\in\{0,\cdots,k\}$ and $s\in\omega^{-k,\sm}_{\LT}$, the product $\mathrm{t}^{n_1}e_1^ie_2^{k-i} s$ defines an element in $\cO^{\la,(n_1,n_2)}_{\LT}$. We denote by
\[\cO^{\lalg,(n_1,n_2)}_{\LT}\subseteq \cO^{\la,(n_1,n_2)}_{\LT}\]
the subsheaf spanned by sections of this form. As pointed out in \ref{lalg0k} and suggested by the notation, it exactly consists of $\GL_2(\Z_p)$-locally algebraic vectors. Equivalently, it is the image of the natural map 
\[H^0(\Fl,\omega^k_{\Fl})(k)\otimes_C \omega^{-k,\sm}_{\LT}\cdot\mathrm{t}^{n_1}\to \cO^{\la,(n_1,n_2)}_{\LT}.\]
From this, we see that $\cO^{\lalg,(n_1,n_2)}_{\LT}=\Sym^k V(k)\otimes_{\Q_p}  \omega^{-k,\sm}_{\LT}\cdot\mathrm{t}^{n_1}$ as a representation of $\GL_2(\Z_p)$.
\end{para}

\begin{para}\label{dLT}
Let
\[\cO^{\sm}_{\LT}\subseteq \cO_{\LT}\]
be the subsheaf of $\GL_2(\Q_p)^0$-smooth sections. Equivalently, this is also 
\[(\pi^{(0)}_{\LT,\HT})_{*} (\varinjlim_{n}(\pi_n)^{-1} \cO_{\mathcal{M}^{(0)}_{\LT,n}}),\] 
where $\pi_n:\mathcal{M}_{\LT,\infty}\to \mathcal{M}_{\LT,n}$ is the natural projection. It is naturally an $\cO_{\Omega}$-module. Similarly, we can define $\omega^{\la}_{\LT},\omega^{\sm}_{\LT}$ and their tensor powers $\omega^{k,\la}_{\LT},\omega^{k,\sm}_{\LT}$.  They are natural $\cO^{\sm}_{\LT}\otimes_{C}\cO_{\Omega}$-modules. Denote by 
\[\Omega^{1}_{\LT}:=(\pi^{(0)}_{\LT,\HT})_{*} (\varinjlim_{n}(\pi_n)^{-1} \Omega^1_{\mathcal{M}^{(0)}_{\LT,n}}).\]
It is non-canonically isomorphic to $\omega^{2,\sm}_{\LT}$ via the Kodaira-Spencer isomorphism and we have the usual derivation
\[d_{\LT}:\cO^{\sm}_{\LT}\to \Omega^{1}_{\LT}.\]
Let $k$ be a non-negative integer. As in \ref{XCY},  the Kodaira-Spencer isomorphism implies that the $(k+1)$-th power of  $d_{\LT}$ defines a map
\[(d_{\LT})^{k+1}:\omega^{-k,\sm}_{\LT}\to \omega^{-k,\sm}_{\LT}\otimes_{\cO^{\sm}_{\LT}} (\Omega^{1}_{\LT})^{\otimes k+1}.\]
We  can repeat our work in \ref{Do1} with everything with subscript $K^p$ replaced by $\LT$. In particular, we have the following result which can be viewed as the restriction of Theorem \ref{I1} to the supersingular locus. \end{para}

\begin{thm} \label{LTI1}
Suppose $\chi=(n_1,n_2)\in\Z^2$ and $k=n_2-n_1\geq 0$. Then there exists a unique natural continuous operator
\[d_{\LT}^{k+1}:\cO^{\la,\chi}_{\LT}\to \cO^{\la,\chi}_{\LT}\otimes_{\cO^{\sm}_{\LT}}(\Omega^1_{\LT})^{\otimes k+1}\]
satisfying the following properties:
\begin{enumerate}
\item $d_{\LT}^{k+1}$ is $\cO_{\Omega}$-linear;
\item $d_{\LT}^{k+1}(\mathrm{t}^{n_1}e_1^ie_2^{k-i} s)=\mathrm{t}^{n_1}e_1^ie_2^{k-i}  (d_{\LT})^{k+1}(s)$ for any section $s\in \omega^{-k,\sm}_{\LT}$ and $i=0,1,\cdots,k$.
\end{enumerate}
\end{thm}

Similarly, we can repeat our work in \ref{DII}, cf.  Theorem \ref{I2} and Proposition \ref{dbarsurj}.

\begin{thm} \label{LTI2}
Suppose $\chi=(n_1,n_2)\in\Z^2$ and $k=n_2-n_1\geq 0$. 
%Let $-w\cdot(-\chi)=(n_2+1,n_1-1)$. 
Then there exists a natural continuous operator
\[\bar{d}_{\LT}^{k+1}:\cO^{\la,\chi}_{\LT}\to \cO^{\la,\chi}_{\LT}\otimes_{\cO_{\Omega}}(\Omega^1_{\Omega})^{\otimes k+1}\]
formally can be regarded as $(d_{\Omega})^{k+1}$ satisfying following properties:
\begin{enumerate}
\item $\bar{d}_{\LT}^{k+1}$ is $\cO^{\sm}_{\LT}$-linear;
\item there exists a non-zero constant $c\in\Q$ such that $\bar{d}_{\LT}^{k+1}(s)=c(u^+)^{k+1}(s)\otimes (dx)^{k+1}$ for any $s\in \cO^{\la,\chi}_{\LT}$ ,where $u^+=\begin{pmatrix} 0 & 1 \\ 0 & 0 \end{pmatrix} \in \Lie(\GL_2(\Q_p))$.
\end{enumerate}
Moreover,
\begin{enumerate}
\item $\bar{d}_{\LT}^{k+1}$ is surjective and commutes with $\GL_2(\Z_p)$.
\item $\ker(\bar{d}_{\LT}^{k+1})=\cO^{\lalg,\chi}_{\LT}$.
\end{enumerate}
\end{thm}

In the next section, we will show that there is a similar description of $\ker d_{\LT}^{k+1}$ and $\coker d_{\LT}^{k+1}$ by swapping the roles of $\GL_2(\Q_p)$ and $D^\times_p$.

\subsection{The Drinfeld space at the infinite level} \label{Drinfty}

\begin{para}
In this subsection, we study the action of $D_p^\times$ on $\mathcal{M}_{\LT,\infty}$. To do this, we use the isomorphism between the Lubin-Tate space and Drinfeld space at the infinite level. See Theorem E of \cite{SW13} and the discussion below it for some historical remarks. Roughly speaking, all of the results for Lubin-Tate towers obtained in the previous subsection are also true for the Drinfeld towers by exactly the same arguments.
\end{para}

\begin{para}[Compare \ref{LTsetup}] \label{Drsetup}
We begin by recalling the construction of Drinfeld towers, cf. \cite{Dr76,BC91}. 

Let $H_1$ be a special formal $\cO_{D_p}$-module over $\bar{\F}_p$. It is unique up to isogeny and  is isomorphic to $H_0\times H_0$ as a formal group.  Consider the functor which  assigns $R\in\mathrm{Nilp}_{W(\bar{\F}_p)}$ to the set of isomorphism classes of deformations of $H_1$ to $R$, i.e. pairs $(G,\rho)$ where $G$ is a special formal $\cO_{D_p}$-module over $R$ and $\rho:H_1\otimes_{\bar{\F}_p} R/p \to G \otimes_{R} R/p$ is a quasi-isogeny. Drinfeld showed that this is pro-represented by a formal scheme $\mathcal{M}'$ over $\Spf W(\bar{\F}_p)$. We denote by $\mathcal{M}_{\Dr}$ the base change of the generic fiber of $\mathcal{M}'$ (viewed as an adic space)  from $\Spa(W(\bar{\F}_p)[\frac{1}{p}],W(\bar{\F}_p))$  to $\Spa(C,\cO_C)$. Then $\mathcal{M}_{\Dr}$ is the disjoint union of $\mathcal{M}^{(i)}_{\Dr}$, $i\in\Z$, where $\mathcal{M}^{(i)}_{\Dr}$ denotes the subset with quasi-isogeny of degree $i$.

Let $(\mathcal{G}',\rho')$ be a universal deformation of $H_1$ on $\mathcal{M}'$. Its covariant Dieudonn\'e crystal defines a vector bundle $M(\mathcal{G}')$ of rank $4$ on $\mathcal{M}_{\Dr}$ equipped with an action of $D_p$ and an integrable connection $\nabla_{\Dr,0}$, the (dual of) Gauss-Manin connection. Let $M(H_1)$ denote the Dieudonn\'e module of $H_1$. Then $\rho'$ induces a natural trivialization $M(\mathcal{G}')\cong M(H_1)\otimes_{W(\bar{\F}_p)}\cO_{\mathcal{M}_{\Dr}}$ under which $\nabla_{\Dr,0}$ is identified with the standard connection on $\cO_{\mathcal{M}_{\Dr}}$. 

The Lie algebra of $\mathcal{G}'$ defines a rank-$2$ vector bundle $\Lie(\mathcal{G}')\otimes C$ on $\mathcal{M}_{\Dr}$. By the Grothendieck-Messing theory, $(\mathcal{G}',\rho')$ gives rises to a $D_p$-equivariant surjection
\[M(H_1)\otimes_{W(\bar{\F}_p)}\cO_{\mathcal{M}_{\Dr}} \cong M(\mathcal{G}')\to \Lie(\mathcal{G}')\otimes C. \]
This induces the so-called \textit{Gross-Hopkins period map} \cite{GH94}
\[\pi'_{\GM}:\mathcal{M}_{\Dr}\to \Fl'_{\GM},\]
where $\Fl'_{\GM}$ is the adic space over $\Spa(C,\cO_C)$ associated to the flag variety parametrizing $2$-dimensional $D_p$-equivariant quotients of the $4$-dimensional $C$-vector space $M(H_1)\otimes_{W(\bar{\F}_p)} C$. Drinfeld proved that $\pi'_{\GM}|_{\mathcal{M}^{(0)}}$ is an isomorphism onto its image. Moreover, one can naturally identify $\Fl'_{\GM}$ with $\mathbb{P}^1$ (cf. Theorem \ref{LTDrdual} below) and the image of $\pi'_{\GM}$ is exactly $\Omega$, the Drinfeld upper half plane. 

As in  \ref{LTsetup}, we denote by $D(\mathcal{G}')$ the dual of $M(\mathcal{G}')$. Then there is a natural inclusion $(\Lie(\mathcal{G}')\otimes C)^\vee\subseteq D(\mathcal{G}')$ which defines the Hodge filtration on $D(\mathcal{G}')$. Again we have the Kodaira-Spencer isomorphism, which can be easily checked on $\Omega$.
\end{para}

\begin{para}[Compare \ref{LTet}]
Note that $\mathcal{G}'$ defines a $\Z_p$-local system of rank $4$ equipped with an action of $D_p$ on the \'etale site of $\mathcal{M}_{\Dr}$, whose dual will be denoted by  $V_\Dr$.
For $n\geq0$, we get an $(\cO_{D_p}/p^n)^\times=\cO_{D_p}^\times/(1+p^n\cO_{D_p})$-\'etale covering $\mathcal{M}_{\Dr,n}$ of $\mathcal{M}_{\Dr}$ by considering $\cO_{D_p}$-equivariant isomorphisms between $V_\Dr/p^n V_\Dr$ and $\cO_{D_p}/p^n$. Scholze-Weinstein \cite[Theorem 6.5.4]{SW13} shows that there exists a perfectoid space $\mathcal{M}_{\Dr,\infty}$ over $\Spa(C,\cO_C)$ such that
\[\mathcal{M}_{\Dr,\infty} \sim\varprojlim_n \mathcal{M}_{\Dr,n}.\]
It admits a natural continuous right action of  $\cO_{D_p}^\times$, which can be extended naturally to action of $D_p^\times$. On the other hand, we can fix an isomorphism $\End(H_1)\otimes_{\Z}\Q\cong M_2(\Q_p)$. Then  the multiplicative group $\GL_2(\Q_p)$ is the group of self-quasi-isogenies of $H_1$ and acts naturally on $\mathcal{M}_{\Dr},\,\mathcal{M}_{\Dr,n}, \,\mathcal{M}_{\Dr,\infty}$ and $\Fl'_{\GM}$ (on the right) through its action on $H_1$. All of these actions are  continuous. We denote by
\[\pi_{\Dr,\GM}: \mathcal{M}_{\Dr,\infty}\to \Fl'_{\GM}\]
the composite of the projection map $ \mathcal{M}_{\Dr,\infty}\to  \mathcal{M}_{\Dr}$ and $\pi'_{\GM}$. It is $\GL_2(\Q_p)\times D_p^\times$-equivariant with respect to the trivial action of $D_p^\times$ on $\Fl'_{\GM}$.

Let $\Lie(\mathcal{G}')\otimes\cO_{\mathcal{M}_{\Dr,\infty}}$ be the pull-back of  $\Lie(\mathcal{G}')\otimes C$ to $\mathcal{M}_{\Dr,\infty}$. Since the Tate module of $\mathcal{G}'$ gets trivialized on $\mathcal{M}_{\Dr,\infty}$ with $\cO_{D_p}$, the  Hodge-Tate sequence induces a $D_p$-equivariant surjection
\[ \cO_{D_p}\otimes_{\Z_p} \cO_{\mathcal{M}_{\Dr,\infty}} \to (\Lie(\mathcal{G}')\otimes\cO_{\mathcal{M}_{\Dr,\infty}})^{\vee}(-1),\]
whose kernel is isomorphic to $\Lie(\mathcal{G}')\otimes\cO_{\mathcal{M}_{\Dr,\infty}}$ if we choose an isomorphism between $\mathcal{G}'$ and its dual, i.e. a principal polarization.
Let $\Fl'$ be the adic space over $\Spa(C,\cO_C)$ associated to the flag variety parametrizing $2$-dimensional $D_p$-equivariant quotients of the $4$-dimensional $C$-vector space $ \cO_{D_p}\otimes_{\Z_p} C$. Then  the Hodge-Tate sequence defines a $D_p^\times$-equivariant map, called the \textit{Hodge-Tate period map}
\[\pi_{\Dr,\HT}: \mathcal{M}_{\Dr,\infty}\to \Fl'.\]

We have the following duality result due to Scholze-Weinstein \cite[Proposition 7.2.2, Theorem 7.2.3]{SW13}.
\end{para}

\begin{thm} \label{LTDrdual}
There are natural $\GL_2(\Q_p)\times D_p^\times$-equivariant isomorphisms 
\[\mathcal{M}_{\LT,\infty}\cong \mathcal{M}_{\Dr,\infty},\] 
and $\Fl\cong\Fl'_{\GM}$,  $\Fl'\cong\Fl_{\GM}$ so that $\pi_{\LT,\HT}$ (resp. $\pi_{\Dr,\HT}$) is getting identified with
$\pi_{\Dr,\GM}$ (resp. $\pi_{\LT,\GM}$) and $\mathcal{M}_{\LT,\infty}^{(0)}$ is identified with $\mathcal{M}^{(0)}_{\Dr,\infty}$.
\end{thm}

This shows that $\pi_{\LT,\HT}(\mathcal{M}_{\LT,\infty})=\Omega$. Under the duality isomorphism, there is a natural isomorphism $\omega_{\LT,\infty}(-1)\cong\omega_{\Dr,\infty}^{-1}$ because both are the pull-backs of $\omega_{\Fl}$ along $\pi_{\LT,\HT}$ and $\pi_{\Dr,\GM}$. We will freely use these isomorphisms from now on.

\begin{para}[Compare \ref{AA+}] \label{DrBB+}
Let $Y$ be an open affinoid subset of $\mathcal{M}_{\Dr}$ and $\tilde{Y}=\Spa(B,B^+)$ be its preimage in $\mathcal{M}_{\Dr,\infty}$. Assume that
\begin{itemize}
\item $\pi_{\Dr,\HT}(\tilde{Y})\neq \Fl'$, hence $\pi_{\Dr,\HT}(\tilde{Y})$ is contained in an affinoid open subset of $\Fl'$ as it is compact.
\end{itemize}
We note that such $Y$ forms a basis of open affinoid subsets of $\mathcal{M}_{\Dr}$. Indeed, since $\GL_2(\Q_p)$ acts transitively on $\pi_{\LT,\GM}^{-1}(y)$ for any $\Spa(C,\cO_C)$-point $y$ of $\Fl_{\GM}$ (cf. Proposition 23.28 of \cite{GH94} and the discussion below it), it is enough to consider those $Y$'s such that there exists a $\Spa(C,\cO_C)$-point of $\mathcal{M}_{\Dr}$ not contained in any translations of $Y$ under $\GL_2(\Q_p)$.

$Y$ is a one-dimensional smooth affinoid adic space over $\Spa(C,\cO_C)$ and $\tilde{Y}$ is a $H=\cO^\times_{D_p}$-Galois pro-\'etale perfectoid covering of $Y$. Again by \cite[Theorem 3.1.2]{Pan20}, the $\cO^\times_{D_p}$-locally analytic vectors $B^{D^\times_p-\la}\subseteq B$ satisfy a first-order differential equation $\theta_{\tilde{Y}}=0$ for some $\theta_{\tilde{Y}}\in B\otimes_{\Q_p}\Lie(H)$. It turns out that as in the case of Lubin-Tate tower, $\theta_{\tilde{Y}}$ essentially comes from the pull-back of the horizontal nilpotent subalgebra along the Hodge-Tate period map. More precisely, for a $\Spa(C,\cO_C)$-point $y$ of $\Fl'$,  it follows from the  construction that $y$ corresponds to a $2$-dimensional $D_p$-stable $C$-subspace of $D_p\otimes_{\Q_p}C$, equivalently, corresponds to a Borel subalgebra $\mathfrak{b}_y\subseteq C\otimes_{\Q_p} \Lie(D_p^\times)$. Denote the nilpotent subalgebra of $\mathfrak{b}_y$ by $\kn_y$. Let
\begin{eqnarray*}
\mathfrak{g}_{D_p}^0&:=&\cO_{\Fl'}\otimes_{\Q_p}\Lie(D_p^\times),\\
\mathfrak{b}^0_{D_p}&:=&\{f\in \mathfrak{g}_{D_p}^0\,| \, f_y\in \mathfrak{b}_y,\mbox{ for all }\Spa(C,\cO_C)\mbox{-point }y\in\Fl'\},\\
\mathfrak{n}^0_{D_p}&:=&\{f\in \mathfrak{g}_{D_p}^0\,| \, f_y\in \mathfrak{n}_y,\mbox{ for all }\Spa(C,\cO_C)\mbox{-point }y\in\Fl'\}.
\end{eqnarray*}
Then we have a natural action of $\mathfrak{g}_{D_p}^0$ on $B^{D^\times_p-\la}$ just as at  the end of \ref{AA+}. 
\end{para}

\begin{thm} \label{Drde}
Let $\tilde{Y}=\Spa(B,B^+)$ be as above. Then $\theta_{\tilde{Y}}$ is given by a generator of $\mathfrak{n}_{D_p}^0$ up to $B^\times$ i.e.  $B^{D^\times_p-\la}$ is annihilated by $\mathfrak{n}_{D_p}^0$. 
\end{thm}

\begin{proof}
Same proof as Theorem \ref{LTde} by observing that the Kodaira-Spencer isomorphism also holds for the Drinfeld towers.
\end{proof}

\begin{para}[Compare \ref{khLT}] \label{khDr}
Theorem \ref{Drde} implies that there is a natural action of $\mathfrak{b}^0_{D_p}/\mathfrak{n}^0_{D_p}$ on $B^{D_p^\times-\la}$. Fix an isomorphism $\iota:C\otimes_{\Q_p} \Lie(D^\times_p)\cong \mathfrak{gl}_2(C)$. Then we can identify $\Fl$ with $\Fl'$, hence get an isomorphism $\mathfrak{b}^0_{D_p}/\mathfrak{n}^0_{D_p}\cong \cO_{\Fl'}\otimes_C \kh$. We denote the induced action of $\kh\to \cO_{\Fl'}\otimes_C \kh\cong \mathfrak{b}^0_{D_p}/\mathfrak{n}^0_{D_p}$ on $B^{D_p^\times-\la}$ by $\theta_{\Dr,\kh}$. It is easy to see that  $\theta_{\Dr,\kh}$ does not depend on $\iota$.  Similarly as in \ref{omegakla}, the same construction defines a natural action of $\kh$ on the $D_p^\times$-locally analytic vectors of $\omega_{\Dr,\infty}$, which is also denoted by $\theta_{\Dr,\kh}$.

We can also describe elements in $B^{D_p^\times-\la}$ in exactly the same way as $B^{\la}$ in Theorem \ref{expGL2}. Using the chosen polarization, we can define global sections  $\mathrm{t}'$ of $\cO_{\mathcal{M}_{\Dr,\infty}}$ and $\mathrm{t}'_n\in H^0(\mathcal{M}_{\Dr,n},\cO_{\mathcal{M}_{\Dr,n}})$ which factors through $\pi_0(\mathcal{M}_{\Dr,n})$ so that $||\mathrm{t}'-\mathrm{t}'_n||\leq p^{-n}$, $n\geq 0$. In fact, we can simply choose $\mathrm{t}'=\mathrm{t}$ because $\cO_{D_p}^\times$ acts on $\mathrm{t}$ via the reduced norm map, cf. proof of Corollary \ref{BlasubsetBDpla}.

Let $W_p$ be a $2$-dimensional irreducible representation of $D_p$ over $C$ in $D_p\otimes_{\Q_p} C$, which is unique up to isomorphisms. Then $\Hom_{C[D_p]}(W_p,\Lie(\mathcal{G}')\otimes C)$
defines a line bundle on $\mathcal{M}_{\Dr}$. The inverse of its pull-back to $\mathcal{M}_{\Dr,*}$ will be denoted by $\omega_{\Dr,*}$ for $*=0,1,\cdots$ and $\infty$. Note that $\omega_{\Dr,\infty}$ is isomorphic to the pull-back of the tautological ample line bundle $\omega_{\Fl'}$ on $\Fl'$ along $\pi_{\Dr,\HT}$. Choose a basis $f_1,f_2$ of $H^0(\Fl',\omega_{\Fl'})$ so that $\pi_{\Dr,\HT}^*f_1$ is invertible on $\tilde{Y}$. This is possible as $\pi_{\Dr,\HT}(\tilde{Y})\neq \Fl'$. Then $y=f_2/f_1$ is an invertible function on $\tilde{Y}$. Let $H_n=1+p^n\cO_{D_p}$. As in \ref{khLT},  for $n\geq 0$, we can find 
\begin{itemize}
\item an integer $r(n)>r(n-1)>0$;
\item $y_n\in B^{H_{r(n)}}$ such that  $\|y-y_n\|_{H_{r(n)}}=\|y-y_n\|\leq p^{-n}$ in $B$;
\item $f_{1,n}\in \omega_{\Dr,\infty}(\tilde{Y})^{H_{r(n)}}$ invertible such that  $\|1-f_{1}/f_{1,n}\|_{H_{r(n)}}=\|1-f_{1}/f_{1,n}\|\leq p^{-n}$ in $B$. This implies that $\log(\frac{f_1}{f_{1,n}})$ makes sense.
\item $||\mathrm{t}-\mathrm{t}'_n||_{H_{r(n)}}=||\mathrm{t}-\mathrm{t}'_n||\leq p^{-n}$.
\end{itemize}
\end{para}

\begin{thm} \label{expDp}
For any $n\geq0$, given a sequence of sections 
\[d_{i,j,k}\in B^{H_{r(n)}},i,j,k=0,1,\cdots\] 
such that the norms of $d_{i,j,k}p^{(n-1)(i+j+k)}$ are uniformly bounded, then
\[f=\sum_{i,j,k\geq 0} c_{i,j,k}(y-y_n)^i \left(\log(\frac{f_1}{f_{1,n}})\right)^j\left(\log(\frac{\mathrm{t}}{\mathrm{t}'_{n}})\right)^k\]
converges in $B^{H_{r(n)-\an}}$ and any $H_{n}$-analytic vector in $B$ arises in this way. 
\end{thm}

\begin{proof}
Same proof as \cite[Theorem 4.3.9]{Pan20}.
\end{proof}

Note that by continuity, $Y$ is $K$-stable for some open subgroup $K\subseteq \GL_2(\Q_p)$. As before, we denote by $B^{\la}$ the subspace of $K$-locally analytic vectors. 

\begin{cor} \label{GL2Dpan}
Let $\Spa(B,B^+)$ be as in \ref{DrBB+}. Then $B^{\la}=B^{D_p^\times-\la}$, i.e. the subspace of $\GL_2(\Q_p)$-locally analytic vectors is the same as the subspace of $D_p^\times$-locally analytic vectors on $\mathcal{M}_{\LT,\infty}$. 
\end{cor}

\begin{rem}
Our proof relies on explicit calculations. It will be very interesting and conceptually satisfying to have a more intrinsic proof.
\end{rem}

\begin{rem}
I was informed by Gabriel Dospinescu that this result was also obtained by himself and Juan Esteban Rodriguez Camargo.
\end{rem}

\begin{proof}
The same proof of Corollary \ref{BlasubsetBDpla} shows that as  a direct consequence of Theorem \ref{expDp}, we have $B^{D_p^\times-\la}\subseteq B^{\la}$. On the other hand,  $B^{\la}\subseteq B^{D_p^\times-\la}$ by Corollary \ref{BlasubsetBDpla}. Hence we have an equality here. Strictly speaking, in Corollary \ref{BlasubsetBDpla}, we assume that $\tilde{Y}$ is the preimage of some open affinoid subset of $\mathcal{M}_{\LT,n}$.  This is not a problem here because it follows from the discussion in \ref{OLTla} that we can find a finite cover of $Y$ by open affinoid subsets of $\mathcal{M}_{\Dr}$, whose preimages in $\mathcal{M}_{\Dr,\infty}$ can be applied with Corollary \ref{BlasubsetBDpla}.
\end{proof}

\begin{para}
We can rewrite these results on the sheaf level as in \ref{OLTla}. Recall that $\cO_{\LT}=\pi^{(0)}_{\LT,\HT}{}_* \cO_{\mathcal{M}_{\LT,\infty}^{(0)}}$,   and $\cO_{\LT}^{\la}$ (resp. $\cO_{\LT}^{\sm}$) denotes the subsheaf of $\GL_2(\Z_p)$-locally analytic sections (resp. $\GL_2(\Z_p)$-smooth sections). By Corollary \ref{GL2Dpan}, $\cO_{\LT}^{\la}$ is also the subsheaf of $\cO_{D_p}^\times$-locally analytic sections.  Let
\[\cO^{D_p^\times-\sm}_{\LT}\subseteq\cO^{\la}_{\LT}\]
denote the subsheaf of $\cO_{D_p}^\times$-smooth sections. Concretely, let $\pi^{(0)}_{\Dr,n}:\mathcal{M}^{(0)}_{\Dr,n}\to \Omega$ denote the Gross-Hopkins map on $\mathcal{M}^{(0)}_{\Dr,n}$. Then
\[\cO^{D_p^\times-\sm}_{\LT}=\varinjlim_n \pi^{(0)}_{\Dr,n}{}_* \cO_{\mathcal{M}^{(0)}_{\Dr,n}}.\]
Let $\omega_{\LT}:=\pi^{(0)}_{\LT,\HT}{}_* \omega_{\LT,\infty}|_{\mathcal{M}_{\LT,\infty}^{(0)}}$ and $\omega_{\Dr}:=\pi^{(0)}_{\LT,\HT}{}_* \omega_{\Dr,\infty}|_{\mathcal{M}_{\Dr,\infty}^{(0)}}$. We can define  subsheaves $\omega^{\la}_{\LT},\omega^{\sm}_{\LT},\omega^{D_p^\times-\sm}_{\LT}\subseteq \omega_{\LT}$ and $\omega^{\la}_{\Dr},\omega^{\sm}_{\LT},\omega^{D_p^\times-\sm}_{\Dr}\subseteq \omega_\Dr$ similarly. Note that the isomorphism $\omega_{\LT,\infty}(-1)\cong\omega_{\Dr,\infty}^{-1}$ induces isomorphisms between these sheaves.

It is easy to check that the action $\theta_{\Dr,\kh}$ introduced in \ref{khDr} acts on  $\cO_{\LT}^{\la}$ and $\omega^{\la}_{\LT}$, $\omega^{\la}_{\Dr}$.

\end{para}

\begin{cor} \label{LTDrkh}
$\theta_{\LT,\kh}=\theta_{\Dr,\kh}$ on $\cO_{\LT}^{\la}$.
\end{cor}

\begin{proof}
The most ``natural'' way to prove this should be using the moduli interpretations of $\mathcal{M}_{\LT,\infty}$ and $\mathcal{M}_{\Dr,\infty}$. Here we give a proof by explicit calculations. Note that the centers of $\GL_2(\Q_p)$ and $D_p^\times$ act in the same way on $B^{\la}$. It suffices to show that $\theta_{\LT,\kh}(h)=\theta_{\Dr,\kh}(h)$, where $h=\begin{pmatrix} 1 & 0\\ 0 & -1\end{pmatrix}$. Let $Y$ be an open affinoid subset of $\Omega$ such that $\tilde{Y}:=(\pi^{(0)}_{\LT,\HT})^{-1}(Y)$ is the preimage of some open affinoid subset $X$ of $\mathcal{M}^{(0)}_{\LT,m}$ for some $m\geq0$. Write $\tilde{Y}=\Spa(B,B^+)$. Fix $f\in B^{\la}$. Then by Theorem \ref{expGL2}, we can write 
\[f=\sum_{i,j,k\geq 0} c_{i,j,k}(x-x_n)^i \left(\log(\frac{e_1}{e_{1,n}})\right)^j\left(\log(\frac{\mathrm{t}}{\mathrm{t}_{n}})\right)^k.\]
for some $c_{i,j,k}\in B^{G_{r(n)}},i,j,k=0,1,\cdots$. See \ref{khLT} for the constructions $x_n,e_{1,n},\mathrm{t}_n$. Since both $\theta_1:=\theta_{\LT,\kh}(h)$ and $\theta_2:=\theta_{\Dr,\kh}(h)$ are first-order differential operators, it is enough to show that they  both agree on $B^{G_{r(n)}},x,e_1/e_{1,n},\mathrm{t}$.  This is clear for $\mathrm{t}$ because the action of $\GL_2(\Q_p)$ (resp.  $D_p^\times$) on it factors through the determinant (resp. reduced norm) map. For other elements, we have the following lemma. (Note that $B^{G_{r(n)}}\subseteq \cO^{\sm}_{\LT}(Y)$, $x\in \cO^{D_p^\times-\sm}_{\LT}(Y)$, $e_{1,n}\in \omega^{\sm}_{\LT}(Y)$, $e_1\in  \omega^{D_p^\times-\sm}_{\LT}(Y)$ and $1/e_{1,n}\in \omega^{\sm}_{\Dr}(Y)$, $1/e_1\in  \omega^{D_p^\times-\sm}_{\Dr}(Y)$.)
\end{proof}

\begin{lem}
\begin{enumerate}
\item $\theta_{\LT,\kh}$ and $\theta_{\Dr,\kh}$ are zero on $\cO^{\sm}_{\LT}$ and $\cO^{D_p^\times-\sm}_{\LT}$. 
\item $\theta_1=0$ on $\omega^{\sm}_{\LT}$ and $\theta_1=-1$ on $ \omega^{D_p^\times-\sm}_{\LT}$.
\item $\theta_2=0$ on $ \omega^{D_p^\times-\sm}_{\Dr}$ and $\theta_2=-1$ on $\omega^{\sm}_{\Dr}$.
\end{enumerate}
\end{lem}

\begin{proof}
We will only prove results for $\theta_{\LT,\kh}$ because the same argument will also work for  $\theta_{\Dr,\kh}$. It follows from the construction that $\theta_{\LT,\kh}$ acts trivially on $\GL_2(\Z_p)$-smooth vectors. Hence $\theta_{\LT,\kh}=0$ on $\cO^{\sm}_{\LT}$, $\omega^{\sm}_{\LT}$. 

For $\cO^{D_p^\times-\sm}_{\LT}$, we first observe that $\theta_{\LT,\kh}=0$ on $\cO_{\Fl}$. This can be seen by writing $\Fl$ as the quotient of $N\setminus\GL_2$ by $H$, where $H=\{\begin{pmatrix} * & 0 \\ 0 & * \end{pmatrix}\}$ and $N=\{\begin{pmatrix} 0 & * \\ 0 & 0 \end{pmatrix}\}$, and $\theta_{\LT,\kh}$ essentially comes from the infinitesimal action of $H$. One can also prove this by a direct calculation, cf. \cite[5.1.1]{Pan20}.  Now this implies that  $\theta_{\LT,\kh}$ is zero on $\cO^{D_p^\times-\sm}_{\LT}$ because  sections of $\cO^{D_p^\times-\sm}_{\LT}$ are analytic functions on some finite \'etale coverings of $\Omega$ and $\theta_{\LT,\kh}$ is a first-order differential operator. For $ \omega^{D_p^\times-\sm}_{\LT}$, the same argument reduces to showing that $\theta_1=-1$ on $\omega_{\Fl}$. Again this can be proved by a simple calculation.
\end{proof}

\begin{para}[Compare Remark \ref{LTlalg}]
Let $\chi:\kh\to C$ be a character. As before, we denote by $\cO^{\la,\chi}_{\LT}\subseteq\cO^{\la}_{K^p}$ the subsheaf of $\chi$-isotypic with respect to $\theta_{\LT,\kh}=\theta_{\Dr,\kh}$.
Suppose $\chi=(n_1,n_2)\in\Z^2$ and $k=n_2-n_1\geq 0$. By Theorem \ref{expDp}, for $i\in\{0,\cdots,k\}$ and $s\in\omega^{-k,D_p^\times-\sm}_{\Dr}\cong\omega^{k,D_p^\times-\sm}_{\LT}(-k)$, the product $\mathrm{t}^{n_1}f_1^if_2^{k-i} s$ defines an element in $\cO^{\la,(n_1,n_2)}_{\LT}$. We denote by
\[\cO^{D_p^\times-\lalg,(n_1,n_2)}_{\LT}\subseteq \cO^{\la,(n_1,n_2)}_{\LT}\]
the subsheaf spanned by sections of this form. It is easy to see that it exactly consists of $\cO_{D_p}^\times$-locally algebraic vectors. Equivalently, it is the image of the natural inclusion
\[H^0(\Fl',\omega^{k}_{\Fl'})\otimes_C \omega^{k,D_p^\times-\sm}_{\LT}\cdot \mathrm{t}^{n_1}\cong H^0(\Fl',\omega^{k}_{\Fl'})(k)\otimes_C \omega^{-k,D_p^\times-\sm}_{\Dr}\cdot \mathrm{t}^{n_1}\subseteq  \cO^{\la,(n_1,n_2)}_{\LT}.\]
\end{para}

\begin{para}
We can repeat the construction in  Theorem \ref{I2} and prove the following analogue of Theorem \ref{LTI2}. Note that the $\Omega$ in the Lubin-Tate picture corresponds to $\Fl'$, or its \'etale coverings $\mathcal{M}_{\LT,n} $ in this Drinfeld picture. In particular, $\cO_{\Omega}$ is replaced by  $\cO^{\sm}_{\LT}$ and $\Omega^1_{\Omega}$ is replaced by $\Omega^1_{\LT}$ here. Also note that since there is no canonical choice of bases of $H^0(\Fl',\omega_{\Fl'})$, we need to make a choice as  in \ref{khDr}.
\end{para}

\begin{thm} \label{DrI2}
Suppose $\chi=(n_1,n_2)\in\Z^2$ and $k=n_2-n_1\geq 0$.  Then there exists a natural continuous operator
\[\bar{d}_{\Dr}^{k+1}:\cO^{\la,\chi}_{\LT}\to \cO^{\la,\chi}_{\LT}\otimes_{\cO^{\sm}_{\LT}}(\Omega^1_{\LT})^{\otimes k+1}\]
which formally can be regarded as $(d_{\Fl'})^{k+1}$ satisfying following properties:
\begin{enumerate}
\item $\bar{d}_{\LT}^{k+1}$ is $\cO^{D_p^\times-\sm}_{\LT}$-linear;
\item Suppose that $Y$ is an open affinoid subset of $\Omega$ and  $\{f_1,f_2\}$ is a basis of $H^0(\Fl',\omega_{\Fl'})$ such that $f_1$ is a generator of $\omega_{\Dr}$ on $Y$.  Then there exists a non-zero constant $c\in\Q$ such that $\bar{d}_{\Dr}^{k+1}(s)=c(u'^+)^{k+1}(s)\otimes (dy')^{k+1}$ for any $s\in \cO^{\la,\chi}_{\LT}$, where $y'=f_2/f_1$ and $u'^+$ denotes the unique nilpotent element of $\Lie(D_p^\times)\otimes_{\Q_p}C$ sending $f_2$ to $f_1$.
\end{enumerate}
Here $d_{\Fl'}$ denotes the derivation on $\Fl'$.
Moreover,
\begin{enumerate}
\item $\bar{d}_{\Dr}^{k+1}$ is surjective and commutes with $\cO_{D_p}^\times$.
\item $\ker(\bar{d}_{\Dr}^{k+1})=\cO^{D_p^\times-\lalg,\chi}_{\LT}=H^0(\Fl',\omega^{k}_{\Fl'})\otimes_C \omega^{k,D_p^\times-\sm}_{\LT}\cdot \mathrm{t}^{n_1}$.
\end{enumerate}
\end{thm}

We remark that the set of $Y$ considered in the Theorem is a basis of open affinoid subsets of $\Omega$, cf. \ref{DrBB+}.

\begin{rem}\label{dDR}
Although this will not be used, we can also redo the construction of Theorem \ref{LTI1} and prove that 
there exists a unique natural continuous  $\cO^{D_p^\times-\sm}_{\LT}$-linear operator
\[d_{\Dr}^{k+1}:\cO^{\la,\chi}_{\LT}\to \cO^{\la,\chi}_{\LT}\otimes_{\cO^{D_p^\times-\sm}_{\LT}}(\Omega^1_{\Dr})^{\otimes k+1}\]
extending $(d_{\Dr})^{k+1}$ on $\omega^{-k,D_p^\times-\sm}_{\Dr}$,
where
\[\Omega^1_{\Dr}:=\varinjlim_n \pi^{(0)}_{\Dr,n}{}_*\Omega^1_{\mathcal{M}^{(0)}_{\Dr,n}}.\]
and $d_{\Dr}$ denotes the derivation $\cO^{D_p^\times-\sm}_{\LT}\to \Omega^1_{\Dr}$.
\end{rem}

\begin{para}
Now we would like to compare $\bar{d}_{\Dr}^{k+1}$ and $d_{\LT}^{k+1}$ introduced in Theorem \ref{LTI1}. Keep the notation in Theorem \ref{DrI2}. Note that $u'^+=\frac{d}{dy'}$ on $\Fl'$. Hence this also holds on its \'etale coverings $\mathcal{M}_{\LT,n}$. From this, it is easy to see that $\bar{d}_{\Dr}^{k+1}$ satisfies the properties in Theorem \ref{LTI1} up to a unit. Hence  by the uniqueness of  $d_{\LT}^{k+1}$, we have 
\end{para}

\begin{thm}
$\bar{d}_{\Dr}^{k+1}=c'd_{\LT}^{k+1}$ for some $c'\in \Q^\times$.
\end{thm}

\begin{cor} \label{dLTkcok}
$d_{\LT}^{k+1}$ is surjective and 
\[\ker(d_{\LT}^{k+1})=\cO^{D_p^\times-\lalg,\chi}_{\LT}=H^0(\Fl',\omega^{k}_{\Fl'})(k)\otimes_C \omega^{-k,D_p^\times-\sm}_{\Dr}\cdot \mathrm{t}^{n_1}.\]
\end{cor}

\subsection{\texorpdfstring{$I_k$}{Lg} on \texorpdfstring{$\Omega$}{Lg}: uniformization of supersingular locus}

\begin{para}
To apply the results obtained in the previous two sections to modular curves, we note that the supersingular locus of modular curves has a uniformization by Lubin-Tate towers and $\pi_{\HT}^{-1}(\Omega)$ is simply a finite disjoint union of $\mathcal{M}^{(0)}_{\LT}$. This was  used by Deligne in his proof \cite{De73} of the local Langlands correspondence for $\GL_2(\Q_l)$. More precisely, the uniformization theorem relates the supercuspidal part of the $\ell$-adic cohomology of modular curves with the $\ell$-adic cohomology of Lubin-Tate towers in terms of Jacquet-Langlands correspondence. We will see a similar picture here with the $\ell$-adic cohomology replaced by the completed cohomology.

We now recall the  uniformization theorem. Let $E_0$ be a supersingular elliptic curve  over $\bar{\F}_p$ whose $p$-divisible group is isomorphic to $H_0$ considered in \ref{LTsetup}. Fix such an isomorphism. Let $D:=\End(E_0)\otimes\Q$. It is well-known that $D$ is a quaternion algebra which is only ramified at $p$ and $\infty$. The chosen isomorphism gives a natural identification $D\otimes_{\Q}\Q_p\cong D_p$. Fix an $\A^p_f$-linear isomorphism
\[H^1_{\mathrm{\acute{e}t}}(E_0,\A^p_f)\cong (A^p_f)^{\oplus 2}.\]
Note that $D$ acts naturally on the right of it. We switch this to a left action by applying the main involution $g\mapsto \iota(g)$.
This induces  an   isomorphism $D\otimes_{\Q}\A_f^{p}\cong M_2(\A_f^p)$ of $\A_f^p$-algebras and $K^p\subseteq\GL_2(\A_f^p)$ is considered as an open subgroup of $(D\otimes_{\Q}\A_f^{p})^\times$ through this isomorphism.

Recall that $\pi_{\HT}:\mathcal{X}_{K^p}\to \Fl$ denotes the Hodge-Tate period morphism for modular curves, and  we have the Lubin-Tate space at infinite level $\mathcal{M}_{\LT,\infty}$ in \ref{LTet} equipped with a right action of $\GL_2(\Q_p)\times D_p^\times$ and a $\GL_2(\Q_p)\times D_p^\times$-equivariant period map $\pi_{\LT,\HT}:\mathcal{M}_{\LT,\infty}\to\Fl$.
\end{para}

\begin{thm} \label{ssunif}
There is a $\GL_2(\Q_p)$-equivariant isomorphism 
\[\pi_{\HT}^{-1}(\Omega)\cong D^\times \setminus [\mathcal{M}_{\LT,\infty}\times (D\otimes_{\Q}\A_f^{p})^\times/K^p]\]
which is functorial in $K^p$ and identifies $\pi_{\HT}: \pi_{\HT}^{-1}(\Omega)\to \Fl$ with $\pi_{\LT,\HT}\circ \pi_1$ on $\mathcal{M}_{\LT,\infty}\times (D\otimes_{\Q}\A_f^{p})^\times/K^p$ and identifies $\omega_{K^p}|_{\Omega}$ with $\pi_{1}{}_*\omega_{\LT}$,  where $\pi_1:\mathcal{M}_{\LT,\infty}\times (D\otimes_{\Q}\A_f^{p})^\times/K^p \to\mathcal{M}_{\LT,\infty}$ denotes the projection map onto the first factor. Here $D^\times$ acts diagonally on the left of  $\mathcal{M}_{\LT,\infty}\times (D\otimes_{\Q}\A_f^{p})^\times/K^p$ via  the embedding 
\[D^\times\to D_p^\times\]
sending $g$ to $\iota(g)$ (hence $D^\times$ acts on the left of $\mathcal{M}_{\LT,\infty}$) and the natural embedding $D^\times\subseteq (D\otimes_{\Q}\A_f^{p})^\times$.
\end{thm}

\begin{proof}
It was shown by \cite[Lemma III.3.6]{Sch15} that $\pi_{\HT}^{-1}(\Omega)$ is exactly the preimage of the supersingular locus of $\mathcal{X}_{K^p\GL_2(\Z_p)}$. The supersingular locus in $\mathcal{X}_{K^p\Gamma(p^n)}$ has a uniformization by $\mathcal{M}_{\LT,n}$ by Theorem III of \cite{RZ96}. See \cite[6.13]{RZ96} for the construction. The isomorphism on the infinite level follows from taking the inverse limit. The key point is that all supersingular elliptic curves are isogenous to $E_0$ and the isomorphism is guaranteed by the Serre-Tate theory. 
The identifications of the Hodge-Tate period maps and $\omega_{K^p}|_{\Omega}$ with $\pi_{1}{}_*\omega_{\LT}$   are clear in view of the constructions.

Since the normalization here is slightly different from the reference, we briefly explain where the main involution of $D^\times$ in the embedding  $D^\times\to D_p^\times$ comes from. Note that $D^\times$ acts naturally on the left of $H^1_{\mathrm{\acute{e}t}}(E_0,\A^p_f)$ via $g\mapsto(g^{-1})^*$, $g\in D^\times$. Then under the fixed isomorphism $H^1_{\mathrm{\acute{e}t}}(E_0,\A^p_f)\cong (A^p_f)^{\oplus 2}$ and $D\otimes_{\Q}\A_f^{p}\cong M_2(\A_f^p)$, this induces an embedding $i:D^\times\to  (D\otimes_{\Q}\A_f^{p})^\times$ which is nothing but $g\mapsto \iota(g)^{-1}$. On the other hand, the construction of \cite[6.13]{RZ96}  gives an isomorphism
\[\pi_{\HT}^{-1}(\Omega)\cong D^\times \setminus [\mathcal{M}_{\LT,\infty}\times (D\otimes_{\Q}\A_f^{p})^\times/K^p]\]
via the embedding  $D^\times\to D_p^\times$ sending $g$ to $g^{-1}$ and $i:D^\times\to  (D\otimes_{\Q}\A_f^{p})^\times$. We get the isomorphism in the theorem by applying $g\mapsto \iota(g)^{-1}$ to $D^\times$.

\end{proof}

\begin{para}
The quotient $D^\times \setminus [\mathcal{M}_{\LT,\infty}\times (D\otimes_{\Q}\A_f^{p})^\times/K^p]$ can be made more explicit. Consider
\[D^\times\setminus (D\otimes_{\Q}\A_f)^\times/K^p.\]
which admits a natural action of $D_p^\times$  on the right. We note that
\begin{itemize}
\item the action of $\cO_{D_p^\times}$ on it is free.
\end{itemize}
This follows easily from our assumption in \ref{brr} that $K^p$ is  contained in the level-$N$-subgroup $\{g\in\GL_2(\hat{\Z}^p)=\prod_{l\neq p}\GL_2(\Z_l)\,\vert\, g\equiv1\mod N\}$ for some $N\geq 3$ prime to $p$. Indeed, for any $g\in (D\otimes_{\Q_p} \A_f^p)^\times$, the intersection $gD^\times g^{-1}\cap K^p\cO_{D_p^\times}$ is finite, hence has to be trivial by the assumption. By the finiteness of the class group, we  can write
\[(D\otimes_{\Q}\A_f)^\times=\bigsqcup_{i\in I}D^\times \gamma_i K^p\cO_{D_p}^\times\]
for some elements $\gamma_i\in (D\otimes_{\Q}\A_f)^\times$ and finite set $I$. Write $\gamma_i=\gamma_i^p \gamma_{i,p}$ where $\gamma_i^p\in (D\otimes_{\Q}\A^p_f)^\times$ and $\gamma_{i,p}\in D_p^\times$. Consider the map 
\[\bigsqcup_{i\in I} \mathcal{M}^{(0)}_{\LT,\infty}\to D^\times \setminus [\mathcal{M}_{\LT,\infty}\times (D\otimes_{\Q}\A_f^{p})^\times/K^p]\]
sending $z_i\in\mathcal{M}^{(0)}_{\LT,\infty},{i\in I}$ to the coset containing $(z_i\iota(\gamma_{i,p}),\gamma_{i}^p)\in \mathcal{M}_{\LT,\infty}\times (D\otimes_{\Q}\A_f^{p})^\times/K^p$.
\end{para}

\begin{lem} \label{expuf}
This is an isomorphism, i.e. $u:\pi_{\HT}^{-1}(\Omega)\cong \bigsqcup_{i\in I} \mathcal{M}^{(0)}_{\LT,\infty}$ and this isomorphism is $\GL_2(\Q_p)^0$-equivariant.
\end{lem}

\begin{proof}
The trivial map  sending $x\in \mathcal{M}_{\LT,\infty}$ to $(x,1)\in  \mathcal{M}_{\LT,\infty}\times D_p^\times$ induces an isomorphism
\[ \mathcal{M}_{\LT,\infty}\cong ( \mathcal{M}_{\LT,\infty}\times D_p^\times)/D_p^\times\]
where $D_p$ acts diagonally on the right hand side via $g\mapsto \iota(g)^{-1}$ on $ \mathcal{M}_{\LT,\infty}$ and the right multiplication on $D_p^\times$. The (right) action of $D_p^\times$ on the left hand side is identified with $g\cdot(z,a)=(z,\iota(g)a)$, $g\in D_p^\times,~(z,a)\in\mathcal{M}_{\LT,\infty}\times D_p^\times$ on the right hand side. On the other hand,
\[ ( \mathcal{M}_{\LT,\infty}\times D_p^\times)/D_p^\times= ( \mathcal{M}^{(0)}_{\LT,\infty}\times D_p^\times)/\cO_{D_p}^\times.\]
Hence 
\[
D^\times \setminus [\mathcal{M}_{\LT,\infty}\times (D\otimes_{\Q}\A_f^{p})^\times/K^p]\cong \left[\mathcal{M}_{\LT,\infty}^{(0)} \times [D^\times\setminus (D\otimes_{\Q}\A_f)^\times/K^p]\right]/\cO_{D_p}^\times
\]
where $\cO_{D_p}^\times$ acts via $g\mapsto \iota(g)^{-1}$ on $ \mathcal{M}^{(0)}_{\LT,\infty}$ and the right multiplication on $D^\times\setminus (D\otimes_{\Q}\A_f)^\times/K^p$.
Our claim follows immediately.
\end{proof} 

\begin{para}
The uniformization isomorphism $u:\pi_{\HT}^{-1}(\Omega)\cong \bigsqcup_{i\in I} \mathcal{M}^{(0)}_{\LT,\infty}$ induces a $\GL_2(\Q_p)^0 $-equivariant isomorphism $\cO_{K^p}|_{\Omega}\cong \bigoplus_{i\in I}\cO_{\LT}$ of $\cO_{\Omega}$-modules. In particular, it is easy to see that after taking the locally analytic vectors, the horizontal Cartan action $\theta_{\kh}$ is identified with $\theta_{\LT,\kh}$. Let $\chi=(n_1,n_2)\in\Z^2$ be a character of $\kh$ with $k=n_2-n_1\geq 0$. Then the differential operator $d^{k+1}|_{\Omega}$ in Theorem \ref{I1} gets identified with $d_{\LT}^{k+1}$ introduced in Theorem \ref{LTI1}. In particular, the surjectivity of $d_{\LT}^{k+1}$ (Corollary \ref{dLTkcok}) shows that 
\end{para}
 
\begin{cor} \label{dk+1Os}
 $d^{k+1}|_{\Omega}$ is surjective.
\end{cor}

In particular, $H^0(\Fl,\coker d^{k+1})=H^0(\mathbb{P}^1(\Q_p),\coker d^{k+1})$. Hence by Corollary \ref{cokerdk+1P1}, \ref{cokerd'k+1}, we have the following result. Recall that $i$ denotes the embedding $\mathbb{P}^1(\Q_p)\subseteq\Fl$.

\begin{cor} \label{cokerIkinfty}
For $\chi=(-k,0)$, 
\[\coker d^{k+1}=i_*(\coker d^{k+1})|_{\mathbb{P}^1(\Q_p)}\cong \omega^k_{\Fl}\otimes_C i_*\mathcal{H}^1_{\ord}(K^p,k)(k)\cdot \mathrm{t}^{-k},\]
\[\coker d'^{k+1}=i_*(\coker d^{k+1})|_{\mathbb{P}^1(\Q_p)}\cong \omega^{-k-2}_{\Fl}\otimes_C i_*\mathcal{H}^1_{\ord}(K^p,k)(k)\otimes\det{}^{k+1}\cdot \mathrm{t}^{-k},\]
\[H^0(\Fl,\coker d^{k+1})=\Ind_B^{\GL_2(\Q_p)} H^1_{\rig}(\Ig(K^p),\Sym^k) (k)\cdot {e'_2}^k\mathrm{t}^{-k},\]
\[H^0(\Fl,\coker d'^{k+1})=\Ind_B^{\GL_2(\Q_p)} H^1_{\rig}(\Ig(K^p),\Sym^k) (k)\cdot {e'_1}^{k+1}{e'_2}^{-1}\mathrm{t}^{-k}.\]
\end{cor}

Recall that $I_k= d'^{k+1}\circ \bar{d}^{k+1}$ and $\bar{d}^{k+1}$ is surjective by Proposition \ref{dbarsurj}. 

\begin{cor} \label{cokik}
$\coker I_k$ is supported on $\mathbb{P}^1(\Q_p)\subseteq\Fl$. Moreover,
\[I^1_k:H^1(\Fl,\cO^{\la,(n_1,n_2)}_{K^p})\to H^1(\Fl,\cO^{\la,(n_2+1,n_1-1)}_{K^p}(k+1))\] 
is surjective.
\end{cor}

\begin{proof}
Let $i:\mathbb{P}^1(\Q_p)\to \Fl$ denotes the closed embedding (of topological spaces). Then
 $H^1(\Fl,\coker I_k)=H^1(\mathbb{P}^1(\Q_p),i^*\coker I_k)=0$ because $\mathbb{P}^1(\Q_p)$ is a profinite set (``zero-dimensional''). This implies the surjectivity of $I^1_k$ as $\Fl$ is one-dimensional, hence taking $H^1$ is right-exact. 
\end{proof}

\begin{rem}
It is interesting to observe that the complex $\mathcal{I}_k:\cO^{\la,(-k,0)}_{K^p}\xrightarrow{I_k} \cO^{\la,(1,-k-1)}_{K^p}(k+1)$ satisfies the following  perversities: $\mathrm{supp} H^1(\mathcal{I}_k)=\mathbb{P}^1(\Q_p)$ is ``zero-dimensional'' and $\mathrm{supp} H^0(\mathcal{I}_k)\subseteq \Fl$ is at most ``one-dimensional''. 

\end{rem}

\begin{para}
We can also use Corollary \ref{dLTkcok} to give a description of $\ker d^{k+1}$. Let $X_{K^p}:=D^\times\setminus (D\otimes_{\Q}\A_f)^\times/K^p$, which is  isomorphic to $\bigsqcup_{i\in I}\cO_{D_p}^\times$ as an $\cO_{D_p}^\times$-space. Recall that in the proof of Lemma \ref{expuf}, we showed that 
\[
\pi_{\HT}^{-1} (\Omega)\cong [\mathcal{M}_{\LT,\infty}^{(0)} \times X_{K^p}]/\cO_{D_p}^\times
,\]
where $\cO_{D_p}^\times$ acts on $\mathcal{M}_{\LT,\infty}^{(0)}$ via $g\mapsto \iota(g)^{-1}$.
From this, it is easy to see that $\cO_{K^p}|_{\Omega}\cong\mathrm{Map}_{\cO_{D_p}^\times}(X_{K^p},\cO_{\LT})$, the set of $\cO_{D_p}^\times$-equivariant maps from $X_{K^p}$ to $\cO_{\LT}$. (In fact, this holds for general  $\cO_{D_p}^\times$-equivariant sheaves on $\mathcal{M}_{\LT,\infty}^{(0)}$ but we don't need this generality here.) Similarly, we have 
\[\ker d^{k+1}|_{\Omega}\cong\mathrm{Map}_{\cO_{D_p}^\times}(X_{K^p},\ker d_{\LT}^{k+1}).\] 
To further simplify it, by Corollary \ref{dLTkcok},  we have
\[\ker d^{k+1}_{\LT}=H^0(\Fl',\omega^{k}_{\Fl'})(k)\otimes_C \omega^{-k,D_p^\times-\sm}_{\Dr}\cdot \mathrm{t}^{n_1}.\]
Note that $H^0(\Fl',\omega^{k}_{\Fl'})\cdot \mathrm{t}^{n_1}$ is an irreducible algebraic representation of $D_p^\times$ over $C$ with highest weight $(k,0)+(n_1,n_1)=(n_1+k,n_1)=(n_2,n_1)$.  We denote the \textit{dual} of  this irreducible representation by $W_{(-n_1,-n_2)}$ which has highest weight $(-n_1,-n_2)$.
\end{para}

\begin{defn} \label{AFD}
Let $\chi=(n_1,n_2)\in\Z^2$ with $n_2-n_1\geq 0$.
\begin{enumerate}
\item  $\mathcal{A}_{D,-\chi}=\mathcal{A}_{D,(-n_1,-n_2)}$ denotes the set of maps 
\[f: D^\times\setminus (D\otimes_{\Q}\A_f)^\times\to W_{(-n_1,-n_2)}\] such that
$f(dh^ph_p)=h_p^{-1}\cdot f(d)$, for any $d\in D^\times\setminus (D\otimes_{\Q}\A_f)^\times$ and any $h^p$ (resp. $h_p$) in  some open compact subgroup $U^p\subseteq (D\otimes_{\Q}\A_f^p)^\times$ (resp. $U_p\subseteq D_p^\times$). This is usually called the space of quaternionic forms on $(D\otimes_\Q \A)^\times$. It admits a right translation action of $ (D\otimes_{\Q}\A_f^p)^\times$ and an action of $D_p^\times$ via $(h_p\cdot f)(d)=h_p\cdot f(dh_p)$, $h_p\in D_p^\times,d\in D^\times\setminus (D\otimes_{\Q}\A_f)^\times$. Both actions are smooth.
\item $\mathcal{A}^{1}_{D,-\chi}\subseteq\mathcal{A}_{D,-\chi}$  denotes the subset of maps which factor through the reduced norm map. Clearly this is non-zero only when $n_1=n_2$.
\item $\mathcal{A}^c_{D,-\chi}:=\mathcal{A}_{D,-\chi}/\mathcal{A}^1_{D,-\chi}$. Then the Hecke action (cf. \ref{Hd} below) induces a natural splitting $\mathcal{A}_{D,-\chi}\cong\mathcal{A}^c_{D,-\chi}\oplus \mathcal{A}^1_{D,-\chi}$.
\end{enumerate}
\end{defn}
 
\begin{prop} \label{kerdo}
There is a Hecke and $\GL_2(\Q_p)^0$-equivariant isomorphism
\[\ker d^{k+1}|_{\Omega}\cong (\mathcal{A}_{D,(-n_1,-n_2)}^{K^p}\otimes_C  \omega^{-k,D_p^\times-\sm}_{\Dr})^{\cO_{D_p}^\times}(n_2)\cdot {\varepsilon'_p}^{n_1},\] 
where $\cO_{D_p}^\times$ acts diagonally  via the action defined in \ref{AFD} on $\mathcal{A}_{D,(-n_1,-n_2)}^{K^p}$ and the action $g\mapsto \iota(g)^{-1}$ on $ \omega^{-k,D_p^\times-\sm}_{\Dr}$, and $\varepsilon'_p$ denotes the character sending $x\in\GL_2(\Q_p)$ to $|\det{x}|\det{x}$, i.e. the $p$-adic cyclotomic character.
\end{prop} 

\begin{proof}
We denote by $\mathrm{Map}^{\sm}(X_{K^p},\ker d_{\LT}^{k+1})$ the set of maps $X_{K^p} \to \ker d_{\LT}^{k+1}$ invariant with respect to some open subgroup of $\cO_{D_p}^\times$. Then
\[\mathrm{Map}_{\cO_{D_p}^\times}(X_{K^p},\ker d_{\LT}^{k+1})=\mathrm{Map}^{\sm}(X_{K^p},\ker d_{\LT}^{k+1})^{\cO_{D_p}^\times}.\]
Since the action of $\cO_{D_p}^\times$ on $\omega^{-k,D_p^\times-\sm}_{\Dr}$ is smooth, hence locally finite, by definition, we have 
\begin{eqnarray*}
\mathrm{Map}^{\sm}(X_{K^p},\ker d_{\LT}^{k+1})&\cong&\mathrm{Map}^{\sm}(X_{K^p},H^0(\Fl',\omega^{k}_{\Fl'})\cdot \mathrm{t}^{n_1})\otimes_C \omega^{-k,D_p^\times-\sm}_{\Dr}(k)\\
&\cong&\mathcal{A}_{D,(-n_1,-n_2)}^{K^p}\otimes_C  \omega^{-k,D_p^\times-\sm}_{\Dr}(n_1+k)\cdot {\varepsilon'_p}^{n_1}\\
&=&\mathcal{A}_{D,(-n_1,-n_2)}^{K^p}\otimes_C  \omega^{-k,D_p^\times-\sm}_{\Dr}(n_2)\cdot {\varepsilon'_p}^{n_1},
\end{eqnarray*} 
where the second isomorphism holds as  $\cO_{D_p}^\times$ acts $\ker d_{\LT}^{k+1}$ via $g\mapsto \iota(g)^{-1}$ and $\GL_2(\A_f)$ acts on $\mathrm{t}$ via $\varepsilon\circ\det$ with $\varepsilon:\A_f^\times/\Q^\times_{>0}\to\Z_p^\times$ denoting the $p$-adic cyclotomic character, cf. \ref{HEt}, and $G_{\Q_p}$ acts on $\mathrm{t}$ also via the $p$-adic cyclotomic character.
Our claim follows immediately.
\end{proof}

\begin{rem} \label{GL20}
We can extend this isomorphism to be $\GL_2(\Q_p)$-equivariant in the following way. It follows from the construction of $\omega_{\Dr}$ that the action of $\cO_{D_p}^\times\times\GL_2(\Q_p)^0$ can be extended naturally to an action of $D_p^\times\times\GL_2(\Q_p)$ on $\bigoplus_{i\in\Z}  \omega^{-k,D_p^\times-\sm}_{\Dr}$. Then 
\begin{eqnarray*}
\ker d^{k+1}|_{\Omega}&\cong &(\mathcal{A}_{D,(-n_1,-n_2)}^{K^p}\otimes_C  \omega^{-k,D_p^\times-\sm}_{\Dr})^{\cO_{D_p}^\times}(n_2)\cdot {\varepsilon'_p}^{n_1}\\
&\cong & (\mathcal{A}_{D,(-n_1,-n_2)}^{K^p}\otimes_C  \bigoplus_{i\in\Z} \omega^{-k,D_p^\times-\sm}_{\Dr})^{D_p^\times}(n_2)\cdot {\varepsilon'_p}^{n_1}
\end{eqnarray*}
which is $\GL_2(\Q_p)$-equivariant.
\end{rem}

\begin{para}
We denote by 
$j:\Omega\to \Fl$ 
the open embedding. Let $\mathcal{F}$ be a sheaf on $\Omega$. As usual $j_{!}\mathcal{F}$ denotes the extension by zero of $\mathcal{F}$. We now can determine the kernel of $d^{k+1}: \cO^{\la,(n_1,n_2)}_{K^p}\to  \cO^{\la,(n_1,n_2)}_{K^p}\otimes_{\cO^{\sm}_{K^p}} (\Omega^1_{K^p}(\mathcal{C})^{\sm})^{\otimes k+1}$ by Proposition \ref{kerdc} and Proposition \ref{kerdo}.
\end{para}

\begin{prop} \label{kerdk+1s}
\[\ker d^{k+1}\cong j_{!}(\mathcal{A}_{D,(-n_1,-n_2)}^{K^p}\otimes_C  \omega^{-k,D_p^\times-\sm}_{\Dr})^{\cO_{D_p}^\times}(n_2)\cdot  {\varepsilon'_p}^{n_1},~~~~~~k\geq 1,\]
\[\ker d^1\cong j_{!} (\mathcal{A}_{D,(-n_1,-n_1)}^{c,K^p}\otimes_C \cO^{D_p^\times-\sm}_{\LT})^{\cO_{D_p}^\times}(n_1)\cdot  {\varepsilon'_p}^{n_1} \oplus \cO_{\Fl} \otimes_C M_0(K^p)(n_1)\cdot {\varepsilon'_p}^{n_1},\]
where $\mathcal{A}_{D,(0,0)}^{c}$ was introduced in Definition \ref{AFD}.
Recall that 
\[\omega^{-k,D_p^\times-\sm}_{\Dr}:=\varinjlim_{n}(\pi^{(0)}_{\Dr,n})_* (\pi^{(0)}_{\Dr,n})^* \omega_{\Fl}^k,\]
where $\pi^{(0)}_{\Dr,n}:\mathcal{M}_{\Dr,n}^{(0)}\to \Omega$ denotes the projection map. We also remind that the action of ${\cO_{D_p}^\times}$ on $\omega^{k+2,D_p^\times-\sm}_{\Dr}$ is via $g\mapsto \iota(g)^{-1}$.
\end{prop}

\begin{proof}
The only case that requires some extra explanation is when $k=0$. We may assume $n_1=0$ by twisting by $\mathrm{t}^{-n_1}$. There is a natural map
\[\cO_{\Fl} \otimes_C M_0(K^p)=\cO_{\Fl} \otimes_C H^0(\Fl,\cO^{\sm}_{K^p})\to \cO_{\Fl} \otimes_C \cO^\sm_{K^p}\to \cO^{\la,(0,0)}_{K^p}\]
whose image is inside of $\ker d^1$. On the other hand, by Proposition \ref{kerdc} and Proposition \ref{kerdo}, there is an exact sequence
 \[0\to  j_{!} (\mathcal{A}_{D,(0,0)}^{K^p}\otimes_C  \cO^{D_p^\times-\sm}_{\LT})^{\cO_{D_p}^\times} \to \ker d^{1} \to (\cO_{\Fl})|_{\mathbb{P}^1(\Q_p)} \otimes_C M_0(K^p) \to 0.\]
 Here we use that $ \omega^{0,D_p^\times-\sm}_{\Dr}=\cO^{D_p^\times-\sm}_{\LT}$. Hence we get an isomorphism
 \[\ker d^1\cong \left[\cO_{\Fl} \otimes_C M_0(K^p)\right] \oplus  j_{!} \left[(\mathcal{A}_{D,(0,0)}^{K^p}\otimes_C  \cO^{D_p^\times-\sm}_{\LT})^{\cO_{D_p}^\times}/ \cO_{\Omega} \otimes_C M_0(K^p) \right].\]
We claim that the image of the map $ \cO_{\Omega} \otimes_C M_0(K^p)\to (\mathcal{A}_{D,(0,0)}^{K^p}\otimes_C  \cO^{D_p^\times-\sm}_{\LT})^{\cO_{D_p}^\times}$ is exactly $(\mathcal{A}_{D,(0,0)}^{1,K^p}\otimes_C  \cO^{D_p^\times-\sm}_{\LT})^{\cO_{D_p}^\times}$. Recall that $\mathcal{A}_{D,(0,0)}^{1,K^p}$ denotes the subset of smooth functions on $D^\times\setminus(D\otimes_{\Q}\A_f)^\times/K^p$ that factor through the reduced norm map. This implies the Proposition because by definition $\mathcal{A}_{D,(0,0)}^{c,K^p}=\mathcal{A}_{D,(0,0)}^{K^p}/\mathcal{A}_{D,(0,0)}^{1,K^p}$.

To see the claim, we note that there is a natural identification 
\[\mathcal{A}_{D,(0,0)}^{1,K^p}=M_0(K^p)\]
because both sides denote the space of smooth functions on $\Q_{>0}^\times\setminus \A_f^\times/\det(K^p)$. Let $\psi:\cO_{D_p}^\times\to C^\times$ be a smooth character. Then the $\psi$-isotypic part $\cO_{\LT}[\psi]$  of $\cO_{\LT}$ can be identified with $\cO_{\Omega}$. Indeed, suppose that $1+p^n\cO_{D_p}\subseteq \ker \psi$. There is a natural surjective map $c_n:\mathcal{M}_{\Dr,n}^{(0)}\to \pi_0(\mathcal{M}_{\Dr,n}^{(0)})\to (\Z/p^n)^\times$, cf. second paragraph of \ref{khDr}. Then the restriction of $\mathcal{M}_{\Dr,n}^{(0)}$  to $c_n^{-1}(1)$ induces an isomorphism 
\[\cO_{\LT}[\psi]\cong(\pi_{\Dr,n*}^{(0)}\cO_{c_n^{-1}(1)})^{\cO_{D_p}^{\times,1}}=\cO_\Omega\] 
where $\cO_{D_p}^{\times,1}\subseteq \cO_{D_p}^{\times}$ denotes the kernel of the reduced norm map. From this we easily deduce our claim on each $\psi$-isotypic component hence deduce the whole claim.
\end{proof}

There is a natural isomorphism $\omega^{-2}_{\Fl}\otimes\det\cong \Omega_{\Fl}$, cf. \ref{KSFl}. We obtain the following result by twisting  Proposition \ref{kerdk+1s} by $(\Omega_{\Fl})^{\otimes k+1}$. 

\begin{cor} \label{kerd'k+1s}
\hspace{2em}
\[\ker d'^{k+1}\cong j_{!}(\mathcal{A}_{D,(-n_1,-n_2)}^{K^p}\otimes_C  \omega^{k+2,D_p^\times-\sm}_{\Dr})^{\cO_{D_p}^\times}(n_2)\otimes \det{}^{k+1}\cdot {\varepsilon'_p}^{n_1},~~~~~~k\geq 1,\]
\[\ker d'^1\cong j_{!} (\mathcal{A}_{D,(-n_1,-n_1)}^{c,K^p}\otimes_C  \omega^{2,D_p^\times-\sm}_{\Dr})^{\cO_{D_p}^\times}(n_2)\otimes\det\cdot{\varepsilon'_p}^{n_1} \oplus \Omega_{\Fl} \otimes_C M_0(K^p)(n_2)\cdot {\varepsilon'_p}^{n_1}.\]
\end{cor}

%\begin{thm} \label{kerIk}
%Suppose $\chi=(-k,0)$. Consider $I_k:\cO^{\la,(-k,0)}_{K^p}\to \cO^{\la,(1,1-k)}_{K^p}(k+1)$. 
%There are natural short exact sequences
%\[0\to \Sym^k V\otimes_{\Q_p}\omega^{-k,\sm}_{K^p}(k)\cdot\mathrm{t}^{-k}\to \ker I_k \to j_{!}(\mathcal{A}_{D,(k,0)}^{K^p}\otimes_C  \omega^{k+2,D_p^\times-\sm}_{\Dr})^{\cO_{D_p}^\times}\to 0\]
%when $k\geq 1 $, and 
%\[0\to \cO^{\sm}_{K^p}\to \ker I_0 \to j_{!} (\mathcal{A}_{D,(0,0)}^{c,K^p}\otimes_C \omega^{2,D_p^\times-\sm}_{\Dr})^{\cO_{D_p}^\times} \oplus \Omega_{\Fl} \otimes_C M_0(K^p)\to 0.  \]

%\end{thm}

\begin{rem}
It is natural to ask whether we can rewrite these results about $\ker d^{k+1}$, $\ker \bar{d}^{k+1}$ and the cokernels in a uniform way. Let me restrict to the case $k=0$. Roughly speaking the complex $d^{1}: \cO^{\la,(0,0)}_{K^p}\to  \cO^{\la,(0,0)}_{K^p}\otimes_{\cO^{\sm}_{K^p}} \Omega^1_{K^p}(\mathcal{C})^{\sm}$ computes the \textit{de Rham cohomology} of the fiber of the Hodge-Tate period map $\pi_{\HT}$. In particular, since the fiber of each point of $\Omega$ is a profinite set by the uniformization theorem, the cohomology only appears in degree zero. On $\mathbb{P}^1(\Q_p)$, the fiber is closely related to the Igusa curve, and it is not surprising at all that we are  essentially computing the de Rham cohomology of the Igusa curve here, cf. Proposition \ref{kerdc}, \ref{cokdk+1infty}. For $\bar{d}$, similarly it is like computing the de Rham cohomology of the fiber of the projection map $\mathcal{X}_{K^p}\to \mathcal{X}_{K^p\GL_2(\Z_p)}$. Again since the fiber is a profinite set, there is no $H^1$, i.e. $\coker\bar{d}=0$.
\end{rem}

\subsection{Spectral decompositions} \label{Sd}
\begin{para}
We are ready to decompose $\ker I_k^1$ with respect to the Hecke action. Recall that $I^1_k=H^1(I_k)$. In fact, we will present two decompositions. The first comes from $I_k=d'^{k+1}\circ \bar{d}^{k+1}$, while the second comes from $ I_k=\bar{d}'^{k+1}\circ d^{k+1}$. Both give different prospectives of $\ker I_k^1$ and we will see some interesting results when  comparing two decompositions.

First we explore the factorization $I_k=d'^{k+1}\circ \bar{d}^{k+1}$. We have the following easy lemma.
\end{para}

\begin{lem} \label{kerIk1fil}
There are natural short exact sequences
\[0\to \ker H^1(\bar{d}^{k+1})\to \ker I_k^1 \to \ker H^1(d'^{k+1})\to 0,\]
\[0\to  H^1(\Fl,\ker d'^{k+1})\to \ker H^1(d'^{k+1})\to H^0(\coker d'^{k+1})\to 0\]
and isomorphisms $ \ker H^1(d'^{k+1})\cong \bH^1(d'^{k+1})$, where $ \bH^1(d'^{k+1})$ denotes the first hypercohomology of the complex $\cO^{\la,(n_1,n_2)}_{K^p}\otimes_{\cO_{\Fl}}(\Omega^1_{\Fl})^{\otimes k+1}\xrightarrow{d'^{k+1}}\cO^{\la,(n_2+1,n_1-1)}_{K^p}(k+1)$.
\end{lem}

\begin{proof}
Note that $\bar{d}^{k+1}$ is surjective by Proposition \ref{dbarsurj}, hence $H^1(\bar{d}^{k+1})$ is surjective, which clearly implies the first claim.
For the second claim, by \cite[Corollary 5.1.3]{Pan20}, we have $H^0(\Fl,\cO^{\la,(n_2+1,n_1-1)}_{K^p})=0$. In particular, the image of $d'^{k+1}$ also has no global sections. With these vanishing results, our claim follows from Lemma \ref{ninvlem} and Remark \ref{hyclem} with $X=d'^{k+1}$.
\end{proof}

\begin{para}
Some notations.
Note that $G_{\Q_p}$ acts on $\mathrm{t}$ via the $p$-adic cyclotomic character and $\GL_2(\A_f)$ acts on $\mathrm{t}$ via $\varepsilon\circ \det$.  This means that $\mathrm{t}(-1)=\varepsilon\circ \det$ as a representation of $\GL_2(\A_f)\times G_{\Q_p}$. For simplicity, we use $\varepsilon$ to denote twisting by $\mathrm{t}(-1)$.
\end{para}

\begin{thm} \label{dec1}
For $\chi=(-k,0)$, 
there is a natural increasing filtration $\Fil_\bullet$ on $\ker I^1_k$ with $\Fil_n=0$, $n\geq 0$ and $\Fil_n= I^1_k$, $n\geq 3$, and 
\[\gr_n \ker I^1_k=\left\{
	\begin{array}{lll}
		\ker H^1(\bar{d}^{k+1})\cong \Sym^k V\otimes_{\Q_p} H^1(\omega^{-k,\sm}_{K^p})\cdot{\varepsilon}^{-k}, &n=1&\\
		H^1(\ker d'^{k+1})\cong (H^1(j_{!}  \omega^{k+2,D_p^\times-\sm}_{\Dr})\otimes_C \mathcal{A}_{D,(k,0)}^{K^p})^{\cO_{D_p}^\times}\otimes\det^{k+1}\cdot{\varepsilon}_p^{-k}, &n=2& \\
				H^0(\coker d'^{k+1})\cong \Ind_B^{\GL_2(\Q_p)}  H^1_{\rig}(\Ig(K^p),\Sym^k) \cdot{\varepsilon}^{-k} {e'_1}^{k+1}{e'_2}^{-1}, &n=3&
	\end{array},\right.\]
	when $k\geq 1$ and 
\[\gr_n \ker I^1_0=\left\{
	\begin{array}{lll}
		\ker H^1(\bar{d}^1)\cong H^1(\cO^{\sm}_{K^p}) , &n=1&\\
		H^1(\ker d'^{1})\cong (H^1(j_{!} \omega^{2,D_p^\times-\sm}_{\Dr})\otimes_C \mathcal{A}_{D,(0,0)}^{c,K^p})^{\cO_{D_p}^\times}\otimes\det \oplus M_0(K^p) , &n=2&\\
				H^0(\coker d'^{1})\cong \Ind_B^{\GL_2(\Q_p)}  H^1_{\rig}(\Ig(K^p))\cdot {e'_1}{e'_2}^{-1} , &n=3&
	\end{array}.\right.\]
All of the cohomology groups are computed on $\Fl$ and we drop $\Fl$ for simplicity.
All of the isomorphisms are Hecke and $\GL_2(\Q_p)$-equivariant. See \ref{Hd} below for the precise meaning of Hecke actions.
\end{thm}

\begin{proof}
The claim for $H^0(\coker d'^{k+1})$ follows from Corollary \ref{cokerIkinfty}. For $H^1(\ker d'^{k+1})$, when $k\geq 1$, it follows from Corollary \ref{kerd'k+1s} by noting that 
\[ H^1(j_{!}(\mathcal{A}_{D,(k,0)}^{K^p}\otimes_C  \omega^{k+2,D_p^\times-\sm}_{\Dr})^{\cO_{D_p}^\times})\cong\left(H^1( j_{!}  \omega^{k+2,D_p^\times-\sm}_{\Dr})\otimes_C \mathcal{A}_{D,(k,0)}^{K^p}\right)^{\cO_{D_p}^\times}\]
because the actions of $\cO_{D_p}^\times$ on $\mathcal{A}_{D,(k,0)}^{K^p}$ and $\omega^{k+2,D_p^\times-\sm}_{\Dr}$ are smooth, hence semi-simple as $\cO_{D_p}^\times$ is compact.  The same argument plus that $H^1(\Fl,\Omega_{\Fl})=C$ proves the case $k=0$. For $\ker H^1(\bar{d}^{k+1})$, by our description of $\ker \bar{d}^{k+1}$ in Proposition \ref{dbarsurj}, it suffices to show that
\[\ker H^1(\bar{d}^{k+1})=H^1(\ker\bar{d}^{k+1}).\]
Since $\bar{d}^{k+1}:  \cO^{\la,(n_1,n_2)}_{K^p}\to \cO^{\la,(n_1,n_2)}_{K^p}\otimes_{\cO_{\Fl}}(\Omega^1_{\Fl})^{\otimes k+1}$ is surjective, we only need to show
\[H^0( \cO^{\la,(n_1,n_2)}_{K^p}\otimes_{\cO_{\Fl}}(\Omega^1_{\Fl})^{\otimes k+1})=0.\]
Consider the exact sequence
\[0\to H^0(\ker d'^{k+1})\to H^0( \cO^{\la,(n_1,n_2)}_{K^p}\otimes_{\cO_{\Fl}}(\Omega^1_{\Fl})^{\otimes k+1}) \xrightarrow{H^1(d'^{k+1})} H^0(\cO^{\la,(n_2+1,n_1-1)}_{K^p}).\]
The last term is zero by \cite[Corollary 5.1.3]{Pan20}. On the other hand, we have 
\[H^0(j_! \omega_{\Dr}^{k+2,\sm})=0\]
because each connected component of $\mathcal{M}_{\Dr,n}$ maps surjectively onto $\Omega$ and 
\[H^0(\Fl,\Omega_\Fl)=0.\]
Hence $H^0(\ker d'^{k+1})=0$, which implies the vanishing of $H^0( \cO^{\la,(n_1,n_2)}_{K^p}\otimes_{\cO_{\Fl}}(\Omega^1_{\Fl})^{\otimes k+1})$.
\end{proof}

\begin{para} \label{Hd}
To further state our result, we need a digression on Hecke algebra.  Let $S$ be a finite set of rational primes containing $p$ such that $K_l=K^p\cap\GL_2(\Q_l)$ is maximal for $l\notin S$. As usual, we denote by $T_l$ the double coset action of 
\[[K_l\begin{pmatrix} l & 0 \\ 0 & 1 \end{pmatrix} K_l]\]
and by $S_l$  the action of $\begin{pmatrix} l & 0 \\ 0 & l \end{pmatrix}$.  Let 
\[\T^S:=\Z_p[T_l,S_l,l\notin S]\]
be the $\Z_p$-algebra generated by symbols $T_l,S_l,l\notin S$. Then it  acts naturally on $\cO^{\la}_{K^p}$ and commutes with $\theta_\kh$, $I_k$, $d^{k+1}$, $\bar{d}^{k+1}$ as these operators are defined purely at the place $p$. For $k\in\Z$, recall that
\[M_{k}(K^p)=H^0(\Fl,\omega^{k,\sm}_{K^p})=\varinjlim_{K_p} H^0(\mathcal{X}_{K^pK_p},\omega^k)\]
denotes the space of classical holomorphic modular forms of weight $k$.  It is well-known that the action of $\T^S$ on $M_k(K^p)$ is semi-simple, hence there is a natural eigenspace decomposition
\[M_k(K^p)=\bigoplus_{\lambda:\T^S\to C} M_k(K^p)[\lambda].\]
We will denote by 
$\sigma_k^{K^p}$ the spectrum of $\T^S$ on  $M_k(K^p)$.
%It is also the set of systems of  Hecke eigenvalues $\T^S\to C$ associated to automorphic representations $\pi$ of $\GL_2(\A)$ such that $\pi_\infty$ has the same infinitesimal character as the irreducible algebraic representation of $\GL_2$ with highest weight $(-1,1-k)$ and $(\pi^\infty)^{K^p}\neq 0$. (Recall that we fixed an isomorphism between $\bar\Q_p\cong \bC$ throughout this paper.)  
In general,
since $\GL_2(\A_f)$ acts on $\mathrm{t}$ via the cyclotomic character, we also have a spectral decomposition
\[M_k(K^p)\cdot \mathrm{t}^{n}=\bigoplus_{\lambda\in\sigma_{k,n}^{K^p}} M_k(K^p)\cdot \mathrm{t}^{n}[\lambda]\]
where $\sigma_{k,n}^{K^p}$ denotes the spectrum of $\T^S$ on $M_k(K^p)\cdot \mathrm{t}^{n}$.
On the other hand, for $k\geq 2$, there is also a spectral decomposition
for $H^1(\Fl,\omega^{2-k,\sm}_{K^p})=\varinjlim_{K_p} H^1(\mathcal{X}_{K^pK_p},\omega^{2-k})$: 
\[H^1(\Fl,\omega^{2-k,\sm}_{K^p})=\bigoplus_{\lambda\in \sigma_{k,k-1}^{K^p}} H^1(\Fl,\omega^{2-k,\sm}_{K^p})[\lambda],\]
where $\lambda$ runs over systems of  Hecke eigenvalues associated to cuspidal automorphic representations $\pi$ of $\GL_2(\A)$ such that $\pi_\infty$ has the same infinitesimal character as the irreducible algebraic representation of $\GL_2$ with highest weight $(k-2,0)$ and $(\pi^\infty)^{K^p}\neq 0$. Hence the spectrum is a subset of $\sigma_{k,k-1}^{K^p}$. This can be seen from the Hodge-Tate decomposition of $H^1_{\mathrm{\acute{e}t}}(\mathcal{X}_{K^pK_p},\Sym^{k-2} V)$.

Recall that $\mathcal{A}^{K^p}_{(-n_1,-n_2)}$  denotes the space of quaternionic forms, cf. Definition \ref{AFD}. The Hecke algebra $\T^S$ acts naturally on it. By the Jacquet-Langlands correspondence, this action is semi-simple, and moreover for $k\geq 1$, there is  a natural decomposition into eigenspaces
\[\mathcal{A}^{K^p}_{(k,0)}=\bigoplus_{\lambda\in \sigma^{K^p}_{k+2,1}} \mathcal{A}^{K^p}_{(k,0)}[\lambda]\]
where $\lambda:\T^S\to C$ runs over cuspidal automorphic representations $\pi$ of $\GL_2(\A)$ such that $\pi_\infty$ has the same infinitesimal character as the irreducible algebraic representation of $\GL_2$ with highest weight $(0,-k)$ and such that $(\pi^\infty)^{K^p}\neq 0$ and $\pi_p$ is special or supercuspidal. In particular, the spectrum is a subset of $\sigma^{K^p}_{k+2,1}$. When $k=0$, for a Hecke-eigenform, either it is contained in  $\mathcal{A}^{1}_{(0,0)}$, i.e. (as a function on $(D\otimes \A_f)^\times$) it factors through the reduced norm map, hence transfers to an eigenform in $M_0(K^p)$, or it transfers to a cuspidal eigenform of weight $2$ as in the higher weight case. Therefore
\[\mathcal{A}^{K^p}_{(0,0)}=\bigoplus_{\lambda\in\sigma_{2,1}^{K^p}\cup \sigma_{0}^{K^p}} \mathcal{A}^{K^p}_{(0,0)}[\lambda].\]
\end{para}

\begin{para} \label{D0a}
Recall that in Definition \ref{D0}, we introduced
\[D_0:=H^0(\Fl,\wedge^2 D^{\sm}_{K^p}).\] 
It follows from \ref{HEt} that $D_0=M_0(K^p)\cdot c$ and $\GL_2(\A_f)$ acts on $c$ via $|\cdot|^{-1}\circ\det$. 

Now we can state the main result of this section.
\end{para}

\begin{thm} \label{sd}
For $\chi=(-k,0)$, 
there are natural generalized eigenspace decompositions
\[\ker I^1_k=\bigoplus_{\lambda \in\sigma_{k+2,1}^{K^p}} \ker I^1_k\widetilde{[\lambda]}\]
when $k\geq 1$, and
\[\ker I^1_0=\bigoplus_{\lambda \in\sigma_{2,1}^{K^p}\cup \sigma_{0}^{K^p}} \ker I^1_0\widetilde{[\lambda]}.\]
Here $\widetilde{[\lambda]}$ denotes the generalized eigenspace of $\lambda$.
\end{thm}

\begin{proof}
By Theorem \ref{dec1}, it's enough to show that there are generalized eigenspace decompositions for each graded pieces $\ker H^1(\bar{d}^{k+1})$, $H^1(\ker d'^{k+1})$ and $H^0(\coker d'^{k+1})$. This is clear for $\ker H^1(\bar{d}^{k+1})$ and $H^1(\ker d'^{k+1})$ in view of discussion in \ref{Hd}.

For $H^0(\coker d'^{k+1})$, we recall that 
$\coker I_k$ is supported on $\mathbb{P}^1(\Q_p)$ and
\[(\coker d'^{k+1})_{\infty}=(\omega^{-k-2}_\Fl)_\infty\otimes_C H^1_{\rig}(\Ig(K^p),\Sym^k)(k) \cdot \mathrm{t}^{-k}.\]
Since $\GL_2(\Q_p)$ acts transitively on $\mathbb{P}^1(\Q_p)$, it is enough to show that $H^1_{\rig}(\Ig(K^p),\Sym^k)\cdot \mathrm{t}^{-k}$ has a generalized eigenspace decomposition with the spectrum in $\sigma_{k+2,1}^{K^p}$ or $\sigma^{K^p}_{2,1}\cup \sigma_0^{K^p}$ if $k=0$. This can be deduced either from Coleman's result \cite[Theorem 2.1]{Cole97} or from people's works on the $\ell$-adic cohomology of the Igusa curves and ``independence of $\ell$'' results. Here we provide a proof without using these works. See also Theorem \ref{ijdR} for another proof.

By Proposition \ref{Colefin} (see also Definition \ref{H1Ig}),  $H^1_{\rig}(\Ig(K^p),\Sym^k)$ is a direct limit of finite-dimensional Hecke-stable vector spaces. Hence there is a generalized eigenspace decomposition
\[ H^1_{\rig}(\Ig(K^p),\Sym^k)\cdot \mathrm{t}^{-k}=\bigoplus_{\lambda:\T^S\to C}  H^1_{\rig}(\Ig(K^p),\Sym^k)\cdot \mathrm{t}^{-k}\widetilde{[\lambda]},\]
which also implies a decomposition 
\[\coker d'^{k+1} =\bigoplus_{\lambda} (\coker d'^{k+1})\widetilde{[\lambda]}.\]
We denote the spectrum by $\sigma$. It remains to show $\sigma\subseteq \sigma_{k+2,1}^{K^p}$ (or $\sigma_{2,1}^{K^p}\cup \sigma_{0}^{K^p}$ if $k=0$). Suppose $\lambda\in\sigma$ but not contained in $\sigma_{k+2,1}^{K^p}$ ($\sigma_{2,1}^{K^p}\cup \sigma_{0}^{K^p}$ if $k=0$). Then it follows from what we have proved that $\Fil_2 \ker I^1_k[\lambda]=0$ and $(\ker I^1_k)\widetilde{[\lambda]} = H^0(\Fl,(\coker d'^{k+1})\widetilde{[\lambda]})$. 

\begin{lem}
$H^0(\Fl,(\coker d'^{k+1})\widetilde{[\lambda]})^{\kn}\neq 0$. Recall that $\kn=\{\begin{pmatrix} 0 & * \\ 0& 0\end{pmatrix}\} \subseteq \mathfrak{gl}_2(\Q_p)$.
\end{lem}

\begin{proof}
Let $0\in\mathbb{P}^1(\Q_p)$ denote the zero of $e_2$. Hence $x$ is a local coordinate around $0$. Since $(\coker d'^{k+1})\widetilde{[\lambda]}$ is supported on $\mathbb{P}^1(\Q_p)$, a profinite set. It suffices to show that the stalk $(\coker d'^{k+1})_0\widetilde{[\lambda]}$ at $0$ has non-zero $\kn$-invariants. This is clear as
\[(\coker d'^{k+1})_0\widetilde{[\lambda]}\cong(\omega^{-k-2}_\Fl)_0\otimes_C H^1_{\rig}(\Ig(K^p),\Sym^k)(k) \cdot \mathrm{t}^{-k}\widetilde{[\lambda]}\neq0,\]
and $(\omega^{-k-2}_\Fl)_0^{\kn}=Ce_1^{-k-2}\neq 0$.
\end{proof}

The lemma shows that $(\ker I^1_k)\widetilde{[\lambda]}^{\kn}\neq 0$. However, by Theorem \ref{kernk} (after twisting by $\mathrm{t}^{-k}$), we see that one of the following happens
\begin{itemize}
\item $M_{k+2}(K^p)\otimes_{M_0(K^p)}D_0^{\otimes -k-1}\cdot\mathrm{t}^{-k}[\lambda]\neq 0$;
\item $H^1(\Fl,\omega^{-k,\sm}_{K^p})\cdot\mathrm{t}^{-k}[\lambda]\neq 0$;
\item $M_{-k}(K^p)[\lambda]\neq 0$.
\end{itemize} 
In the first case, $\lambda\in\sigma_{k+2,1}^{K^p}$ because $\GL_2(\A^p_f)$ acts on $D_0$ via the inverse of the $p$-adic cyclotomic character. The same conclusion is true for the second case. The third case only can happen when $k$=0  and it will imply that $\lambda\in \sigma_0^{K^p}$. Hence we get a contradiction here. Therefore the spectrum of $\coker I_k$ is contained in $\sigma_{k+2,1}^{K^p}$ if $k\geq 1$ and $\sigma_{2,1}^{K^p}\cup \sigma_{0}^{K^p}$ if $k=0$.
\end{proof}

In Section \ref{VIk}, we introduced $I'_0:H^1(\Fl,\cO_{K^p})^{\la,(0,0)}\to H^1(\Fl,\cO_{K^p})^{\la,(1,-1)}(1)=H^1(\Fl,\cO_{K^p}^{\la,(1,-1)})(1)$ induced by $I^1_0$. 
\begin{cor} \label{sd'}
\[\ker I'_0=\bigoplus_{\lambda \in\sigma_{2,1}^{K^p}\cup \sigma_{0}^{K^p}} \ker I'_0\widetilde{[\lambda]}.\]
\end{cor}

\begin{proof}
This follows from Theorem \ref{sd} because $\ker I'_0$ is a quotient of $\ker I^1_0$,  cf. Subsection \ref{VIk}.
\end{proof}

\begin{para} \label{2nddec}
Next we explore another decomposition of $\ker I^1_k$. Note that $\ker I_k^1=\ker H^1(d^{k+1})$ by Lemma \ref{kerdk+1}. Consider the following (de Rham) complex
\[ DR_k:\cO^{\la,(-k,0)}_{K^p}\xrightarrow{d^{k+1}}  \cO^{\la,(-k,0)}_{K^p}\otimes_{\cO^{\sm}_{K^p}} (\Omega^1_{K^p}(\mathcal{C})^{\sm})^{\otimes k+1}.\]
The decreasing filtration $\Fil^1 DR_k: 0\to  \cO^{\la,(-k,0)}_{K^p}\otimes_{\cO^{\sm}_{K^p}} (\Omega^1_{K^p}(\mathcal{C})^{\sm})^{\otimes k+1}$ defines a natural filtration on $\bH^\bullet(DR_k)$. Then 
\[\ker H^1(d^{k+1})=\bH^1(DR_k)/\Fil^1\bH^1(DR_k)\]
with $\Fil^1\bH^1(DR_k)\cong \coker H^0(d^{k+1})$ by Remark \ref{hyclem}. 
\end{para}

\begin{lem} \label{lemH0}
\hspace{2em}
\begin{enumerate}
\item $H^0(d^{k+1})=0$. In particular, 
\[\Fil^1\mathbb{H}^1(DR_k)\cong\coker H^0(d^{k+1})=H^0( \cO^{\la,(-k,0)}_{K^p}\otimes_{\cO^{\sm}_{K^p}} (\Omega^1_{K^p}(\mathcal{C})^{\sm})^{\otimes k+1}).\]
\item There is a  natural isomorphism
\[H^0( \cO^{\la,(-k,0)}_{K^p}\otimes_{\cO^{\sm}_{K^p}} (\Omega^1_{K^p}(\mathcal{C})^{\sm})^{\otimes k+1})\cong \Sym^k V\otimes_{\Q_p} M_{k+2}(K^p)(k)\otimes D_0^{-k-1}\cdot\mathrm{t}^{-k}\]
induced by the inclusion $\cO^{\lalg,(-k,0)}_{K^p}\subseteq \cO^{\la,(-k,0)}_{K^p}$, cf. \ref{lalg0k}, \ref{lalgKp}.
\end{enumerate}
\end{lem}

\begin{proof}
The first claim  is clear as $H^0( \cO^{\la,(n_1,n_2)}_{K^p})=0$ when $k\neq 0$ and $H^0( \cO^{\la,(n_1,n_1)}_{K^p})=M_0(K^p)\cdot \mathrm{t}^{n_1}$ which is annihilated by $d^1$, cf. \cite[Corollary 5.1.3]{Pan20}. The second claim is the last part of Corollary \ref{gso2k2}.
\end{proof}

\begin{defn}
Denote by $i_{H}$  the induced map 
\[\Sym^k V\otimes_{\Q_p} M_{k+2}(K^p)(k)\otimes D_0^{-k-1}\cdot\mathrm{t}^{-k}\cong\Fil^1\bH^1(DR_k)\subseteq \bH^1(DR_k).\]
Recall that $D_0^{-1}(1)\cdot \mathrm{t}^{-1}=M_0(K^p)\otimes_{\Q_p}\det^{-1}$ as a $\GL_2(\A_f)\times G_{\Q_p}$-representation, cf. \ref{HEt}. We may aslo view $i_H$ as 
\[i_H: \Sym^k V^*\otimes_{\Q_p} M_{k+2}(K^p)\otimes D_0^{-1}\to\bH^1(DR_k)\]
where $V^*$ denotes the dual of $V$. We will see in Remark \ref{HodFil} that $i_H$ describes the Hodge filtration on the de Rham cohomology of the modular curves.
\end{defn}

Note that
\[0\to H^1(\ker d^{k+1}) \to \bH^1(DR_k)\to H^0(\coker d^{k+1})\to 0.\]
Combining with Corollary \ref{cokerIkinfty} and Proposition \ref{kerdk+1s}, we obtain the following result.

\begin{prop} \label{bH1DRk}
There is a natural exact sequence 
\[0\to (H^1(j_{!}  \omega^{-k,D_p^\times-\sm}_{\Dr})\otimes \mathcal{A}_{D,(k,0)}^{c,K^p})^{\cO_{D_p}^\times}\cdot{\varepsilon}_p^{-k} \to \bH^1(DR_k)
 \to H^0( i_* \mathcal{H}^1_{\ord}(K^p,k)\otimes\omega^k_{\Fl}) \cdot\varepsilon^{-k}\to 0
\]
with $H^0( i_* \mathcal{H}^1_{\ord}(K^p,k)\otimes\omega^k_{\Fl}) \cdot\varepsilon^{-k}\cong\Ind_B^{\GL_2(\Q_p)} H^1_{\rig}(\Ig(K^p),\Sym^k) \cdot \varepsilon^{-k}{e'_2}^k$.
\end{prop}

Assuming Corollary \ref{ssIg} below, we have the following result.

\begin{cor} \label{BScon}
Suppose $\lambda\in\sigma^{K^p}_{k+2,1}$ corresponds to a cuspidal automorphic representation  $\pi$ of $\GL_2(\A)$.
\begin{enumerate}
\item If $\pi_p$ is supercuspidal,  then $H^1_{\rig}(\Ig(K^p),\Sym^k)\cdot\mathrm{t}^{-k}[\lambda]=0$ by Corollary \ref{ssIg} below, and
\[\ker I^1_k\widetilde{[\lambda]} \cong  (H^1(j_{!}  \omega^{-k,D_p^\times-\sm}_{\Dr})\otimes_C \mathcal{A}_{D,(k,0)}^{c,K^p}[\lambda])^{\cO_{D_p}^\times}\cdot{\varepsilon}_p^{-k}/ i_H(\Sym^k V^*\otimes_{\Q_p} M_{k+2}(K^p)\otimes D_0^{-1}[\lambda]).
\]
In particular, $\ker I^1_k\widetilde{[\lambda]}=\ker I^1_k[\lambda]$ is the eigenspace of $\lambda$. Moreover, by Theorem \ref{dec1}, there is an exact sequence 
\[\begin{multlined}0\to \Sym^k V\otimes_{\Q_p} H^1(\omega^{-k,\sm}_{K^p})\cdot{\varepsilon}^{-k}[\lambda]\to \ker I^1_k[\lambda] \\
\to 
		 (H^1(j_{!}  \omega^{k+2,D_p^\times-\sm}_{\Dr})\otimes \mathcal{A}_{D,(k,0)}^{K^p}[\lambda])^{\cO_{D_p}^\times}\otimes \det{}^{k+1}\cdot{\varepsilon}_p^{-k}\to 0.\end{multlined}\]
\item If $ \mathcal{A}_{D,(k,0)}^{c,K^p}[\lambda]=0$, i.e. $\pi_p$ is a principal series, then 
\[\ker I^1_k\widetilde{[\lambda]}\cong \Ind_B^{\GL_2(\Q_p)} H^1_{\rig}(\Ig(K^p),\Sym^k) \cdot \varepsilon^{-k}{e'_2}^k[\lambda]/ i_H(\Sym^k V^*\otimes_{\Q_p} M_{k+2}(K^p)\otimes D_0^{-1}[\lambda]).\]
In particular, $\ker I^1_k\widetilde{[\lambda]}=\ker I^1_k[\lambda]$ is the eigenspace of $\lambda$ by Corollary \ref{ssIg} below. Moreover,  by Theorem \ref{dec1},
there is an exact sequence 
\[\begin{multlined}0\to \Sym^k V\otimes_{\Q_p} H^1(\omega^{-k,\sm}_{K^p})\cdot{\varepsilon}^{-k}[\lambda]\to \ker I^1_k[\lambda] \\
\to 
		\Ind_B^{\GL_2(\Q_p)} H^1_{\rig}(\Ig(K^p),\Sym^k) \cdot \varepsilon^{-k}{e'_1}^{k+1}{e'_2}^{-1}[\lambda]\to 0.\end{multlined}\]
\end{enumerate}
\end{cor}

\begin{rem} \label{DRstn}
We believe that $\ker I^1_k\widetilde{[\lambda]}=\ker I^1_k[\lambda]$ also holds when $\pi_p$ is special. This can be proved by using Emerton's local-global compatibility and people's work on the $p$-adic local Langlands correspondence for $\GL_2(\Q_p)$. See Remark \ref{stn} below. It would be very interesting to have a more direct proof. Note that in both cases when $\pi_p$ is either supercuspidal or a  principal series, we actually proved that $\bH^1(DR_k)[\lambda] =  \bH^1(DR_k)\widetilde{[\lambda]}$ and
\[ \ker I^1_k[\lambda]=\bH^1(DR_k)[\lambda]/ \Fil^1\bH^1(DR_k)[\lambda]. \]
It is natural to guess that this is also true when $\pi_p$ is special.
\end{rem}

\begin{rem} \label{KSreal}
 Proposition \ref{bH1DRk} says that $\bH^1(DR_k)$ is an extension of 
 \begin{itemize}
 \item $H^0( i_* \mathcal{H}^1_{\ord}(K^p,k)\otimes\omega^k_{\Fl}) \cdot\varepsilon^{-k}$ by 
 \item $(H^1(j_{!}  \omega^{-k,D_p^\times-\sm}_{\Dr})\otimes \mathcal{A}_{D,(k,0)}^{c,K^p})^{\cO_{D_p}^\times}\cdot{\varepsilon}_p^{-k}$.
 \end{itemize} We can rewrite these two cohomology groups
  in the following more uniform way. On $\Omega$, we may view $\mathcal{A}_{D,(k,0)}^{c,K^p}$ as a $C$-local system $\mathcal{L}^{ss}$ on the \'etale site $\Omega_{\et}$ of $\Omega$. On the other hand $\omega^k_{\Fl}|_{\Omega}=j^*\omega^k_\Fl$ naturally  defines a sheaf on $\Omega_{\et}$, which by abuse of notation will also be denoted by $j^*\omega^k_\Fl$. Let $\nu:\Omega_{\et}\to \Omega$ denote the projection map from the \'etale site to the analytic site. Then
\begin{itemize}
\item $(H^1(j_{!}  \omega^{-k,D_p^\times-\sm}_{\Dr})\otimes \mathcal{A}_{D,(k,0)}^{c,K^p})^{\cO_{D_p}^\times}=H^1(\Fl, j_!\nu_* (\mathcal{L}^{ss}\otimes j^*\omega^k_\Fl))$.
\item $H^0( i_* \mathcal{H}^1_{\ord}(K^p,k)\otimes\omega^k_{\Fl})=H^0(\Fl, i_! (\mathcal{L}^{\ord}\otimes i^*\omega^k_{\Fl}))$.
\end{itemize}
where $\mathcal{L}^{\ord}= \mathcal{H}^1_{\ord}(K^p,k)$ is viewed as a local system on $\mathbb{P}^1(\Q_p)$, and we  note that $i_*=i_!$.
This is very similar to Kashiwara-Schmid's geometric constructions of representations of real groups on the flag variety in \cite{KS94}. It is also very interesting to observe that these representations should be considered as objects in the ``classical local Langlands correspondence'' in the sense that they do not see the information of the Hodge filtration. This suggests that for $p$-adic local Langlands correspondence, it might be necessary to work in certain filtered (equivariant) derived category on $\Fl$.
\end{rem}

\subsection{de Rham cohomology of modular curves}
\begin{para}
The goal of this subsection is to show that most generalized eigensubspaces in Theorem \ref{sd} are eigenspaces. To do this, we need to understand the Hecke action on $H^1_{\rig}(\Ig(K^p),\Sym^k)$. This is well-known for the $\ell$-adic cohomology (essentially the Jacquet module of the automorphic forms). Here we  work with the de Rham cohomology and do all the computations on $\Fl$ using ``the $\bar{d}$-resolution'' in Remark \ref{dbarres}.

For integer $k\geq 0$, recall that  $\theta^{k+1}$ was introduced in \ref{XCY}. Let $D^k$ denote the complex
\[ \omega^{-k,\sm}_{K^p}\xrightarrow{\theta^{k+1}} \omega^{-k,\sm}_{K^p}\otimes_{\cO^{\sm}_{K^p}} (\Omega^1_{K^p}(\mathcal{C})^{\sm})^{\otimes k+1}.\]
It follows from the construction of $\theta^{k+1}$ that the inclusion $\omega^{-k,\sm}_{K^p}\subseteq\Sym^k (D^{\sm}_{K^p})^* $ induces a quasi-isomorphism between $D^k$ and $\Sym^k (D^{\sm}_{K^p})^*\xrightarrow{\nabla} \Sym^k (D^{\sm}_{K^p})^* \otimes_{\cO^{\sm}_{K^p}} \Omega^1_{K^p}(\mathcal{C})^{\sm}$. The vanishing of $H^0(\theta^{k+1})$ and $H^1(\omega^{k+2,\sm}_{K^p})$ implies that there is a natural exact sequence
\[0\to H^0( \omega^{-k,\sm}_{K^p}\otimes_{\cO^{\sm}_{K^p}} (\Omega^1_{K^p}(\mathcal{C})^{\sm})^{\otimes k+1}) \to \bH^1(D^k) \to H^1(\omega^{-k,\sm}_{K^p}) \to 0 \]
and $\mathbb{H}^2(D^k)=0$,
where $\mathbb{H}^*(D^k)$ denotes the hypercohomology of $D^k$. It is easy to see that 
$\displaystyle \bH^1(D^k)=\varinjlim_{K_p\subseteq\GL_2(\Q_p)} \bH^1(\mathcal{X}_{K^pK_p}, \Sym^k D^*\xrightarrow{\nabla} \Sym^k D^*\otimes\Omega^1(\mathcal{C}))$. Note that $\Sym^k D^*\otimes\Omega^1(\mathcal{C})$ is defined over $\Q$ and  we have the usual Shimura isomorphism for the cohomology of its base change to complex numbers, which implies that the Hecke action of $\T^S$ on its cohomology is  \textit{semi-simple}. In particular, we just proved the following result.
\end{para}

\begin{prop}
The Hecke action of $\T^S$ on $\mathbb{H}^1(D^k)$ is {semi-simple}.
\end{prop}

\begin{para}
We also have the following complexes on $\Fl$.
\[ DR_k:\cO^{\la,(-k,0)}_{K^p}\xrightarrow{d^{k+1}}  \cO^{\la,(-k,0)}_{K^p}\otimes_{\cO^{\sm}_{K^p}} (\Omega^1_{K^p}(\mathcal{C})^{\sm})^{\otimes k+1},\]  
\[DR'_k:\cO^{\la,(-k,0)}_{K^p}\otimes_{\cO_{\Fl}}(\Omega^1_{\Fl})^{\otimes k+1}\xrightarrow{d'^{k+1}}\cO^{\la,(1,-k-1)}_{K^p}(k+1).\]
Then by Proposition \ref{ddbarcom}, $\bar{d}^{k+1}$ induces an exact triangle
\[ D^k\otimes_{\Q_p} \Sym^k V\cdot\varepsilon^{-k}\to DR_k \xrightarrow{\bar{d}^{k+1}} DR'_k \xrightarrow{+1} D^k[1]\otimes_{\Q_p} \Sym^k V\cdot\varepsilon^{-k}.\]
It was shown in the proof of Theorem \ref{dec1} that $H^0(\cO^{\la,(-k,0)}_{K^p}\otimes_{\cO_{\Fl}}(\Omega^1_{\Fl})^{\otimes k+1})=0$. Hence $\bH^0(DR'_k)=0$. Thus 
\[0\to \bH^1(D^k)\otimes_{\Q_p} \Sym^k V\cdot\varepsilon^{-k}\to \bH^1(DR_k)\to \bH^1(DR'_k)\to 0. \]
On the other hand, we have $0\to H^1(\ker d^{k+1}) \to \bH^1(DR_k)\to H^0(\coker d^{k+1})\to 0$ as there is no $H^2$ on $\Fl$, and similar statement holds for $\bH^1(DR'_k)$. An application of the snake lemma to
\[\begin{tikzcd}
0 \arrow[r] & H^1(\ker d^{k+1})   \arrow[d,"SS^k"]\arrow[r] & \bH^1(DR_k) \arrow[d,"\bH^1(\bar{d}^{k+1})"] \arrow[r]& H^0(\coker d^{k+1}) \arrow[d,"ORD^k"]  \arrow[r] & 0 \\
0 \arrow[r] & H^1(\ker d'^{k+1})   \arrow[r] & \bH^1(DR'_k)  \arrow[r]& H^0(\coker d'^{k+1})  \arrow[r] & 0
\end{tikzcd}.\]
shows that
\begin{eqnarray} \label{dRstr}
0\to \ker SS^k \to \bH^1(D^k)\otimes_{\Q_p} \Sym^k D \cdot\varepsilon^{-k}  \to \ker ORD^k \to \coker SS^k\to 0
\end{eqnarray}
where both vertical maps $SS^k,ORD^k$ are induced by $\bar{d}^{k+1}$. 
\end{para}

\begin{lem} \label{ORDk}
$\ker ORD^k\cong \Sym^k V \otimes_{\Q_p} \Ind_B^{\GL_2(\Q_p)} H^1_{\rig}(\Ig(K^p),\Sym^k)\cdot\varepsilon^{-k}$. 
\end{lem}

\begin{proof}
By Corollary \ref{cokerIkinfty}, we have
$\coker d^{k+1}\cong \omega^k_{\Fl}\otimes_C \mathcal{H}^1_{\ord}(K^p,k)\cdot\varepsilon^{-k}$,
$\coker d'^{k+1}\cong \omega^{-k-2}_{\Fl}\otimes_C \mathcal{H}^1_{\ord}(K^p,k)\otimes\det{}^{k+1}\cdot\varepsilon^{-k}$, and $ORD^k$ is induced by $\bar{d}^{k+1}:\omega^k_{\Fl}\to \omega^{-k-2}_{\Fl}\otimes\det{}^{k+1}$, whose kernel is $H^0(\Fl,\omega^k_{\Fl})=\Sym^k V\otimes_{\Q_p} C$. This shows that 
\begin{eqnarray*}
\ker ORD^k &\cong& \Sym^k V \otimes_{\Q_p} H^0(\mathcal{H}^1_{\ord}(K^p,k))\cdot\varepsilon^{-k}\\
 &\cong & \Sym^k V \otimes_{\Q_p} \Ind_B^{\GL_2(\Q_p)} H^1_{\rig}(\Ig(K^p),\Sym^k) \cdot\varepsilon^{-k}. \end{eqnarray*}
 \end{proof}
For $SS^k$, recall that $\displaystyle \cO_{\LT}^{D_p^\times-\sm}=\varinjlim_{n}(\pi^{(0)}_{\Dr,n})_* \cO_{\mathcal{M}_{\Dr,n}^{(0)}}$,
where $\pi^{(0)}_{\Dr,n}:\mathcal{M}_{\Dr,n}^{(0)}\to \Omega$ denotes the projection map, and $\displaystyle \Omega^1_{\Dr}=\varinjlim_{n}(\pi^{(0)}_{\Dr,n})_* \Omega^1_{\mathcal{M}_{\Dr,n}^{(0)}}$
cf. Remark \ref{dDR}. Then $\Omega^1_\Dr\cong \omega^{2,D_p^\times-\sm}_{\Dr}\otimes\det$ by the Kodaira-Spencer isomorphism.  There is a natural derivation  $d_{\dR}: \cO_{\LT}^{D_p^\times-\sm}\to \Omega^1_\Dr$. We denote this complex by $DR_{\dR}$.

\begin{defn}
For $i\geq 0$, we define
\[H^{i}_{\dR,c}(\mathcal{M}_{\Dr,\infty}^{(0)}):=\bH^i(j_! \cO_{\LT}^{D_p^\times-\sm}\xrightarrow{j_! d_{\dR}} j_!\Omega^1_\Dr).\]
Since $H^i(j_! \cO_{\LT}^{D_p^\times-\sm})=H^i(j_! \Omega^1_{\Dr})=0$ unless $i=1$, we have
\[H^{i}_{\dR,c}(\mathcal{M}_{\Dr,\infty}^{(0)})=H^{i-1}\left(H^1(j_! \cO_{\LT}^{D_p^\times-\sm})\xrightarrow{H^1(j_!d_\dR)} H^1(j_! \Omega^1_{\Dr})\right).\]
\end{defn}

$\bar{d}^{k+1}$ induces a map $\omega_{\Dr}^{-k,D_p^\times-\sm}\to \omega_{\Dr}^{-k,D_p^\times-\sm}\otimes_{\cO_{\LT}^{D_p^\times-\sm}}(\Omega^1_{\Dr})^{\otimes k+1}\cong \omega_{\Dr}^{k+2,D_p^\times-\sm}\otimes\det^{k+1}$.

\begin{lem} \label{SSk}
There is a natural morphism of complexes
\[\left[\omega_{\Dr}^{-k,D_p^\times-\sm}\xrightarrow{\bar{d}^{k+1}}  \omega_{\Dr}^{k+2,D_p^\times-\sm}\otimes\det{}^{k+1}\right]\to \Sym^k V\otimes_{\Q_p} DR_{\dR},\]
which is a quasi-isomorphism and has a left inverse.
In particular, applying $j_!$, we have 
\[\ker SS^k\cong (H^{1}_{\dR,c}(\mathcal{M}_{\Dr,\infty}^{(0)})\otimes_C \mathcal{A}_{D,(k,0)}^{c,K^p})^{\cO_{D_p}^\times}\otimes_{\Q_p}\Sym^k V\cdot{\varepsilon}_p^{-k}, \]
\[\coker SS^k \cong  \left\{
	\begin{array}{lll}
		(H^{2}_{\dR,c}(\mathcal{M}_{\Dr,\infty}^{(0)})\otimes_C \mathcal{A}_{D,(k,0)}^{K^p})^{\cO_{D_p}^\times}\otimes_{\Q_p}\Sym^k V\cdot{\varepsilon}_p^{-k}, &k\geq 1&\\
		(H^{2}_{\dR,c}(\mathcal{M}_{\Dr,\infty}^{(0)})\otimes_C \mathcal{A}_{D,(0,0)}^{c,K^p})^{\cO_{D_p}^\times}\oplus M_0(K^p) , &k=0&
	\end{array}.\right.\]
\end{lem}

\begin{proof}
The quasi-isomorphism follows from the construction of $\bar{d}^{k+1}$, cf. \ref{dbar}. More precisely, there is a natural action of $Z(U(\mathfrak{g}))$ on $\Sym^k V\otimes_{\Q_p}DR_{\dR}$. For an infinitesimal character $\tilde\chi:Z(U(\mathfrak{g}))\to C$,  a simple calculation shows that the $\tilde{\chi}$-isotypic part of $\Sym^k V\otimes_{\Q_p}DR_{\dR}$ is non-zero only when $\tilde{\chi}$ is the infinitesimal character of $\omega^k_{\Fl}$, and in this case, the $\tilde{\chi}$-isotypic part is naturally isomorphic to $\omega_{\Dr}^{-k,D_p^\times-\sm}\xrightarrow{\bar{d}^{k+1}}  \omega_{\Dr}^{k+2,D_p^\times-\sm}\otimes\det{}^{k+1}$.

By  Proposition \ref{kerdk+1s} and Corollary \ref{kerd'k+1s}, we may identify $SS^k$ with 
\[(\mathcal{A}_{D,(k,0)}^{K^p}\otimes  H^1( j_{!}\omega^{-k,D_p^\times-\sm}_{\Dr}))^{\cO_{D_p}^\times}\cdot  {\varepsilon'_p}^{-k} \xrightarrow{H^1(\bar{d}^{k+1})} (\mathcal{A}_{D,(k,0)}^{K^p}\otimes  H^1( j_{!}\omega^{k+2,D_p^\times-\sm}_{\Dr}))^{\cO_{D_p}^\times}\otimes\det{}^{k+1}\cdot  {\varepsilon'_p}^{-k}\]
when $k\geq 1$ and a similar result holds for $k=0$. From this we easily deduce our claim.
\end{proof}

\begin{thm} \label{ijdR}
For $k\geq 1$,  there is a natural exact sequence
\begin{eqnarray*}
0\to  (H^{1}_{\dR,c}(\mathcal{M}_{\Dr,\infty}^{(0)})\otimes_C \mathcal{A}_{D,(k,0)}^{K^p})^{\cO_{D_p}^\times}\to \bH^1(D^k) \to  \Ind_B^{\GL_2(\Q_p)} H^1_{\rig}(\Ig(K^p),\Sym^k)\\
 \to   (H^{2}_{\dR,c}(\mathcal{M}_{\Dr,\infty}^{(0)})\otimes_C \mathcal{A}_{D,(k,0)}^{K^p})^{\cO_{D_p}^\times}\to 0.
 \end{eqnarray*}
When $k=0$, we have 
\begin{eqnarray*}
0\to  (H^{1}_{\dR,c}(\mathcal{M}_{\Dr,\infty}^{(0)})\otimes_C \mathcal{A}_{D,(0,0)}^{c,K^p})^{\cO_{D_p}^\times}\to \bH^1(D^k) \to  \Ind_B^{\GL_2(\Q_p)} H^1_{\rig}(\Ig(K^p))\\
 \to   (H^{2}_{\dR,c}(\mathcal{M}_{\Dr,\infty}^{(0)})\otimes_C \mathcal{A}_{D,(0,0)}^{c,K^p})^{\cO_{D_p}^\times}\oplus M_0(K^p)\to 0.
 \end{eqnarray*}
\end{thm}

\begin{proof}
Apply the results of Lemma \ref{ORDk} and Lemma \ref{SSk} to sequence \eqref{dRstr} and remove $\otimes_{\Q_p}\Sym^k V\cdot {\varepsilon'_p}^{-k}$ in every term of the exact sequence by taking the $\GL_2(\Q_p)$-smooth parts of the tensor product of sequence \eqref{dRstr} with $ \Sym^k V^*\cdot {\varepsilon'_p}^{k}$.
\end{proof}

\begin{rem}
Theorem \ref{ijdR} can be viewed as a $p$-adic analogue of the usual vanishing cycle exact sequence in the $\ell$-adic case \cite[4.2]{Car86}. It can also be obtained from taking the hypercohomology of the exact triangle $ j_! D^k|_{\Omega}\to D^k \to i_*i^* D^k\xrightarrow{+1}  j_! D^k|_{\Omega}[1]$. It's clear that $\bH^*( i_*i^* D^k)$ computes $ \Ind_B^{\GL_2(\Q_p)} H^*_{\rig}(\Ig(K^p),\Sym^k)$. For $ j_! D^k|_{\Omega}$, note that $\bH^i(j_! D^k|_{\Omega})$ computes the compactly supported de Rham cohomology of the supersingular locus of modular curves, or essentially the compactly supported de Rham cohomology of the Lubin-Tate towers. To compare with Theorem  \ref{ijdR},  we use \cite[Th\'eor\`eme 4.4]{CDN20}
 which says that  the Lubin-Tate towers and  Drinfeld towers have the same compactly supported de Rham cohomology.
\end{rem}

\begin{para} \label{H1dRc}
As the notation suggests, $H^{i}_{\dR,c}(\mathcal{M}_{\Dr,\infty}^{(0)})$ should be viewed as the compactly supported de Rham cohomology of the Drinfeld towers in the following sense. Let 
\[H^i_{\dR,c}(\mathcal{M}_{\Dr,n}^{(0)}):= H^{i-1}\left(H^1(j_! \pi_{\Dr,n\,*}^{(0)}\cO_{\mathcal{M}_{\Dr,n}^{(0)}})\xrightarrow{d_{\dR}} H^1(j_! \pi_{\Dr,n\,*}^{(0)}\Omega^1_{\mathcal{M}_{\Dr,n}^{(0)}}) \right).\]
Then $\displaystyle H^i_{\dR,c}(\mathcal{M}_{\Dr,\infty}^{(0)})=\varinjlim_n H^i_{\dR,c}(\mathcal{M}_{\Dr,n}^{(0)})$. Let 
\[H^i_{\dR}(\mathcal{M}_{\Dr,n}^{(0)}):= H^{i}\left(H^0(\mathcal{M}_{\Dr,n}^{(0)},\cO_{\mathcal{M}_{\Dr,n}^{(0)}})\xrightarrow{d_{\dR}} H^0(\mathcal{M}_{\Dr,n}^{(0)},\Omega^1_{\mathcal{M}_{\Dr,n}^{(0)}}) \right).\]
Since $\mathcal{M}_{\Dr,n}^{(0)}$ is partially proper, the higher cohomology groups of $\cO_{\mathcal{M}_{\Dr,n}^{(0)}}$ and $\Omega^1_{\mathcal{M}_{\Dr,n}^{(0)}}$ vanish. Hence $H^i_{\dR}(\mathcal{M}_{\Dr,n}^{(0)})=\bH^i(\cO_{\mathcal{M}_{\Dr,n}^{(0)}}\xrightarrow{d_{\dR}}\Omega^1_{\mathcal{M}_{\Dr,n}^{(0)}} )$, i.e. the de Rham cohomology of $\mathcal{M}_{\Dr,n}^{(0)}$. As explained in \cite[\S 4.3]{CDN20}, the Serre duality on $\mathcal{M}_{\Dr,n}^{(0)}$ induces an isomorphism
\[H^i_{\dR}(\mathcal{M}_{\Dr,n}^{(0)})\cong \Hom_C(H^{2-i}_{\dR,c}(\mathcal{M}_{\Dr,n}^{(0)}),C),\]
where $\Hom$ is calculated between $C$-vectors spaces. Comparing our notations here with the reference, we note that since $\pi_{\Dr,n}^{(0)}$ is a finite morphism, it's easy to see that the cohomology group $H^1(j_! \pi_{\Dr,n\,*}^{(0)}\cO_{\mathcal{M}_{\Dr,n}^{(0)}})$ agrees with the $H^1_c(\mathcal{M}_{\Dr,n}^{(0)},\cO_{\mathcal{M}_{\Dr,n}^{(0)}})$ introduced in \cite[\S 4.3.2]{CDN20}.

$H^0_{\dR}(\mathcal{M}_{\Dr,n}^{(0)})=\sC(\pi^0(\mathcal{M}_{\Dr,n}^{(0)}),C)$, where $\pi^0(\mathcal{M}_{\Dr,n}^{(0)})$ denotes the set of connected components of $\mathcal{M}_{\Dr,n}^{(0)}$. Hence $H^{2}_{\dR,c}(\mathcal{M}_{\Dr,n}^{(0)})$ is finite-dimensional and we have the following result. 
\end{para}

\begin{lem} \label{detnm}
The action of $\GL_2(\Q_p)^0$ on $H^2_{\dR,c}(\mathcal{M}_{\Dr,\infty}^{(0)})$ factors through the determinant map. Similarly, the action of  $\cO_{D_p}^\times$ factors through the reduced norm map.
\end{lem}

\begin{cor} \label{ssIg}
Let $\lambda:\T^S\to C$ correspond to a cuspidal automorphic representation $\pi$ of $\GL_2(\A)$. Suppose that the generalized eigenspace $H^1_{\rig}(\Ig(K^p),\Sym^k)\widetilde{[\lambda]}\neq 0$.
Then 
\begin{enumerate}
\item $H^1_{\rig}(\Ig(K^p),\Sym^k)\widetilde{[\lambda]}=H^1_{\rig}(\Ig(K^p),\Sym^k)[\lambda]$ is the eigenspace associated to $\lambda$.
\item $\pi_p$ is either a principal series or special;
\end{enumerate} 
\end{cor}

\begin{proof}
By Theorem \ref{ijdR}, there is an exact sequence
\[\bH^1(D^k)\widetilde{[\lambda]} \to  \Ind_B^{\GL_2(\Q_p)} H^1_{\rig}(\Ig(K^p),\Sym^k)\widetilde{[\lambda]}
 \to   (H^{2}_{\dR,c}(\mathcal{M}_{\Dr,\infty}^{(0)})\otimes_C \mathcal{A}_{D,(k,0)}^{K^p})^{\cO_{D_p}^\times}\widetilde{[\lambda]}\to 0.\]
Note that $\T^S$ acts semi-simply on the first and third terms. Hence if $\T^S$ does not act semi-simply on $H^1_{\rig}(\Ig(K^p),\Sym^k)\widetilde{[\lambda]}$, we see that $(H^{2}_{\dR,c}(\mathcal{M}_{\Dr,\infty}^{(0)})\otimes_C \mathcal{A}_{D,(k,0)}^{K^p})^{\cO_{D_p}^\times}\widetilde{[\lambda]}$ at least contains $\Ind_B^{\GL_2(\Q_p)} W$ for some smooth representation $W$ of $B$, which contradicts Lemma \ref{detnm}. This proves the first claim. Lemma \ref{detnm} also implies that 
\[\Hom_{C[\GL_2(\Q_p)]}(\pi_p,\Ind_B^{\GL_2(\Q_p)} H^1_{\rig}(\Ig(K^p),\Sym^k)[\lambda])\neq 0.\]
Hence $\pi_p$ is either a principal series or special. (To rule out the possibility that $\pi_p$ is a character, either we use the fact that $\pi$ is generic, or we observe that this would imply that a  special representation appears in $(H^{2}_{\dR,c}(\mathcal{M}_{\Dr,\infty}^{(0)})\otimes_C \mathcal{A}_{D,(k,0)}^{K^p})^{\cO_{D_p}^\times}$, which again contradicts Lemma \ref{detnm}.)
\end{proof}

\begin{rem} \label{HodFil}
It follows from Lemma \ref{lemH0} that the natural map $\bH^1(D^k)\to \bH^1(DR_k)$ induces $\Fil^1 \bH^1(D^k)=\Fil^1 \bH^1(DR_k)$. In particular, the inclusion $\Fil^1 \bH^1(DR_k)\subseteq  \bH^1(DR_k)$ exactly describes the position of the Hodge filtration on $\bH^1(D^k)$.
\end{rem}

\section{Intertwining operators:  \texorpdfstring{$p$}{Lg}-adic Hodge-theoretic interpretation} \label{IopHti}
In this section, we give a $p$-adic Hodge-theoretic meaning of the intertwining operators $I_k$ and $I^1_k$ introduced in Section \ref{Ioef}. Our main result says that it essentially agrees with an operator introduced by Fontaine  \cite{Fo04} in the classifications of almost de Rham $B_\dR$-representations, which we will call the Fontaine operator. We first recall its construction and generalize it to our setting. 

\subsection{The Fontaine operator}
\begin{para}[Construction of $B_{\dR}^+$] \label{CBdR+}
We start with the classical construction of $B^+_{\dR}$.
Let $\displaystyle \cO_{C^{\flat}}=\varprojlim_{x\mapsto x^p}\cO_C/p\cong \varprojlim_{x\mapsto x^p}\cO_C$ be the tilt of $\cO_C$ (classically denoted by $R$) and $A_\inf=W(\cO_{C^{\flat}})$ its ring of Witt vectors. We have the usual surjective ring homomorphism 
\[\theta:A_\inf\to \cO_C\]
sending the Teichm\"{u}ller lifting $[x]$ of $\displaystyle x=(x_n)_{n\geq 0}\in\varprojlim_{x\mapsto x^p}\cO_C$, to $x_0$.
By abuse of notation, we also denote by $\theta:A_\inf[\frac{1}{p}]\to C$ the rational version.
The period ring $B_{\dR}^+$ is defined as the $\ker(\theta)$-completion of $A_\inf[\frac{1}{p}]$. This is a complete discrete valuation ring with residue field isomorphic to $C$ via $\theta$. It has a natural decreasing filtration with $\Fil^k B_\dR^+=\ker(\theta)^k B_\dR^+, k\geq 0$. By Hensel's lemma, for any finite extension $K$ of $\Q_p$ in $C$, we have a canonical lifting of $K\subseteq C$ to $K\subseteq B_{\dR}^+$. The Galois group $G_{\Q_p}$ acts on everything and $\theta$ is $G_{\Q_p}$-equivariant. 

Fix a compatible system of primitive $p^n$-th roots of unity $\{\zeta_{p^n}\}_{n\geq 0}$ in $\cO_C$ such that $\zeta_{p^{n+1}}^p=\zeta_{p^n}$, or equivalently a generator of $\Z_p(1)$. Then $(\zeta_{p^n})_{n\geq 0}$ defines an element $\varepsilon$ in $\cO_{C^\flat}$. It is well-known that $[\varepsilon]-1$ generates the kernel of $\theta:A_\inf[\frac{1}{p}]\to C$. The element
\[t:=\log([\varepsilon])=-\sum_{i=1}^{+\infty} \frac{(1-[\varepsilon])^i}{i}\in B_{\dR}^+\]
is Fontaine's ``$2\pi i$'' in $p$-adic Hodge theory on which $G_{\Q_p}$ acts via the cyclotomic character. There is a canonical isomorphism $\gr^kB_{\dR}^+=t^k B_\dR^+/t^{k+1}B_{\dR}^+\cong C(k)$,  $k\geq 0$.
\end{para}

\begin{para}[Banach $B_{\dR}^+/(t^k)$-modules] \label{BBdrtkmod}
For the purpose of this paper, we will only consider modules over $B_{\dR}^+/(t^k)$.  There is a natural $\Q_p$-Banach algebra structure on $B_{\dR}^+/(t^k)$ with the unit open ball given by the image of $A_\inf\to B_{\dR}^+/(t^k)$. We say a $\Q_p$-Banach space $W$ is a Banach $B_{\dR}^+/(t^k)$-module if it is equipped with a $B_{\dR}^+/(t^k)$-module structure such that the multiplication map $B_{\dR}^+/(t^k)\times W\to W$ is jointly continuous. We equip $W$ with the $t$-adic filtration: $\Fil^iW=t^iW,i\geq 0$ and denote by 
\[W_i:=\gr^i W=W\otimes_{B_{\dR}^+/(t^k)} t^iB_{\dR}^+/t^{i+1}B_{\dR}^+,~~~~~~i\geq0.\]
Hence $W_0=W/tW$. Moreover if $W$ is flat over $B_{\dR}^+/(t^k)$, then
$t^iW\subseteq W$ is a closed subspace because it can be identified with the kernel of $W\xrightarrow{\times t^{k-i}} W$. Hence in this case $W_i$ is a $C$-Banach space for $i\geq 0$ and there are natural isomorphisms
\[W_i\cong W_0\otimes_C C(i)=W_0(i),~~~~~~i\geq 0.\]
\end{para}

\begin{para}[Assumption] \label{setup}
Throughout this subsection, we fix a finite extension $K$ of $\Q_p$ in $C$.
Now suppose $W$ is a Banach $B_{\dR}^+/(t^k)$-module equipped with a continuous semilinear action of $G_{K}$. Then $G_{K}$ also acts on $W_i$. Recall that $K_\infty\subset \overbar\Q_p$ denotes the maximal $\Z_p$-extension of $K$ in $K(\mu_{p^\infty})$, cf. \ref{HTsetup}. Let
\[W^{K}_i\subseteq W_i\]
be the subspace of $G_{K_\infty}$-fixed, $G_K$-analytic vectors in $W_i$. This is naturally a $K$-Banach space. The $C$-Banach space structure on $W_i$ induces a map
\[\varphi_{W,i}^K:C\widehat{\otimes}_K W^{K}_i \to W_i.\]
We will write $\varphi_i^K$ instead of $\varphi_{W,i}^K$ if there is no confusion.
We make the following assumptions from now on
\begin{itemize}
\item $W$ is flat over $B_{\dR}^+/(t^k)$;
\item $\varphi_0^K$ is an isomorphism.
\end{itemize}
A standard argument using Tate's normalized trace shows that the second assumption implies that the natural map
\[L\otimes_{K}W^K_0\xrightarrow{\cong} W^L_0\]
is an isomorphism for any finite extension $L$ of $K$ in $C$. In particular, $\varphi^L_0$ is an isomorphism for any finite extension $L$ of $K$. 

%If $K$ is moreover a Galois extension of $\Q_p$, we have a natural action of $\Gal({K(\mu_{p^\infty})}/\Q_p)$ on $W^K_0$. On the other hand, the Lie algebra $\Lie(\Gal({K(\mu_{p^\infty})}/K))\cong \Lie(\Gal(\Q_p(\mu_{p^\infty})/\Q_p))$ also acts on $W^K_0$. We remark that both actions commute with each other by observing that $\Gal({K(\mu_{p^\infty})}/\Q_p)$ is a subgroup of $\Gal(\Q_p(\mu_{p^\infty})/\Q_p)\times \Gal(K/\Q_p)$.
\end{para}

\begin{para}[Fontaine' generalization of Sen theory]
We denote by 
\[W^K\subseteq W\] 
the subspace  of $G_{K_\infty}$-fixed, $G_K$-analytic vectors. It has a natural $K$-linear structure coming from the inclusion $K\subseteq B_{\dR}^+$ and inherits the $K$-Banach space structure and the decreasing filtration  from  $W$.  %Hence we have a natural map
%\[\varphi^K: B_\dR^+/t^kB_\dR^+\widehat\otimes_K W^K\to W.\]
Taking the graded pieces, we get a natural injective map
\[g_i^K:\gr^i W^K\to W_i^K\]
for $i\geq 0$. Here is Fontaine's generalization of Sen theory to $B_{\dR}^+$-representations. 
\end{para}

\begin{prop} \label{FonSen}
Let $W$ be a flat Banach $B_{\dR}^+/(t^k)$-module equipped with a continuous semilinear action of $G_{K}$ and $\varphi_0^K$ is an isomorphism. Then there exists a finite extension $K'$ of $K$ in $K_\infty$ such that 
\begin{enumerate}
\item $g_i^{K'}$ is an isomorphism for $i\geq0$.
\item $g^L_i$ is an isomorphism  for $i\geq 0$ and any finite extension $L$ of $K'$. In particular, there is a natural isomorphism 
\[L\otimes_{K'}W^{K'}\cong W^L.\]
\end{enumerate}
\end{prop}

\begin{proof}
Denote by $K_n\subseteq K_\infty$ the unique $\Z/(p^n)$-extension of $K$ and $\Gal(K_\infty/K_n)$ by $\Gamma_n$. 

\begin{lem} \label{surj^K}
Let $W$ be as in Proposition \ref{FonSen}. Then for $j\geq 0$,
\begin{enumerate}
\item $W^{G_{K_\infty}}\to (W/t^jW)^{G_{K_\infty}}$ is surjective.
\item The image of $W^{K_n}\to (W/t^jW)^{K_n}$ contains $(W/t^jW)^K$ for some $n\geq 0$.
\end{enumerate}
\end{lem}

\begin{proof}
The first claim is a direct consequence of the well-known result  $H^1_\cont(G_{K_\infty},C)=0$. Indeed, consider the $G_{K_\infty}$-invariants of  the exact sequence
\[0\to t^jW\to W\to W/t^jW\to 0.\]
It suffices to show that $H^1_{\cont}(G_{K_\infty},t^jW)=0$. Since $t^jW$ is filtered by $W_i$, it is enough to know that $H^1_{\cont}(G_{K_\infty},W_i)=0$. Since $W_i= C\widehat{\otimes}_K W^{K}_i$ and $G_{K_\infty}$ fixes $W^K_i$, 
\[H^1_{\cont}(G_{K_\infty},W_i)= H^1_{\cont}(G_{K_\infty},C\widehat{\otimes}_K W^{K}_i)
=H^1_{\cont}(G_{K_\infty},C)\widehat{\otimes}_K W^{K}_i=0\]
where the second equality follows from Proposition \ref{tenscomHcont} and the third equality is a classical result of Tate. Note that this shows that $W^{G_{K_\infty}}$ is filtered by $W_i^{G_{K_\infty}}$. 

To see the second part of the lemma, consider the exact sequence
\[0\to (t^jW)^{G_{K_\infty}}\to W^{G_{K_\infty}}\to (W/t^jW)^{G_{K_\infty}}\to 0.\]
Let $n\geq0$ be an integer. Now take the $\Gamma_n$-analytic vectors of this sequence, equivalently take the completed tensor product with $\sC^{\an}(\Gamma_n,\Q_p)$ over $\Q_p$ then take the $\Gamma_n$-invariants. Hence we get an exact sequence
\[W^{K_n}\to (W/t^jW)^{K_n}\to H^1_\cont\left(\Gamma_n,(t^jW)^{G_{K_\infty}}\widehat{\otimes}_{\Q_p}\sC^{\an}(\Gamma_n,\Q_p)\right)\]
functorial in $n$. 
\begin{lem} \label{Gamacyc}
There exists $n\geq 0$ such that the image of 
\[H^1_\cont\left(\Gamma_0,(t^jW)^{G_{K_\infty}}\widehat{\otimes}_{\Q_p}\sC^{\an}(\Gamma_0,\Q_p)\right)\to H^1_\cont\left(\Gamma_n,(t^jW)^{G_{K_\infty}}\widehat{\otimes}_{\Q_p}\sC^{\an}(\Gamma_n,\Q_p)\right)\]
is zero. 
\end{lem}
This lemma implies that the image of $W^{K_n}\to (W/t^jW)^{K_n}$ contains $(W/t^jW)^{K_0}=(W/t^jW)^{K}$, which is exactly what we want.
\end{proof}

\begin{proof}[Proof of Lemma \ref{Gamacyc}]
Since $(t^jW)^{G_{K_\infty}}$ is filtered by $W_i^{G_{K_\infty}}$, it suffices to show that for any $i,m\geq 0$, we can find $n\geq m$ such that the image of
\[H^1_\cont\left(\Gamma_m,W_i^{G_{K_\infty}}\widehat{\otimes}_{\Q_p}\sC^{\an}(\Gamma_m,\Q_p)\right)\to H^1_\cont\left(\Gamma_n,W_i^{G_{K_\infty}}\widehat{\otimes}_{\Q_p}\sC^{\an}(\Gamma_n,\Q_p)\right)\]
is zero. By the Ax-Sen-Tate theorem,  $C^{G_{K_\infty}}=\widehat{K_\infty}$, the $p$-adic completion of $K_\infty$ . Since $\varphi_i^K$ is an isomorphism, we have
\[W_i^{G_{K_\infty}}=W^K_i\widehat{\otimes}_{K} \widehat{K_\infty}.\]
By our construction, the action of $\Gamma_m$ on $W_i^K$ is analytic, hence there is a natural $\Gamma_m$-equivariant isomorphism
\[W_i^K\widehat{\otimes}_{\Q_p}\sC^{\an}(\Gamma_m,\Q_p) \cong W_i^K\widehat{\otimes}_{\Q_p}\sC^{\an}(\Gamma_m,\Q_p)\]
functorial in $m$, where $\Gamma_m$ acts on everything except $W_i^K$ on the right hand side, cf. \ref{GanBan}. We use $W'$ to denote $W_i^K$ equipped with the trivial action of $\Gamma_m$. Thus it suffices to find $n\geq m$ such  that the image of 
\[H^1_\cont\left(\Gamma_m,\widehat{K_\infty}\widehat\otimes_K W'\widehat{\otimes}_{\Q_p}\sC^{\an}(\Gamma_m,\Q_p)\right)\to H^1_\cont\left(\Gamma_n,\widehat{K_\infty}\widehat\otimes_K W'\widehat{\otimes}_{\Q_p}\sC^{\an}(\Gamma_n,\Q_p)\right)\]
is zero, i.e. $\widehat{K_\infty}\widehat\otimes_K W'$ is strongly $\mathfrak{LA}$-acyclic with respect to the action of $\Gamma_0$ in the sense of \cite[\S 2.2]{Pan20}. Denote by $\bar{\tr}_n:\widehat{K_\infty}\to K_n$ the Tate's normalized trace, which exists for sufficiently large $n$ and gives a continuous left inverse of the inclusion $K_n\subseteq \widehat{K_\infty}$. See for example \cite[4.1]{BC08}. Fix a generator $\gamma$ of $\Gamma_m$. Then $\gamma-1$ is invertible on $\ker(\bar\tr_n)$ and the  norm of its inverse $\|(\gamma-1)^{-1}\|$ converges to $1$ when $n\to+\infty$. Since  $(\gamma-1)^p$ has norm $\leq p^{-1}$ on $\sC^{\an}(\Gamma_m,\Q_p)$, by the same argument as in the proof of \cite[Lemma 3.6.6]{Pan20}, we have
\[H^1_\cont(\Gamma_m,\ker(\bar\tr_n)\widehat\otimes_K W'\widehat{\otimes}_{\Q_p} \sC^{\an}(\Gamma_m,\Q_p))=0\]
if $\|(\gamma-1)^{-1}\|< p^{1/p}$ on $\ker(\bar\tr_n)$. Fix such an $n$. Since $\widehat{K_\infty}=\ker(\bar\tr_n)\oplus K_n$, it is enough to find $n'\geq n$ such that the image of 
\[H^1_\cont(\Gamma_m, K_n\otimes_K W'\widehat\otimes_{\Q_p}\sC^\an(\Gamma_m,\Q_p))\to H^1_\cont(\Gamma_{n'}, K_n\otimes_K W'\widehat\otimes_{\Q_p}\sC^\an(\Gamma_{n'},\Q_p)) \]
is zero. This is clear if we take $n'=n+1$ because the map $H^1_{\cont}(\Gamma_n,\sC^\an(\Gamma_n,\Q_p))\to H^1_{\cont}(\Gamma_{n+1},\sC^\an(\Gamma_{n+1},\Q_p))$ is zero.
\end{proof}

Back to the proof of Proposition \ref{FonSen}. We argue by induction on $k$. When $k=1$, we have $W=W_0$ and all the claims are clear. Suppose we already proved Proposition \ref{FonSen} for $k\leq l$. We will deduce the analogous statements for $k=l+1$. Apply our induction hypothesis to the $B_{\dR}^+/(t^{k-1})$-module $tW$. We can find a finite extension $K_m$ of $K$ in $K_\infty$ such that $\gr^i\left( (tW)^{K_m} \right)\cong W_{i+1}^{K_m}$ for $i\geq 0$, and for any finite extension $L$ of $K_m$, we have $(tW)^L=(tW)^{K_m}\otimes_{K_m} L$. 

By Lemma \ref{surj^K}, the image of $W^{K_n}\to (W/tW)^{K_n}=W_0^{K_n}$ contains $W_0^{K_m}$ for some $n\geq m$. Note that  $W_0^{K_n}=W_0^{K_m}\otimes_{K_m} K_n$. Hence
\[W^{K_n}\to W_0^{K_n}\]
is surjective as this map is $K_n$-linear. It follows that
\[g_0^{K_n}:\gr^0 W^{K_n}\to W_0^{K_n}\]
is an isomorphism. Thus $g_i^{K_n}$ is an isomorphism for $i\geq 0$ because $g_i^{K_n}$ is already an isomorphism for $i\geq 1$ by the induction hypothesis.

Now we observe that we can repeat the same argument with $K_n$ replaced by a finite extension $L$ of $K_n$ and prove that $g_i^{L}$ is an isomorphism for $i\geq 0$ and any finite extension $L$ of $K_n$. This implies that the natural map $L\otimes_{K_n}W^{K_n}\to W^L$ is an isomorphism by looking at the graded pieces which is nothing but $L\otimes_{K_n}W_i^{K_n}=W^L_i$.
\end{proof}

Our next result concerns the kernel and cokernel of a morphism between such $W$'s.
\begin{prop} \label{Bmor}
Let $f:X\to Y$ be a continuous $B_{\dR}^+/(t^k)$-linear, $G_K$-equivariant maps between Banach flat $B^{+}_\dR/(t^k)$-modules equipped with continuous semilinear actions of $G_K$ such that $\varphi^K_{X,0}, \varphi^K_{Y,0}$ are isomorphisms. Suppose that
\begin{itemize}
\item $f$ is strict with respect to the topology on $X,Y$; hence $\ker f, \coker f$ are natural Banach $B^{+}_\dR/(t^k)$-modules equipped with continuous semilinear actions of $G_K$. 
\item $f$ is strict with respect to  the $t$-adic filtrations on $X,Y$, that is $f(t^i X)=f(X)\cap t^i Y$ for any $i\geq 0$.
\end{itemize}
Then $\ker f$ and $\coker f$ are also flat over $B^{+}_\dR/(t^k)$, and $\varphi^K_{\ker f,0}, \varphi^K_{\coker f,0}$ are isomorphisms. Moreover, there exists a finite extension $L$ of $K$ in $K_\infty$ such that the natural sequence 
\begin{eqnarray} \label{fkc}
0\to (\ker f)^L\to X^L\to Y^L\to (\coker f)^L\to 0
\end{eqnarray}
is exact.
\end{prop}

\begin{proof}
The second assumption implies that $0\to t^i\ker f\to t^i X\to t^i Y\to t^i \coker f\to 0$ is exact. Hence the flatness of $\ker f$ and $\coker f$ follows from the flatness of $X,Y$ over $B^{+}_\dR/(t^k)$. Moreover this implies that $0\to (\ker f)_0\to X_0\xrightarrow{f_0} Y_0\to (\coker f)_0\to 0$ is exact with strict homomorphisms. Since $X_0= C\widehat\otimes_K X_0^K$ and  $Y_0= C\widehat\otimes_K Y_0^K$, it follows from Corollary \ref{strHT} that $\varphi^K_{\ker f,0}, \varphi^K_{\coker f,0}$ are isomorphisms and $0\to (\ker f)_0^K\to X_0^K\xrightarrow{f_0} Y_0^K\to (\coker f)_0^K\to 0$ is exact. For the last part, by Proposition \ref{FonSen}, we can find a finite extension $L$ of $K$ in $K_\infty$ such that $\gr^iW^L=W_i^L$ for $W=\ker f,X,Y,\coker f$. Then it's clear that the sequence \eqref{fkc} is exact for such $L$.
\end{proof}

For the purpose of this paper, we will only consider Hodge-Tate $B_{\dR}^+/(t^k)$-representations of weights $0,l$ in the following sense. 

\begin{defn} \label{HTBdR}
Let $W,K$ be as in \ref{setup} and $l\geq 1$ a positive integer. We say $W$ is Hodge-Tate of weights $0,l$ if 
$W_0$ is Hodge-Tate of weights $0,l$  in the sense of Definition \ref{HTg}, equivalently, there exists a finite extension $L$ of $K$ in $C$ such that  $W_0=C\widehat\otimes_L W_0^L$  and  the action of the Sen operator (i.e. $1\in\Q_p=\Lie(\Z_p^\times)\cong \Lie(\Gal(L_\infty/L))$) on $W_0^L$ is semi-simple with eigenvalues $0$ and $-l$.
\end{defn}

\begin{para} \label{HT0l}
Suppose $W$ is Hodge-Tate of weights $0,l$. Following the notation in Definition \ref{HTg}, we decompose $W^L_0$ into the direct sum of eigenspaces
\[W^L_0=W^L_{0,0}\oplus W^L_{0,-l},\]
where $W^L_{0,j}$ denotes the eigenspace for $j\in\{0,-l\}$. Let $W_{0,j}=C\widehat\otimes_L W^L_{0,j}$. Then 
\[W_0=W_{0,0}\oplus W_{0,-l}\]
and this decomposition does not depend on the choice of $L$. Similarly we can define
\[W^L_i=W^L_{i,0}\oplus W^L_{i,-l},\]
\[W_i=W_{i,0}\oplus W_{i,-l}\cong W_{0,0}(i)\oplus W_{0,-l}(i),~~~~~~i\geq 0\]
via the isomorphism $W_i\cong W_0(i)$. We remark that  the action of the Sen operator on $W^L_{i,j}$ is multiplication by $i+j$. Hence $ \Lie(\Gal({L_\infty/L}))$ acts trivially on $W^L_{i,j}$ if and only if $W^L_{i,j}=W^L_{0,0}$ or $W^L_{l,-l}(l)$ if it exists.
\end{para}

By Remark \ref{clHT}, we can take $L=K$ in the definition \ref{HTBdR}. In fact, a slightly stronger result holds.

\begin{lem} \label{HTK}
Let $W$ be a flat Banach $B_{\dR}^+/(t^k)$-module equipped with a continuous semilinear action of $G_{K}$. Suppose $W$ is Hodge-Tate of weights $0,l$. Then $W_0=C\widehat\otimes_K W_0^K$ and $g_i^K:\gr^i W^K\to W_i^K$ are isomorphisms for $i\geq 0$.
\end{lem}

\begin{proof}
It follows from Remark \ref{clHT} that $W_0=C\widehat\otimes_K W_0^K$. It remains to show that  $g_i^K:\gr^i W^K\to W_i^K$ are isomorphisms for $i\geq 0$.
By Proposition \ref{FonSen}, $g_i^{K_n},i\geq 0$ are isomorphisms for some $n\geq 0$. An induction argument on $i$ shows that it is enough to prove that 
\begin{eqnarray} \label{twKn}
0\to (tW)^{K_n}\to W^{K_n} \to W_0^{K_n}\to 0
\end{eqnarray}
remains exact when passing to the $\Gal(K_\infty/K)$-analytic vectors.
Fix a topological generator $\sigma\in\Gal(K_\infty/K)$. Then $\sigma^{p^n}$ topologically generates $\Gal(K_\infty/K_n)$. Since the action of $\Gal(K_\infty/K_n)$ on $W_i^{K_n}$ is analytic, the eigenvalues of $\sigma^{p^n}$ on $W^{K_n}$ are of the form $\varepsilon_p(\sigma^{p^n})^{i+j}$, $i=0,\cdots, k-1$, $j=0,-l$ and there is a natural decomposition 
\[W^{K_n}=\bigoplus_m E_m,\] 
where  $E_m$ denotes the generalized eigenspace associated to  $\varepsilon_p(\sigma^{p^n})^{m}$.  Note that  each $E_m$ has a natural generalized eigenspace decomposition with respect to the action of $\sigma$ and $E_m^{\Gal(K_\infty/K)-\an}$ is nothing but the generalized eigenspace associated to the eigenvalue $\varepsilon_p(\sigma)^{m}$. Using this interpretation of the $\Gal(K_\infty/K)$-analytic vectors, we see easily that \eqref{twKn} remains exact after taking $\Gal(K_\infty/K)$-analytic vectors.
\end{proof}

\begin{para} \label{FonOpsetup}
Finally we can define the Fontaine operator. Our setup is as follows. Let $k$ be a positive integer. Let $W$ be a flat Banach $B_{\dR}^+/(t^{k+1})$-module equipped with a continuous semilinear action of $G_{K}$. Moreover, we assume that 
\begin{itemize}
\item $W$ is Hodge-Tate of weights $0,k$. 
\end{itemize}
By Lemma \ref{HTK}, $W_0=C\widehat\otimes_K W^K_0$ and $\gr^iW^K= W_i^K$, $i\geq 0$. Hence $W^K$ is filtered by $W^K_{i,0}$ and $W^K_{i,-k}$.

$\Lie(\Gal(K_\infty/K))$ acts naturally on $W^K$. Denote by $\nabla\in \End_K(W^K)$ the bounded operator defined by the action of $1\in\Q_p\cong\Lie(\Gal(K_\infty/K))$. (See for example \cite[Lemma 2.1.8]{Pan20} for the boundedness.) Consider the subspace $E_0(W^K)$ of $W^K$ annihilated by some power of $\nabla$. Equivalently, this is the generalized eigenspace $E_0(W^K)$ associated to the eigenvalue $0$. It follows from our discussion in \ref{HT0l} that there is a natural exact sequence
\[0\to W^K_{k,-k}\to E_0(W^K)\to W^K_{0,0}\to 0\]
and the induced actions of $\nabla$ on $W^K_{k,-k}$ and $W^K_{0,0}$ are trivial. In particular, if we apply $\nabla$ to $E_0(W^K)$, it induces a $K$-linear map
\[N_W^K:W^K_{0,0}\to W^K_{k,-k}\cong W^K_{0,-k}(k).\]
Since $W_{0,0}=W^K_{0,0}\widehat{\otimes}_{K} C$ and $W_{k,-k}=W^K_{k,-k}\widehat{\otimes}_{K} C$, we can $C$-linearly extend $N_W^K$ to a map
\[N_W:W_{0,0}\to W_{k,-k}\cong W_{0,-k}(k).\]
It follows from Proposition \ref{FonSen} that $N_W$ does not depend on the choice of $K$.  Since  $\nabla$ commutes with the action of $G_{K}$ on $W^K$, we conclude that $N_W$ also commutes with $G_{K}$.
\end{para}

\begin{defn} \label{defnNW}
The $G_{K}$-equivariant, $C$-linear map 
\[N_W:W_{0,0}\to W_{0,-k}(k)\]
 is called the Fontaine operator associated to $W$. 
%We also define $\tilde{N}_W\in\End(W)$ as the composite map 
%\[W\to W/tW=W_0=W_{0,0}\oplus W_{0,-k}\xrightarrow{\mathrm{pr}_0} W_{0,0}\xrightarrow{N_W} W_{k,-k}\subseteq W_k=t^{k}W\subseteq W\]
%where $\mathrm{pr}_0$ denotes the projection of $W_0$ onto $W_{0,0}$. Clearly $\tilde{N}_W$ is nilpotent as  $\tilde{N}_W^2=0$.
\end{defn}

\begin{rem}
It is clear from the definition that $N_W$ is functorial in $W$. More precisely, suppose $Z$ is another flat Banach $B_{\dR}^+/(t^{k+1})$-module equipped with a continuous semilinear action of $G_{K}$ and  $Z$ is Hodge-Tate of weights $0,k$. Let $f:W\to Z$ be a $C$-linear $G_{K}$-equivariant continuous map. Then $f$ intertwines $N_W$ and $N_Z$.
\end{rem}

\begin{para}[The connection with Fontaine's work] \label{relFon}
We explain the meaning of $N_W$ when $W$ comes from a $2$-dimensional $p$-adic Galois representation of $G_{\Q_p}$. Let $E$ be a finite extension of $\Q_p$ and $V$ a $2$-dimensional $E$-vector space equipped with a continuous action of $G_{\Q_p}$. We assume that 
\begin{itemize}
\item $V$ is Hodge-Tate of weights $0,k$ for some positive integer $k$. 
\end{itemize}
In this case, it simply means that 
\[(V\otimes_{\Q_p}C)^{G_{\Q_p}}\neq 0~~~~~~~~~~~\mbox{ and } ~~~~~~~~~\left(V\otimes_{\Q_p}C(k)\right)^{G_{\Q_p}}\neq 0.\]
Let
\[W:=V\otimes_{\Q_p}B_{\dR}^+/(t^{k+1}).\]
This is a finite free $B_{\dR}^+/(t^{k+1})$-module. Choose a lattice $V^o$ of $V$. We get a $\Q_p$-Banach space structure on  $W$  by  requiring the image of $V^o\otimes_{\Z_p} A_{\inf}\to V\otimes_{\Q_p}B_{\dR}^+/(t^{k+1})$ to be the unit ball. The $p$-adic topology on $W$ defined in this way does not depend on the choice of $V^o$. Note that
\[W_0=W/tW=V\otimes_{\Q_p}C.\]
It follows from our assumption that there is a natural decomposition 
\[W_0=C\otimes_{\Q_p}W^{\Q_p}_{0,0}\oplus C\otimes_{\Q_p}W^{\Q_p}_{0,-k}\]
where $W^{\Q_p}_{0,j}\subseteq W\otimes_{\Q_p} C$ denotes the (one-dimensional) $E$-subspace on which $G_{\Q_p}$ acts via $\tilde\varepsilon_p^{j}$ for $j=0,-k$, where $\tilde\varepsilon_p:G_{\Q_p}\to\Z_p^\times$ is a twist of $\varepsilon_p$ introduced in \ref{HTsetup}. Hence $\dim_E E_0(W^{\Q_p})=2$.

Comparing with Fontaine's work, we remark that our $E_0(W^{\Q_p})$ is $D_{\mathrm{pdR}}(V)$ in \cite[\S 4]{Fo04}and our $\nabla|_{E_0(W^{\Q_p})}$ agrees with his $-\nu$. Fontaine also proved the following result.
\end{para}

\begin{thm} \label{FondR}
Let  $V$ be a two-dimensional continuous representation of $G_{\Q_p}$ over a finite extension of $\Q_p$ with Hodge-Tate weights $0,k>0$. Let $W=V\otimes_{\Q_p}B_{\dR}^+/(t^{k+1})$. Then $N_W=0$ if and only $V$ is a de Rham representation.
\end{thm}

\begin{proof}
Since this result will play an important role later, we sketch a proof here. Recall that in this case, $V$ is de Rham if and only if $\dim_E D_{\dR}(V)=\dim_E V=2$, where $D_\dR(V)=(V\otimes_{\Q_p}B_{\dR}^+)^{G_{\Q_p}}$. By our assumption, $\left(V\otimes_{\Q_p}C(l)\right)^{G_{\Q_p}}=0$ when $l\geq k+1$. Hence 
\[(V\otimes_{\Q_p}B_{\dR}^+)^{G_{\Q_p}}\cap V\otimes_{\Q_p}t^{k+1}B_{\dR}^+=\{0\}.\]
In particular, the natural map 
\[(V\otimes_{\Q_p}B_{\dR}^+)^{G_{\Q_p}}\to V\otimes_{\Q_p}B_{\dR}^+/(t^{k+1})=W\]
is injective.  Hence it induces an inclusion
\[ D_{\dR}(V)\subseteq E^0(W^{\Q_p}).\]
If $\dim_E D_{\dR}(V)=2$, a consideration of the dimensions shows that this is in fact an equality $D_{\dR}(V)= E_0(W^{\Q_p})$. Since $G_{\Q_p}$ acts trivially on $D_{\dR}(V)$, it follows from the definition that $N_W=0$. 

Now suppose $N_W=0$.  Then $E_0(W^{\Q_p})$ is fixed by $G_{\Q_p}$, i.e. $W^{G_{\Q_p}}=E_0(W^{\Q_p})$.
The kernel of $B_{\dR}^+/(t^l)\otimes_{\Q_p}V\to W$ for $l\geq {k+1}$ is filtered by $C(i),i>0$ which has no $H^1_{\cont}(G_{\Q_p},\bullet)$. Hence the natural map $(B_{\dR}^+/(t^l)\otimes_{\Q_p}V)^{G_{\Q_p}}\to W^{G_{\Q_p}}$ is an isomorphism when $l\geq k+1$. Passing to the limit over $l$, we conclude that $\dim_E D_\dR(V)=2$.
\end{proof}

\begin{para}[Generalizations to LB-spaces] \label{LBsetup}
We briefly discuss how to generalize the construction of the Fontaine operator to Hausdorff LB-$B_{\dR}^+/(t^{k+1})$-modules. Let 
\[W^0\subseteq W^1\subseteq W^2\subseteq\cdots\]
be an increasing sequence of  Banach $B_{\dR}^+/(t^{k+1})$-modules $W^j$ equipped with continuous semilinear actions of $G_{K}$ whose transition maps are continuous, $G_{K}$-equivariant and $B_{\dR}^+/(t^{k+1})$-linear.  Denote the direct limit $\displaystyle \varinjlim_j W^j$ by $W$ and assume that this is a Hausdorff LB-space. We also denote the quotient $W/tW$ by $W_0$.
Assume  that
\begin{itemize}
\item $W$ is flat over $B_{\dR}^+/(t^{k+1})$. This implies that $W_0$ is also a Hausdorff LB-space by the same argument as in \ref{BBdrtkmod}.
\end{itemize}
We remark that we don't require each $W^j$ is flat over $B_{\dR}^+/(t^{k+1})$. As before, we denote by $W^K$ (resp. $W_0^K$) the subspace of $G_{K_\infty}$-invariant, $G_K$-analytic vectors in $W$ (resp. $W_0$). This definition makes sense in view of the discussion in \ref{roHLB}. The $t$-adic filtration on $W$ induces a decreasing filtration  $\Fil^\bullet$ on $W^K$.
 
We say $W$ is Hodge-Tate of weights $0,k$ if 
\begin{itemize}
\item $W_0$ is Hodge-Tate of weights $0,k$ in the sense that the natural map
\[W_0^K\widehat\otimes_K C\to W_0\]
is an isomorphism and the action of the Sen operator $1\in\Q_p=\Lie(\Gal(K_\infty/K))$ on $W_0^K$ is semi-simple with eigenvalues $0,-k$. Here $W_0^K\widehat\otimes_K C$ is understood as $\displaystyle \varinjlim_{j=0} Z^j \widehat\otimes_K C$ by writing  the LB-space $W_0^K=\bigcup_{j=0}^{+\infty}Z^j$ as an increasing sequence of $K$-Banach subspaces of $W_0^K$.
\end{itemize}
Following our notation in \ref{HT0l}, we have an eigenspace decomposition 
\[W_0^K=W_{0,0}^{K}\oplus W_{0,-k}^K\]
and an induced decomposition of $W_0$:
\[W_0=W_{0,0}\oplus W_{0,-k}.\]
Similarly, let $W_i=t^iW/t^{i+1}W$. Then $W_i\cong W_0(i)$ and there is a natural decomposition 
\[W_i=W_{i,0}\oplus W_{i,-k}.\]
\end{para}

\begin{prop} \label{Kcmgr}
Let $W$ be as in \ref{LBsetup} and assume that $W$ is Hodge-Tate of weights $0,k$. Then the natural maps
\[\gr^i W^K\to W^K_i\]
are isomorphisms for $i\geq 0$. Moreover for any finite extension $L$ of $K$ in $C$, we have
\[W^L=W^K\otimes_K L.\]
\end{prop}

\begin{proof}
The key observation is as follows. By our assumption, we can write 
\[W_0=\bigcup_{j=0}^{+\infty}Y^j\]
as an increasing sequence of $C$-Banach Hodge-Tate representations of $G_K$ of weights $0,k$.  Hence 
\[W_i=\bigcup_{j=0}^{+\infty}Y^j(i)=\bigcup_{j=0}^{+\infty}\gr^i W^j.\]
By Proposition \ref{1.1.10}, this implies that for any $j\geq 0$, there exist $j',j''\geq j$ such that 
\[\gr^i W^j\subseteq Y^{j'}(i)\subseteq \gr^i W^{j''}.\]

The rest of the argument is almost the same as  the proof of Proposition \ref{FonSen} and Lemma \ref{HTK}.  For example, using the vanishing of $H^1_\cont(G_{K_\infty}, Y^{j'}(i))$ in the proof of the first part of Lemma \ref{surj^K}, we see that $H^1_\cont(G_{K_\infty}, W_i)=0$ and $W^{G_{K_\infty}}\to W_0^{G_{K_\infty}}$ is surjective. We omit the details here.
\end{proof}

\begin{para} \label{NWLB}
Now we can carry out the same construction in \ref{FonOpsetup}. Let $W$ be as in \ref{LBsetup} and assume that $W$ is Hodge-Tate of weights $0,k$. Denote by $\nabla$ the action of $1\in\Q_p=\Lie(\Gal(K_\infty/K))$ on $W^K$ and denote by $E_0(W^K)$ the generalized eigenspace associated to the eigenvalue $0$. There is a natural exact sequence
\[0\to W^K_{k,-k}\to E_0(W^K)\to W^K_{0,0}\to 0.\]
Then $\nabla$ induces a map $W^K_{0,0}\to W^K_{k,-k}$. We can $C$-linearly extend it to a map
\[N_W:W_{0,0}\to W_{k,-k}=W_{0,-k}(k)\]
which is called the Fontaine operator. Clearly $N_W$ is functorial in $W$.

Our last result extends Proposition \ref{Bmor} to this setup. 
\end{para}

\begin{prop} \label{LBfXY}
Let $f:X\to Y$ be a continuous $B_{\dR}^+/(t^k)$-linear, $G_K$-equivariant maps between Hausdorff  flat LB $B^{+}_\dR/(t^k)$-modules equipped with continuous semilinear actions of $G_K$ such that $\varphi^K_{X,0}, \varphi^K_{Y,0}$ are isomorphisms. Suppose that
\begin{itemize}
\item  $X=\bigcup_n X_n$ and $Y=\bigcup_n Y_n$ are unions of increasing sequences of $G_K$-stable $B^{+}_\dR/(t^k)$-Banach modules;
\item $X_n$ and $Y_n$ are Hodge-Tate of weights $0,k$ for $n\geq 0$;
\item $f(X_n)\subseteq Y_n$ and $f|_{X_n}:X_n\to Y_n$ is strict; hence $\ker f=\bigcup_n \ker f|_{X_n}$ and $\coker f=\bigcup_n Y_n/f(X_n)$ are  natural $B^{+}_\dR/(t^k)$-LB modules.
\item $\coker f$ is Hausdorff, (while $\ker f$ is automatically Hausdorff;)
\item $f$ is strict with respect to  the $t$-adic filtrations on $X,Y$.
\end{itemize}
Then $\ker f$ and $\coker f$ are also Hodge-Tate of weights $0,k$ and flat over $B^{+}_\dR/(t^k)$. Moreover, $\varphi^K_{\ker f,0}, \varphi^K_{\coker f,0}$ are isomorphisms and the natural sequence
\[
0\to (\ker f)^K\to X^K\to Y^K\to (\coker f)^K\to 0
\]
is exact.
\end{prop}

\begin{proof}
Same proof as Proposition \ref{Bmor}.
\end{proof}

\begin{cor} \label{Ncokerf}
Same setup as in Proposition \ref{LBfXY}. Let $N_W:W_{0,0}\to W_{k,-k}$ denote the Fontaine operator of $W$ for $W=\ker f, X,Y, \coker f$. Then the following diagram commutes.
\[
\begin{tikzcd}
0 \arrow[r] & (\ker f)_{0,0}  \arrow[d,"N_{\ker f}"]\arrow[r] & X_{0,0} \arrow[r] \arrow[d, "N_X"] & Y_{0,0}  \arrow[r] \arrow[d,"N_Y"]  & (\coker f)_{0,0}  \arrow[d,"N_{\coker f}"]\arrow[r]& 0 \\
0 \arrow[r] & (\ker f)_{k,-k} \arrow[r] & X_{k,-k} \arrow[r] & Y_{k,-k}  \arrow[r]   & (\coker f)_{k,-k} \arrow[r]& 0
\end{tikzcd}.
\]
\end{cor}

\begin{proof}
Direct consequences of Proposition \ref{LBfXY} and the construction of the Fontaine operator.
\end{proof}

\subsection{de Rham period sheaves} \label{dRps}
\begin{para}
In this subsection, we apply our discussion in the previous subsection to the infinite level modular curves. To do this, we need the construction of de Rham period sheaves, i.e. a relative construction of $B_{\dR}^+$, cf. \ref{CBdR+}. Our reference here is \cite[\S 6]{Sch13}.
% \cite[\S 2.2]{DLLZ2}. 

First we have the sheaf $\A_{\inf,\mathcal{X}_{K^p}}$ on $\mathcal{X}_{K^p}$ which assigns an 
affinoid perfectoid open subset $V_\infty=\Spa(B,B^+)\subseteq\mathcal{X}_{K^p}$ to $W(B^{\flat+})$,  the ring of Witt vectors of $\displaystyle B^{\flat+}:=\varprojlim_{x\mapsto x^p} B^+/p$, the tilt of $B^+$.
There is the usual surjective ring homomorphism
\[\theta: \A_{\inf,\mathcal{X}_{K^p}}\to\cO^+_{\mathcal{X}_{K^p}}.\]
By abuse of notation, we also denote by $\theta: \A_{\inf,\mathcal{X}_{K^p}}[\frac{1}{p}]\to \cO_{\mathcal{X}_{K^p}}$ the rational version. The (positive) de Rham ring $\B_{\dR,\mathcal{X}_{K^p}}^+$ is defined as the $\ker(\theta)$-adic completion of $\A_{\inf,\mathcal{X}_{K^p}}[\frac{1}{p}]$. The natural inclusion $C\subseteq \cO_{\mathcal{X}_{K^p}}$ defines a map $B_{\dR}^+\to \B_{\dR,\mathcal{X}_{K^p}}^+$. As explained in \cite[Lemma 6.3]{Sch13}, $t\in B_{\dR}^+$ is also a generator of the kernel of $\B_{\dR,\mathcal{X}_{K^p}}\to \cO_{\mathcal{X}_{K^p}}$. We define a decreasing filtration on $\B_{\dR,\mathcal{X}_{K^p}}^+$ by $\Fil^i\B_{\dR,\mathcal{X}_{K^p}}^+=t^i\B_{\dR,\mathcal{X}_{K^p}}^+,i\geq 0$. Then $\gr^i\B_{\dR,\mathcal{X}_{K^p}}^+\cong \cO_{\mathcal{X}_{K^p}}(i),i\geq 0$. Set
\[\B_{\dR}^+:={\pi_{\HT}}_*\B_{\dR,\mathcal{X}_{K^p}}^+.\]
Clearly $\gr^i \B_{\dR}^+\cong\cO_{K^p}(i)$.

We will only consider truncated sheaves. More precisely, let $k\geq 0$ be an integer. Set
\[\B_{\dR,k}^+:=\B_{\dR}^+/(t^k).\]
Let $U\in\mathfrak{B}$ be  an affinoid open subset of $\Fl$ with $V_\infty:=\pi_{\HT}^{-1}(U)=\Spa(B,B^+)$. Then 
$\B_{\dR}^+/(t^k)(U)$ is naturally a flat $B_{\dR}^+/(t^k)$-Banach module with the unit open ball given by the image of $\A_{\inf,\mathcal{X}_{K^p}}(V_\infty)\to \B_{\dR,\mathcal{X}_{K^p}}^+(V_\infty)/(t^k)$. It is clear from the construction that $\GL_2(\Q_p)$ acts on $\B_{\dR,k}^+$. We denote by
\[\B_{\dR,k}^{+,\la}\subset \B_{\dR,k}^+\]  
the subsheaf of $\GL_2(\Q_p)$-locally analytic sections. There is a natural decreasing filtration on $\B_{\dR,k}^{+,\la}$ induced from the $t$-adic filtration on $\B_{\dR,k}^+$. By the same argument in \ref{stpHT}, we see that $G_K$ acts naturally on $\B_{\dR,k}^+(U)$ for some finite extension $K$ of $\Q_p$ in $C$.

The Lie algebra $\Lie(\GL_2(\Q_p))=\mathfrak{gl}_2(\Q_p)$ acts naturally on $\B_{\dR,k}^{+,\la}$. Let $Z:=Z(U(\mathfrak{gl}_2(\Q_p)))$ the centre of the universal enveloping algebra of $\mathfrak{gl}_2(\Q_p)$ over $\Q_p$. Given an infinitesimal character $\tilde\chi:Z\to\Q_p\subseteq B_{\dR}^+$, \footnote{We can allow infinitesimal characters valued in a finite extension of $\Q_p$. But since only integral weights are considered in this paper, we restrict ourselves to characters valued in $\Q_p$ here.} we denote by 
\[\B_{\dR,k}^{+,\la,\tilde\chi}\subset \B_{\dR,k}^{+,\la}\] 
the $\tilde\chi$-isotypic part. Then there is also an induced decreasing filtration on $\B_{\dR,k}^{+,\la,\tilde\chi}$. On the other hand,  we  denote by $\cO^{\la,\tilde\chi}_{K^p}$  the $\tilde\chi$-isotypic part of $\cO^{\la}_{K^p}$. There are natural maps
\[\gr^i \B_{\dR,k}^{+,\la} \to \cO^{\la}_{K^p}(i),\]
\[\gr^i\B_{\dR,k}^{+,\la,\tilde\chi}\to \cO^{\la,\tilde\chi}_{K^p}(i)\]
for $\gr^{i}\B_{\dR,k}^{+,\la}=\cO^{\la}_{K^p}(i),i=0,\cdots,k-1$.
\end{para}

\begin{lem} \label{BdRcmgr}
For $i=0,\cdots,k-1$, the natural maps
\begin{enumerate}
\item $\gr^i \B_{\dR,k}^{+,\la} \to \cO^{\la}_{K^p}(i)$,
\item $\gr^i\B_{\dR,k}^{+,\la,\tilde\chi}\to \cO^{\la,\tilde\chi}_{K^p}(i)$
\end{enumerate}
are isomorphisms. In particular, $\B_{\dR,k}^{+,\la}$ and $\B_{\dR,k}^{+,\la,\tilde\chi}$ are flat over $B_{\dR}^+/(t^k)$.
\end{lem}

\begin{proof} 
Both are local statements on $\Fl$. Fix $U\in\mathfrak{B}$ which is $G_0=1+p^m M_2(\Z_p)$-stable. For the first part, since $\B_{\dR,k}^+(U)$ is filtered by $\cO_{K^p}(U)(i)$, it suffices to show that $\cO_{K^p}(U)$ is $\mathfrak{LA}$-acyclic respect to the action of $G_0$. This is \cite[Proposition 4.3.15]{Pan20}.

For the second part, by the same argument, it is enough to show that 
\[\Ext^i_{Z}(\tilde\chi, \cO^{\la}_{K^p}(U))=\Ext^i_Z(\Q_p,\cO^{\la}_{K^p}(U)(\tilde\chi^{-1}))=0,i\geq 1.\]
Recall that there is the horizontal action $\theta_\kh$ of $\kh=\begin{pmatrix} * & 0\\ 0 &*\end{pmatrix}\subseteq\mathfrak{gl}_2(C)$ on $\cO^{\la}_{K^p}(U)$, cf. \ref{brr}. Moreover it follows from \cite[Corollary 4.2.8]{Pan20}  that  the action of $Z$ on $\cO^{\la}_{K^p}(U)$ factors through  $S(\kh)$, the symmetric algebra of $\kh$ (over $C$). In fact, by Harish-Chandra's isomorphism, we can choose an isomorphism $S(\kh)\cong C[x_1,x_2]$ and identify $Z$ with the subalgebra $\Q_p[x_1,x_2^2]$. Hence $\Ext^i_{Z}(\Q_p,\cdot)$ (resp. $\Ext^i_{S(\kh)}(C,\cdot)$) can be computed by the Koszul complex with respect to $x_1,x_2^2$ (resp. $x_1,x_2$).

Choosing a character $\chi:S(\kh)\to C$ extending $\tilde\chi$. By \cite[Lemma 5.1.2.(1)]{Pan20},
\[\Ext^i_{S(\kh)}(\chi, \cO^{\la}_{K^p}(U))=\Ext^i_{S(\kh)}(C, \cO^{\la}_{K^p}(U)(\chi^{-1}))=0,i\geq 1.\]
It is easy to see that this implies the vanishing of $\Ext^i_Z(\Q_p,\cO^{\la}_{K^p}(U)(\tilde\chi^{-1})),i\geq 1$ using the Koszul complexes.
\end{proof}

\begin{para} \label{tilchik}
By Harish-Chandra's isomorphism, $Z\otimes_{\Q_p}C=Z(U(\mathfrak{gl}_2(C)))\cong S(\kh)^W$, where $W$ denotes the Weyl group of $\mathfrak{gl}_2$ and acts on $S(\kh)$ via the dot action: $w\cdot \mu=w(\mu+\rho)-\rho,\mu\in\kh^*,w\in W$. Here $\rho$ denotes the half sum of the positive roots as usual. In particular, $\tilde\chi$ is identified with a $W$-orbit of $\kh^*$.

Let $k$ be a positive integer. Consider 
\[\tilde\chi_k=\{(0,1-k), (-k,1)\}\subseteq \kh^*,\]
the infinitesimal character  of the $(k-1)$-th symmetric power of the dual of the standard representation.
It follows from the relation between $\theta_\kh$ and the infinitesimal character  \cite[Corollary 4.2.8]{Pan20} that on $\cO^{\la,\tilde\chi_k}_{K^p}$, 
\[\left(\theta_\kh(\begin{pmatrix} a & 0\\ 0 & d\end{pmatrix})-(1-k)a \right)\left(\theta_\kh(\begin{pmatrix} a & 0\\ 0 & d\end{pmatrix})-(a-kd) \right)=0.\]
Hence we  have a natural decomposition
\[\cO^{\la,\tilde\chi_k}_{K^p}=\cO^{\la,(1-k,0)}_{K^p}\oplus \cO^{\la,(1,-k)}_{K^p}.\]
By Proposition \ref{GnanHT}, $\cO^{\la,(1-k,0)}_{K^p}(U)$ (resp. $ \cO^{\la,(1,-k)}_{K^p}(U)$) is  Hodge-Tate of weight $0$ (resp. $k$)  for any $U\in\Fl$.   This shows that $\cO^{\la,\tilde\chi_k}_{K^p}(U)$ is Hodge-Tate of weights $0,k$.
\end{para}

\begin{para} \label{noBdRK}
Let $U\in\mathfrak{B}$. Then  $G_K$ acts naturally on $\B_{\dR,k+1}^{+,\la,\tilde\chi_k}(U)$ for some finite extension $K$ of $\Q_p$ in $C$. We just verified that $\B_{\dR,k+1}^{+,\la,\tilde\chi_k}(U)$  satisfies all the assumptions in \ref{LBsetup}. Therefore  we can carry out the construction in \ref{NWLB} and get the Fontaine operator 
\[\cO^{\la,(1-k,0)}_{K^p}(U) \to \cO^{\la,(1,-k)}_{K^p}(U)(k).\]
By the functorial property of the Fontaine operator, this actually defines a map of sheaves
\[\cO^{\la,(1-k,0)}_{K^p} \to \cO^{\la,(1,-k)}_{K^p}(k).\]
\end{para}

\begin{defn}
Let $k$ be a positive integer. We denote this map by 
\[N_k: \cO^{\la,(1-k,0)}_{K^p} \to \cO^{\la,(1,-k)}_{K^p}(k).\] 
%and by 
%\[\tilde{N}_k\in \End(\B_{\dR,k+1}^{+,\la,\tilde\chi_k})\]
%the composite map $\B_{\dR,k+1}^{+,\la,\tilde\chi_k}\to\cO^{\la,\tilde\chi_k}_{K^p}\to \cO^{\la,(1-k,0)}_{K^p}\xrightarrow{N_k} \cO^{\la,(1,-k)}_{K^p}(k)\subseteq \B_{\dR,k+1}^{+,\la,\tilde\chi_k}$, cf. Definition \ref{defnNW}.
\end{defn}

Now we can state the main theorem of this section. Recall that we have introduced an intertwining operator $I_{k-1}:\cO^{\la,(1-k,0)}_{K^p} \to \cO^{\la,(1,-k)}_{K^p}(k)$ in  Definition \ref{I}.
\begin{thm} \label{MT}
Let $k$ be a positive integer. Then $N_k=c_kI_{k-1}$ for some $c_k\in\Q_p^\times$. 
\end{thm}

The proof of this theorem will be given in the rest of this section.

\begin{rem}It's possible to show that $c_k\in\Q^\times$ but we don't need this here. See Remark \ref{Qstr} below.\end{rem}

\begin{para}
Theorem \ref{MT} also implies a similar result on the cohomology level. We need some constructions from subsection \ref{LBH}. Let $U_1,U_2,U_{12}$ be the open subsets of $\Fl$ introduced in \ref{Cech}. Fix an integer $k$. 
\end{para}

\begin{lem} \label{Bdrce}
The natural map $\B_{\dR,k}^+(U_1)\to \B_{\dR,k}^+(U_{12})$ is a closed embedding and is strict with respect to the $t$-adic filtrations. 
\end{lem}

\begin{proof}
Since $t^i\B_{\dR,k}^+/t^{i+1}\B_{\dR,k}^+\cong\cO_{K^p}(i)$, both claims follow directly from Lemma \ref{rescle}. 
\end{proof}

\begin{cor}
The \v{C}ech complex $\B_{\dR,k}^{+}(U_1)\oplus B_{\dR,k}^{+}(U_2)\to \B_{\dR,k}^{+}(U_{12})$ is strict with respect to the topology on $\B_{\dR,k}^{+}(U_1)\oplus B_{\dR,k}^{+}(U_2),\B_{\dR,k}^{+}(U_{12})$ and is strict with respect to the $t$-adic filtrations.
\end{cor}

\begin{proof}
An application of Lemma \ref{keyweyl} and Lemma \ref{Bdrce}. We note that the same argument of Lemma \ref{keyweyl} also proves the strictness for the $t$-adic filtrations.
\end{proof}

Set $G_0=1+p^2M_2(\Z_p)$ and $G_n=G_0^{p^n}$. Let $\tilde\chi:Z\to\Q_p$ be a character.

\begin{cor} \label{strCeLB}
The natural maps 
%\[\B_{\dR,k}^{+}(U_1)^{G_n-\an}\to \B_{\dR,k}^{+}(U_{12})^{G_n-\an},\]
\[\B_{\dR,k}^{+}(U_1)^{G_n-\an,\tilde\chi}\to \B_{\dR,k}^{+}(U_{12})^{G_n-\an,\tilde\chi},\]
%\[\B_{\dR,k}^{+,\la}(U_1)\to \B_{\dR,k}^{+,\la}(U_{12}),\]
\[\B_{\dR,k}^{+,\la,\tilde\chi}(U_1)\to \B_{\dR,k}^{+,\la,\tilde\chi}(U_{12})\]
 are closed embeddings. The second map is also strict with respect to the $t$-adic filtrations.
 Hence by Lemma \ref{keyweyl}, the following \v{C}ech complexes are strict
\[\B_{\dR,k}^{+}(U_1)^{G_n-\an,\tilde\chi}\oplus B_{\dR,k}^{+}(U_2)^{G_n-\an,\tilde\chi}\to \B_{\dR,k}^{+}(U_{12})^{G_n-\an,\tilde\chi},\]
\[\B_{\dR,k}^{+,\la,\tilde\chi}(U_1)\oplus B_{\dR,k}^{+,\la,\tilde\chi}(U_2)\to \B_{\dR,k}^{+,\la,\tilde\chi}(U_{12}),\]
and the second complex is strict with respect to the $t$-adic filtrations.
\end{cor}

\begin{proof}
All of the claims follow  from Lemma \ref{Bdrce} except for the strictness of $\B_{\dR,k}^{+,\la,\tilde\chi}(U_1)\to \B_{\dR,k}^{+,\la,\tilde\chi}(U_{12})$ for the $t$-adic filtrations. To see this, we use Lemma \ref{BdRcmgr} which says that $\gr^i \B_{\dR,k}^{+,\la,\tilde\chi}=\cO^{\la,\tilde\chi}_{K^p}(i)$.
\end{proof}

\begin{defn}
Let $\mathcal{F}$ be one of $\B_{\dR,k}^+,\B_{\dR,k}^{+,G_n-\an,\tilde\chi},\B_{\dR,k}^{+,\la,\tilde\chi}$. We define
$\check{H}^\bullet(\mathcal{F})$ as  the cohomology of the \v{C}ech complex $\mathcal{F}(U_1)\oplus \mathcal{F}(U_2)\to \mathcal{F}(U_{12})$, where $\B_{\dR,k}^{+,G_n-\an,\tilde\chi}(U_i)$ is understood as $\B_{\dR,k}^{+}(U_i)^{G_n-\an,\tilde\chi}$. These are natural  $B_{\dR}^+/(t^k)$-Banach modules when $\mathcal{F}=\B_{\dR,k}^+,\B_{\dR,k}^{+,G_n-\an,\tilde\chi}$. It follows that
\[\check{H}^\bullet(\B_{\dR,k}^{+,\la,\tilde\chi})=\varinjlim_n \check{H}^\bullet(\B_{\dR,k}^{+,G_n-\an,\tilde\chi})\]
which defines a $B_{\dR}^+/(t^k)$-LB module structure on $\check{H}^\bullet(\B_{\dR,k}^{+,\la,\tilde\chi})$.
\end{defn}

\begin{prop}
For $\mathcal{F}=\B_{\dR,k}^+,\B_{\dR,k}^{+,\la,\tilde\chi}$, there are natural isomorphisms 
\[\check{H}^\bullet(\mathcal{F})\cong H^{\bullet}(\Fl,\mathcal{F}).\]
These cohomology groups are flat over $B_{\dR}^+/(t^k)$ and we equip them with the $t$-adic filtrations. Then
\[\gr^i H^{\bullet}(\Fl,\mathcal{F})\cong H^{\bullet}(\Fl,\gr^i\mathcal{F}).\]
Moreover, $H^{\bullet}(\Fl,\B_{\dR,k}^{+,\la,\tilde\chi})$ is Hausdorff.
\end{prop}

\begin{proof}
For the first isomorphism, it suffices to prove that $\check{H}^\bullet(\gr^i\mathcal{F})\cong H^{\bullet}(\Fl,\gr^i\mathcal{F})$, where $\gr^i\mathcal{F}=\cO_{K^p}(i)$ or $\cO^{\la,\tilde\chi}_{K^p}(i)$. When $\gr^i\mathcal{F}=\cO_{K^p}(i)$, this was shown in the first paragraph of the proof of \cite[Theorem 4.4.6]{Pan20}. When $\gr^i\mathcal{F}=\cO^{\la,\tilde\chi}_{K^p}(i)$, this was explained in section \ref{Celachi}. Since $\mathcal{F}=\B_{\dR,k}^+,\B_{\dR,k}^{+,\la,\tilde\chi}$ is flat over $B_{\dR}^+/(t^k)$, the strictness of $\mathcal{F}(U_1)\oplus \mathcal{F}(U_2)\to \mathcal{F}(U_{12})$ for $t$-adic filtrations implies the flatness of $\check{H}^\bullet(\mathcal{F})$ over $B_{\dR}^+/(t^k)$ and $\gr^i H^{\bullet}(\Fl,\mathcal{F})\cong H^{\bullet}(\Fl,\gr^i\mathcal{F})$. It remains to show that $H^{\bullet}(\Fl,\B_{\dR,k}^{+,\la,\tilde\chi})$ is Hausdorff. This is clear as $H^{\bullet}(\Fl,\gr^i \B_{\dR,k}^{+,\la,\tilde\chi})\cong H^\bullet(\Fl,\cO^{\la,\tilde\chi}_{K^p}(i))$ is filtered by $H^\bullet (\Fl,\cO^{\la,\chi}_{K^p}(i))$ for some characters $\chi:\kh\to\Q_p$, which is Hausdorff by Proposition \ref{Haus}.
\end{proof}

\begin{para}
Now set $\tilde\chi=\tilde\chi_k$ defined in section \ref{tilchik}. Let $K$ be a finite extension of $\Q_p$ such that $G_K$ acts on $\B_{\dR,k+1}^+(U_1),\B_{\dR,k+1}^+(U_{12})$. As explained in section \ref{noBdRK}, $\B_{\dR,k+1}^{+,\la,\tilde\chi_k}(U_{?})$ satisfies all the assumptions in \ref{LBsetup} for $?=1,2,12$. Hence by applying Proposition \ref{LBfXY} to $\B_{\dR,k+1}^{+,\la,\tilde\chi_k}(U_1)\oplus B_{\dR,k+1}^{+,\la,\tilde\chi_k}(U_2)\to \B_{\dR,k+1}^{+,\la,\tilde\chi_k}(U_{12})$, we have the following result.
\end{para}

\begin{cor} \label{I1Fono}
$H^1(\Fl,\B_{\dR,k+1}^{+,\la,\tilde\chi_k})$ is Hodge-Tate of weights $0,k$ and satisfies all of the assumptions in \ref{LBsetup}. The Fontaine operator  in this case (cf. section \ref{NWLB}) agrees with
\[I^1_{k-1}: H^1(\Fl,\cO^{\la,(1-k,0)}_{K^p})\to H^1(\Fl,\cO^{\la,(1,-k)}_{K^p}(k))\]
up to $\Q_p^\times$. Recall that $I^1_{k-1}=H^1(I_{k-1})$, cf. Definition \ref{I^1k}.
\end{cor}

\begin{proof}
A direct consequence of Theorem \ref{MT} and Corollary \ref{Ncokerf}.
\end{proof}

\begin{para}
It will be useful to relate $H^1(\Fl,\B_{\dR,k+1}^{+,\la,\tilde\chi_k})$ with $H^1(\Fl,\B_{\dR,k+1}^{+})^{\la,\tilde\chi_k}$. Since $\B_{\dR,k+1}^{+}$ is filtered by $\cO_{K^p}(i)$ which is $\mathfrak{LA}$-acyclic as a representation of $\GL_2(\Q_p)$ by \cite[Proposition 4.3.15]{Pan20}. Hence by the same proof of \cite[Theorem 4.4.6]{Pan20}, there are natural isomorphisms
\[H^i(\Fl,\B_{\dR,k+1}^{+})^{\la}\cong H^i(\Fl,\B_{\dR,k+1}^{+,\la}).\]
\end{para}

\begin{prop} \label{compinfch}
Let $k$ be a positive integer. If $k\geq 2$, then
\[H^1(\Fl,\B_{\dR,k+1}^{+,\la,\tilde\chi_k})\cong H^1(\Fl,\B_{\dR,k+1}^{+})^{\la,\tilde\chi_k}. \]
When $k=1$, there is a natural exact sequence
\[0\to \Ext^1_{Z}(\tilde\chi_1,H^0(\B_{\dR,2}^{+})^{\la}) \to 
H^1(\B_{\dR,2}^{+,\la,\tilde\chi_1})\to
H^1(\B_{\dR,2}^{+})^{\la,\tilde\chi_1}\to
\Ext^2_{Z}(\tilde\chi_1,H^0(\B_{\dR,2}^{+})^{\la}), \]
where all of the sheaf cohomology groups are computed on $\Fl$.
\end{prop}

\begin{proof}
Since $\B_{\dR,k}^{+,\la}$ is filtered by $\cO^{\la}_{K^p}(i)$, it follows from the proof of Lemma \ref{BdRcmgr} that 
\[\Ext^i_Z(\tilde\chi, \B_{\dR,k}^{+,\la}(U))=0\]
for $i\geq 1$ and $U\in\mathfrak{B}$. Hence by exactly the same argument as the proof of  \cite[Corollary 5.1.3]{Pan20}, there is an exact sequence
\[0\to \Ext^1_{Z}(\tilde\chi_k,H^0(\B_{\dR,k+1}^{+})^{\la}) \to 
H^1(\B_{\dR,k+1}^{+,\la,\tilde\chi_k})\to
H^1(\B_{\dR,k+1}^{+})^{\la,\tilde\chi_k}\to
\Ext^2_{Z}(\tilde\chi_k,H^0(\B_{\dR,k+1}^{+})^{\la}). \]
Note that $H^0(\Fl,\B_{\dR,k+1}^{+,\la})$ is filtered by $H^0(\Fl, \cO^{\la}_{K^p})(i)$ on which the action of $\mathfrak{sl}_2(\Q_p)\subseteq \mathfrak{gl}_2(\Q_p)$ is trivial  by its relation with the completed cohomology (See Theorem \ref{BdRcomp} below). Hence  $\Ext^i_{Z}(\tilde\chi_k,H^0(\B_{\dR,k+1}^{+})^{\la})=0$ if $k\geq 2$.
\end{proof}

\subsection{Proof of Theorem \ref{MT} I} \label{PfMTI}
\begin{para}
Recall that in Definition \ref{I}, $I_{k-1}$ was defined as the composite $\bar{d}'^k\circ d^k$. We will also factorize $N_k$ as a composite of two maps and identify each map with $d^k$ or $\bar{d}'^k$. Not surprisingly, the basic idea is that instead of working directly with $\B_{\dR}^{+}$, we replace $\B_{\dR,k+1}^{+}$ by a resolution using the \textit{Poincar\'e lemma sequence}. We first  sketch a proof when $k=1$ in the following discussion. I hope this can help the reader understand the general arguments.

Consider the Poincar\'e lemma sequence on $\mathcal{X}_{K^p}$ (viewed as log pro-\'etale covering of $\mathcal{X}=\mathcal{X}_{K^pK_p}$ for some $K_p$.)
\[0\to \B_{\dR}^+\to \cO\B_{\dR}^+\to \cO\B_{\dR}^+\otimes\Omega_{\mathcal{X}}^1(\mathcal{C})\to 0.\]
See section \ref{setupOBdR} \ref{PlsFe} below for more details. This sequence is strict with respect to the natural decreasing filtrations on each term. Consider the quotient by $\Fil^2$ and note that $\gr^0\cO\B_{\dR}^+=\cO_{\mathcal{X}_{K^p}}$.
\[0\to \B_{\dR,2}^+\to \cO\B_{\dR}^+/\Fil^2 \cO\B_{\dR}^+\to \cO_{\mathcal{X}_{K^p}}\otimes\Omega_{\mathcal{X}}^1(\mathcal{C})\to 0.\]
It is easy to see that this sequence remains exact after restricting to the $\GL_2(\Q_p)$-locally analytic vectors and  taking the $\tilde{\chi}_1$-parts.
\[0\to \B_{\dR,2}^{+,\la,\tilde{\chi}_1}\to (\cO\B_{\dR}^+/\Fil^2 \cO\B_{\dR}^+)^{\la,\tilde{\chi}_1}\xrightarrow{\nabla_{\log}} \cO_{\mathcal{X}_K^p}^{\la,\tilde{\chi}_1}\otimes\Omega_{\mathcal{X}}^1(\mathcal{C})\to 0.\]

Now given a section $s$ of $\cO_{\mathcal{X}_K^p}^{\la,\tilde{\chi}_1}=\B_{\dR,2}^{+,\la,\tilde{\chi}_1}/ \Fil^1$ fixed by $G_K$ for some finite extension $K$ of $\Q_p$, we would like to compute $N_k(s)$. By definition, we need to find a lifting of $s$ in $\B_{\dR,2}^{+,\la,\tilde{\chi}_1,K}$ and compute its image under $\nabla$. It turns out that there exists
\begin{itemize}
\item  $\tilde{s}\in (\cO\B_{\dR}^+/\Fil^2 \cO\B_{\dR}^+)^{\la,\tilde{\chi}_1}$ fixed by $G_K$ such that $\tilde{s}\mod\Fil^1=s$.
\end{itemize}
The lifting $\tilde{s}\notin \B_{\dR,2}^{+,\la,\tilde{\chi}_1,K}$ in general, hence we need a modification. There always exists
\begin{itemize}
\item  $s'\in(\Fil^1\cO\B_{\dR}^+/\Fil^2 \cO\B_{\dR}^+)^{\la,\tilde{\chi}_1,K}$ such that $\nabla_{\log}(s')=\nabla_{\log}(\tilde{s})$.
\end{itemize}
This is essentially because the first graded piece of the Poincar\'e lemma sequence is exact.
\[0\to \cO_{\mathcal{X}_{K^p}}(1)\to \gr^1\cO\B_{\dR}^+\to  \cO_{\mathcal{X}_{K^p}}\otimes\Omega_{\mathcal{X}}^1(\mathcal{C})\to 0.\]
Clearly $\tilde{s}-s'\in \B_{\dR,2}^{+,\la,\tilde{\chi}_1}$ is a lifting of $s$. Moreover since $\tilde{s}$ is fixed by $G_K$,
\[N_k(s)=\nabla(\tilde{s}-s')=-\nabla(s').\]
It remains to compare $\nabla(s')$ with $\bar{d}'^1\circ d^1 (s)$. This is a consequence of the following results.
\begin{enumerate} \label{parta}
\item $\nabla_{\log}(\tilde{s})=d^1(s)$. \label{partb}
\item $\nabla(s')=c_1\bar{d}'^1 (\nabla_{\log}(s'))$ for some $c_1\in\Q_p^\times$.
\end{enumerate}
Indeed, combining with our choice of $s'$, we see that 
\[N_k(s)=-\nabla(s')=-c_1 \bar{d}'^1\circ d^1(s).\]
Part (1) in fact follows from the definition of $\tilde{s}$ which is constructed from certain $\cO\B_{\dR}^+$-comparison isomorphism. (This is can be viewed as  a ``$B_\dR^+$-deformation'' of our construction of elements in $\cO_{\mathcal{X}_{K^p}}$ using the Hodge-Tate sequence in section \ref{exs}.) For Part (2), one key observation (due to Faltings) is that the first graded piece of the Poincar\'e lemma sequence, also known as the Faltings's extension, is essentially isomorphic to the pull-back of the non-split exact sequence $0\to \cO_{\Fl}\to (\omega^1)^{\oplus 2}\to \omega^2_{\Fl}\to 0$ on $\Fl$ via $\pi_\HT$. Hence all the calculations are reduced to $\Fl$ and it's not hard to compare with $\bar{d}$, the derivation on $\Fl$.
\end{para}

\begin{para} \label{setupM'MM''}
Now we treat the general case. The abstract linear algebra setup is as follows.
Let $M$ be an abelian group equipped with an endomorphism $N\in \End(M)$. Suppose that $N^2=0$. Let $M'\subseteq \ker(N)$ be a subgroup such that $N(M)\subseteq M'$. Hence $N$ becomes trivial on $M'':=M/M'$ and induces a map
\[N_{M'}: M''\to M'.\]
This recovers $N^K_W$ in \ref{FonOpsetup} with $M=E_0(W^K),M'=W_{k,-k}^K,M''=W_{0,0}^K$ and $N=\nabla$.

Now suppose $(M_1,N_1,\theta:M_1\to M'')$ is a triple extending the triple $(M,N,M\to M'')$. More precisely, $M_1$ is an abelian group containing $M$ as a subgroup, $N_1\in \End(M_1)$ with $N_1^2=0$ and $N_1|_M=N$, and $\theta:M_1\to M''$ is a map such that $\theta|_M$ agrees with the natural quotient map $M\to M''$. Suppose that $M'_1\subseteq \ker \theta$ is an $N_1$-stable subgroup and  $M'_1\cap M=M'$. In particular, $M'_1$ contains $M'$. We denote the quotient by $M'_2\subseteq M_2:=M_1/M$. 

\begin{eqnarray} \label{absdiag}
\begin{tikzcd}
& M''=M/M' &&&\\
0 \arrow[r] & M  \arrow[u]\arrow[r] & M_1 \arrow[r,"\pi"] \arrow[lu, "\theta" ']& M_2  \arrow[r] & 0 \\
0 \arrow[r] & M' \arrow[r] \arrow[u] & M'_1 \arrow[r] \arrow[u] & M'_2 \arrow[r] \arrow[u] & 0 
\end{tikzcd}.
\end{eqnarray}
Assume that
\begin{itemize}
\item  there is a right inverse $s:M''\to \ker(N_1)$ of $\theta:M_1\to M''$, i.e. $\theta\circ s=\Id_{M''}$ such that $\pi\circ s(M'')\subseteq M'_2$, where $\pi:M_1\to M_2$ denotes the quotient map.

\begin{tikzcd}
&&& M'' \arrow[d,"s"] \arrow[r,"\pi\circ s"] & M'_2 \arrow[d] \\
&&& \ker(N_1) \arrow[r,"\pi"] & M_2
\end{tikzcd};
\item $N_1(M'_1)\subseteq M'$. Hence $N_1$ induces a map $N_{1.M'}:M'_{2}\to M'$.
\end{itemize}
\end{para}

\begin{lem} \label{decomN}
$N_{M'}=-N_{1,M'}\circ (\pi\circ s)$ as maps $M''\to M'$
\end{lem}

\begin{proof}
The basic idea is very simple.
Let $a\in M''=M/M'$.  By defintion, $N_{M'}(a)=N(\tilde{a})$ if $\tilde{a}\in M$ is a lifting of $a$. Note that $s(a)$ is a lifting of $a$ in $M_1$ but not necessarily in $M$. To make it inside of $M$, we will modify it with an element in $M'_1$.

Let $b\in M'_1$ be a lifting of $\pi(s(a))\in M'_2=M'_1/M'$. By definition, 
\[-N_{1,M'}\circ \pi\circ s(a)=-N_1(b).\]
On the other hand, consider $s(a)-b\in M_1$. Since 
\[\pi(s(a)-b)=\pi(s(a))-\pi(s(a))=0\in M_2\]
by our choice of $b$, we have $s(a)-b\in M$. Note that $b\in M'_1\subseteq\ker(\theta)$. Hence $\theta(s(a)-b)=a$, i.e. $s(a)-b$ is a lifting of $a$ in $M$. By definition, 
\[N_{M'}(a)=N(s(a)-b)=-N_1(b)=-N_{1,M'}\circ \pi\circ s(a)\]
where the last equality follows from $s(M'')\subseteq \ker(N_1)$. This finishes the proof.
\end{proof}

\begin{rem}
In our application, using the notation in \ref{noBdRK} and \ref{FonOpsetup} , we will take $M=E_0(\B_{\dR,k+1}^{\la,\tilde\chi_k}(U)^K)$, $M'=\cO^{\la,(1,-k)}_{K^p}(U)(k)^K$, $M''=\cO^{\la,(1-k,0)}_{K^p}(U)^K$, $N=\nabla$ and  the exact sequence $0\to M\to M_1\to M_2\to 0$ (resp. $0\to M'\to M'_1\to M'_2\to 0$) will be constructed from the Poincar\'e lemma sequence (resp. its $k$-th graded piece).  It remains to construct $M_1, M'_1,N_1$ and $s$. 
\end{rem}

\begin{para} \label{setupOBdR}
Next we collect some results about the positive structural de Rham sheaf $\cO\B_{\dR}^+$ (and its log version). We will not recall the construction here. People can find the construction in  \cite[\S 2.2]{DLLZ2} (with a log structure) and \cite[(3)]{Sch13e} (without log structures). 

Let $U\in\mathfrak{B}$ and $V_\infty=\pi_\HT^{-1}(U)$. Then $V_\infty$ is the inverse image of some affinoid $V_{K_p}\subseteq\mathcal{X}_{K^pK_p}$ for sufficiently small open subgroup $K_p\subseteq \GL_2(\Z_p)$.  Let $K$ be a finite extension of $\Q_p$ in $C$ over which all $V_{K_p}$'s are defined. Hence we can write $V_{K_p}=V_{K_p,K}\times_{\Spa(K,\cO_K)} \Spa(C,\cO_C)$. 

Fix a sufficiently small open subgroup $G_0\subseteq\GL_2(\Z_p)$. After enlarging $K$, we may assume that all cusps in $V_{G_0}$ are defined over $K$. Let $V_0=V_{G_0,K}$. The cusps in $V_0$ define a natural log structure on $V_{0}$, cf. \cite[Example 2.1.2]{DLLZ2}. Consider the pro-Kummer \'etale site $(V_{0})_{\proket}$ of $V_{0}$. Then we can regard $V_\infty$ as the open covering $\displaystyle \varprojlim_{L,K_p} V_{K_p,L}$
in $(V_{0})_{\proket}$, where $K_p$ runs through all open subgroups of $G_0$ and $L$ runs through all finite extension of $K$ in $C$. In particular, we can evaluate the positive structural de Rham sheaf $\cO\B_{\dR,\log,V_0}^+$ (called the geometric de Rham period in \cite[Definition 2.2.10]{DLLZ2}) at $V_\infty$ and denote it by $\cO\B_{\dR}^+(V_\infty)$. It has the following ring-theoretic properties.

\begin{enumerate} 
\item There is a natural surjective ring homomorphism $\theta_{\log}:\cO\B_{\dR}^+(V_\infty)\to\cO_{\mathcal{X}_{K^p}}(V_\infty)$. It induces a decreasing  filtration on $\cO\B_{\dR}^+(V_\infty)$ with $\Fil^i\cO\B_{\dR}^+(V)=\ker\theta_{\log}^i$. This filtration is exhausted and $\cO\B_{\dR}^+(V_\infty)$  is complete with respect to it. Moreover, the natural map $\Sym^i_{\gr^0 \cO\B_{\dR}^+(V_\infty)} \gr^1\cO\B_{\dR}^+(V_\infty)\to\gr^i\cO\B_{\dR}^+(V_\infty)$ is an isomorphism for $i\geq 0$.
\item There is a natural injection $\B_\dR^+(V_\infty)\to\cO\B_{\dR}^+(V_\infty)$ which is strict with respect to the filtrations such that its composite with $\theta_{\log}$ agrees with $\theta:\B_{\dR}^+(V_\infty)\to \cO_{\mathcal{X}_{K^p}}(V_\infty)$.
\item There is a natural $\Q_p$-algebra injection $\displaystyle \varinjlim_{L,K_p} \cO_{V_{K_p,L}}(V_{K_p,L})\to \cO\B_{\dR}^+(V_\infty)$ such that its composite with $\theta_{\log}$ agrees with the natural inclusion $\cO_{V_{K_p,L}}(V_{K_p,L})\subseteq \cO_{\mathcal{X}_{K^p}}(V_\infty)$.
\item $G_{K}$ and $G_0$ act naturally on $\cO\B_{\dR}^+(V_\infty)$ and all the maps above are $G_K\times G_0$-equivariant.
\end{enumerate}
The last two properties can be seen directly from the construction. The first two properties follow from the explicit description of $\cO\B_{\dR}^+(V_\infty)$ in \cite[Proposition 2.3.15]{DLLZ2}. We will also give a description of $\cO\B_{\dR}^+(V_\infty)$ in terms of $\B_{\dR}^+(V_\infty)$ below, cf \eqref{locBdR}.

We remark that $\cO\B_{\dR}^+(V_\infty)$ does not depend on the choice of $G_0$ and $K$. See the paragraph below the proof of Proposition 2.3.15 of \cite{DLLZ2} and also \cite[(3)]{Sch13e}. 
\end{para}
 
\begin{para}[Poincar\'e lemma sequence and Faltings's extension] \label{PlsFe}
There is a natural logarithmic connection on $\cO\B_{\dR}^+(V_\infty)$, cf. \cite[2.2.15]{DLLZ2}. Following the notation in \ref{XCY}, we use $\mathcal{C}$ to denote the set of cusps. Then the logarithmic connection
\[\cO_{V_0}\to\Omega^1_{V_0}(\mathcal{C})\]
can be extended $\B_\dR^+(V_\infty)$-linearly to a logarithmic connection
\[\nabla_{\log}:\cO\B_{\dR}^+(V_\infty)\to \cO\B_{\dR}^+(V_\infty)\otimes_{\cO_{V_0}}\Omega^1_{V_0}(\mathcal{C})\]
where by abuse of notation $\cO_{V_0}=\cO_{V_0}(V_0)$ and $\Omega^1_{V_0}(\mathcal{C})=\Omega^1_{V_0}(\mathcal{C})(V_0)$. As the notation suggests, $\nabla_{\log}$ also does not depend on the choice of $V_0$. More precisely, $\nabla_{\log}$ also extends the log connection $\cO_{V_{K_p,L}}\to \Omega^1_{{V_{K_p,L}}}(\mathcal{C})$ for any $K_p,L$ under the natural identification $ \Omega^1_{{V_{K_p,L}}}(\mathcal{C})=\cO_{V_{K_p,L}}\otimes_{\cO_{V_0}}\Omega^1_{V_0}(\mathcal{C})$. We sometimes will drop the $\cO_{V_0}$ in the tensor products below.

$\nabla_{\log}$ induces a short strict exact sequence (Poincar'e lemma sequence)
\[0\to  \B_\dR^+(V_\infty) \to \cO\B_{\dR}^+(V_\infty)\to \cO\B_{\dR}^+(V_\infty)\otimes_{\cO_{V_0}}\Omega^1_{V_0}(\mathcal{C})\to 0\]
where $\Omega^1_{V_0}(\mathcal{C})$ has degree $1$. Taking the first graded piece, we obtain an exact sequence of $\cO_{\mathcal{X}_{K^p}}(V_\infty)$-modules
\begin{eqnarray} \label{Fext}
0\to \cO_{\mathcal{X}_{K^p}}(V_\infty)(1)\to \gr^1 \cO\B_{\dR}^+(V_\infty)\to \cO_{\mathcal{X}_{K^p}}(V_\infty)\otimes_{\cO_{V_0}}\Omega^1_{V_0}(\mathcal{C})\to 0
\end{eqnarray}
called the log Faltings's extension.  Note that $\Omega^1_{V_0}(\mathcal{C})\cong \omega^2$ by Kodaira-Spencer isomorphism, hence $\Omega^1_{V_0}(\mathcal{C})$ is trivial. Then $\cO_{\mathcal{X}_{K^p}}(V_\infty)\otimes_{\cO_{V_0}}\Omega^1_{V_0}(\mathcal{C})\cong  \cO_{\mathcal{X}_{K^p}}(V_\infty)$ by choosing a generator of $\Omega^1_{V_0}(\mathcal{C})$. Recall that there are natural isomorphisms $\Sym^k_{\gr^0 \cO\B_{\dR}^+(V_\infty)} \gr^1\cO\B_{\dR}^+(V_\infty)\cong\gr^k\cO\B_{\dR}^+(V_\infty)$. Hence the Faltings's extension induces a natural filtration on $\gr^k\cO\B_{\dR}^+(V_\infty)$ with graded pieces isomorphic to $\cO_{\mathcal{X}_{K^p}}(V_\infty)(i)\otimes_{\cO_{V_0}}\Omega^1_{V_0}(\mathcal{C})^{\otimes k-i} \cong \cO_{\mathcal{X}_{K^p}}(V_\infty)(i),i=0,\cdots,k$.

We also have the following explicit description of $ \cO\B_{\dR}^+(V_\infty)$. Choose an element $X\in \Fil^1 \cO\B_{\dR}^+(V_\infty)$ whose image in $\cO_{\mathcal{X}_{K^p}}(V_\infty)\otimes_{\cO_{V_0}}\Omega^1_{V_0}(\mathcal{C})$ is a generator. Then the natural map
\begin{eqnarray} \label{locBdR}
\B_{\dR}^+(V_\infty)[[X]]\to \cO\B_{\dR}^+(V_\infty)
\end{eqnarray}
is an isomorphism of filtered groups where $X$ has degree $1$. 
\end{para}

\begin{para}[$\cO\B_{\dR,k}^+$] \label{cOBdRk+}
It is clear that assigning $U$ to $\cO\B_{\dR}^+(V_\infty)$, where $V_\infty=\pi_{\HT}^{-1}(U)$,  defines a sheaf on $\Fl$, which will be denoted by $\cO\B_{\dR}^+$ by abuse of notation. It is equipped with a natural decreasing filtration and we denote $\cO\B_{\dR}^+/\Fil^k\cO\B_{\dR}^+$ by $\cO\B_{\dR,k}^+$. There is a natural map $\B_{\dR,k}^+\to \cO\B_{\dR,k}^+$. It follows from the discussion in the previous paragraph that $\cO\B_{\dR,k}^+(U)$ is a finite $\B_{\dR,k}^+(U)$-module, $U\in\mathfrak{B}$. Hence we can define a natural $p$-adic topology on $\cO\B_{\dR,k}^+(U)$. More precisely, using the notation in \eqref{locBdR}, we can define a $\Q_p$-Banach algebra structure on $\cO\B_{\dR,k}^+(U)$ with the unit open ball given by the image of 
\[\A_\inf(V_\infty)[[X]]\to \B_{\dR,k}^+(U)[[X]]\to \cO\B_{\dR,k}^+(U).\]
The $p$-adic topology defined in this way does not depend on the choice of $X$.
\end{para}

\begin{para}[subsheaves of $\cO\B_{\dR,k}^+$] \label{subexa}
It follows from the construction that $\cO\B_{\dR,k}^+$ is a $\GL_2(\Q_p)$-equivariant sheaf. We denote by
\[\cO\B_{\dR,k}^{+,\la}\subseteq \cO\B_{\dR,k}^+\]
the subsheaf of $\GL_2(\Q_p)$-locally analytic sections. The Lie algebra $\mathfrak{gl}_2(\Q_p)$ acts naturally on it. Given an infinitesimal character $\tilde\chi:Z\to\Q_p$, we denote by 
\[\cO\B_{\dR,k}^{+,\la,\tilde\chi}\subseteq \cO\B_{\dR,k}^{+,\la}\]
the $\tilde\chi$-isotypic part. Both sheaves $\cO\B_{\dR,k}^{+,\la,\tilde\chi}$ and  $\cO\B_{\dR,k}^{+,\la}$ inherit a decreasing filtration from $\cO\B_{\dR,k}^+$. Note that by Faltings's extension, $\cO\B_{\dR,k}^+$ is naturally filtered by sheaves isomorphic to $\cO_{K^p}$ (up to Tate twists). Hence  the same argument as in the proof of Lemma \ref{BdRcmgr} shows that taking graded pieces commutes with taking locally analytic vectors and $\tilde\chi$-isotypic parts. 

Now take $\tilde\chi=\tilde\chi_{l}$ for some positive integer $l$ defined in \ref{tilchik}. Let $U\in\mathfrak{U}$. Then $G_K$ acts naturally on $\cO\B_{\dR,k}^+(U)$ hence also on $\cO\B_{\dR,k}^{+,\la,\tilde\chi}(U)$  for some finite extension $K$ of $\Q_p$ in $C$. Consider the subspace of $G_{K_\infty}$-fixed, $G_K$-analytic vectors $\cO\B_{\dR,k}^{+,\la,\tilde\chi}(U)^K\subseteq \cO\B_{\dR,k}^{+,\la,\tilde\chi}(U)$ as before. We have seen that $\cO\B_{\dR,k}^{+,\la,\tilde\chi}(U)$ is naturally filtered by LB-spaces isomorphic to $\cO^{\la,\tilde\chi}_{K^p}(U)(i)$ for some integer $i$, which is Hodge-Tate. Therefore taking graded pieces  also commutes with taking $G_{K_\infty}$-fixed, $G_K$-analytic vectors by  the same argument as in the proof  of Proposition \ref{Kcmgr}. Let's summarize these results in the following proposition.
\end{para}

\begin{prop} \label{griFIl}
Let $U\in\mathfrak{B}$ and suppose $G_K$ acts on $\cO\B_{\dR,k}^+(U)$ for some finite extension $K$ of $\Q_p$ in $C$. For integers $i\geq 0$ and $l>0$, the following natural maps
\begin{enumerate}
\item $\gr^i \cO\B_{\dR,k}^{+,\la}\to (\gr^i\cO\B_{\dR}^+)^{\la}$,
\item $\gr^i \cO\B_{\dR,k}^{+,\la,\tilde\chi_l}\to (\gr^i\cO\B_{\dR}^+)^{\la,\tilde\chi_l}$,
\item $\gr^i \cO\B_{\dR,k}^{+,\la,\tilde\chi_l}(U)^K\to (\gr^i\cO\B_{\dR}^+(U))^{\la,\tilde\chi_l,K}$
\end{enumerate}
are isomorphisms. Moreover the $i$-th symmetric power of Faltings's extension induces a natural filtration on $\gr^i \cO\B_{\dR,k}^{+,\la,\tilde\chi_l}(U)^K$  whose graded pieces are  isomorphic to  
\[ \cO^{\la,\tilde\chi_l}_{K^p}(U)(j)^K\otimes_{\cO_{V_0}}\Omega^1_{V_0}(\mathcal{C})^{\otimes i-j},~~~j=0,\cdots,i\]
\end{prop}

For $i\geq 0$, since $\gr^i \cO\B_{\dR,k}^{+,\la}$  does not depend on $k>i$, we will simply write it as $\gr^i\cO\B_{\dR}^{+,\la}$. Similarly, we can define $\gr^i \cO\B_{\dR}^{+,\la,\tilde\chi_l}$, $\gr^i\B_{\dR}^{+,\la}$...

\begin{para}
Consider the action of $1\in\Q_p\cong\Lie(\Gal(K_\infty/K))$ on $\cO\B_{\dR,k}^{+,\la,\tilde\chi_l}(U)^K$. It is easy to see that  $\cO\B_{\dR,k}^{+,\la,\tilde\chi_l}(U)^K$ is filtered by $\cO^{\la,\tilde\chi_l}_{K^p}(U)(i)^K$, hence there is a generalized eigenspace decomposition 
\[\cO\B_{\dR,k}^{+,\la,\tilde\chi_l}(U)^K=\bigoplus_{\lambda} E_{\lambda}(\cO\B_{\dR,k}^{+,\la,\tilde\chi_l}(U)^K)\]
where $\lambda$ runs through all eigenvalues. We will be particularly interested in the  the generalized eigenspace associated to $\lambda=0$.
\end{para}

\begin{para} \label{constdiag}
Now we can construct a diagram similar to \eqref{absdiag}. Consider the following truncated Poincar\'e lemma sequence and its $\Fil^k$-parts:
\[\begin{tikzcd}
& \cO^{\la,\tilde\chi_k}_{K^p}  &&&\\
0 \arrow[r] & \B_{\dR,k+1}^{+,\la,\tilde\chi_k}  \arrow[r] \arrow[u] &  \cO\B_{\dR,k+1}^{+,\la,\tilde\chi_k}  \arrow[r,"\nabla_{\log}"]  \arrow[lu,"\theta_{\log}" ']&  \cO\B_{\dR,k}^{+,\la,\tilde\chi_k} \otimes_{\cO_{V_0}}\Omega^1_{V_0}(\mathcal{C})  \arrow[r] & 0 \\
0 \arrow[r] & \gr^k\B_{\dR}^{+,\la,\tilde\chi_k} \arrow[r] \arrow[u] & \gr^k \cO\B_{\dR}^{+,\la,\tilde\chi_k}  \arrow[r] \arrow[u] & \gr^{k-1} \cO\B_{\dR}^{+,\la,\tilde\chi_k} \otimes_{\cO_{V_0}}\Omega^1_{V_0}(\mathcal{C})   \arrow[r] \arrow[u] & 0 
\end{tikzcd}.\]
Evaluate this diagram at $U$, take  the subspace of $G_{K_\infty}$-fixed, $G_K$-analytic vectors, and take the generalized eigenspace of $0$ with respect to the action of $1\in\Q_p\cong\Lie(\Gal(K_\infty/K))$.  We obtain the following diagram, cf. \eqref{absdiag}.
\[
\begin{tikzcd}
 \cO^{\la,(1-k,0)}_{K^p}(U)^{G_K} &&&\\
 E_0(\B_{\dR,k+1}^{+,\la,\tilde\chi_k}(U)^K)  \arrow[r,hook] \arrow[u] &  E_0(\cO\B_{\dR,k+1}^{+,\la,\tilde\chi_k} (U)^K) \arrow[r,"E_0(\nabla_{\log})"]  \arrow[lu,"E_0(\theta_{\log})" ']&  E_0(\cO\B_{\dR,k}^{+,\la,\tilde\chi_k}(U)^K) \otimes\Omega^1_{V_0}(\mathcal{C})  \arrow[r] & 0 \\
 \cO^{\la,(1,-k)}_{K^p}(U)(k)^{G_K} \arrow[r,hook] \arrow[u] & E_0(\gr^k \cO\B_{\dR}^{+,\la,\tilde\chi_k}(U)^K)  \arrow[r] \arrow[u] & \cO^{\la,(1-k,0)}_{K^p}(U)^{G_K} \otimes\Omega^1_{V_0}(\mathcal{C})^{\otimes k}  \arrow[r] \arrow[u] & 0 
\end{tikzcd}.\]
Both horizontal sequences are exact by exactly the same argument as in \ref{subexa}. Here we identify 
\begin{itemize}
\item $E_0(\cO^{\la,\tilde\chi_k}_{K^p}(U)^K)=\cO^{\la,(1-k,0)}_{K^p}(U)^K=\cO^{\la,(1-k,0)}_{K^p}(U)^{G_K}$;
\item $E_0( \gr^k\B_{\dR}^{+,\la,\tilde\chi_k} (U)^K)= E_0(\cO^{\la,\tilde\chi_k}_{K^p}(U)(k)^K)=\cO^{\la,(1,-k)}_{K^p}(U)(k)^K=\cO^{\la,(1,-k)}_{K^p}(U)(k)^{G_K}$;
\item $E_0(\gr^{k-1} \cO\B_{\dR}^{+,\la,\tilde\chi_k}(U)^K) \otimes_{\cO_{V_0}}\Omega^1_{V_0}(\mathcal{C})=\cO^{\la,(1-k,0)}_{K^p}(U)^{G_K} \otimes_{\cO_{V_0}}\Omega^1_{V_0}(\mathcal{C})^{\otimes k}$.
\end{itemize}
The first two follow from that $\cO^{\la,\tilde\chi_k}_{K^p}(U)^K$ is Hodge-Tate of weights $0,k$, cf. \ref{tilchik}. For the last one, we note that by Proposition \ref{griFIl}, $\gr^{k-1} \cO\B_{\dR}^{+,\la,\tilde\chi_k}(U)^K$ is filtered  by 
\[ \cO^{\la,\tilde\chi_k}_{K^p}(U)(j)^K\otimes_{\cO_{V_0}}\Omega^1_{V_0}(\mathcal{C})^{\otimes k-1-j} \cong \cO^{\la,\tilde\chi_k}_{K^p}(U)(j)^K,~~~j=0,\cdots,k-1.\]
Hence the only Hodge-Tate weight-$0$ part is contributed by $ \cO^{\la,\tilde\chi_k}_{K^p}(U)^K\otimes_{\cO_{V_0}}\Omega^1_{V_0}(\mathcal{C})^{\otimes k-1} $, i.e. when $j=0$. This implies our claim.
\end{para}

\begin{para}
We basically have everything in \ref{setupM'MM''} here with $N$ and $N_1$ given by the action of $1\in\Q_p\cong\Lie(\Gal(K_\infty/K))$, except that we still need to construct a section of $E_0(\theta_{\log})$ (also denoted by $s$ in  \ref{setupM'MM''})
\[s_{k+1}: \cO^{\la,(1-k,0)}_{K^p}(U)^{G_K}\to \cO\B_{\dR,k+1}^{+,\la,\tilde\chi_k} (U)^{G_K} \]
such that $\im(E_0(\nabla_{\log})\circ s_{k+1})\subseteq \Fil^{k-1}\cO\B_{\dR,k}^{+,\la,\tilde\chi_k} (U)^{G_K}\otimes\Omega^1_{V_0}(\mathcal{C})^{\otimes k} = \cO^{\la,(1-k,0)}_{K^p}(U)^{G_K} \otimes\Omega^1_{V_0}(\mathcal{C})^{\otimes k}  $.
(Note that $\ker(N_1)$ is exactly the subspace of $G_K$-invariants here.) 
This will be done in the next subsection, cf. \ref{defnsk}. Let's assume its existence at the moment. Then by Lemma \ref{decomN},
$N_k(U): \cO^{\la,(1-k,0)}_{K^p}(U)^{G_K} \to   \cO^{\la,(1,-k)}_{K^p}(U)(k)^{G_K}$ (corresponding to $N_{M'}$ in Lemma \ref{decomN}) can be written as
\[N_k(U)=-N'_k\circ (E_0(\nabla_{\log})\circ s_{k+1}),\]
where $N'_k: \cO^{\la,(1-k,0)}_{K^p}(U)^{G_K} \otimes\Omega^1_{V_0}(\mathcal{C})^{\otimes k}\to \cO^{\la,(1,-k)}_{K^p}(U)(k)^{G_K}$ is obtained by applying $1\in\Q_p\cong\Lie(\Gal(K_\infty/K))$ to the middle term of the exact sequence 
\[0 \to \cO^{\la,(1,-k)}_{K^p}(U)(k)^{G_K} \to E_0(\gr^k \cO\B_{\dR,k+1}^{+,\la,\tilde\chi_k}(U)^K)  \to \cO^{\la,(1-k,0)}_{K^p}(U)^{G_K} \otimes\Omega^1_{V_0}(\mathcal{C})^{\otimes k}  \to 0. \]
Thus in order to prove that $N_k=c_kI_{k-1}=c_k\bar{d}'^k\circ d^k$ for some $c_k\in\Q^\times$,  it remains to compare $E_0(\nabla_{\log})\circ s_{k+1}$ with $d^k$ and $N'_k$ with $\bar{d}'^k$. 
\end{para}

\begin{prop} \label{comdk}
Let $s_{k+1}: \cO^{\la,(1-k,0)}_{K^p}(U)^{G_K}\to \cO\B_{\dR,k+1}^{+,\la,\tilde\chi_k} (U)^{G_K}$ be the map defined in Definition \ref{defnsk} below. Then
\[E_0(\nabla_{\log})\circ s_{k+1}=d^k.\]
Both sides are  viewed as maps $\cO^{\la,(1-k,0)}_{K^p}(U)^{G_K}\to \cO^{\la,(1-k,0)}_{K^p}(U)^{G_K} \otimes\Omega^1_{V_0}(\mathcal{C})^{\otimes k} $.
\end{prop}

\begin{prop} \label{comdbark}
There exists $c_k\in\Q_p^\times$ such that 
\[N'_k=c_k\bar{d}'^k.\] 
Both sides are  viewed as maps $\cO^{\la,(1-k,0)}_{K^p}(U)^{G_K} \otimes\Omega^1_{V_0}(\mathcal{C})^{\otimes k}\to \cO^{\la,(1,-k)}_{K^p}(U)(k)^{G_K}$.
\end{prop}

We will prove these two propositions in the next two subsections.

\subsection{Proof of Theorem \ref{MT} II} \label{PTII}
\begin{para} \label{reldRcomp}
In this subsection, we prove Proposition \ref{comdk}.
Keep the same notation as in the previous subsection. As claimed before,  we will construct a map
\[s_{k+1}: \cO^{\la,(1-k,0)}_{K^p}(U)^{G_K}\to \cO\B_{\dR,k+1}^{+,\la,\tilde\chi_k} (U)^{G_K} \]
which is a section of $E_0(\theta_{\log}):\cO\B_{\dR,k+1}^{+,\la,\tilde\chi_k} (U)^{G_K}\to   \cO^{\la,(1-k,0)}_{K^p}(U)^K$, and we will compute its composite with $E_0(\nabla_{\log})$. Essentially we need to construct some explicit elements in $\cO\B_{\dR}^+$. Our strategy is to use a relative de Rham comparison theorem on modular curves. One reference here is \cite[4.1.3]{Pan20}. 

Keep the same setup and notation as in \ref{setupOBdR}. The first relative $p$-adic \'etale cohomology of the universal family of elliptic curves on $V_0$ defines a rank two $\hat\Z_p$-local system $\hat{\underline{V}}$ (denoted by $\hat{\underline{V}}_{\log}$ in \cite{Pan20}) on $(V_0)_{\proket}$. Recall that in \ref{XCY}, we have a rank two filtered vector bundle $D$ on $V_0$ equipped with a logarithmic connection $\nabla$, which is defined as the canonical extension of the first de Rham cohomology of the universal family of elliptic curves. Then $\hat{\underline{V}}$ and $(D,\nabla)$ are associated. Concretely, evaluating everything at $V_\infty\in (V_0)_{\proket}$, we have a natural isomorphism 
\[V\otimes_{\Q_p} \cO\B_{\dR}^+(V_\infty)[\frac{1}{t}]\cong D\otimes_{\cO_{V_0}}\cO\B_{\dR}^+(V_\infty)[\frac{1}{t}]\]
preserving the filtrations and logarithmic connections. Here $V=\Q_p^{\oplus 2}$ as $\hat{\underline{V}}$ becomes trivial on $V_\infty$ and $\frac{1}{t}$ has degree $-1$. Under this isomorphism, we even have a more refined result
\begin{eqnarray} \label{relOBdR}
V\otimes_{\Q_p} t\cO\B_{\dR}^+(V_\infty) \subseteq 
D\otimes_{\cO_{V_0}}\cO\B_{\dR}^+(V_\infty)
\subseteq V\otimes_{\Q_p} \cO\B_{\dR}^+(V_\infty), 
\end{eqnarray}
cf. \cite[(4.1.2)]{Pan20}. Thus there is a natural $\cO_{V_0}$-linear, $G_0$-equivariant map
\begin{eqnarray}\label{l_1}
l_1:D(V_0)\otimes_{\Q_p} V^*\to \cO\B_{\dR}^+(V_\infty)=\cO\B_{\dR}^+(U)
\end{eqnarray}
preserving filtrations and logarithmic connections on both sides. The images of this $l_1$ should be considered as the ``$\cO\B_{\dR}^+$-periods of $V$''. Take the composite of this map with $\theta_{\log}$.
\[ \theta_{\log}\circ l_1: D(V_0)\otimes_{\Q_p} V^* \to \gr^0 D(V_0)\otimes_{\Q_p} V^*\to \cO_{\mathcal{X}_{K^p}}(V_\infty).\]
Note that this induces a map 
\[\gr^0 D(V_0)\cong \omega^{-1}\otimes\wedge^2 D(V_0)\to \cO_{\mathcal{X}_{K^p}}(V_\infty)\otimes_{\Q_p} V\] 
which is nothing but the injection in \eqref{rHTH} (when restricting to $\omega^{-1}\otimes\wedge^2 D(V_0)$) up to a Tate twist by \cite[Proposition 7.9]{Sch13}. Let $(1,0)^*,(0,1)^*\in V^*$ be the dual basis of $V=\Q_p^{\oplus 2}$. Then it follows from our discussion in \ref{HEt} and the construction of $e_1,e_2,\mathrm{t}$ in \ref{exs} that for $f\in D(V_0)$,
\[\theta_{\log}\circ l_1(f\otimes (1,0)^*)=\frac{\bar{f}e_2}{\mathrm{t}c},\]
\[\theta_{\log}\circ l_1(f\otimes (0,1)^*)=-\frac{\bar{f}e_1}{\mathrm{t}c},\]
where $\bar{f}\equiv f \mod \Fil^1D$, viewed as an element in $\gr^0 D\cong \omega^{-1}\otimes \wedge^2 D$. Recall that $c$ is our chosen global non-vanishing section of $\wedge^2 D$ used in the construction of $\mathrm{t}$. See \ref{HEt} for more details. In the below, sometimes we will also use $c^i$ to denote twisting by $(\wedge^2 D)^i$. For example, $\gr^0 D=\omega^{-1}c$.

In general, we can consider the $k$-th symmetric powers $\Sym^k\hat{\underline{V}}$ and $\Sym^k D$ and obtain a natural map
\begin{eqnarray} \label{l_k}
l_k:\Sym^k D(V_0)\otimes_{\Q_p} \Sym^k V^*\to \cO\B_{\dR}^+(V_\infty)=\cO\B_{\dR}^+(U).
\end{eqnarray}
Similarly, it has the property that for $f\in \Sym^k D(V_0)$,
\begin{eqnarray} \label{thetalk}
\theta_{\log}\circ l_k\left(f\otimes ((1,0)^*)^{\otimes i}(-(0,1)^*)^{\otimes k-i}\right) =\frac{\bar{f}e_1^{k-i}e_2^i}{\mathrm{t}^kc^k},~~~~~i=0,\cdots,k
\end{eqnarray}
where $\bar{f}\equiv f\mod \Fil^1$.
We will use these maps to produce elements in $\cO\B_{\dR}^+$. First we need several easy lemmas.
\end{para}

\begin{lem} \label{thetainv}
$f\in \cO\B_{\dR}^+(U)$ is invertible if and only if $\theta_{\log}(f)\in\cO_{K^p}(U)$ is invertible.
\end{lem}

\begin{proof}
This is clear as $\cO\B_{\dR}^+(U)$ is complete with respect to the $\ker(\theta_{\log})$-adic topology.
\end{proof}

\begin{lem} \label{OBdRconv}
Let $\{a_n\}_{n\geq 0}$ be a sequence of sections of $\cO_{V_0}$ on $V_0$ and $y\in\cO\B_{\dR,k}^+(U)$. Suppose $\theta_{\log}(y)\in\cO_{\mathcal{X}_{K^p}}^+(V_\infty), i.e. ||\theta_{\log}(y)||\leq 1$ and $\displaystyle \lim_{n\to+\infty} a_n=0$. Then
\[\sum_{n=0}^{+\infty} a_n y^n\]
converges in $\cO\B_{\dR,k}^+(U)$. In other words, $\sum_{n=0}^{+\infty} a_n y^n$ converges if and only if $\sum_{n=0}^{+\infty} a_n \theta_{\log}(y)^n$ converges.
\end{lem}

\begin{proof}
This is essentially \cite[Lemme 4.9]{BC16}.  We give a sketch of their argument here. Since $\theta:\A_\inf(V_\infty)\to \cO_{\mathcal{X}_{K^p}}^+(V_\infty)$ is surjective, we can find $\tilde{y}\in \A_\inf(V_\infty)$ such that $z=y-\tilde{y}\in\ker(\theta_{\log})$. Hence $\sum_{n=0}^{+\infty} a_n \tilde{y}^n$ converges in view of the topology defined in \ref{cOBdRk+}. Using that $z^k=0$, we deduce that
\[\sum_{n=0}^{+\infty} a_n y^n=\sum_{n=0}^{+\infty} a_n (z+\tilde{y})^n=\sum_{i=0}^{k-1}z^i\sum_{n=0}^{+\infty} \binom{n}{i} a_n\tilde{y}^{n-i}\]
converges in $\cO\B_{\dR,k}^+(U)$.
\end{proof}

\begin{para}
We start with the case $k=1$.
Assume that $e_1$ is invertible on $V_\infty$. Fix a section $f_1\in D(V_0)$ whose image in $\gr^0 D(V_0)$ is a generator. Consider $l_1(f_1\otimes (0,1)^*)\in\cO\B_{\dR}^+(U)$. Note that $\theta_{\log}\circ l_1(f_1\otimes (0,1)^*)\in\cO_{K^p}(U)^\times$ by our choice of $f_1$. Hence by Lemma \ref{thetainv}, $l_1(f_1\otimes (0,1)^*)$ is invertible in $\cO\B_{\dR}^+(U)$. Let 
\[\tilde{x}:=-\frac{l_1(f_1\otimes (1,0)^*)}{l_1(f_1\otimes (0,1)^*)}\in\cO\B_{\dR}^+(U).\]
Then it follows from the discussion in \ref{reldRcomp} that $\theta_{\log}(\tilde{x})=x$. Clearly $G_K$ fixes $\tilde{x}$ and the action of $G_0$ on $\tilde{x}$ is analytic. Shrinking $G_0$ if necessary, we may assume the $||x||=||x||_{G_0}$, the norm as a $G_n$-analytic vector. A direct calculation shows that the Lie algebra $\mathfrak{gl}_2(\Q_p)$ acts on $\tilde{x}$ via
\[\begin{pmatrix} a & 0 \\ 0 & d \end{pmatrix}\cdot \tilde{x}=(d-a)\tilde{x},~~~~~~~~~~~\begin{pmatrix} 0 & 1 \\ 0 & 0\end{pmatrix} \cdot \tilde{x}=1,~~~~~~~~~~~\begin{pmatrix} 0 & 0 \\ 1& 0\end{pmatrix} \cdot \tilde{x}=-\tilde{x}^2\]
(in the same way as on $x$). In particular, $Z$ acts on powers of $\tilde{x}$ via $\tilde{\chi}_1$.

We are going to construct maps
\[ \cO^{\la,(0,0)}_{K^p}(U)^{G_K}\to \cO\B_{\dR,l}^{+,\la,\tilde\chi_1} (U)^{G_K},l\geq 0\]
which are compatible when varying $l$. When $l=2$, this will give us the desired map $s_2: \cO^{\la,(0,0)}_{K^p}(U)^{G_K}\to \cO\B_{\dR,1}^{+,\la,\tilde\chi_1} (U)^{G_K}$. In fact we will construct $G_K$-equivariant maps 
\begin{eqnarray} \label{tildesl}
 \tilde{s}_l:\bigcup_{L}\cO^{\la,(0,0)}_{K^p}(U)^{G_L}\to \bigcup_{L}\cO\B_{\dR,l}^{+,\la,\tilde\chi_1} (U)^{G_L},l\geq 0
 \end{eqnarray}
 where $L$ runs through all finite extensions of $K$ in $C$.  Now  we need the explicit description of $\cO^{\la,(0,0)}_{K^p}(U)$ in Theorem \ref{str}. As in \ref{exs}, we choose $x_n\in\cO_{V_{G_{r(n)}}}(V_{G_{r(n)}}),n\geq 0$ such that $\|x-x_n\|_{G_{r(n)}}=\|x-x_n\|\leq p^{-n}$. We may assume that $x_n$ is defined over some finite extension $K_n$ of $K$ in $C$ and $\bigcup_n K_n=\overbar\Q_p$. Then an element $f\in \cO^{\la,(0,0)}_{K^p}(U)^{G_L}$ can be written as
 \[f=\sum_{i=0}^{+\infty} c_i(x-x_n)^i\]
 for some $n\geq 0$ and $c_i\in\cO_{V_{G_{r(n)}},K_n}$ such that $c_ip^{(n-1)i}$ is uniformly bounded. Define 
 \[\tilde{s}_l(f):=\sum_{i=0}^{+\infty} c_i(\tilde{x}-x_n)^i\in\cO\B_{\dR,l}^+(U).\]
 Note that this series is convergent by Lemma \ref{OBdRconv}. It is easy to see that $\tilde{s}_l(f)$ does not depend on the choice of $x_n$ and $\tilde{s}_l(f)$ is a $G_{r(n)}$-analytic vector. Moreover $Z$ acts on $\tilde{s}_l(f)$ via $\tilde\chi_1$. Hence we obtain the map $\tilde{s}_l$ claimed in \eqref{tildesl}. Clearly the inverse limit of $\tilde{s}_l$ over $l$ defines a map
\[\phi_1:\bigcup_{L}\cO^{\la,(0,0)}_{K^p}(U)^{G_L}\to \cO\B_{\dR}(U).\]
\end{para}

\begin{para}
Our next step is to compute $\nabla_{\log}\circ \phi_1:\cO^{\la,(0,0)}_{K^p}(U)^{G_K}\to \cO\B_{\dR}^+(U) \otimes_{\cO_{V_0}}\Omega^1_{V_0}(\mathcal{C})$. Essentially we need to understand the image of $x$. In fact, the following lemma will be enough for our purpose in view of Lemma \ref{tpowt} below.  
\end{para}

\begin{lem}[$p$-adic Legendre's relation] \label{pLegen}
Recall that the map $l_1$ was introduced in \eqref{l_1}.
\begin{enumerate}
\item For any $g_1,g_2\in D(V_0)$ and $v_1,v_2\in V^*$, the periods satify
\[l_1(g_1\otimes v_1)l_1(g_2\otimes v_2)-l_1(g_1\otimes v_2)l_1(g_2\otimes v_1)\in t\cO\B_{\dR}^+(U).\]
\item $\nabla_{\log}\circ\phi_1(x)\in t\cO\B_{\dR}^+(U)\otimes\Omega^1_{V_0}(\mathcal{C})$.
\end{enumerate}
\end{lem}

\begin{proof}
For the first part, it suffices to show that 
\[\wedge^2D\otimes_{\cO_{V_0}}\cO\B_{\dR}^+(U)
\subseteq t\left(\wedge^2V\otimes_{\Q_p} \cO\B_{\dR}^+(U)\right),\]
cf. \eqref{relOBdR}. Let $\mathbb{M}_0=(D\otimes_{\cO_{V_0}}\cO\B_{\dR}^+(U))^{\nabla_{\log}}$. It was shown in the proof of Theorem 7.6 of \cite{Sch13} that $\mathbb{M}_0\otimes_{\B_{\dR}^+(U)}\cO\B_{\dR}^+(U)=D\otimes_{\cO_{V_0}}\cO\B_{\dR}^+(U)$. Then
\[\mathbb{M}_0\subseteq V\otimes_{\Q_p}\B_{\dR}^+(U).\]
By Proposition 7.9 \textit{ibid.} or \cite[(4.1.2)]{Pan20}, the cokernel of this inclusion is nothing but $\gr^1D\otimes_{\cO_{V_0}}\cO_{\mathcal{X}_{K^p}}(V_\infty)(-1)\neq 0$. Hence 
\[\wedge^2 \mathbb{M}_0\subseteq t(\wedge^2 V\otimes_{\Q_p}\B_{\dR}^+(U)),\]
which implies our claim by taking the tensor product with $\cO\B_{\dR}^+(U)$.

For the second part, we recall that $\tilde{x}:=-\frac{l_1(f_1\otimes (1,0)^*)}{l_1(f_1\otimes (0,1)^*)}$.  Fix a generator $\omega$ of $\Omega^1_{V_0}(\mathcal{C})$ and write $\nabla_{\log}(f_1)=f_2\omega, f_2\in\cO_{V_0}$. Then
\[\nabla_{\log}(\tilde{x})=\frac{1}{l_1(f_1\otimes (0,1)^*)^2}(l_1(f_1\otimes v_1)l_1(f_2\otimes v_2)-l_1(f_1\otimes v_2)l_1(f_2\otimes v_1))\otimes\omega,\]
where $v_1=(1,0)^*,v_2=(0,1)^*$. Hence by the first part, we have $\nabla_{\log}(\tilde{x})\in t\cO\B_{\dR}^+(U)\omega$.
\end{proof}

\begin{para}
It is clear from the construction that $\tilde{x}$ and $\phi_1$ depends on the choice of $f_1\in D(V_0)$. However  this lemma shows that $\tilde{x}$ is actually well-defined up to $t\cO\B_{\dR}^+(U)$, i.e. $\tilde{x}\mod t\cO\B_{\dR}^+(U)$ is independent of the choice of $f_1$. To see this, we consider the Poincar\'e lemma sequence modulo $t$
\[0\to \B_{\dR}^+(U)/(t) \to \cO\B_{\dR}^+(U)/(t)\xrightarrow{\nabla_{\log}\mod t}  \cO\B_{\dR}^+(U)/(t)\otimes\Omega^1_{V_0}(\mathcal{C})\to 0\]
which is exact as all the terms in the Poincar\'e lemma sequence are free over $\B_{\dR}^+(U)$. Lemma \ref{pLegen} implies that $\tilde{x}\mod (t)\in \B_{\dR}^+(U)/(t)=\cO_{K^p}(U)$. On the other hand, $\theta_{\log}(\tilde{x})=x$ by our construction. Hence
\[\tilde{x}\mod (t)=x.\]
We summarize what we obtained so far.
\end{para}

\begin{prop} \label{phi1}
Given  a section $f_1\in D(V_0)$ whose image in $\gr^0 D(V_0)$ is a generator, we can define a map $G_K\times G_0$-equivariant map
\[\phi_1:\bigcup_{L}\cO^{\la,(0,0)}_{K^p}(U)^{G_L}\to \cO\B_{\dR}^+(U)\]
satisfying following the conditions
\begin{enumerate}
\item $\theta_{\log}\circ\phi_1=\Id$.
\item $\phi_1$ is $\cO_{V_{K_p},L}$-linear for any open subgroup $K_p\subseteq G_0$ and finite extension $L$ of $K$ in $C.$
\item the composite map $\bigcup_{L}\cO^{\la,(0,0)}_{K^p}(U)^{G_L}\xrightarrow{\phi_1} \cO\B_{\dR}^+(U)\to \cO\B_{\dR,l}^+(U)$ is continuous for $l\geq 0$.
\item $\phi_1(x)\equiv x \mod (t)$. In particular, $\nabla_{\log}\circ\phi_1(x)\in t\cO\B_{\dR}^+(U)\otimes\Omega^1_{V_0}(\mathcal{C})$ and $\phi_1\mod (t)$ is independent of the choice of $f_1$.
\end{enumerate}
\end{prop}

As mentioned before, we denote by $s_2$ the map $\phi_1|_{\cO^{\la,(0,0)}_{K^p}(U)^{G_K}}\mod\Fil^2$, hence
\[s_2: \cO^{\la,(0,0)}_{K^p}(U)^{G_K}\to \cO\B_{\dR,2}^{+,\la,\tilde{\chi_1}}(U)^{G_K}.\]
Moreover the composite map
\[ \cO^{\la,(0,0)}_{K^p}(U)^{G_K}\xrightarrow{s_2} \cO\B_{\dR,2}^{+,\la,\tilde{\chi_1}}(U)^{G_K} \xrightarrow{\nabla_{\log}}  \cO^{\la,(0,0)}_{K^p}(U)^{G_K} \otimes\Omega^1_{V_0}(\mathcal{C})\]
sends $x$ to $0$ and sections $f\in\cO_{V_{G_{n}},K}\subseteq \cO^{\la,(0,0)}_{K^p}(U)^{G_K}$ to $df$. This shows that $\nabla_{\log}\circ s_2=d^1$ in view of the construction of $d^1$, cf. Theorem \ref{I1}. This proves Proposition \ref{comdk} when $k=1$.

\begin{para} \label{kgeq2}
Now for $k\geq 2$, we follow the same strategy as in Subsection \ref{Do1}. Recall that we have \eqref{l_k}
\[l_{n}:\Sym^{n} D(V_0)\otimes_{\Q_p} \Sym^{n} V^* \to \cO\B_{\dR}^+(U),~~~~n\geq 0\]
which is $\cO_{V_0}$-linear. Therefore using $\phi_1$, we can extend this map $\cO^{\la,(0,0)}_{K^p}(U)^{G_K}$-linearly to a map
\[l'_{n}:\Sym^{n} D(V_0)\otimes_{\cO_{V_0}}\cO^{\la,(0,0)}_{K^p}(U)^{G_K}\otimes_{\Q_p} \Sym^{n} V^* \to \cO\B_{\dR}^+(U).\]
Note that there are natural surjective maps coming from the Hodge filtration on $\Sym^n D$ and the Hodge-Tate filtration on $\Sym^n V^*\otimes \cO_{\mathcal{X}_{K^p}}$
\[\Sym^n D\to \gr^0 \Sym^n D=\omega^{-k}\otimes (\wedge^2 D)^{\otimes n},\]
\[\cO^{\la,(0,0)}_{K^p}\otimes_{\Q_p} \Sym^{n} V^*=\cO^{\la,(0,0)}_{K^p}\otimes_{\Q_p} \Sym^{n} V\otimes\det{}^{-n}
 \to \omega^{k,\la,(0,n)}_{K^p}\otimes\det{}^{-n}(-n),\]
(cf. \ref{genklambdak})  and the natural identification \eqref{rmtwddet1},
\[\cO^{\la,(n_1+1,n_2+1)}_{K^p}=\cO^{\la,(n_1,n_2)}_{K^p}\cdot \mathrm{t}\cong \cO^{\la,(n_1,n_2)}_{K^p}\otimes_{\cO_{K^p}^{\sm}}(\wedge^2 D^{\sm}_{K})^{-1}\otimes \det(1).\]
These induce a natural map
\[\Sym^{n} D(V_0)\otimes_{\cO_{V_0}}\cO^{\la,(0,0)}_{K^p}(U)^{G_K}\otimes_{\Q_p} \Sym^{n} V^*\to \cO^{\la,(-n,0)}_{K^p}(U)^{G_K}.\]
which  is nothing but the composite map
$\theta_{\log}\circ l'_{n}$ by our discussion in \ref{l_1}. The BGG constructions for $\Sym^n D$ and  $\Sym^n V^*$ give  a natural left inverse of this map
\begin{eqnarray} \label{psin}
\psi_n:\cO^{\la,(-n,0)}_{K^p}(U)^{G_K}\to\Sym^{n} D(V_0)\otimes_{\cO_{V_0}}\cO^{\la,(0,0)}_{K^p}(U)^{G_K}\otimes_{\Q_p} \Sym^{n} V^*.
\end{eqnarray}
More precisely, recall that $\Sym^n D(V_0)\otimes_{\cO_{V_0}}\cO^{\la,(0,0)}_{K^p}(U)$ was denoted by $\Sym^n D^{(0,0)}(U)$ in the proof of Theorem \ref{I1}. The filtration on $\Sym^n D$ naturally extends to $\Sym^n D^{(0,0)}(U)$ in an $\cO^{\la,(0,0)}_{K^p}$-linear way. By Lemma \ref{KSsec}, we have a natural $U(\mathfrak{gl}_2(\Q_p))$-equivariant map 
\[r'_n:\gr^0\Sym^n D^{(0,0)}(U)^{G_K}\xrightarrow{r'_n} \Sym^k D^{(0,0)}(U)^{G_K}\]
which is a section of $\Sym^k D^{(0,0)}(U)^{G_K}\to \gr^0 \Sym^k D^{(0,0)}(U)^{G_K}$ and was constructed using the Kodaira-Spencer isomorphism. On the other hand,
by Lemma \ref{secinf}, the $\tilde\chi_{k+1}$-isotopic part of 
\[\gr^0\Sym^n D^{(0,0)}(U)^{G_K}\otimes_{\Q_p} \Sym^{n} V^*=\omega^{-n}(V_0)c^n\otimes_{\cO_{V_0}}\cO^{\la,(0,0)}_{K^p}(U)^{G_K}\otimes_{\Q_p} \Sym^{n} V^*\]
 is canonically isomorphic to $\cO^{\la,(-n,0)}_{K^p}(U)^{G_K}$ via $\theta_{\log}\circ l'_n$. Thus we obtain a natural injective $U(\mathfrak{gl}_2(\Q_p))$-equivariant  map
 \[\cO^{\la,(-n,0)}_{K^p}(U)^{G_K}\to  \gr^0\Sym^n D^{(0,0)}(U)^{G_K}\otimes_{\Q_p} \Sym^{n} V^*\]
 and its composite with $r'_n\otimes 1$ gives the  map $\psi_n$ we are looking for. We remark that all the maps defines here are $G_0$-equivariant and continuous. 
\end{para}

\begin{defn} \label{defnsk}
Let $s_{k+1}$ be the composite map 
\[ \cO^{\la,(1-k,0)}_{K^p}(U)^{G_K}\xrightarrow{\psi_{k-1}}\Sym^{n} D(V_0)\otimes\cO^{\la,(0,0)}_{K^p}(U)^{G_K}\otimes \Sym^{n} V^*
\xrightarrow{l'_{k-1}\mod\Fil^{k+1}}\cO\B_{\dR,k+1}^{+}(U)^{G_K}\]
It follows from our previous discussion that $s_{k+1}$ actually defines a map 
\[s_{k+1}:\cO^{\la,(1-k,0)}_{K^p}(U)^{G_K}\to\cO\B_{\dR,k+1}^{+,\la,\tilde{\chi}_k}(U)^{G_K}\]
such that
 $\theta_{\log} \circ s_{k+1}=\Id$.
\end{defn}

We are ready to prove Proposition \ref{comdk}. We restate it here.

\begin{prop} \label{nlogsk+1dk}
$\im(\nabla_{\log}\circ s_{k+1})\subseteq \Fil^{k-1} \cO\B_{\dR,k}^{+,\la,\tilde{\chi}_k}(U)^{G_K}\otimes \Omega^1_{V_0}(\mathcal{C})$. Moreover, under the isomorphism $\Fil^{k-1}\cO\B_{\dR}^{+,\la,\tilde\chi_k} (U)^{G_K}\otimes\Omega^1_{V_0}(\mathcal{C}) = \cO^{\la,(1-k,0)}_{K^p}(U)^{G_K} \otimes\Omega^1_{V_0}(\mathcal{C})^{\otimes k}$, we have 
\[\nabla_{\log}\circ s_{k+1}={d}^{k}:\cO^{\la,(1-k,0)}_{K^p}(U)^{G_K}\to \cO^{\la,(1-k,0)}_{K^p}(U)^{G_K} \otimes\Omega^1_{V_0}(\mathcal{C})^{\otimes k}.\]
\end{prop}

\begin{proof}
Our first observation is that it suffices to prove the proposition modulo $t$.

\begin{lem} \label{tpowt}
The natural map 
\[\cO\B_{\dR,k}^{+,\la,\tilde\chi_k}(U)^{G_K}\to \left(\cO\B_{\dR,k}^{+,\la,\tilde\chi_k}(U)/(t)\right)^{G_K}\] 
is an isomorphism.
\end{lem}

\begin{proof}
By the same argument as in the proof of Proposition \ref{griFIl}, 
\[0\to t\cO\B_{\dR,k}^{+,\la,\tilde\chi_k}(U)^K\to \cO\B_{\dR,k}^{+,\la,\tilde\chi_k}(U)^K\to \left(\cO\B_{\dR,k}^{+,\la,\tilde\chi_k}(U)/(t)\right)^{K}\to 0\] 
is an exact sequence. Moreover each term has a generalized eigenspace decomposition with respect to the action of $1\in\Q_p\cong\Lie(\Gal(K_\infty/K))$. Hence we only need to show that $0$ does not appear in the spectrum on $t\cO\B_{\dR,k}^{+,\la,\tilde\chi_k}(U)^K$, equivalently,
\[t\cO\B_{\dR,k}^{+,\la,\tilde\chi_k}(U)^{G_K}=0.\]
Note that there is a natural filtration  on $t\cO\B_{\dR,k}^{+,\la,\tilde\chi_k}(U)$ coming from the filtration on $\cO\B_{\dR}^+$. It suffices to show that
\[\left(\gr^i t\cO\B_{\dR,k}^{+,\la,\tilde\chi_k}(U)\right)^{G_K}=0,~~~~~~i\geq 0.\]
We may assume $i\in\{1,\cdots,k-1\}$ because $\gr^i\cO\B_{\dR,k}^+=0$ if $i$ is not in this range. Again by the proof of Proposition \ref{griFIl}, it follows from Faltings's extension \eqref{Fext} that
\[0\to \cO_{K^p}^{\la,\tilde{\chi}_k}(U)(1)^K\to \gr^1\cO\B_{\dR}^{+,\la,\tilde\chi_k}(U)^K\to \cO_{K^p}^{\la,\tilde{\chi}_k}(U)^K\otimes_{\cO_{V_0}}\Omega^1_{V_0}(\mathcal{C})\to 0\]
is exact and $\cO_{K^p}^{\la,\tilde{\chi}_k}(U)(1)^K = (\gr^1 t\cO\B_{\dR,k}^{+,\la,\tilde\chi_k}(U))^K$. In general, $(\gr^i t\cO\B_{\dR,k}^{+,\la,\tilde\chi_k}(U))^K$ is filtered by
\[\cO_{K^p}^{\la,\tilde{\chi}_k}(U)(j)^K\otimes_{\cO_{V_0}}\Omega^1_{V_0}(\mathcal{C})^{\otimes i-j},~~~~~~~~j=1,\cdots,i.\]
(cf. the paragraph below \eqref{Fext}.) Recall that  the eigenvalues of the Sen operator on $\cO_{K^p}^{\la,\tilde{\chi}_k}(U)^K$ are $0,-k$, cf. \ref{tilchik}. Therefore $0$ is not an eigenvalue on $\cO_{K^p}^{\la,\tilde{\chi}_k}(U)(j)^K\otimes_{\cO_{V_0}}\Omega^1_{V_0}(\mathcal{C})^{\otimes i-j}$ as $j\in\{1,\cdots,k-1\}$ by our assumption. This certainly implies our claim.
\end{proof}

Recall the construction of $d^k$ in \ref{Do1}, cf. Lemma \ref{KSsec}. The composite map
\[\omega^{-k+1,\la,(0,0)}_{K^p}(U)c^{k-1} \xrightarrow{r'_{k-1}} \Sym^{k-1} D^{(0,0)}(U) \xrightarrow{\nabla'_{k-1}} \Sym^{k-1} D^{(0,0)}(U)\otimes_{\cO_{V_0}}\Omega^1_{V_0}(\mathcal{C})\]
has images in $\Fil^{k-1} \Sym^{k-1} D^{(0,0)}(U)\otimes_{\cO_{V_0}}\Omega^1_{V_0}(\mathcal{C})$, where $\nabla'_{k-1}$ is $\cO_{\Fl}$-linear and extends the connection on $\Sym^{k-1} D$. Hence $\nabla'_{k-1}$ induces a map $d'^k:$
\[\omega^{-k+1,\la,(0,0)}_{K^p}(U)c^{k-1}\to \Fil^{k-1} \Sym^{k-1} D^{(0,0)}(U)\otimes\Omega^1_{V_0}(\mathcal{C})=\omega^{-k+1,\la,(0,0)}_{K^p}(U)c^{k-1}\otimes\Omega^1_{V_0}(\mathcal{C})^{\otimes k}.\]
Note that by \ref{kgeq2}, the $\tilde\chi_k$-part of $\omega^{-k+1,\la,(0,0)}_{K^p}(U)c^{k-1}\otimes_{\Q_p} \Sym^{k-1} V^*$ is canonically isomorphic to $\cO^{\la,(1-k,0)}_{K^p}(U)$. This shows that if we take the tensor product of $d'^k$ with $\Sym^{k-1} V^*$ and take the $\tilde\chi_k$-part, we get a map
\[\cO^{\la,(1-k,0)}_{K^p}(U)\to \cO^{\la,(1-k,0)}_{K^p}(U)\otimes \Omega^1_{V_0}(\mathcal{C})^{\otimes k}\]
which is nothing but $d^k$. In other words, recall that there is an injection
\[\psi_{k-1}: \cO^{\la,(1-k,0)}_{K^p}(U)^{G_K} \to \Sym^{k-1} D(V_0)\otimes_{\cO_{V_0}}\cO^{\la,(0,0)}_{K^p}(U)^{G_K}\otimes_{\Q_p} \Sym^{k-1} V^*
\]
cf.  \eqref{psin}. Then $d^k|_{\cO^{\la,(1-k,0)}_{K^p}(U)^{G_K}}$ is obtained by restricting $\nabla'_{k-1}\otimes 1$ to $\cO^{\la,(1-k,0)}_{K^p}(U)^{G_K}$ via $\psi_{k-1}$. For simplicity, we write $DV_{k}$ for $\Sym^{k-1} D^{(0,0)}(U)^{G_K}\otimes_{\Q_p} \Sym^{k-1} V^*$. Consider the diagram 
\[
\begin{tikzcd}
\cO^{\la,(1-k,0)}_{K^p}(U)^{G_K}  \arrow[r,"{\psi_{k-1}}"]  \arrow[rd,"s_{k+1}\mod t" ']
& DV_k   \arrow[r,"\nabla'_{k-1}\otimes 1"] \arrow[d," l'_{k-1}\mod t"] & DV_k\otimes\Omega^1_{V_0}(\mathcal{C})  \arrow[d,"{l}'_{k-1}\otimes 1\mod t"]\\
&\cO\B_{\dR,k+1}^+(U)/(t) ~~~~~~~~~~~~\arrow[r,"\nabla_{\log}~ \mathrm{ mod}~ t"] & ~~~~~~~~~~~\cO\B_{\dR,k}^+(U)/(t)\otimes_{\cO_{V_0}}\Omega^1_{V_0}(\mathcal{C})
\end{tikzcd}.
\]

\begin{lem} \label{comdiag}
This diagram is commutative.
\end{lem}

Note that this lemma will imply that 
\[\nabla_{\log}\circ s_{k+1}={d}^{k}\mod t,\]
hence $\nabla_{\log}\circ s_{k+1}={d}^{k}$  by Lemma \ref{tpowt}, which is exactly what we need to show.
\end{proof}

\begin{proof}[Proof of Lemma] \ref{comdiag}
The left triangle is commutative by the definition of $s_{k+1}$. For the right square, we observe that ${l}'_{k-1}\mod t$ is a $K$-linear mapping between $C$-LB spaces. Hence we can extend ${l}'_{k-1}\mod t$ to a $C$-linear map $l':\Sym^{k-1} D^{(0,0)}(U)\otimes_{\Q_p} \Sym^{k-1} V^*\to \cO\B_{\dR,k+1}^+(U)/(t)$. Here we use the fact $\Sym^k D^{(0,0)}(U)=\Sym^k D \otimes \cO^{\la,(0,0)}_{K^p}(U)$ is Hodge-Tate of weight $0$. Now it suffices to show
\[
\begin{tikzcd}
\Sym^{k-1} D^{(0,0)}(U)\otimes_{\Q_p} \Sym^{k-1} V^*   \arrow[r,"\nabla'_{k-1}"] \arrow[d,"{l}'"] & \Sym^{k-1} D^{(0,0)}(U) \otimes_{\cO_{V_0}}\Omega^1_{V_0}(\mathcal{C}) \otimes_{\Q_p} \Sym^{k-1} V^* \arrow[d,"l'\otimes 1"]\\
\cO\B_{\dR,k+1}^+(U)/(t) \arrow[r,"\nabla_{\log}\mod t"] & \cO\B_{\dR,k}^+(U)/(t)\otimes_{\cO_{V_0}}\Omega^1_{V_0}(\mathcal{C})
\end{tikzcd}
\]
is commutative, i.e. $(\nabla_{\log}\mod t)\circ l'=(l'\otimes 1)\circ\nabla'_{k-1}$. (Here we use $\nabla'_{k-1}$ instead of $\nabla'_{k-1}\otimes 1$ for simplicity.) First, this is clearly true when restricted to $\Sym^{k-1} D\otimes_{\Q_p} \Sym^{k-1} V^*\subseteq \Sym^{k-1} D^{(0,0)}(U)\otimes_{\Q_p} \Sym^{k-1} V^*$, cf. \ref{reldRcomp}. Second, all of the maps in the diagram commute with multiplication by $x\in\cO^{\la,(0,0)}_{K^p}$: 
\begin{itemize}
\item for $\nabla'_{k-1}$, this is because $\nabla'_{k-1}$ is $\cO_{\Fl}$-linear;
\item for $l'$, this is because $l'$ is $\cO^{\la,(0,0)}_{K^p}(U)$-linear by our construction;
\item for $\nabla_{\log}\mod t$, this is the last part of Proposition \ref{phi1}.
\end{itemize}
Third, both connections $(\nabla_{\log}\mod t)$ and $\nabla'_{k-1}$ satisfy the Leibniz rule for  functions on $V_{K_p,L}$ for any $K_p,L$. See \ref{PlsFe} for $\nabla_{\log}$ and the proof of Theorem \ref{I1} for $\nabla'_{k-1}$.
Now our claim follows from the observation that all the maps are continuous and elements of the form 
\[\sum_{i=0}^n x^i \sum_{j=1}^m a_{ij}b_j\in\Sym^{k-1} D^{(0,0)}(U),~~~~~~~~~~~a_{ij}\in \varinjlim_{K_p,L}\cO_{V_{K_p,L}}(V_{K_p,L}),b_j\in \Sym^{k-1} D(V_0),\]
are dense in $\Sym^{k-1} D^{(0,0)}(U)$ by the explicit description of $\cO^{\la,(0,0)}_{K^p}(U)$ in Theorem \ref{str}.
\end{proof}
  
\subsection{Proof of Theorem \ref{MT} III}
\begin{para}
In this subsection, we prove Proposition \ref{comdbark}. Keep the same notation in \ref{PfMTI}. We may assume that $e_1$ is invertible on $V_\infty$. First we recall the definition of $N'_k$. Consider the locally analytic vectors in the $k$-th graded piece of the Poincar'e lemma sequence
\begin{eqnarray} \label{grkPls}
~~~~~~~~0\to \gr^k \B^{+,\la}_{\dR}(U)\to \gr^k \cO\B_{\dR}^{+,\la}(U) \xrightarrow{\gr^k\nabla_{\log}} \gr^{k-1}\cO\B_{\dR}^{+,\la}(U)\otimes\Omega^1_{V_0}(\mathcal{C})\to 0.
\end{eqnarray}
Take the $\tilde\chi_k$-isotypic part  and the $G_{K_\infty}$-fixed, $G_K$-analytic vectors of this sequence.
\[0\to \gr^k \B^{+,\la,\tilde\chi_k}_{\dR}(U)^K\to \gr^k \cO\B_{\dR}^{+,\la,\tilde\chi_k}(U)^K \to \gr^{k-1}\cO\B_{\dR}^{+,\la,\tilde\chi_k}(U)^K\otimes_{\cO_{V_0}}\Omega^1_{V_0}(\mathcal{C})\to 0.\]
Note that $\Lie(\Gal(K_\infty/K))$ acts semi-simply on every term except the middle one, cf. \ref{constdiag}. The action of $1\in\Q_p\cong \Lie(\Gal(K_\infty/K))$ induces a map between the kernels of $\Lie(\Gal(K_\infty/K))$ , i.e the $G_K$-invariants,
\[N'_k:\gr^{k-1}\cO\B_{\dR}^{+,\la,\tilde\chi_k}(U)^{G_K}\otimes_{\cO_{V_0}}\Omega^1_{V_0}(\mathcal{C}) \to 
\gr^k \B^{+,\la,\tilde\chi_k}_{\dR}(U)^{G_K}.\]
We obtain the form of $N'_k$ in Proposition \ref{comdbark} by identifying $\gr^{k-1}\cO\B_{\dR}^{+,\la,\tilde\chi_k}(U)^{G_K}$ with $\cO^{\la,(1-k,0)}_{K^p}(U)^{G_K}$ and $\gr^{k-1}\cO\B_{\dR}^{+,\la,\tilde\chi_k}(U)^{G_K}\otimes_{\cO_{V_0}}\Omega^1_{V_0}(\mathcal{C})$ with $\cO^{\la,(1-k,0)}_{K^p}(U)^{G_K}\otimes \Omega^1_{V_0}(\mathcal{C})^{\otimes k}$.

In other words, $N'_k$ is  obtained by looking at the action of $G_K$ on the $\tilde\chi_k$-isotypic parts. It turns out that we can ``switch'' the roles of $G_K$ and $Z$ here. In fact, we will show that $N'_k$ can also be obtained by looking at the action of $Z$ on the $G_K$-invariants up to a non-zero constant. Then it will be quite straightforward to relate it  with $\bar{d}^k$ using this representation-theoretic interpretation of $N'_k$.

We introduce some notations. Let $\mathfrak{z}=\{\begin{pmatrix} a & 0\\ 0 & a\end{pmatrix}\}$ be the centre of $\mathfrak{gl}_2(\Q_p)$. Then there is a natural inclusion $U(\mathfrak{z})\subseteq Z$ and $\tilde\chi_k$ induces a character $z_k:\mathfrak{z}\to\Q_p$. For a $U(\mathfrak{z})$-module $M$, we denote by $M^{z_k}$ the $z_k$-isotypic part of $M$.
\end{para}

\begin{prop} \label{GKinvex}
The sequence \eqref{grkPls} remains exact after taking the $G_K$-invariants and $z_k$-isotypic parts, i.e. we have a short exact sequence
\[
0\to \gr^k \B^{+,\la}_{\dR}(U)^{G_K,z_k}\to \gr^k \cO\B_{\dR}^{+,\la}(U)^{G_K,z_k} \to \gr^{k-1}\cO\B_{\dR}^{+,\la}(U)^{G_K,z_k}\otimes \Omega^1_{V_0}(\mathcal{C})\to 0.
\]
\end{prop}

\begin{proof}
When $k=1$,  the exact sequence \eqref{grkPls} is simply
\begin{eqnarray} \label{Fextla}
0\to \cO^{\la}_{K^p}(U)(1)\to \gr^1 \cO\B_{\dR}^{+,\la}(U)\to \cO^{\la}_{K^p}(U)\otimes_{\cO_{V_0}}\Omega^1_{V_0}(\mathcal{C})\to 0.
\end{eqnarray}

\begin{lem} \label{GKzspli}
The surjective map $ \gr^1 \cO\B_{\dR}^{+,\la}(U)\to \cO^{\la}_{K^p}(U)\otimes_{\cO_{V_0}}\Omega^1_{V_0}(\mathcal{C})$ has a $G_K\times U(\mathfrak{z})$-equivariant section. Equivalently, 
\[\gr^1 \cO\B_{\dR}^{+,\la}(U)\cong  \cO^{\la}_{K^p}(U)(1)\oplus \cO^{\la}_{K^p}(U)\otimes_{\cO_{V_0}}\Omega^1_{V_0}(\mathcal{C}) \]
as $G_K\times U(\mathfrak{z})$-modules.
\end{lem}

Clearly this lemma implies our claim when $k=1$. It also implies the general case  because
\[\gr^k \cO\B_{\dR}^{+,\la}\cong \Sym^k \gr^1\cO\B_{\dR}^{+,\la}.\]
\end{proof}

\begin{proof}[Proof of Lemma \ref{GKzspli}]
Faltings proved that  (the locally analytic vectors in) Faltings's  extension \eqref{Fextla} are (up to a sign) isomorphic to the (locally analytic vectors in) a twist of the relative Hodge-Tate filtration sequence \eqref{rHT}
\begin{eqnarray} \label{twrHT}
0\to \cO^{\la}_{K^p}(U)(1)\xrightarrow{p} V(1)\otimes_{\Q_p} \omega^{1,\la}_{K^p}(U) \xrightarrow{q}\omega^{2,\la}_{K^p}(U)\to 0,
\end{eqnarray}
after identifying $\omega^2(V_0)$ with $\Omega^1_{V_0}(\mathcal{C})$ using the Kodaira-Spencer isomorphism, cf. \cite[Theorem 4.2.2]{Pan20}. Here, we choose a trivialization of $\wedge^2 D$ which is implicitly done in the reference, cf. \ref{XCY}, \ref{KSsm}. 

The surjective map $V(1)\otimes_{\Q_p} \omega^{1,\la}_{K^p}(U) \to \omega^{2,\la}_{K^p}(U)$ sends the standard basis of $V=\Q_p^2$ to $e_1,e_2$, cf. \ref{exs}. (Recall that we fix a basis of $\Q_p(1)$ in \ref{exs}.) Then a $G_K\times U(\mathfrak{z})$-equivariant section is given by mapping $f\in \omega^{2,\la}_{K^p}(U)$ to $(1,0)\otimes \frac{f}{e_1}\in V(1)\otimes_{\Q_p} \omega^{1,\la}_{K^p}(U)$.
\end{proof}

We need to understand how $Z$ acts on the exact sequence in Proposition \ref{GKinvex}. By Harish-Chandra's isomorphism, $Z$ is generated by $z=\begin{pmatrix} 1 & 0 \\ 0 & 1 \end{pmatrix} \in \mathfrak{z}$ and the Casimir operator
\[\Omega=\begin{pmatrix} 0 & 1 \\  0 & 0\end{pmatrix}\begin{pmatrix} 0 & 0 \\ 1 & 0 \end{pmatrix}+\begin{pmatrix} 0 & 0 \\  1 & 0\end{pmatrix}\begin{pmatrix} 0 & 1 \\ 0 & 0 \end{pmatrix}+\frac{1}{2}\begin{pmatrix} 1 & 0 \\  0 & -1\end{pmatrix}^2\in U(\mathfrak{gl}_2(\Q_p)).\]
Recall that we identify a character $\tilde\chi:Z\to \Q_p$ with a pair of weights  of the form $\{(a,b),(b-1,1+a)\}$ in \ref{tilchik}. Then $\tilde\chi(z)=a+b$ and $\tilde\chi(\Omega)=\frac{1}{2}(a-b+1)^2-\frac{1}{2}$.
For example, $\tilde\chi_k(\Omega)=\frac{1}{2}k^2-\frac{1}{2}$ and $\tilde\chi_k(z)=1-k$. 

\begin{prop} \label{SenopCasz}
Let $\theta_{\Sen}$ denote the action of $1\in\Q_p\cong\Lie(\Gal(K_\infty/K))$ on $\cO^{\la}_{K^p}(U)^K$. Then it satisfies the following quadratic relation
\[(2\theta_\Sen-z+1)^2-1=2\Omega\]
\end{prop}

\begin{proof}
By \cite[Theorem 5.1.8]{Pan20}, $\theta_\Sen=\theta_\kh(\begin{pmatrix} 0 & 0 \\ 0 & 1\end{pmatrix})$. Hence this result follows directly from the relation between $\theta_\kh$ and the action of $Z$, cf. \ref{tilchik}, \cite[Corollary 4.2.8]{Pan20}.
\end{proof}

\begin{cor}\label{ZactonGKinv}
\begin{enumerate}
\item
$Z$ acts on $\gr^k \B^{+,\la}_{\dR}(U)^{G_K,z_k}$ via $\tilde\chi_k$, i.e. 
\[\gr^k \B^{+,\la}_{\dR}(U)^{G_K,z_k}=\gr^k\B_{\dR}^{+,\la,\tilde\chi_k}(U)(k)^{G_K}=\cO^{\la,(1,-k)}_{K^p}(U)(k)^{G_K}.\]
\item The action of $Z$ on $\gr^{k-1}\cO\B_{\dR}^{+,\la}(U)^{G_K,z_k}$ has a generalized eigenspace decomposition. The $\tilde\chi_k$-generalized eigenspace  is equal to the $\tilde\chi_k$-eigenspace. In particular,
\[\gr^{k-1}\cO\B_{\dR}^{+,\la}(U)^{G_K,z_k,\tilde\chi_k}=\gr^{k-1}\cO\B_{\dR}^{+,\la,\tilde\chi_k}(U)^{G_K}=\cO^{\la,(1-k,0)}_{K^p}(U)^{G_K} \otimes\Omega^1_{V_0}(\mathcal{C})^{\otimes k-1}.\]
\end{enumerate}
\end{cor}

\begin{proof}
For the first part, we note that  $\gr^k \B^{+,\la}_{\dR}(U)^{G_K,z_k}=\cO^{\la}_{K^p}(U)(k)^{G_K,z_k}$, hence  it follows directly from  Proposition \ref{SenopCasz}. For the second part, Proposition \ref{griFIl} implies that 
$\gr^i \cO\B_{\dR,k-1}^{+,\la,\tilde\chi_l}(U)^{G_K,z_k}$  is naturally  filtered by $U(\mathfrak{gl}_2(\Q_p))$- submodules 
\[ \cO^{\la}_{K^p}(U)(j)^{G_K,z_k}\otimes_{\cO_{V_0}}\Omega^1_{V_0}(\mathcal{C})^{\otimes k-1-j},~~~j=0,\cdots,k-1.\]
Again by Proposition \ref{SenopCasz}, we see that $\Omega$ acts on $\cO^{\la}_{K^p}(U)(j)^{G_K,z_k}\otimes_{\cO_{V_0}}\Omega^1_{V_0}(\mathcal{C})^{\otimes k-1-j}$ via $\frac{1}{2}(k-2j)^2-\frac{1}{2}$, which is equal to $\frac{1}{2}k^2-\frac{1}{2}$ only when $j=0$.
\end{proof}

\begin{para}
Now we can take the $\tilde\chi_k$-generalized eigenspace $E_{\tilde\chi_k}$ of $\gr^k \cO\B_{\dR}^{+,\la}(U)^{G_K,z_k}$, or equivalently, by Corollary \ref{ZactonGKinv}, this is the same as  the kernel of $(\Omega-\tilde\chi_k(\Omega))^2$. By  Proposition \ref{GKinvex} and Corollary \ref{ZactonGKinv}, it sits inside the exact sequence 
\[0\to \cO^{\la,(1,-k)}_{K^p}(U)(k)^{G_K} \to E_{\tilde\chi_k}\to \cO^{\la,(1-k,0)}_{K^p}(U)^{G_K} \otimes\Omega^1_{V_0}(\mathcal{C})^{\otimes k}\to 0.\]
The action of $\Omega-\tilde\chi_k(\Omega)$ on $E_{\tilde\chi_k}$ induces a natural map
\[\tilde{N}_k:\cO^{\la,(1-k,0)}_{K^p}(U)^{G_K} \otimes\Omega^1_{V_0}(\mathcal{C})^{\otimes k}\to \cO^{\la,(1,-k)}_{K^p}(U)(k)^{G_K}.\]
It is continuous, hence can be extended $C$-linearly to a map (which is also denoted by $\tilde{N}_k$ by abuse of notation)
\[\tilde{N}_k:\cO^{\la,(1-k,0)}_{K^p}(U)\otimes\Omega^1_{V_0}(\mathcal{C})^{\otimes k}\to \cO^{\la,(1,-k)}_{K^p}(U)(k).\]
\end{para}

\begin{lem}
$N'_k=\frac{1}{2k}\tilde{N}_k$.
\end{lem}

\begin{proof}
Let $f\in \cO^{\la,(1-k,0)}_{K^p}(U)^{G_K} =\gr^{k-1}\cO\B_{\dR}^{+,\la,\tilde\chi_k}(U)^{G_K}\otimes \Omega^1_{V_0}(\mathcal{C})$. By definition, $N'_k(f)$ is computed as follows: we may choose $f_1\in \gr^k \cO\B_{\dR}^{+,\la,\tilde\chi_k}(U)^K$ such that 
\begin{itemize}
\item $\theta_K^2 (f_1)=0$, where $\theta_K$ denotes $1\in\Q_p\cong\Lie(\Gal(K_\infty/K))$;
\item $\gr^k\nabla_{\log}(f_1)=f$, where  $\gr^k\nabla_{\log}:\gr^k \cO\B_{\dR}^{+,\la}(U) \to \gr^{k-1}\cO\B_{\dR}^{+,\la}(U)\otimes \Omega^1_{V_0}(\mathcal{C})$, cf. \eqref{grkPls}.
\end{itemize}
Then $N'_k(f)=\theta_K(f_1)$.
Similarly, $\tilde{N}_k(f)$ is computed as follows: we may choose $f_2\in \gr^k \cO\B_{\dR}^{+,\la}(U)^{G_K,z_k}$ such that 
\begin{itemize}
\item $(\Omega-\tilde\chi_k(\Omega))^2\cdot f_2=0$;
\item $\gr^k\nabla_{\log}(f_2)=f$.
\end{itemize}
Then $\tilde{N}_k(f)=(\Omega-\tilde\chi_k(\Omega))\cdot f_2$. 

Consider $f_c=f_1-f_2\in \gr^k \cO\B_{\dR}^{+,\la}(U)^{z_k,K}$. Since $\gr^k\nabla_{\log}(f_c)=0$, by  \eqref{grkPls}, we have 
\[f_c\in \gr^k \B^{+,\la}_{\dR}(U)^{z_k,K}=\cO^{\la}_{K^p}(U)(k)^{z_k,K}. \]
By Proposition \ref{SenopCasz}, there is the quadratic relation
\[(2(\theta_K-k)-(1-k)+1)^2-1=2\Omega\]
on $\cO^{\la}_{K^p}(U)(k)^{z_k,K}$. Apply it to $f_c$. A simple computation shows that
\[2\theta_K^2\cdot f_c-2k\theta_K(f_c)=(\Omega-\tilde\chi_k(\Omega))\cdot f_c\]
Note that $\theta_K^2$ annihilates $f_1$ and $\theta_K$ annihilates $f_2$. Hence $\theta_{K}^2(f_c)=0$. On the other hand,
\begin{itemize}
\item $\theta_K(f_c)=\theta_K(f_1)=N'_k(f)$ as $\theta_K$ annihilates $f_2$;
\item $(\Omega-\tilde\chi_k(\Omega))\cdot f_c=-(\Omega-\tilde\chi_k(\Omega))\cdot f_2=-\tilde{N}_k(f)$ as $f_1$ has infinitesimal character $\tilde\chi_k$.
\end{itemize}
We conclude that 
\[N'_k(f)=\frac{1}{2k}\tilde{N}_k(f).\]
\end{proof}

\begin{para}
It remains to compare $\tilde{N}_k$ with $\bar{d}'^k$. Note that both are continuous maps
\[\cO^{\la,(1-k,0)}_{K^p}(U)\otimes\Omega^1_{V_0}(\mathcal{C})^{\otimes k}\cong \omega^{2k,\la,(1-k,0)}_{K^p}(U)\to \cO^{\la,(1,-k)}_{K^p}(U)(k)\]
which are $\cO^{\sm}_{K^p}(U)$-linear.  Here we use the Kodaira-Spencer isomorphism and our fixed trivialization of $\wedge^2 D$, cf. \ref{HEt}. Recall that in Corollary \ref{density}, we introduced 
\[\omega^{2k,\la,(1-k,0)}_{K^p}(U)^{\kn-fin}\subseteq \omega^{2k,\la,(1-k,0)}_{K^p}(U)\]
which consist of elements of the form
$(\frac{e_1}{\mathrm{t}})^{k-1}\sum_{i=0}^n a_ix^i,~~~~~a_i\in\omega^{k+1,\sm}_{K^p}(U)$. There is a natural $\mathfrak{gl}_2(\Q_p)$-equivariant isomorphism 
\[\omega^{2k,\la,(1-k,0)}_{K^p}(U)^{\kn-fin}=\omega^{k+1,\sm}_{K^p}(U)\otimes_{\Q_p} M^{\vee}_{(0,1-k)},\]
where  $M^{\vee}_{(0,k-1)}\subseteq \omega^{k-1}_{K^p}(U)$ consists of $(\frac{e_1}{\mathrm{t}})^{k-1}\sum_{i=0}^n a_ix^i,~~~~~a_i\in\Q_p$. Similarly, we have 
\[\cO^{\la,(1,-k)}_{K^p}(U)^{\kn-fin}\subseteq \cO^{\la,(1,-k)}_{K^p}(U)\]
and 
\[\cO^{\la,(1,-k)}_{K^p}(U)^{\kn-fin}=\omega^{k+1,\sm}_{K^p}(U)\otimes_{\Q_p} M^{\vee}_{(-k,1)},\]
where $M^{\vee}_{(-k,1)}$ consists of $\mathrm{t}e_1^{-k-1}\sum_{i=0}^n a_ix^i,~~~~~a_i\in\Q_p$.
\end{para}

\begin{lem} \label{redVerma}
Both $\tilde{N}_k$ and $\bar{d}'^k$ send $\omega^{2k,\la,(1-k,0)}_{K^p}(U)^{\kn-fin}$ to $\cO^{\la,(1,-k)}_{K^p}(U)^{\kn-fin}(k)$, i.e.
\[\tilde{N}_k,~\bar{d}'^k: \omega^{k+1,\sm}_{K^p}(U)\otimes_{\Q_p} M^{\vee}_{(0,1-k)}\to \omega^{k+1,\sm}_{K^p}(U)\otimes_{\Q_p} M^{\vee}_{(-k,1)}(k).\]
Moreover, both $\cO^{\sm}_{K^p}(U)$-linear maps are induced from maps
\[M^{\vee}_{(0,1-k)}\to M^{\vee}_{(-k,1)}(k)\]
which by abuse of notation are still denoted by $\tilde{N}_k$ and $\bar{d}'^k$.
\end{lem}

\begin{proof}
We first compute $\bar{d}'^k$.
By Theorem \ref{I2}, for $a_n\in\omega^{k+1,\sm}_{K^p}(U)$, there exists a constant $\bar{c}\in\Q^\times$ such that 
\[\bar{d}^k(a_n(\frac{e_1}{\mathrm{t}})^{k-1}x^n)=\bar{c}a_n(\frac{e_1}{\mathrm{t}})^{k-1}(\frac{d}{dx})^k(x^n)(dx)^k=\bar{c}a_n(\frac{e_1}{\mathrm{t}})^{k-1} \binom{n}{k}x^{n-k} (dx)^k.\]
Recall that we identify $dx$ with $\frac{\mathrm{t}}{e_1^2}$ in \ref{HEt}. Hence
\[\bar{d}'^k(a_n(\frac{e_1}{\mathrm{t}})^{k-1}x^n)=\bar{c}a_n \binom{n}{k}\frac{\mathrm{t}x^{n-k}}{e_1^{k+1}}\in\cO^{\la,(1,-k)}_{K^p}(U)^{\kn-fin}(k),\]
i.e. $\bar{d}'^k((\frac{e_1}{\mathrm{t}})^{k-1}x^n)=\bar{c} \binom{n}{k}{\mathrm{t}e_1^{-k-1}x^{n-k}}$.

Next we consider $\tilde{N}_k$. By identifying Faltings's extension with the relative Hodge-Tate filtration sequence (cf. the proof of \ref{GKzspli}), we can identify $\gr^k \cO\B_{\dR}^{+,\la}(U)^{G_K,z_k}$ with
\[\left(\Sym^k V\otimes_{\Q_p}\omega^{k,\la}_{K^p}(U)(k)\right)^{G_K,z_k}=\Sym^k V\otimes_{\Q_p} \omega^{k,\la,(1-k,-k)}_{K^p}(U)(k)^{G_K}\]
so that 
\begin{itemize}
\item  the injection $i:\cO^{\la,(1,-k)}_{K^p}(U)(k)^{G_K}\to \gr^k \cO\B_{\dR}^{+,\la}(U)^{G_K,z_k}$ is identified with the map $\Sym^k p|_{\cO^{\la,(1,-k)}_{K^p}(U)(k)^{G_K}}$, where $p$ is the injective map in \eqref{twrHT};
\item the surjection $j:\gr^k \cO\B_{\dR}^{+,\la}(U)^{G_K,z_k}\to \omega^{2k,\la,(1-k,0)}_{K^p}(U)^{G_K}$ is getting identified with the $k$-th symmetric power of the surjective map $q$ in \eqref{twrHT} (when restricted to the $G_K$-invariants and $z_k$-isotypic parts) up to a sign.
\end{itemize}
Extending everything $C$-linearly, we get
\begin{eqnarray}\label{kthsymwall}
~~~~~0\to\cO^{\la,(1,-k)}_{K^p}(k)\xrightarrow{i} \Sym^k V\otimes \omega^{k,\la,(1-k,-k)}_{K^p}(k)\xrightarrow{j} \omega^{2k,\la,(1-k,0)}_{K^p}\to 0
\end{eqnarray}
which becomes exact when passing to the $\tilde\chi_k$-generalized eigenspaces.

\begin{lem}
The sequence \eqref{twrHT} induces an exact sequence
\[0\to\cO^{\la,(1,-1)}_{K^p}(U)^{\kn-fin}(1)\to V(1)\otimes_{\Q_p} \omega^{1,\la,(0,-1)}_{K^p}(U)^{\kn-fin} \to \omega^{2,\la,(0,0)}_{K^p}(U)^{\kn-fin}\to 0.\]
Moreover under the isomorphism $\cO^{\la,(1,-1)}_{K^p}(U)^{\kn-fin}=\omega^{2,\sm}_{K^p}(U)\otimes_{\Q_p} M^{\vee}_{(-1,1)}$ and similar isomorphisms for $\omega^{1,\la,(0,-1)}_{K^p}(U)^{\kn-fin},\omega^{2,\la,(0,0)}_{K^p}(U)^{\kn-fin}$, this exact sequence is the tensor product of $\omega^{2,\sm}_{K^p}(U)$ with an exact sequence
\[0\to M^{\vee}_{(-1,1)}(1)\to V(1)\otimes_{\Q_p} M^{\vee}_{(-1,0)} \to M^{\vee}_{(0,0)}\to 0.\]
Similar results also hold for the $k$-th symmetric power of the sequence \eqref{twrHT}.
\end{lem}

\begin{proof}
In \eqref{twrHT}, we have $p(f)=(\frac{fe_2}{\mathrm{t}},-\frac{fe_1}{\mathrm{t}})$ up to a sign and $q(f_1,f_2)=e_1f_1+e_2f_2$ for $f\in\cO^{\la}_{K^p}(U)(1)$ and $f_1,f_2\in \omega^{1,\la}_{K^p}(U)(1)$. This implies all the claims here.
\end{proof}

By this lemma, there is a $\mathfrak{gl}_2(\Q_p)$-equivariant sequence
\begin{eqnarray} \label{wallcr}
0\to M^{\vee}_{(-k,1)}(k) \to \Sym^k V\otimes_{\Q_p} M^{\vee}_{(-k,1-k)}(k)\to M^{\vee}_{(0,1-k)}\to 0
\end{eqnarray}
which is induced from  the $k$-th symmetric power of \eqref{twrHT} and becomes exact after passing to the $\tilde\chi_k$-generalized eigenspaces. Its tensor product with $\omega^{k+1,\sm}_{K^p}(U)$ is naturally identified with taking $\kn-fin$ vectors  of \eqref{kthsymwall} on $U$. Note that $\tilde{N}_k$ is essentially obtained by applying the action of $Z$ to the middle term of \eqref{kthsymwall}. It follows from our discussion here that $\tilde{N}_k$ is induced from a map $M^{\vee}_{(0,1-k)}\to M^{\vee}_{(-k,1)}(k)$.
\end{proof}

\begin{para} \label{CatO}
Now we can prove Proposition \ref{comdbark}.
As we explained before, we only need to compare $\tilde{N}_k$ and $\bar{d}'^k$.
By Corollary \ref{density}, $\omega^{2k,\la,(1-k,0)}_{K^p}(U)^{\kn-fin}$ is a  dense subset of $\omega^{2k,\la,(1-k,0)}_{K^p}(U)$. Since both $\tilde{N}_k$ and $\bar{d}'^k$ are continuous, by Lemma \ref{redVerma}, it suffices to prove that
\[\tilde{N}_k,\bar{d}'^k:M^{\vee}_{(0,1-k)}\to M^{\vee}_{(-k,1)}(k)\]
differ by a non-zero constant.

There are natural actions of $\mathfrak{gl}_2(\Q_p)$ on $M^{\vee}_{(0,1-k)}$ and $M^{\vee}_{(-k,1)}(k)$ and  both maps $\tilde{N}_k,\bar{d}'^k$ are $\mathfrak{gl}_2(\Q_p)$-equivariant. We shall ignore the Tate twist here and only restrict to the action of $\mathfrak{sl}_2(\Q_p)$. Recall that $M^{\vee}_{(0,1-k)}=\{(\frac{e_1}{\mathrm{t}})^{k-1}\sum_{i=0}^n a_ix^i,~~~~~a_i\in\Q_p\}$. The action of $u^+=\begin{pmatrix} 0 & 1 \\ 0 & 0 \end{pmatrix}$ corresponds to the derivation with respect to $x$. From this, it is easy to see that 
\begin{itemize}
\item $M^{\vee}_{(0,1-k)}$ is the dual  of the Verma module of $\mathfrak{sl}_2(\Q_p)$ with highest weight $k-1$;
\item Similarly, $M^{\vee}_{(-k,1)}$  is the dual  of the Verma module with highest weight $-k-1$ (which is also isomorphic to the Verma module with highest weight $-k-1$).
\end{itemize}
Note that $k\geq 1$, hence  $-k-1$ is anti-dominant and $k-1$ is dominant. This shows that $M^{\vee}_{(-k,1)}$  is irreducible and $M^{\vee}_{(0,1-k)}$ is an extension of $M^{\vee}_{(-k,1)}$ by a finite-dimensional irreducible representation of $\mathfrak{sl}_2(\Q_p)$. Therefore,
\[\dim_{\Q_p}\Hom_{U(\mathfrak{sl}_2(\Q_p))}(M^{\vee}_{(0,1-k)},M^{\vee}_{(-k,1)})=1.\]
Since both $\tilde{N}_k$ and $\bar{d}'^k$ belong to this Hom space, it suffices to show the following lemma.
\end{para}

\begin{lem}
Both maps $\tilde{N}_k,\bar{d}'^k$ are non-zero.
\end{lem}

\begin{proof}
For $\bar{d}'^k$, this follows from the explicit formula of $\bar{d}'^k$ in Theorem \ref{I2}. For $\tilde{N}_k$, recall that it is induced from the action of $Z$ on  the $\tilde\chi_k$-generalized eigenspace $E_k$ of $\Sym^k V\otimes_{\Q_p} M^{\vee}_{(-k,1-k)}$. Note that $M^{\vee}_{(-k,1-k)}$ is the Verma module of $\mathfrak{sl}_2(\Q_p)$ with highest weight $-1=-\rho$. Hence in the language of translation functors, $E$ is the translation of $M^{\vee}_{(-k,1-k)}$ from  $-1$ to $-k-1$. It is well-known that $E_k$ is a projective envelop of $M^{\vee}_{(-k,1)}$ in category $\mathcal{O}$ and the action of $Z$ on $E_k$ is non-semisimple. See for example \cite[3.12]{Hum08}. For reader's convenience, we sketch a proof here. When $k=1$, then $E_1=V\otimes_{\Q_p} M^{\vee}_{(-1,0)}$ and a simple computation shows that $Z$ acts non-semisimply on it. This implies the general case $k\geq2$ because $E_k$ is the translation of $E_1$ from $-2$ to $-k-1$ and this translation functor is an equivalence of category between $\cO_{-2}$ and $\cO_{-k-1}$, cf. \cite[7.8]{Hum08}.
\end{proof}

\begin{rem} \label{Qstr}
Clearly $M^{\vee}_{(i,j)}$ has a natural $\Q$-structure preserved by the action of $\mathfrak{sl}_2(\Q)\subseteq\mathfrak{sl}_2(\Q_p)$. Hence this proof actually shows that  $\tilde{N}_k,\bar{d}^k$  differ by a constant in $\Q$. Therefore the constant $c_k$ in Proposition \ref{comdbark} can be chosen to be in $\Q^\times$.
\end{rem}

\section{Regular de Rham representations in the completed cohomology} \label{mrs}
In this section, we will combine results obtained in the previous two sections and study regular de Rham representations appearing in the completed cohomology. 

\subsection{Completed cohomology}
\begin{para}
In order to state our main result, we first recall the completed cohomology of modular curves introduced by Emerton \cite{Eme06}. Let $K^p=\prod_{l\neq p}K_l$ be an open compact subgroup of $\GL_2(\A_f^p)$.  Set
\[\tilde{H}^i(K^p,\Z/p^n):=\varinjlim_{K_p\subseteq\GL_2(\Q_p)} H^i(X_{K^pK_p}(\bC),\Z/p^n),\]
where $K_p$ runs through all open compact subgroups of $\GL_2(\Q_p)$. Recall that $X_{K^pK_p}$ denotes the complete modular curve of level $K^pK_p$ over $\Q$. The completed cohomology of tame level $K^p$ is defined as
\[\tilde{H}^i(K^p,\Z_p):=\varprojlim_n \tilde{H}^i(K^p,\Z/p^n).\]
In general, its torsion part is annihilated by $p^n$ for some $n$ by \cite[Theorem 1.1]{CE12}. But in the case of modular curves, it is well-known that $\tilde{H}^i(K^p,\Z_p)$ is torsionfree.
Suppose $W$ is a $\Q_p$-Banach space. We denote by 
\[\tilde{H}^i(K^p,W):=\tilde{H}^i(K^p,\Z_p)\widehat\otimes_{\Z_p} W.\]
This is naturally a $\Q_p$-Banach space equipped with a continuous action of $\GL_2(\Q_p)$. If $W=\Q_p$,  then $\tilde{H}^i(K^p,\Q_p)$ is an admissible representation. Since $X_{K^pK_p}$ is defined over $\Q$, using the \'etale cohomology, we have  a natural Galois action of $G_{\Q}$ on everything, which commutes with the action of $\GL_2(\Q_p)$. 

Recall that $\kn=\{\begin{pmatrix} 0 & *\\ 0 & 0 \end{pmatrix}\}\subseteq \mathfrak{gl}_2(\Q_p)=\Lie(\GL_2(\Q_p))$.
\end{para}

\begin{thm} \label{MTCH}
Let $E$ be a finite extension of $\Q_p$ and 
\[\rho:G_{\Q}\to\GL_2(E)\]
be a two-dimensional continuous absolutely irreducible representation of $G_{\Q}$. By abuse of notation, we  also use $\rho$ to denote its underlying representation space. Suppose that
\begin{enumerate}
\item $\rho$ appears in $\tilde{H}^1(K^p,E)$, i.e. $\Hom_{E[G_{\Q}]}(\rho,\tilde{H}^1(K^p,E))\neq 0$. We denote by 
\[\tilde{H}^1(K^p,E)[\rho]\subseteq \tilde{H}^1(K^p,E)\] 
the image of the evaluation map $\rho\otimes_{E}\Hom_{E[G_{\Q}]}(\rho,\tilde{H}^1(K^p,E))\to \tilde{H}^1(K^p,E)$.
\item $\rho|_{G_{\Q_p}}$ is de Rham of Hodge-Tate weights $0,k$ for some integer $k>0$.
\end{enumerate}
Then 
\begin{enumerate}
\item (classicality) $\rho$ arises from a cuspidal eigenform of weight $k+1$.
\item (finiteness) Let $K_p=1+p^nM_2(\Z_p)$ for some $n\geq 2$. Denote by  $\tilde{H}^1(K^p,E)[\rho]^{K_p-\an}\subseteq \tilde{H}^1(K^p,E)[\rho]$ the subspace of $K_p$-analytic vectors. Then the subspace of $\kn$-invariants
\[\tilde{H}^1(K^p,E)[\rho]^{K_p-\an,\kn}\]
is a finite-dimensional vector space over $E$.
\end{enumerate}
\end{thm}

A description of $\tilde{H}^1(K^p,E)^{\la}[\rho]\widehat\otimes_{\Q_p} C$ as a representation of $\GL_2(\Q_p)$ will be given in Subsection \ref{rbpLLC} below.

\begin{rem} \label{relDPS}
The main result of  \cite{DPS20} (see also \cite[Proposition 6.1.5]{Pan20}) says that $\tilde{H}^1(K^p,E)[\rho]^{\la}$ has an infinitesimal character. In their follow-up work \cite{DPS22}, Dospinescu-Pa\v{s}k\=unas-Schraen  show that this is actually enough to deduce our finiteness result. %More precisely,  they show that if $W$ is a unitary admissible Banach representation of $\GL_2(\Q_p)$ and $W^{\la}$ admits an infinitesimal character, then $W^{K_p-\an.\kn}$ is finite-dimensional. 
In fact, they even prove a finiteness result on the $K_p$-analytic vectors, cf. \cite[Corolllary 6.7]{DPS22}. Their method is purely representation-theoretic.
\end{rem}

\begin{para}[Hecke algebra]
The proof of Theorem \ref{MTCH} will be given in the next subsection. To better interpret that $\rho$ arises from a cuspidal eigenform, we need the (big) Hecke algebra.  See \cite[6.1.1]{Pan20} for more details. Let $S$ be a finite set of rational primes containing $p$ such that $K_l\subseteq\GL_2(\Q_l)$ is maximal for $l\notin S$. We have the usual Hecke operators $T_l$ and $S_l$, cf.  \ref{Hd}.  For an open compact subgroup $K_p$ of $\GL_2(\Q_p)$, we define
\[\T(K^pK_p)\subseteq\End_{\Z_p}(H^1(X_{K^pK_p}(\bC),\Z_p))\]
as the $\Z_p$-subalgebra generated by Hecke operators $T_l,S_l^{\pm1},l\notin S$, which does not depend on the choice of $S$. The (big) Hecke algebra of tame level $K^p$ is defined as the projective limit
\[\T(K^p):=\varprojlim_{K_p} \T(K^pK_p).\]
This is a complete semi-local noetherian $\Z_p$-algebra and $\T(K^p)/\km$ is a finite field for any maximal ideal $\km$. It acts faithfully on $\tilde{H}^1(K^p,\Z_p)$ and commutes with  the action of $\GL_2(\Q_p)\times G_{\Q}$. Moreover, there is a continuous $2$-dimensional determinant $D_S$ of $G_{\Q,S}$ valued in $\T(K^p)$  in the
sense of Chenevier \cite{Che14} such that the characteristic polynomial of $D_S(\Frob_l)$ is 
\[X^2-l^{-1}T_l+l^{-1}S_l.\]
Its twist by the inverse of the cyclotomic character is also the Eichler-Shimura congruence relation, i.e. 
\[\Frob_l^2-T_l\Frob_l+lS_l=0\]
on $\tilde{H}^1(K^p,\Z_p)$. The Poincar\'e duality on $X_{K^pK_p}(\bC)$ implies that $\T(K^p)$  acts on $\tilde{H}^0(K^p,\Z_p)$. 

 Let $\lambda:\T(K^p)\to E$ be a $\Z_p$-algebra homomorphism. We can associate a semi-simple Galois representation (unique up to conjugation)
\[\rho_{\lambda}:G_{\Q,S}\to \GL_2(E)\]
whose determinant is $\lambda\circ D_S$. Denote the $\lambda$-isotypic part by
\[\tilde{H}^1(K^p,E)[\lambda]\subseteq \tilde{H}^1(K^p,E).\]
Then it follows from the Eichler-Shimura relation that 
\[\Hom_{E[G_{\Q}]}(\rho_\lambda(-1),\tilde{H}^1(K^p,E))=\Hom_{E[G_{\Q}]}(\rho_\lambda(-1),\tilde{H}^1(K^p,E)[\lambda]).\]
Assume that $\rho_\lambda$ is absolutely irreducible. Then by the main result of \cite{BLR91}, $\tilde{H}^1(K^p,E)[\lambda]$ is also $\rho_{\lambda}(-1)$-isotypic, that is
\[\tilde{H}^1(K^p,E)[\lambda]=\Hom_{E[G_{\Q}]}(\rho_\lambda(-1),\tilde{H}^1(K^p,E)[\lambda])\otimes_E \rho_\lambda(-1).\]
\end{para}

\subsection{Fontaine operator on completed cohomology} 
\begin{para} \label{rholambda}
Let $\rho$ be as in Theorem \ref{MTCH} and $\lambda':\T(K^p)\to E$ such that $\rho_{\lambda'}\cong \rho$. The classicality part of  Theorem \ref{MTCH}  is equivalent with saying that $M_{k+1}(K^p)\otimes_{\Q_p} E[\lambda']\neq 0$. Let $\lambda:\T(K^p)\to E$ be the twist of $\lambda'$ by the cyclotomic character, i.e. $\lambda(T_l)=l^{-1}\lambda'(T_l)$, $l\notin S$. Then $\rho_\lambda(-1)\cong \rho$ and 
\[\tilde{H}^1(K^p,E)[\lambda]=\tilde{H}^1(K^p,E)[\rho].\]
Since $\rho$ is de Rham of Hodge-Tate weights $0,k$, there is a Hodge-Tate decomposition
\[\tilde{H}^1(K^p,C\otimes_{\Q_p}E)[\lambda]=\tilde{H}^1(K^p,E)[\rho]\widehat{\otimes}_{\Q_p}C=\tilde{H}^1(K^p,C\otimes_{\Q_p}E)[\lambda]_0\oplus \tilde{H}^1(K^p,C\otimes_{\Q_p}E)[\lambda]_k,\]
where $\tilde{H}^1(K^p,C\otimes_{\Q_p}E)[\lambda]_i\neq 0$ denotes the Hodge-Tate weight $i$ part for $i=0,k$. Note that $\GL_2(\Q_p)$ acts continuously on $\tilde{H}^1(K^p,C\otimes_{\Q_p}E)[\lambda]$. We will prove the following result in this subsection. Recall that $I^1_{k-1}$ was introduced in Definition \ref{I^1k}.
\end{para}

\begin{thm} \label{HT0kerF}
Let $\rho$ be as in Theorem \ref{MTCH}. Then there is a natural $\GL_2(\Q_p)$-equivariant isomorphism of $\lambda$-isotypic parts
\[\tilde{H}^1(K^p,C\otimes_{\Q_p}E)[\lambda]_0^{\la}\cong \ker I^1_{k-1}\otimes_{\Q_p} E [\lambda]\]
\end{thm}

\begin{proof}[Proof of Theorem  \ref{MTCH}]
By Theorem \ref{HT0kerF}, $\ker I^1_{k-1}\otimes_{\Q_p} E [\lambda]\neq 0$. Hence Theorem \ref{sd} implies that $M_{k+1}(K^p)\otimes_{\Q_p} E[\lambda']\cong M_{k+1}\cdot D_0^{-1}[\lambda]\neq 0$, which proves the classicality part. The finiteness part follows from Corollary \ref{knfin}. 
\end{proof}

To prove this result, we recall the comparison theorem between the completed cohomology and the  cohomology of the $\mathcal{X}_{K^p}$, which is essentially due to Scholze \cite[\S IV.2]{Sch15}.

\begin{prop} \label{BdRcomp}
There is a natural $G_{\Q_p}$-equivariant isomorphism of $B_{\dR,k}^+$-modules
\[\tilde{H}^i(K^p,B_{\dR,k}^+)\cong H^i(\Fl,\B_{\dR,k}^+),\]
where $\B_{\dR,k}^+$ denotes the truncated de Rham period sheaves introduced in Subsection \ref{dRps}.
\end{prop}

\begin{proof}
When $k=1$, this isomorphism is reduced to 
\begin{eqnarray} \label{cis}
\tilde{H}^i(K^p,C)\cong H^i(\Fl,\cO_{K^p}),
\end{eqnarray}
which was established in \cite[Corollary 4.4.3]{Pan20}. For $k\geq 2$, we need the following lemma.

\begin{lem} \label{cmap}
There is a natural $G_{\Q_p}$-equivariant homomorphism of $B_{\dR,k}^+$-modules
\[\tilde{H}^i(K^p,B_{\dR,k}^+)\to H^i(\Fl,\B_{\dR,k}^+)\]
whose reduction modulo $t$ agrees with the isomorphism \eqref{cis}.
\end{lem}
Assuming this lemma at the moment. Since $B_{\dR,k}^+$ and $\B_{\dR,k}^+$ are flat over $B_{\dR,k}^+$, an induction on $k$ shows that this map is an isomorphism.
\end{proof}

\begin{proof}[Proof of Lemma \ref{cmap}]
The most natural proof  is to use the primitive comparison theorem  on modular curves
\[H^i_{\et}(\mathcal{X}_{K^pK_p},\Z/{p^n})\otimes_{\Z_p} A^a_{\inf} \cong H^i_{\proet}(\mathcal{X}_{K^pK_p},\A^a_\inf/p^n)\]
cf. \cite[Proof of Theorem 8.4]{Sch13} and to argue as in \cite[Proof of Theorem IV.2.1]{Sch15}. Here we sketch a construction which only uses the \'etale site instead of the pro-\'etale site. 

For $X=\mathcal{X}_{K^pK_p}$ or $\mathcal{X}_{K^p}$, we denote by $W_n(\cO_X^+/p)$ the sheaf on $X_{\et}$ which is the sheafication of $U\mapsto W_n(\cO_X^+(U)/p)$, where $W_n$ denotes the ring of $n$-truncated Witt vectors. There are natural maps
\[\begin{multlined}
H^i_{\et}(\mathcal{X}_{K^pK_p},\Z/p^n)\otimes_{\Z_p}W_n(\cO_C/p)\to H^i_{\et}(\mathcal{X}_{K^pK_p},W_n(\cO_C/p))\\
\to H^i_{\et}(\mathcal{X}_{K^pK_p},W_n(\cO^+_{\mathcal{X}_{K^pK_p}}/p))
\to H^i_{\et}(\mathcal{X}_{K^p},W_n(\cO^+_{\mathcal{X}_{K^p}}/p)).
\end{multlined}\]
The direct limit over $K_p$ gives a map 
\[\tilde{H}^i(K^p,\Z/p^n)\otimes_{\Z_p} W_n(\cO_C/p)\to H^i_{\et}(\mathcal{X}_{K^p},W_n(\cO^+_{\mathcal{X}_{K^p}}/p)).\]
There is an almost isomorphism
\[H^i_{\et}(\mathcal{X}_{K^p},W_n(\cO^+_{\mathcal{X}_{K^p}}/p))^a\cong H^i(\mathcal{X}_{K^p},W_n(\cO^+_{\mathcal{X}_{K^p}}/p))^a\]
where the right hand side is computed on the analytic site. Indeed, by induction on $n$, it suffices to prove the case $n=1$, that is $H^i_{\et}(\mathcal{X}_{K^p},\cO^{+a}_{\mathcal{X}_{K^p}}/p)\cong H^i(\mathcal{X}_{K^p},\cO^{+a}_{\mathcal{X}_{K^p}}/p)$. But this is clear as on affinoid subsets, $\cO^{+a}_{\mathcal{X}_{K^p}}$ has no higher cohomology on both \'etale and analytic sites, cf. \cite[Proposition 6.14, 7.13]{Sch12}. Hence there is a map  of $W_n(\cO_C/p)^a$-modules
\[f_n:\tilde{H}^i(K^p,\Z/p^n)\otimes_{\Z_p} W_n(\cO_C/p)^a\to H^i(\mathcal{X}_{K^p},W_n(\cO^+_{\mathcal{X}_{K^p}}/p))^a.\]

Recall that $A_\inf=W(\cO_{C^\flat})$ and $\displaystyle \cO_{C^\flat}=\varprojlim_{x\mapsto x^p} \cO_C/p$. We denote elements in $\cO_{C^\flat}$ by $(\cdots,a_1,a_0),a_i\in\cO_C/p$ satisfying $a_{n+1}^p=a_n$. For $n\geq 0$, the projection map sending $(\cdots,a_1,a_0)$ to $a_n$ defines a ring homomorphism $\cO_{C^\flat}\to\cO_C/p$ and induces a quotient map
$\phi'_n:A_\inf\to W(\cO_C/p)$. Denote by $\phi_n:A_\inf\to W_n(\cO_C/p)$ the composite of $\phi'_n$ with the natural projection $W(\cO_C/p)\to W_n(\cO_C/p)$. We claim that the quotient map
\[A_\inf\to A_\inf/((\ker\theta)^k,p^m)\]
factors through $\phi_l$ for $l$ sufficiently large. It suffices to show that this map factors through $\phi'_n$.
For $a=(\cdots,a_1,a_0)\in\cO_{C^\flat}$ with $a_0=0$, it is clear that $[a]\in(\ker\theta,p)$. Now if $a_n=0$, then $a=b^{p^n}$ for some $b=(\cdots,b_1,b_0)\in\cO_{C^\flat}$ with $b_0=0$ and $[a]=[b]^{p^n}\in((\ker\theta)^{p^n},p^{n+1})$ by an easy induction on $n$.  This implies our claim. Similarly the quotient map $\A_{\inf,\mathcal{X}_{K^p}}\to \A_{\inf,\mathcal{X}_{K^p}}/((\ker\theta)^k,p^m)$ factors through some $W_n(\cO^+_{\mathcal{X}_{K^p}}/p)$. Therefore $f_n$ induces maps of $A^a_\inf/((\ker\theta)^k,p^m)$-modules
\[g_{k,m}:\tilde{H}^i(K^p,\Z/p^m)\otimes_{\Z_p}  A^a_\inf/((\ker\theta)^k,p^m)\to H^i(\mathcal{X}_{K^p}, \A^a_{\inf,\mathcal{X}_{K^p}}/((\ker\theta)^k,p^m))\]
such that $g_{k,m}\equiv g_{k,m+1}\mod p^m$ and $g_{k,m}\equiv g_{k+1,m}\mod (\ker\theta)^k$.

\begin{lem}
For $k\geq 1$,
\begin{enumerate}
\item The torsion part of $H^i(\mathcal{X}_{K^p}, \A^a_{\inf,\mathcal{X}_{K^p}}/(\ker\theta)^k)$ is killed by $p^n$ for some $n$;
\item $\displaystyle H^i(\mathcal{X}_{K^p}, \A^a_{\inf,\mathcal{X}_{K^p}}/(\ker\theta)^k)\cong\varprojlim_{m}H^i(\mathcal{X}_{K^p}, \A^a_{\inf,\mathcal{X}_{K^p}}/((\ker\theta)^k,p^m)).$
\end{enumerate}
\end{lem}

\begin{proof}
First assume $k=1$. Then $\A_{\inf,\mathcal{X}_{K^p}}/\ker\theta= \cO^+_{\mathcal{X}_{K^p}}$. By \cite[Corollary 4.4.3]{Pan20},
\[\tilde{H}^i(K^p,\cO^a_C)\cong H^i(\mathcal{X}_{K^p}, \cO^{+a}_{\mathcal{X}_{K^p}})\cong\varprojlim_m H^i(\mathcal{X}_{K^p}, \cO^{+a}_{\mathcal{X}_{K^p}}/p^m).\]
The torsion part of $\tilde{H}^i(K^p,\cO^a_C)$ is killed by some $p^n$ which implies our claims when $k=1$. In general we do induction on $k$. Note that $\ker\theta$ is principal, hence  $\A_{\inf,\mathcal{X}_{K^p}}/(\ker\theta)^k$ is torsionfree as $\A_{\inf,\mathcal{X}_{K^p}}/\ker\theta=\cO^+_{\mathcal{X}_{K^p}}$ is torsionfree. Let $\mathcal{F}_k=\A^a_{\inf,\mathcal{X}_{K^p}}/(\ker\theta)^k$ and fix a generator $\tilde{a}$ of $\ker\theta$. Then $\mathcal{F}_1=\cO^{+a}_{\mathcal{X}_{K^p}}$. There is an exact sequence
\[0\to \mathcal{F}_{k-1}\xrightarrow{\times \tilde{a}} \mathcal{F}_{k}\to \mathcal{F}_1\to 0. \]
The first claim follows by induction on $k$. To see the second claim, consider
\[0\to H^i(\mathcal{X}_{K^p},\mathcal{F}_k)/p^n \to H^i(\mathcal{X}_{K^p},\mathcal{F}_k/p^n) \to H^{i+1}(\mathcal{X}_{K^p},\mathcal{F}_k)[p^n]\]
which forms a projective system when $n$ varies.
The transition map $H^{i+1}(\mathcal{X}_{K^p},\mathcal{F}_k)[p^{n+1}]\to H^{i+1}(\mathcal{X}_{K^p},\mathcal{F}_k)[p^n]$ is multiplication by $p$. Hence $\displaystyle \varprojlim_{n} H^{i+1}(\mathcal{X}_{K^p},\mathcal{F}_k)[p^n] =0$ by what we have proved. In particular, $\displaystyle \varprojlim_{n} H^i(\mathcal{X}_{K^p},\mathcal{F}_k)/p^n \cong \varprojlim_{n} H^i(\mathcal{X}_{K^p},\mathcal{F}_k/p^n) $ with torsion parts annihilated by some $p^n$. From this and \cite[Lemma 4.4.4]{Pan20}, we deduce that
\[ H^i(\mathcal{X}_{K^p},\mathcal{F}_k)\cong \varprojlim_{n} H^i(\mathcal{X}_{K^p},\mathcal{F}_k)/p^n \cong \varprojlim_{n} H^i(\mathcal{X}_{K^p},\mathcal{F}_k/p^n)\]
in exactly the same way as the proof of  \cite[Corollary 4.4.3]{Pan20}.
\end{proof}

Now taking the inverse limit of $g_{k,m}$ with respect to $m$, we get a map
\[\varprojlim_m\tilde{H}^i(K^p,\Z/p^m)\otimes_{\Z_p}  A^a_\inf/((\ker\theta)^k,p^m)\to H^i(\mathcal{X}_{K^p}, \A^a_{\inf,\mathcal{X}_{K^p}}/(\ker\theta)^k)\]
which is clearly $G_{\Q_p}$-equivariant and $A^a_\inf$-linear. When $k=1$,  this map agrees with the isomorphism \eqref{cis}. In general, this gives  the desired map in Lemma \ref{cmap} after inverting $p$ modulo the following lemma.
\end{proof}

\begin{lem}
There are natural isomorphisms
\begin{enumerate}
\item $\displaystyle \tilde{H}^i(K^p, A_\inf/(\ker\theta)^k)\cong \varprojlim_m\tilde{H}^i(K^p,\Z/p^m)\otimes_{\Z_p}  A_\inf/((\ker\theta)^k,p^m)$.
\item $H^i(\mathcal{X}_{K^p}, \A^a_{\inf,\mathcal{X}_{K^p}}/(\ker\theta)^k)\cong H^i(\Fl,\pi_{\HT*}\A^a_{\inf,\mathcal{X}_{K^p}}/(\ker\theta)^k )$.
\end{enumerate}
\end{lem}

\begin{proof}
For the first claim, we prove the following stronger result: 
\[\tilde{H}^i(K^p,M)\cong \varprojlim_m \tilde{H}^i(K^p,\Z/p^m)\otimes_{\Z_p} M\]
for any $p$-adically complete torsion-free $\Z_p$-module $M$. Note that there exists a complex $\tilde{S}^\bullet$ of $p$-adically complete torsion-free $\Z_p$ modules such that $H^i(\tilde{S}^\bullet)\cong \tilde{H}^i(K^p,\Z_p)$ and   $H^i(\tilde{S}^\bullet/p^m)\cong \tilde{H}^i(K^p,\Z/p^m)$, cf. \cite[p. 27, 28]{Eme06}. Hence there is an projective system of short exact sequences
\[0\to \tilde{H}^i(K^p,\Z_p)/p^m\to \tilde{H}^i(K^p,\Z/p^m) \to \tilde{H}^{i+1}(K^p,\Z_p)[p^m],\]
and the transition map $\tilde{H}^{i+1}(K^p,\Z_p)[p^{m+1}]\to \tilde{H}^{i+1}(K^p,\Z_p)[p^m]$ is multiplication by $p$. From this we easily deduce our claim by taking the tensor product with $M$. 

For the second claim, it suffices to show that $R^i\pi_{\HT*}\A^a_{\inf,\mathcal{X}_{K^p}}/(\ker\theta)^k=0$ when $i\geq 1$. This follows as for $U\subseteq\mathfrak{B}$, the preimage $\pi_{\HT}^{-1}(U)$ is affinoid perfectoid, hence we have $H^i(\pi_{\HT}^{-1}(U),\A^a_{\inf,\mathcal{X}_{K^p}}/(\ker\theta)^k)=0,i\geq 1$ by induction on $k$.
\end{proof}

We need one more lemma for proving Theorem \ref{HT0kerF}.
\begin{lem} \label{dRcompne}
There is a natural $G_{\Q_p}$-equivariant isomorphism of $B_{\dR,k+1}^+$-modules
\[\tilde{H}^1(K^p,B_{\dR,k+1}^+\otimes_{\Q_p}E)^{\la}[\lambda]\cong 
H^1(\Fl,\B_{\dR,k+1}^{+,\la,\tilde\chi_k})\otimes_{\Q_p} E[\lambda].\]
\end{lem}

\begin{proof}
Recall that $Z$ denotes the center of $U(\mathfrak{gl}_2(\Q_p))$ and  $\tilde\chi_k:Z\to\Q_p$ denotes the infinitesimal character  of the $(k-1)$-th symmetric power of the dual of the standard representation, cf. Section \ref{tilchik}. Since $\rho_\lambda=\rho(1)$ has Hodge-Tate weights $-1,k-1$, it follows that the subspace of $\GL_2(\Q_p)$-locally analytic vectors $\tilde{H}^1(K^p,E)^{\la}[\lambda]$ has infinitesimal character $\tilde\chi_k$ by \cite[Proposition 6.1.5]{Pan20}. Hence 
\[\tilde{H}^1(K^p,B_{\dR,k+1}^+\otimes_{\Q_p}E)^{\la}[\lambda]\cong\tilde{H}^1(K^p,B_{\dR,k+1}^+\otimes_{\Q_p}E)^{\la,\tilde\chi_k}[\lambda]\cong H^1(\Fl,\B_{\dR,k+1}^{+})^{\la,\tilde\chi_k}\otimes_{\Q_p} E[\lambda],\]
where the second isomorphism follows from Proposition \ref{BdRcomp}. It remains to show
\[H^1(\Fl,\B_{\dR,k+1}^{+})^{\la,\tilde\chi_k}\otimes_{\Q_p} E[\lambda]\cong H^1(\Fl,\B_{\dR,k+1}^{+,\la,\tilde\chi_k})\otimes_{\Q_p} E[\lambda].\]
This is clear when $k>1$ by Proposition \ref{compinfch}. When $k=1$, note that
\[H^0(\Fl,\B_{\dR,2}^+)^{\la}\cong \tilde{H}^0(K^p,B_{\dR,2}^+)^{\la}=\tilde{H}^0(K^p,\Q_p)^{\la}\widehat\otimes_{\Q_p}B_{\dR,2}^+.\]
Again by Proposition \ref{compinfch}, it suffices to show that the localization $\tilde{H}^0(K^p,E)_{\mathfrak{p}_\lambda}=0$ where $\mathfrak{p}_\lambda=\ker \lambda$. But this follows from our assumption that $\rho_\lambda$ is irreducible. Indeed we can find an element $g\in [G_{\Q},G_\Q]$ such that $\tr(\rho_\lambda(g))\neq 2$.
 Let $X^2-a_gX+1$ be the characteristic polynomial of $D_S(g)$. Then $a_g-2\notin\mathfrak{p}_\lambda$ annihilates $\tilde{H}^0(K^p,E)$ because the action of $G_\Q$ on $\tilde{H}^0(K^p,E)$ factors through its abelianization $G_{\Q}^{ab}$. 
\end{proof}

\begin{proof}[Proof of Theorem \ref{HT0kerF}]
The $\gr^0$-part of the isomorphism in Lemma \ref{dRcompne} gives
\[\tilde{H}^1(K^p, C\otimes_{\Q_p}E)[\lambda]^{\la}\cong H^1(\Fl,\cO^{\la,\tilde\chi_k}_{K^p})\otimes_{\Q_p} E[\lambda].\]
Recall that $H^1(\Fl,\cO^{\la,\tilde\chi_k}_{K^p})=H^1(\Fl,\cO^{\la,(1-k,0)}_{K^p})\oplus H^1(\Fl,\cO^{\la,(1,-k)}_{K^p})$  with $H^1(\Fl,\cO^{\la,(1-k,0)}_{K^p})$ being the Hodge-Tate weight zero part, cf. Section \ref{tilchik}. It is enough to show that 
\[I^1_{k-1}\otimes 1: H^1(\Fl,\cO^{\la,(1-k,0)}_{K^p})\otimes E\to H^1(\Fl,\cO^{\la,(1,-k)}_{K^p})\otimes E(k)\]
 is zero on $H^1(\Fl,\cO^{\la,(1-k,0)}_{K^p})\otimes E[\lambda]\cong\tilde{H}^1(K^p, C\otimes_{\Q_p}E)[\lambda]^\la_0$.

Let $N=\Hom_{E[G_\Q]}(\rho,\tilde{H}^1(K^p,E)^{\la})$ and $W$ be
\[\tilde{H}^1(K^p,B_{\dR,k+1}^+\otimes_{\Q_p}E)^{\la}[\lambda]\cong ( B_{\dR,k+1}^+\otimes_{\Q_p}\rho)\widehat\otimes_E N.\]
We can apply the construction in Section \ref{NWLB}  to $W$ and get the Fontaine operator
\[N_W:W_{0,0}=\tilde{H}^1(K^p,C\otimes_{\Q_p} E)[\lambda]^{\la}_0\to W_{0,-k}(k)=\tilde{H}^1(K^p,C\otimes_{\Q_p} E)[\lambda]^{\la}_k(k)\]
Since $\rho$ is \textit{de Rham} by our assumption, Fontaine's result (cf. Theorem \ref{FondR}) implies that $N_W=0$. On the other hand,  it follows from Corollary \ref{I1Fono} that the Fontaine operator on $H^1(\Fl,\B_{\dR,k+1}^{+,\la,\tilde\chi_k})\otimes_{\Q_p} E$ agrees with $I^1_{k-1}\otimes 1$ up to a unit. Thus under the isomorphism in Lemma \ref{dRcompne}, it follows from the functorial property of the Fontaine operator that 
\[I^1_{k-1}\otimes 1|_{\tilde{H}^1(K^p,C\otimes_{\Q_p} E)[\lambda]^{\la}_0}=N_W=0.\]
\end{proof}

\subsection{Relation with the \texorpdfstring{$p$}{Lg}-adic local Langlands correspondence for  \texorpdfstring{$\GL_2(\Q_p)$}{Lg}} \label{rbpLLC}
\begin{para} \label{noraf}
Let $\rho$ be as in Theorem \ref{MTCH} and keep the same notation in Section \ref{rholambda}. Combining Theorem \ref{HT0kerF} with results in Subsection \ref{Sd}, we obtain a geometric description of the $\GL_2(\Q_p)$-locally analytic representation $\tilde{H}^1(K^p,E)[\lambda]^{\la}$. When the  automorphic representation corresponding to $\rho$ is supercuspidal at $p$, our result is a special case of the Breuil-Strauch conjecture, which was fully solved by Dospinescu-Le Bras in \cite{DLB17}. See Remark \ref{scla}. When  the automorphic representation is a principal series at $p$, our result is a special case of a conjecture of Berger-Breuil and Emerton, solved by Liu-Xie-Zhang and Colmez. See Remark \ref{psla}. In fact, our work shows that both cases can be formulated in a uniform way, cf. Remark \ref{scpsConj} below for more details. We introduce some notations first.

For simplicity, we fix an embedding $\tau:E\to C$ and denote the composite map $\tau\circ\lambda$ by 
\[\lambda_\tau:\T(K^p)\to C.\]
Consider $\displaystyle M_{k+2}:=\varinjlim_{K\subseteq\GL_2(\A_f)} M_{k+2}(K)=\varinjlim_{K^p\subseteq\GL_2(\A^p_f)} M_{k+2}(K^p)$. It admits a natural smooth action of $\GL_2(\A_f)$ and $M_{k+2}(K^p)=M_{k+2}^{K^p}$. We have just  shown that $M_{k+2}(K^p)\otimes D_0^{-1}[\lambda_\tau]\neq 0$. (Recall that $\otimes D_0$ is the same as twisting with  $|\cdot|^{-1}\circ \det$, cf. Section \ref{D0a}.) Hence $\lambda_\tau$ corresponds to an irreducible representation of $\GL_2(\A_f)$
\[\pi^\infty=\otimes'\pi_l\subseteq M_{k+2}\otimes D_0^{-1}\]
such that $(\pi^{\infty})^{K^p}=M_{k+2}(K^p)[\lambda_\tau]$.

Now assume that $\pi_p$ is special or supercuspidal. In particular $\pi^\infty$ can be transferred to $(D\otimes\A_f)^\times$ by the Jacquet-Langlands correspondence. Recall that $\mathcal{A}_{D,(k,0)}$ was introduced action in Definition \ref{AFD}. Then there is  a $(D\otimes\A_f)^\times$-equivariant embedding
\[\pi'{}^{\infty}:=\pi^{\infty,p}\otimes_C \pi'_p\to \mathcal{A}_{D,(k,0)}\]
for some irreducible smooth representation $\pi'_p$ (corresponding to $\pi_p$ under the local Jacquet-Langlands correspondence in suitable normalization) of $D_p^\times$. Here $\pi^{\infty,p}=\otimes'_{l\neq p}\pi_l$.
\end{para}

\begin{thm} \label{scch}
Suppose $\pi_p$ is supercuspidal. There is a $\GL_2(\Q_p)^0$-equivariant isomorphism
between $\tilde{H}^1(K^p,C)[\lambda_\tau]_0^{\la}$ and 
\[(\pi^{\infty,p})^{K^p}\otimes_C  \left[\left(H^1(\Fl,j_{!}  \omega^{-k+1,D_p^\times-\sm}_{\Dr})\otimes_C \pi'_p\right)^{\cO_{D_p}^\times}\cdot{\varepsilon'_p}^{-k+1}/ \Sym^{k-1} V^*\otimes_{\Q_p} \pi_p\right].\]
\end{thm}

\begin{rem} \label{GL2oGL2}
As explained in Remark \ref{GL20}, this isomorphism is $\GL_2(\Q_p)$-equivariant by writing 
$(H^1(\Fl,j_{!}  \omega^{-k+1,D_p^\times-\sm}_{\Dr})\otimes_C \pi'_p)^{\cO_{D_p}^\times}$ as $\displaystyle (\bigoplus_{i\in\Z} H^1(\Fl,j_{!}  \omega^{-k+1,D_p^\times-\sm}_{\Dr})\otimes_C \pi'_p)^{{D_p}^\times}$.
\end{rem}

\begin{proof}[Proof of Theorem \ref{scch}]
By Theorem \ref{HT0kerF} and Corollary \ref{BScon}, $\tilde{H}^1(K^p,C)[\lambda_\tau]_0^{\la}$ is isomorphic to
\[  \left[(\pi^{\infty,p})^{K^p}\otimes \left(H^1(\Fl,j_{!}  \omega^{-k+1,D_p^\times-\sm}_{\Dr})\otimes_C \pi'_p\right)^{\cO_{D_p}^\times}\cdot{\varepsilon'_p}^{-k+1}\right]/  
\left[(\pi^{\infty,p})^{K^p}\otimes\Sym^{k-1} V^*\otimes \pi_p\right].\]
It remains to move the tensor product with $(\pi^{\infty,p})^{K^p}$ outside of the quotient. This can be done by considering new vectors of $\pi_l^{K_l},l\neq p$.
\end{proof}

\begin{rem}
By Corollary \ref{BScon}, there is a  $\GL_2(\Q_p)^0$-equivariant exact sequence
\[\begin{multlined}0\to   (\pi^{\infty})^{K^p}\otimes_C  \Sym^{k-1} V^*\to \tilde{H}^1(K^p,C)[\lambda_\tau]_0^{\la} 
\to \\
		 (H^1(\Fl,j_{!}  \omega^{k+1,D_p^\times-\sm}_{\Dr})\otimes \pi'{}^{\infty})^{K^p\times\cO_{D_p}^\times}\otimes \det{}^{k+1}\cdot{\varepsilon'_p}^{-k}\to 0.\end{multlined}\]
which can be upgraded to a $\GL_2(\Q_p)$-equivariant sequence.
\end{rem}

\begin{rem}
Since $\rho$ is $2$-dimensional, $\rho\otimes_{\Q_p} C\cong (\rho\otimes_{\Q_p} C)_0^{\oplus 2}$ as $C$-vector spaces. Thus there is a non-canonical $\GL_2(\Q_p)$-equivariant isomorphism
\[\left(\Hom_{E[G_{\Q}]}(\rho,\tilde{H}^1(K^p,E)^{\la})\right)\widehat\otimes_E C\cong\tilde{H}^1(K^p,C)[\lambda_\tau]_0^{\la}.\]
Hence Theorem \ref{scch} also gives a description of $\tilde{H}^1(K^p,C)[\lambda_\tau]^{\la}$ as a representation of $\GL_2(\Q_p)$. On the other hand, $\tilde{H}^1(K^p,E)^{\la}$ is an admissible locally analytic $E$-representation of $\GL_2(\Q_p)$  in the sense of Schneider-Teitelbaum \cite{ST02}. We remind the readers  that in  \cite{ST02}, the coefficient field is assumed to be spherically  complete. In particular $C$ does not satisfy their assumption. Theorem \ref{scch} nonetheless implies certain admissibility of $(H^1(\Fl,j_{!}  \omega^{-k+1,D_p^\times-\sm}_{\Dr})\otimes_C \pi'_p)^{\cO_{D_p}^\times}$ as a locally analytic representation of $\GL_2(\Q_p)$ over $C$.
\end{rem}

\begin{rem} \label{scla}
We explain the connection between Theorem \ref{scch} and the Breuil-Strauch conjecture in this case. For simplicity we restrict ourselves to the case $k=1$. Note that $\pi'_p$ is fixed by $1+p^n\cO_{D_p}$ for some $n\geq 0$. Hence
\[(H^1(\Fl,j_{!}  \omega^{0,D_p^\times-\sm}_{\Dr})\otimes_C \pi'_p)^{\cO_{D_p}^\times}
=\left[H^1\left(\Fl,j_!\pi^{(0)}_{\Dr,n*} \cO_{\mathcal{M}_{\Dr,n}^{(0)}}\right)\otimes_C \pi'_p\right]^{\cO_{D_p}^\times}\]
where $\pi^{(0)}_{\Dr,n}:\mathcal{M}_{\Dr,n}^{(0)}\to \Omega$ denotes the projection map, cf. Subsection \ref{Drinfty}. As explained in Section \ref{H1dRc}, there is an exact sequence
\[0\to H^1_{\dR,c}(\mathcal{M}_{\Dr,n}^{(0)}) \to H^1(\mathcal{M}_{\Dr,n}^{(0)},\cO_{\mathcal{M}_{\Dr,n}^{(0)}})\xrightarrow{d_{\dR}} H^1(\mathcal{M}_{\Dr,n}^{(0)},\Omega^1_{\mathcal{M}_{\Dr,n}^{(0)}}).\]
Taking the $\lambda$-isotypic part in Theorem \ref{ijdR}, we deduce that 
\[\left(H^1_{\dR,c}(\mathcal{M}_{\Dr,n}^{(0)})\otimes_C \pi'_p\right)^{\cO_{D_p}^\times}\cong D_{\dR}(\rho|_{G_{\Q_p}})\otimes_E \pi_p.\]
See also \cite[Th\'{e}or\`{e}me 1.4]{DLB17}. Here $D_{\dR}(\rho|_{G_{\Q_p}})=(\rho\otimes_{\Q_p}B_{\dR})^{G_{\Q_p}}$.  The inclusion map 
$\pi_p\subseteq \left[H^1\left(\Fl,j_!\pi^{(0)}_{\Dr,n*} \cO_{\mathcal{M}_{\Dr,n}^{(0)}}\right)\otimes_C \pi'_p\right]^{\cO_{D_p}^\times}$ in Theorem \ref{scch}
is nothing but 
\[\pi_p\cong\Fil^1 D_{\dR}(\rho|_{G_{\Q_p}})\otimes_E \pi_p\subseteq D_{\dR}(\rho|_{G_{\Q_p}})\otimes_E \pi_p\cong \left(H^1_{\dR,c}(\mathcal{M}_{\Dr,n}^{(0)})\otimes_C \pi'_p\right)^{\cO_{D_p}^\times}.\]

On the other hand, Emerton's local-global compatibility result (\cite[Corollary 6.3.6]{Pan20}) implies that 
\[\tilde{H}^1(K^p,C)[\lambda_\tau]_0^{\la}\cong (\pi^{\infty,p})^{K^p}\otimes_E \Pi(\rho|_{G_{\Q_p}})^{\la}\]
where $ \Pi(\rho|_{G_{\Q_p}})$ denotes the unitary $E$-Banach space representation of $\GL_2(\Q_p)$ attached to $\rho|_{G_{\Q_p}}$ under the $p$-adic local Langlands correspondence for $\GL_2(\Q_p)$. (Note that $\rho|_{G_{\Q_p}}$ is absolutely irreducible as $\pi_p$ is supercuspidal.) Therefore Theorem \ref{scch} is equivalent with 
\[\Pi(\rho|_{G_{\Q_p}})^{\la}\widehat\otimes_E C\cong \left[H^1\left(\Fl,j_!\pi^{(0)}_{\Dr,n*} \cO_{\mathcal{M}_{\Dr,n}^{(0)}}\right)\otimes_C \pi'_p\right]^{\cO_{D_p}^\times}/ \Fil^1 D_{\dR}(\rho|_{G_{\Q_p}})\otimes_E \pi_p.\]
This is exactly what was conjectured by Breuil-Strauch for $\rho|_{G_{\Q_p}}$. In general, for any de Rham representation $\rho_p:G_{\Q_p}\to\GL_2(E)$ of Hodge-Tate weights $0,1$ such that $\rho_p$ corresponds to $\pi_p$ under the usual local Langlands correspondence, then there is an isomorphism $D_{\dR}(\rho_p)\cong D_{\dR}(\rho|_{G_{\Q_p}})$ unique up to $E^\times$, and $\rho_p$ is completely determined by the position of $\Fil^1 D_{\dR}(\rho_p)$. Breuil-Strauch conjectured that there is a $\GL_2(\Q_p)$-isomorphism
\[\Pi(\rho_p)^{\la}\widehat\otimes_E C\cong \left[H^1\left(\Fl,j_!\pi^{(0)}_{\Dr,n*} \cO_{\mathcal{M}_{\Dr,n}^{(0)}}\right)\otimes_C \pi'_p\right]^{\cO_{D_p}^\times}/ \Fil^1 D_{\dR}(\rho_p)\otimes_E \pi_p.\]
The full conjecture was proved by Dospinescu-Le Bras \cite[Th\'eor\`em 1.4]{DLB17}, taking into account of the Serre duality
\[\Hom_{C}^{\cont} \left(H^1(\Fl,j_!\pi^{(0)}_{\Dr,n*}  \cO_{\mathcal{M}_{\Dr,n}^{(0)}}),C \right)\cong H^0(\mathcal{M}_{\Dr,n}^{(0)}, \Omega^1_{\mathcal{M}_{\Dr,n}^{(0)}}).\]
\end{rem}

It follows from  Theorem \ref{HT0kerF} and the second part of Corollary \ref{BScon} that there is also a description of $\tilde{H}^1(K^p,C)[\lambda_\tau]_0^{\la}$ when $\pi_p$ is a principal series.

\begin{thm} \label{PSla}
Suppose $\pi_p$ is a principal series. There is a $\GL_2(\Q_p)$-equivariant isomorphism
\[\tilde{H}^1(K^p,C)[\lambda_\tau]_0^{\la} \cong \Ind_B^{\GL_2(\Q_p)} H^1_{\rig}(\Ig(K^p),\Sym^{k-1}) \cdot \varepsilon^{-k+1}{e'_2}^{k-1}[\lambda]/(\pi^{\infty})^{K^p}\otimes \Sym^{k-1} V^*,\]
where $\Ind_B^{\GL_2(\Q_p)}$ denotes the locally analytic induction.
\end{thm}

\begin{rem}
By Corollary \ref{BScon}, there is also a  $\GL_2(\Q_p)$-equivariant exact sequence
\[\begin{multlined}0\to   (\pi^{\infty})^{K^p}\otimes_C  \Sym^{k-1} V^*\to \tilde{H}^1(K^p,C)[\lambda_\tau]_0^{\la} 
\to \\
		\Ind_B^{\GL_2(\Q_p)} H^1_{\rig}(\Ig(K^p),\Sym^{k-1}) \cdot \varepsilon^{-k+1}{e'_1}^{k}{e'_2}^{-1}[\lambda]\to 0.\end{multlined}\]

\end{rem}

\begin{rem} \label{psla}
Again we assume $k=1$ here for simplicity.  But the discussion below works for general $k$. In this case, $H^1_{\rig}(\Ig(K^p))$ can be identified with the log rigid cohomology of the tower of Igusa curves. The comparison theorem between the log rigid cohomology and log crystalline cohomology implies that there is an isomorphism $H^1_{\rig}(\Ig(K^p)) [\lambda]\cong (\pi^{\infty,p})^{K^p}\otimes_E D_{\mathrm{cris}}(\rho|_{G_{\Q_p}})$, where  $D_{\mathrm{cris}}(\rho|_{G_{\Q_p}}):=(\rho\otimes_{\Q_p} B_{\mathrm{cris}})^{G_{\Q_p(\mu_{p^n})}}$ for some sufficiently large $n$  and carries an action of the Weil group $W_{\Q_p}$ by Fontaine's recipe. Using the congruence relation of $B\times W_{\Q_p}^{ab}$ on the Igusa cruves (cf. \cite[\S 1.6.4]{Car86}), we can equip $D_{\mathrm{cris}}(\rho|_{G_{\Q_p}})$ with a $B$-action and this isomorphism becomes $B$-equivariant. Hence 
\[\tilde{H}^1(K^p,C)[\lambda_\tau]_0^{\la} \cong (\pi^{\infty,p})^{K^p}\otimes \Ind_B^{\GL_2(\Q_p)} D_{\mathrm{cris}}(\rho|_{G_{\Q_p}})/\pi_p,\]
where again the embedding $\pi_p\to \Ind_B^{\GL_2(\Q_p)} D_{\mathrm{cris}}(\rho|_{G_{\Q_p}})$ comes from the $\Fil^1$.
If we further assume $\rho|_{G_{\Q_p}}$ is absolutely irreducible, then Emerton's local-global compatibility result implies that
\[\Pi(\rho|_{G_{\Q_p}})^{\la}\widehat\otimes_{E} C\cong \Ind_B^{\GL_2(\Q_p)} D_{\mathrm{cris}}(\rho|_{G_{\Q_p}})/\pi_p.\]
The local form of this isomorphism was conjectured by Berger-Breuil in \cite[Conjecture 5.3.7]{BB10} and by Emerton \cite[Conjecture 6.7.3]{Eme06C}, and was proved by Liu-Xie-Zhang \cite{LXZ12} and Colmez \cite{Col14}.
\end{rem}

\begin{rem} \label{scpsConj}
As noted in Remark \ref{DRstn}, we have
\[\tilde{H}^1(K^p,C)[\lambda_\tau]_0^{\la} \cong \bH^1(DR_k)[\lambda_\tau]/ \Fil^1 \bH^1(DR_k)[\lambda_\tau]\]
which gives the descriptions of $\tilde{H}^1(K^p,C)[\lambda_\tau]_0^{\la}$ in Theorem \ref{scch} and \ref{PSla}, when $\pi_p$ is either supercuspidal or a principal series. This suggests that the Breuil-Strauch conjecture in the supercuspidal case and the conjecture of Berger-Breuil and Emerton in the principal series case can be formulated in a uniform way. We believe this is also true when $\pi_p$ is special.
\end{rem}

\begin{rem}\label{stn}
Finally we remark on the case when $\pi_p$ is special. Again let me assume $k=1$ for simplicity. By Theorem \ref{dec1}, $\ker I^1_0[\tilde\lambda_\tau]$ has a natural increasing filtration $\Fil_\bullet$ with 
\[\gr_n \ker I^1_0[\tilde\lambda_\tau]=\left\{
	\begin{array}{lll}
		H^1(\cO^{\sm}_{K^p})[\lambda]\cong (\pi^{p,\infty})^{K^p}\otimes_{C} \pi_p , &n=1&\\
		 (\pi^{p,\infty})^{K^p}\otimes_C (H^1(j_{!} \omega^{2,D_p^\times-\sm}_{\Dr})\otimes_C \pi'_p)^{\cO_{D_p}^\times}\otimes\det, &n=2&\\
				 \Ind_B^{\GL_2(\Q_p)}  H^1_{\rig}(\Ig(K^p))[\lambda]\cdot {e'_1}{e'_2}^{-1} , &n=3&
	\end{array}.\right.\]
If $\rho|_{G_{\Q_p}}$ is absolutely irreducible, then Emerton's local-global compatibility result implies that $\tilde{H}^1(K^p,C)[\lambda_\tau]_0^{\la}\cong (\pi^{p,\infty})^{K^p}\widehat\otimes_{E} \Pi(\rho|_{G_{\Q_p}})^{\la}$. A description of $ \Pi(\rho|_{G_{\Q_p}})^{\la}$ was conjectured by Emerton  \cite[Conjecture 6.7.7]{Eme06C} and proved in \cite{LXZ12,Col14}. By Theorem \ref{HT0kerF}, there is an inclusion $\tilde{H}^1(K^p,C)[\lambda_\tau]_0^{\la}\subseteq \ker I^1_0[\tilde\lambda_\tau]$. We claim that it is actually an equality:  a careful analysis using Serre duality shows that  $\Fil_2\ker I^1_0[\tilde\lambda_\tau]$ is essentially $(\pi^{p,\infty})^{K^p}\widehat\otimes_{E}\Sigma(2,\mathcal{L})$, where $\Sigma(2,\mathcal{L})$ was introduced in \cite{Bre04}. See  the construction of Breuil and Morita's duality in \S2, \S3 of the reference, especially Remarque 3.5.7. On the other hand, it can be shown that $\gr_3\ker I^1_0[\tilde\lambda_\tau]$ essentially agrees with the locally analytic induction in the quotient of the sequence in \cite[Conjecture 6.7.7]{Eme06C} (after taking the completed tensor product with $(\pi^{p,\infty})^{K^p}$). Hence It follows that $\ker I^1_0[\tilde\lambda_\tau]=\tilde{H}^1(K^p,C)[\lambda_\tau]_0^{\la}$, i.e. $\ker I^1_0[\tilde\lambda_\tau]$ is the eigenspace of $\lambda$.
\end{rem}

\bibliographystyle{amsalpha}

\bibliography{bib}

\end{document}